\documentclass[11pt]{article}%
\usepackage[a4paper,margin=3cm]{geometry}%
\usepackage[pagebackref]{hyperref}%
\usepackage[T1]{fontenc}%
\usepackage[utf8]{inputenc}%
\usepackage{lmodern,mathtools}
\usepackage{amsmath,amsfonts,amssymb,amsthm}
\usepackage{amsmath}
\usepackage{tikz}
\usepackage{dsfont}
\usepackage{appendix}

\newcommand\tab[1][0.5cm]{\hspace*{#1}}
\newtheorem{theorem}{Theorem}[section]%
\newtheorem{definition}[theorem]{Definition}
\newtheorem{corollary}[theorem]{Corollary}%
\newtheorem{prop}[theorem]{Proposition}%
\newtheorem{lemma}[theorem]{Lemma}%
\newtheorem{Remark}[theorem]{Remark}%
\newtheorem{example}[theorem]{Example}
\let\ensembleNombre\mathbb
\newcommand*\mes{\mathcal P(\R)}

\newcommand*\N{\ensembleNombre{N}}
\newcommand*\Z{\ensembleNombre{Z}}

\newcommand*\R{\ensembleNombre{R}}
\newcommand*\C{\ensembleNombre{C}}

\newcommand*\dis{\displaystyle}
\newcommand*\E{\mathbb{E}}
\newcommand*\mP{\mathbb{P}}

\newcommand*\conv{\underset{n\to\infty}{\lim}}
\newcommand*\limite{\underset{\varepsilon\to 0}{\lim}}
\newcommand*\convf{\underset{n\to+\infty}{\longrightarrow}}

\newcommand*\convlaw{\overunderset{\mathcal L}{n\to+\infty}{\longrightarrow}}
\newcommand*\equalaw{\underset{law}{=}}
\DeclareMathOperator{\Var}{Var} 
\DeclareMathOperator{\diag}{diag} 
\DeclareMathOperator{\Tr}{Tr} 
\DeclareMathOperator{\leb}{Leb} 
 
\newcommand{\SPAN}{\mathrm{span}}

\renewcommand{\Im}{\mathfrak{Im}}
\renewcommand{\Re}{\mathfrak{Re}}

\makeatletter

\numberwithin{equation}{section}

\usepackage{titling}

\pretitle{\begin{center}\Huge\bfseries}
\posttitle{\par\end{center}\vskip 0.5em}
\preauthor{\begin{center}\ttfamily}
\postauthor{\end{center}}
\predate{\par\large\centering}
\postdate{\par}

\title{From random matrices to systems of particles in interaction}
\author{Valentin Pesce \footnote[1]{address: CMAP, Ecole Polytechnique, UMR 7641, 91120 Palaiseau, France
\newline
\url{mailto:valentin.pesce@polytechnique.edu}}}
\date{\today}
\providecommand{\keywords}[1]{\textbf{\textit{Keywords---}} #1}

\begin{document}

\maketitle
\thispagestyle{empty}

\begin{abstract}
The goal of these expository notes is to give an introduction to random matrices for non-specialist of this topic focusing on the link between random matrices and systems of particles in interaction. We first recall some general results about the random matrix theory that create a link between random matrices and systems of particles through the knowledge of the law of the eigenvalues of certain random matrices models. We next focus on a continuous in time approach of random matrices called the Dyson Brownian motion. We detail some general methods to study the existence of system of particles in singular interaction and the existence of a mean field limit for these systems of particles. Finally, we present the main result of large deviations when studying the eigenvalues of random matrices. This method is based on the fact that the eigenvalues of certain models of random matrices can be viewed as log gases in dimension 1 or 2. 
\end{abstract}

\newpage 

{\footnotesize\tableofcontents}

\begin{table}
  \footnotesize
  \begin{center}
    \begin{tabular}[c]{|r|l|} \hline
    $x:=y$
    & $x$ is equal to $y$ as a definition \\
    $\N$
    & set of non negative integers $\{0,1,...\}$ \\
    $\N_{\ge 1}$
    & set of positive integers $\{1,2,...\}$ \\
     $\mathds 1_A$
    & indicator function of the set $A$ \\
    $A[X]$
    & set of polynomials with coefficients in the ring $A$\\
    $M_{n,m}(A)$
     & set of matrices of size $n\times m$ with coefficients in the ring $A$\\
     $M_{N}(A)$
     & set of matrices of size $N\times N$ with coefficients in the ring $A$\\
     $\Delta_{i,j}$
     & matrix with a 1 in $(i,j)$ and 0 otherwise\\
     $\Tr,\,\det$
     & trace and determinant of a matrix\\
     $\det[a_{i,j}]_{1\le i,j\le N}$
     & determinant of the matrix $A=(a_{i,j})_{1\le i,j\le N}$\\
     $\overline z, \overline A$
     & conjugate of a complex number $z$, conjugate matrix of matrix $A$\\
     $A^{-1},\, A^T,\, A^*$
     & inverse matrix of matrix $A$, transpose matrix of matrix $A$, conjugate transpose matrix of matrix $A$\\
     $I_N$
     & identity matrix of size $N\times N$ \\
     $\chi_A(X)$
     & characteristic polynomial of $A$\\
     $\mathcal H_N(\C)$
     & set of hermitian complex matrices of size $N\times N$\\
     $S_N(\R)$
     & set of symmetric real matrices of size $N\times N$\\
	$GL_N(\C),\,\mathcal U_N(\C)$
	& set of inversible matrices of $M_N(\C)$, set of unitary matrices of $M_N(\C)$\\    
	$\mathbb H, \,\C-\R$
	& set of complex numbers with positive imaginary part, set of complex numbers that are not real\\
	$\langle u|v\rangle,\,\, ||.||_2$
	& canonical scalar product between two vectors of $\R^N$ or  $\C^N$ and its associated norm\\
	$C^k(I,E)$
	& set of functions from an interval $I$ to a Banach space $E$ with $k$ continuous derivatives\\
	$L^k(E,\mu)$
	& set of complex valued functions such that $|f|^k$ is integrable on $E$ for the measure $\mu$\\
	$d\phi_a$
     & Frechet derivative of $\phi$ in $a$\\
     $\partial_i f, \partial f/\partial x_i$
     & partial derivative according to the $i^{th}$ vector of the canonical basis of the function $f:\R^N\to \R$\\
     $\partial_t f$
     & partial derivative of $f: I\times \R^N$ with $I$ an interval, with respect to the time variable\\
     $\nabla f,\, \nabla^2 f$
     & gradient vector of $f:\R^N\to \R$, Hessian matrix of $f:\R^N\to \R$\\
     $\Delta f$
     & Laplacian of $f$ (define as the trace of the Hessian matrix of $f$)\\
     $\mathcal F(f)$
     & Fourier transform of the function $f: \R^N\to \C$\\
     $\mes,\, \mathcal P(A)$
     & set of probability measures on $\R$, set of probability measures with support included in $A\subset\R$\\
 
     $\leb$
     & Lebesgue measure on $\R^N$\\
     $\mu(dx)\text{ or } d\mu(x)$
     & integration with respect to the measure $\mu$\\
     $supp(f)$, $supp(\mu)$
     & support of the function $f$ or the measure $\mu$ \\
     $\mu\otimes\nu$
     &product measure of $\mu$ and $\nu$\\
     $\mu_n\convlaw\mu$
     & the sequence of probability measures on a metric space $\mathbb X$ $(\mu_n)_{n\ge 1}$ converges in law (or weakly)\\
     $\,$ 
     & to $\mu$ which means that for all $f$ continuous and bounded on $\mathbb X$ we have $\int_\mathbb X fd\mu_n\underset{n\to+\infty}{\rightarrow}\int_\mathbb X fd\mu$\\
     $iid$ 
     & independent and identically distributed\\
     $\mP$
     & a probability\\
     $\E, \Var$
     & expectation and variance \\
	$L^p$
	& in the context of random variables: random variables such that $\E(|X|^p)<+\infty$  \\   
     $X\sim \mu$
     & the law of the random variable $X$ is $\mu$\\
     $\mathcal N(m,\sigma^2)$
     & Normal law of mean $m$ and variance $\sigma^2$\\
     $\E(X|\mathcal F)$
     & conditional expectation of the random variable $X$ according to the filtration $\mathcal F$\\
     $GUE_N/GOE_N$
     & Gaussian Unitary/Orthogonal ensemble of $N\times N$ matrix \\
     $B(t)$
    &  most of the time a standard Brownian motion\\
    $|A|$
    & cardinal of the set $A$\\
    $\mathcal P(E)$
    & in the context of set: the set of subsets of $E$\\
     $A^c$
     & complement of the subset $A$\\
     $\overline A,\, \overset{\circ} {A},\, \partial A$
    & closure of the set $A$, interior of the set $A$, frontier of the set $A$\\
     $i$
     &the element (0,1) of $\C$ or an index of summation if it is an index\\ 
     $\Re(z),\,\Im(z)$
     & real and imaginary part of the complex number $z$\\
     $\Gamma$
     & Euler Gamma function\\
     $\log$
      & natural Neperian logarithm function\\
      $x_+,x_-$
      & positive part, negative part of the real number $x$\\
      $\overline{\lim}$ or $\limsup_N$
     & superior limit \\
      $\underline{\lim}$ or $\liminf_N$
     & inferior limit \\
     $\Delta_N(\lambda_1,...,\lambda_N)$
    & most of the times Vandermonde determinant of the family $(\lambda_1,...,\lambda_N)$\\
    $\mathfrak{S}_N$
    & symmetric group of cardinal $N!$\\
    $\varepsilon(\sigma)$
    & signature of the permutation $   	\sigma\in\mathfrak{S}_N$\\ 
      \hline
    \end{tabular}
    \smallskip
    \caption{Main frequently used notations.}
    \label{tb:notations}
  \end{center}
\end{table}
\newpage
\section{Introduction}
\tab These notes are intended for beginners in random matrices or for people interested in systems of particles in interaction as the Dyson Brownian motion or Coulomb and Riesz gases. It is recommended to have a graduate level in probability and to be familiar with the Itô calculus for Part \ref{part:dyson}. There are many reference textbooks on the general topic of random matrices as for instance \cite{Alicelivre,Tao,Mehta,compactgroup} and the literature is already very wide. These notes shall focus on a very particular aspect of the study of large random matrices which is their links with systems of particles in interaction. For instance the Dyson Brownian motion is often presented as a related topic of random matrices but here we wanted to present the most possible self contained and accessible approach on this topic for people that are non familiar with the background of random matrices.
\newline
\newline
\tab We can consider the works of Wishart in 1928 \cite{wishart1928generalised} as the starting point of the theory of large random matrices. The original motivation of Wishart was a statistical one. Given independent and identically distributed centred random vectors $X_1,...,X_N$ of $\R^d$, he wanted to estimate the covariance matrix $\Sigma=\E[X_1 X_1^T]$. Wishart proposed $\Sigma_N=\frac{1}{N}\sum_{k=1}^N X_kX_k^T=:\frac{1}{N}\mathfrak X_N \mathfrak X_N^T$ as an estimator the empirical covariance matrix where $\mathfrak X_N=(X_1,...,X_N)\in M_{d,N}(\R)$. As an application of the law of large numbers, $\Sigma_N$ converges almost surely towards $\Sigma$ when $N$ becomes large. In order to understand the behaviour of $\Sigma_N$ when for example the size $d_N$ of the data also increases with $N$, one can for instance study the behaviour of its spectrum as in the principal component analysis which is an important topic in statistics. In this sense Wishart was the first one to ask questions about large random matrices.
\newline
To understand the distribution of the spectrum of matrices a natural quantity to introduce is the so-called empirical distribution of the eigenvalues. This is the probability measure given by $\mu_N:=\frac{1}{d_N}\sum_{k=1}^{d_N} \delta_{\lambda_k^N}(dx)$ where $(\lambda_k^N)_{1\le k\le d_N}$ is the spectrum of $\Sigma_N$ and $\delta_z(dx)$ is the Dirac measure in $z$.
Indeed, if we want to know how many eigenvalues of $\Sigma_N$ are in an interval $I=[a,b]\subset\R$, we have $$\int_{a}^b\mu_N(dx)=\cfrac{|\{1\le i\le d_N,\,\lambda_i^N\in[a,b]\}|}{d_N}.$$
In Figure 1, we illustrate the histogram of the eigenvalues of $\Sigma_N$ when $N$ and $d_N$ are large. 
\begin{center}
\includegraphics[scale=0.5]{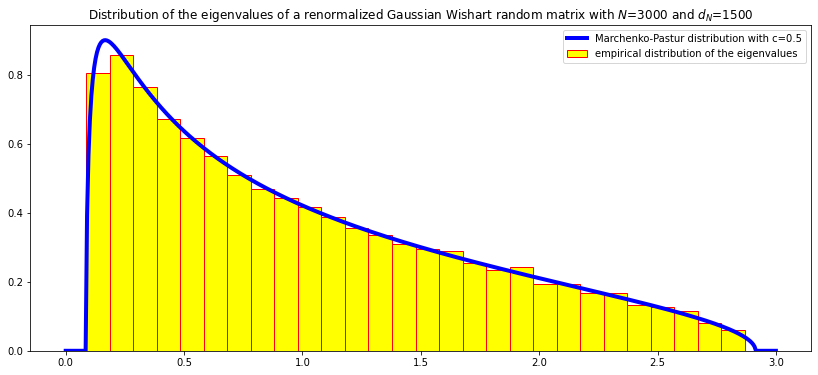}
Figure 1: Illustration of the Marchenko-Pastur theorem.
\end{center}
The notion of convergence of probability measures that corresponds to the convergence of the histogram is the convergence in law (or also called the weak convergence of measures). We recall that $\mu_N$ converges in law towards a probability measure $\mu\in\mes$ if for all $f$ continuous and bounded real function $$\int_\R f d\mu_N \underset{N\to+\infty}{\rightarrow}\int_\R fd\mu$$ and we shall write $\mu_N\overunderset{\mathcal L}{N\to+\infty}{\longrightarrow}\mu$ in this case.
A central result on this topic was obtained by Marchenko and Pastur at the end of the 1960s \cite{marchenko1967distribution}. For the following theorem assume that the entries of $\mathfrak X_N$ are independent, identically distributed, centred and of variance $1$.
\begin{theorem}[Marchenko-Pastur, Figure 1]
Assume that there exists $c>0$ such that $d_N$ satisfies $\frac{d_N}{N}\underset{N\to+\infty}{\longrightarrow} c$. Let $(\lambda_1^N,...,\lambda_{d_N}^N)$ be the $d_N$ eigenvalues of $\Sigma_N$ and $\mu_N:=\frac{1}{d_N}\sum_{k=1}^{d_N} \delta_{\lambda_k^N}(dx)$ be the empirical distribution of the eigenvalues of $\Sigma_N$.
Then, almost surely we have the following convergence $$\mu_N\overunderset{\mathcal L}{N\to+\infty}{\longrightarrow} \rho_c(dx),$$
where $\rho_c$ is an explicit distribution called the Marchenko-Pastur law of parameter $c$ given by 
\begin{equation} 
\label{eq:marchenko}
\rho_c(dx)=\left(1-\cfrac{1}{c}\right)_+\delta_0(dx)+\cfrac{1}{2\pi c x}\,\sqrt{(c^+-x)(x-c^-)}\,\mathds{1}_{[c^-,c^+]}(x)dx,
\end{equation} 
with $c^-=(1-\sqrt{c})^2$ and $c^+=(1+\sqrt{c})^2$.
\end{theorem}
Then, in the 1950s, random matrices became an active field particularly of physic at the instigation of for instance Wigner and Dyson \cite{wigner1958distribution,Dyson}. Random matrices were used to model the levels of energies of complex quantum systems or systems of particles in interactions as Coulomb gases. 
In order to understand the energy levels of atoms physicists study the eigenvalues of Hamiltonian operators that are self-adjoint operators (in finite or infinite dimensions). An idea that emerged in the 1950s is that we could try to approximate the spectrum of such operators by the spectrum of large random matrices. This leads for instance to the question of the behaviour of the spectrum of large random symmetric matrices. The main result obtained by Wigner in \cite{wigner1958distribution} in this direction is the following one.
\begin{theorem}[Wigner, Figure 2]
For all $N\ge 1$, let $M^N\in S_N(\R)$ be a symmetric matrix such that $(M^N_{i,j})_{1\le i\le j\le N}$ is a family of independent centred random variables of variance 1. Let $(\lambda_1^N,...,\lambda_N^N)$ be the real spectrum of $\frac{M^N}{\sqrt{N}}$ and $\mu_N:=\frac{1}{N}\sum_{k=1}^{N} \delta_{\lambda_k^N}(dx)$ be the empirical distribution of the eigenvalues of $\frac{M^N}{\sqrt{N}}$.
Then, almost surely we have the following convergence $$\mu_N\overunderset{\mathcal L}{N\to+\infty}{\longrightarrow} \sigma(dx),$$
where $\sigma$ is an explicit distribution called the Wigner semicircle law given by 
\begin{equation}
\label{eq:semicircledistrib}
\sigma(dx)=\cfrac{1}{2\pi}\,\sqrt{(4-x^2)_+}\,\mathds{1}_{[-2,2]}(x)dx.
\end{equation} 
\end{theorem}
\begin{center}
\includegraphics[scale=0.5]{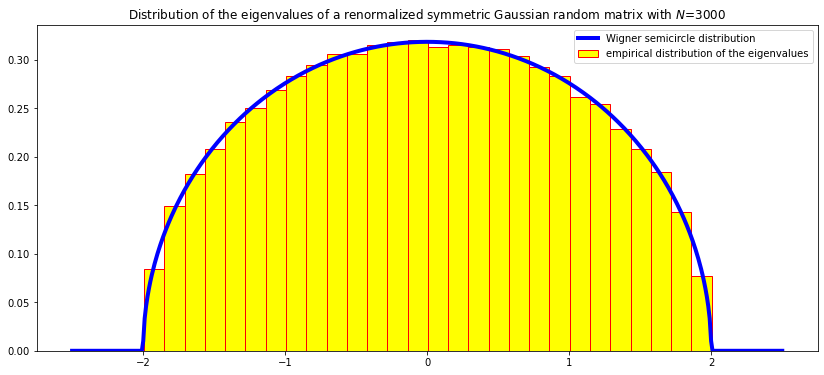}
Figure 2: Illustration of the Wigner theorem.
\end{center}
\tab Random matrix theory is at the intersection of linear algebra and probability explaining its mathematical richness. Moreover, random matrix theory gives new points of view on many topics in mathematics and creates a path between some topics that were a priori not connected.
Since the seminal questions the random matrix theory has been used in various fields. Here is a non exhaustive list of some active fields in which random matrices appear:
\newline
$\bullet$ Ecology and large Lotka–Volterra models \cite{akjouj2024complex,akjouj2024equilibria}
\newline
$\bullet$ Coulomb and Riesz gases as the eigenvalues of some random matrices models can be considered as log gases \cite{Dyson,serfaty2024lectures,Chafai1}
\newline
$\bullet$ Zeros of random polynomials and random power series \cite{hough2009zeros,butez2017polynomes}
\newline
$\bullet$ Physics \cite{guhr1998random} since for instance the distributions of fermions trapped in a harmonic potential is the same as the distribution the eigenvalues of random matrices \cite{dean2019noninteracting} or in the study of levels of energy of atoms as in the original work of Wigner \cite{wigner1958distribution} or for complex systems as many-body quantum systems \cite{vyas2018random}
\newline
$\bullet$ Number theory through a deep connexion between the zeros of the Riemann zeta function and the eigenvalues of random matrices originally discovered by Montgomery and Dyson \cite{montgomery1973pair,bourgade2006random,compactgroup}
\newline
$\bullet$ Random permutations \cite{baik1999distribution}
\newline
$\bullet$ Last passage percolation (which is also linked with random permutations) \cite{baryshnikov2001gues,gravner2001limit}, the Brownian percolation which is connected with the Dyson Brownian motion \cite{o2002representation} and more recently the so-called directed landscape which can be constructed as a limit object of last passage percolation \cite{barraquand2024paysage,dauvergne2022directed}
\newline
$\bullet$ Random graphs through the adjacency matrix \cite{erdos2013spectral,augeri2025large}
\newline
$\bullet$ Combinatorics of maps \cite{guionnet2005combinatorial,dubach2021words}
\newline
$\bullet$ Free probability or non commutative probability introduced by Voiculescu to deal with the free group factors isomorphism problem in operator algebras theory \cite{voiculescu1995free,mingo2017free,Alicelivre}
\newline
$\bullet$ Statistics as in the original work of Wishart to estimate covariance matrices \cite{wishart1928generalised}
\newline
\newline
\tab In these notes, we decided to focus on the link between random matrices and system of particles in interaction. 
\newline
\tab In Part \ref{part:ginibregaussian}, we shall mainly focus on two ensembles of matrices that are the so-called Ginibre ensemble and the Gaussian unitary ensemble. From these models one can compute the law of their eigenvalues. This will allow us to consider the eigenvalues as Riesz or Coulomb log gases in dimension 1 and 2 subjected to an external quadratic potential. Then, we shall detail how we can study the eigenvalues of Ginibre ensemble and Gaussian unitary ensemble as point processes in $\R$ or $\C$, in order to compute the correlation functions of the eigenvalues. From these computations, we shall obtain the so-called circular law and Wigner theorem that gives the statistics of the eigenvalues when the size of random matrices of the Ginibre ensemble or the Gaussian unitary ensemble becomes large. 
\newline
\tab In Part \ref{part:dyson}, we deal with the so-called Dyson Brownian motion which is a continuous model of random matrices introduced by Dyson in 1962 \cite{Dyson}. From this continuous model of random matrices, one can obtain a system of stochastic differential equations for the dynamics of the eigenvalues, so that the eigenvalues can be considered as particles in interaction through a repulsive force and Brownian noise. As we will see, one can interpret the law of the process of the eigenvalues of the Dyson Brownian motion as the law of $N$ independent Brownian motions conditioned by the fact that they do not intersect. This approach gives another point of view and make some other proofs available for some results of random matrices viewed in Part \ref{part:ginibregaussian} as the law of the eigenvalues of matrices of the Gaussian unitary ensemble or the Wigner theorem. 
\newline
\tab Finally, we study the large deviations of the empirical mean of the eigenvalues of matrices of the Gaussian unitary ensemble. This result was proved by Ben Arous and Guionnet in \cite{Alice1}. The method is quite robust and actually only uses the fact that the eigenvalues of random matrices of the Gaussian unitary ensemble is a log gas in dimension 1. The result obtained by Ben Arous and Guionnet is today extended for Coulomb and Riesz gases and we shall explain those results in Part \ref{sec:largedev}. This gives another proof for the Wigner theorem and the circular law presented in Part \ref{part:ginibregaussian}. 
\newline
\newline
Before starting, let us mention that notations between parts are independent except for the notations of table \ref{tb:notations}.
\newpage

\part{Ginibre and Gaussian unitary ensembles}

\label{part:ginibregaussian}
\tab In this section, we shall focus on the two most famous ensembles of matrices which are the Gaussian unitary ensemble and the Ginibre ensemble. They correspond to hermitian and non hermitian matrices with independent Gaussian entries. 
\newline
These models were particularly studied because they have a certain structure that help doing explicit computations as we shall explain. Moreover, in random matrix theory it is expected a so-called "universality" phenomenon which morally says that the behaviour of spectral statistics of a random matrix is similar with the behaviour with independent Gaussian entries when the size of the matrix becomes large. Typically, the circular law or the Wigner theorem, that shall be proved in this section for Ginibre matrices and matrices of the Gaussian unitary ensemble, are relevant examples.
\newline
\tab The remarkable point about matrices of the Gaussian unitary ensemble or the Ginibre ensemble is that the law of the eigenvalues of such matrices can be computed explicitly and has a certain form. Indeed, the eigenvalues of such matrix models are real or complex determinental point processes. Thanks to this remark we can compute the correlation functions of the eigenvalues of matrices in the Ginibre or in the Gaussian unitary ensemble. There are many references for these standard results as for instance \cite{Alicelivre,haagerup2003random,Mehta,Tao,Chafai1,compactgroup}.
\newline
\tab Then, we prove that up to a re normalization, the eigenvalues of Ginibre matrices are distributed uniformly on the disk when the size of the matrix becomes large. This is called the circular law and was originally proved by Mehta and Ginibre \cite{Mehta}. About the circular law, we refer to the detailed revue of Bordenave and Chafaï \cite{bordenave2012around} and the revue of Chafaï \cite{Chafai1}.
\newline
\tab We finally show the Wigner theorem for matrices of the Gaussian unitary ensemble which is the counterpart of the circular law. We refer to \cite{Tao,haagerup2003random,Alicelivre} for this result.
\newline
\tab All the random variables are defined on a same probability space ($\Omega$, $\mathcal F$, $\mP$).
 
\section{Examples of remarkable ensembles of matrices}
In this part we introduce the Ginibre ensemble and the Gaussian unitary ensemble and we point out some basic properties they have in common.
\subsection{Notation and first properties}
\subsubsection{Gaussian unitary ensemble}
Let $N\ge 1$ be a positive integer.
\begin{definition}
\label{def gue,goe}
A random variable $X$ is said to be a $N$ complex (resp. real) Wigner matrix if $X$ takes value in $\mathcal H_N(\C)$ (resp. $S_N(\R)$) and $(X_{i,i},\sqrt2\, \Re (X_{i,j}), \sqrt 2\,\Im(X_{i,j}))_{1\le i < j\le N}$ is a family of independent Gaussian random variables of law $\mathcal N(0,1)$ (resp.  $(X_{i,i}, \sqrt{2} X_{i,j})_{1\le i < j\le N}$ is a family of independent Gaussian random variables of law $\mathcal N(0,1)$). 
\newline 
If $X$ is a $N$ complex (resp. real) Wigner matrix, we say that $X$ is in the Gaussian Unitary Ensemble (resp. Gaussian Orthogonal Ensemble). We shall write $X\sim GUE_N$ (resp. $X\sim GOE_N)$.
\newline
For a random variable in the $GUE_N$ or in the $GOE_N$, we denote $(\lambda_1,...,\lambda_N)\in\R^N$ the vector of its eigenvalues viewed as an exchangeable vector of $\R^N$.
\end{definition}
Let us mention that in the literature what we call here Wigner matrices are most of the times called Gaussian Wigner matrices since the entries of the matrix are Gaussian. Since, in these notes we shall only focus on the case when the entries are Gaussian we decided to use the terminology of Wigner matrices instead of Gaussian Wigner matrices. 
\newline
We will now focus on $N$ complex Wigner matrix.
To compute the law of a $N$ complex Wigner matrix, we first compute the dimension of $\mathcal H_N(\C)$ as a $\R$ vector space.
\begin{lemma}
\label{lemma:dimensionHermitienne}
$\mathcal H_N(\C)$ is a $\R$ vector space of dimension $N^2$ such that $((\Delta_{k,k})_{k\le N}, (\Delta_{k,j}+\Delta_{j,k})_{1\le k<j\le N},  (i(\Delta_{k,j}-\Delta_{j,k})_{1\le k<j\le N}))$ is a basis.
\end{lemma}

\begin{proof}
Let $M=(M_{k,j})_{1\le k,j\le N}\in \mathcal H_N(\C)$. By definition of $\mathcal H_N(\C)$, for all $1\le k\le N$, $M_{k,k}\in\R$ and for all $1\le k<j\le N$, $\Re(M_{k,j})=\Re(M_{j,k})$ and $\Im(M_{k,j})=-\Im(M_{j,k})$. So we can write: $$M=\sum_{1\le k\le N} M_{k,k}\Delta_{k,k}+ \sum_{1\le k<j\le N}\Re(M_{k,j})\left(\Delta_{k,j}+\Delta_{j,k}\right)+\sum_{1\le k<j\le N}\Im(M_{k,j})i\left(\Delta_{k,j}-\Delta_{j,k}\right).$$
This gives the result.
\end{proof}

We shall sometimes identify $\mathcal H_N(\C)$ and $\R^{N^2}$ by expressing a matrix of $\mathcal H_N(\C)$ in the basis of Lemma \ref{lemma:dimensionHermitienne}.

\begin{prop}
\label{prop: loi Gue}
Let $X$ be a $N$ complex Wigner matrix, then $X$ has a density with respect to the Lebesgue measure on $\R^{N^2}$ which is proportional to $M\mapsto\exp(-\Tr(M^2)/2)$.
\end{prop}

\begin{proof}
Let $f\ge0$ be a measurable function defined on $\R^{N^2}=\mathcal H_N(\C)$. By definition of a Wigner matrix, there exists a normalizing constant $c_N$ such that $$\E[f(X)]=c_N\dis\int_{\R^{N^2}}f((Y_{i,j})_{1\le i,\,j\le N})\exp\left(-\dis\sum_{1\le i \le N}\cfrac{Y_{i,i}^2}{2}-\dis\sum_{1\le i<j\le N}\left( \Re(Y_{i,j})^2 +\Im(Y_{i,j})^2\right)\right)d\leb(Y).$$
So we get: $$\E[f(X)]=c_N\dis\int_{\R^{N^2}}f((Y_{i,j})_{1\le i,\,j\le N})\exp\left(-\cfrac{1}{2}\dis\sum_{1\le i,j\le N}|Y_{i,j}|^2\right)d\leb(Y).$$
We write: $$\E(f(X))=c_N\dis\int_{\mathcal H_N(\C)}f((Y_{i,j})_{1\le i,\,j\le N})\exp\left(-\cfrac{1}{2}\Tr(Y^2)\right)d\leb(Y).$$
Since this formula is true for all $f\ge 0$ measurable, we get the result.
\end{proof}

The next proposition explains the terminology of the Gaussian Unitary Ensemble and show an invariance of the law with respect to multiplication by a unitary matrices. A similar result holds for the Gaussian orthogonal ensemble with orthogonal matrices instead of unitary matrices.
We first state a lemma in order do a change a variable.
\begin{lemma}
\label{lemma: det 1}
Let $V\in\mathcal U_N(\C)$ and $\psi_{V}:\mathcal H_N(\C)\to\mathcal H_N(\C)$ defined by $\psi_{V}(X)=V XV^*$. Then we have $|\det(\psi_V)|=1$.
\end{lemma}
\begin{proof}
We define $d_{V'}=|\det(\psi_{V'})|\in\R^+$ for all $V'\in\mathcal U_N(\C)$. Then for all $V'\in\mathcal U_N(\C)$ and $V''\in\mathcal U_N(\C)$ we have $d_{V'}d_{V''}=d_{V'V''}$ since $\psi_{V'}\circ\psi_{V''}=\psi_{V'V''}$.
This gives that $d_Vd_{V^*}=1$ and so $d_V\ne 0$ and for all $k\in\Z$, $d_{V^k}=d_V^k$.
\newline
Since $(V^k)_{k\in\N}$ is a sequence of matrices of $\mathcal U_N(\C)$, which is a compact group, we can extract a subsequence that converges to $\tilde{V}\in\mathcal U_N(\C)$. By continuity, there exists a subsequence of $(d_{V^k})_{k\in\N}$ that converges to $d_{\tilde{V}}\ne 0$. Since $(d_{V^k})_{k\in\N}$ is a non negative real geometric sequence of reason $d_V$, it converges to an element of $\R^{+*}$ if and only if $d_V=1$ which proves the result.
\end{proof} 

\begin{prop}
\label{prop:invarianceloiunitaire}
Let $U\in\mathcal U_N(\C)$ be a deterministic unitary matrix and $X$ be a $N$ complex Wigner matrix, then $U^*XU$ is also a $N$ complex Wigner matrix. 
\end{prop}

\begin{proof}
The application $\phi:$ $\mathcal H_N(\C)\to \mathcal H_N(\C)$ defined by $\phi(X)=U^*XU$ has a Jacobian by Lemma \ref{lemma: det 1}. Using the law of a $N$ complex Wigner matrix obtained in Proposition \ref{prop: loi Gue} and doing a change of variables yields that the density of $U^*XU$ is proportional to $Y\mapsto\exp\left(-\cfrac{1}{2}\Tr((UYU^*)^2)\right)$. 
\newline
By property of the trace $\Tr((UYU^*)^2)=\Tr(UY^2U^*)=\Tr(Y^2U^*U)=\Tr(Y^2)$, which yields that the density of $U^*XU$ is proportional to $Y\mapsto\exp\left(-\cfrac{1}{2}\Tr(Y^2)\right)$. Hence $U^*XU$ and $X$ have the same law.
\end{proof}

An important property about the spectrum of a $N$ complex Wigner matrix is that the spectrum is almost surely simple.
\begin{prop}
\label{prop: spectrum simple}
Let $M$ be a $N$ complex Wigner matrix. Then almost surely the spectrum of $M$ is simple (all the eigenvalues are distinct) and $\det(M)\ne 0$.
\end{prop}
\begin{proof}
Let $\Lambda^0=(\lambda_1^0,...,\lambda_N^0)\in\R^N$ be the random spectrum of $M$. 
Let us introduce a quantity linked with the so-called discriminant of a polynomial. For $\Lambda=(\lambda_1,...,\lambda_n)\in\R^N$, let $\mathcal D (\Lambda):=\prod_{i\ne j}(\lambda_i-\lambda_j)$.
We notice that the spectrum of $M$ is simple if and only if $\mathcal D (\Lambda^0)=\prod_{i\ne j}(\lambda_i^0-\lambda_j^0)\ne 0$.
\newline
Since $\mathcal D$ is a symmetric polynomial in the $\lambda_i$ (which means that for all permutation $\sigma\in \mathfrak S_N$, for all $(\lambda_1,...\lambda_N)\in\R^N$ we have that $\mathcal D(\lambda_1,...,\lambda_N)=\mathcal D(\lambda_{\sigma(1)},...,\lambda_{\sigma(N)})$), the fundamental theorem of the symmetric polynomial (the so-called Newton theorem) gives that $\mathcal D\in \R[\Sigma_1,...,\Sigma_N]$ with $(\Sigma_i)_{1\le i\le N}$ the symmetric polynomials associated to the $\lambda_i$. Given $\Lambda\in\R^N$, let us consider a matrix $M^\Lambda$ in $\mathcal M_N(\C)$ such that the set of its eigenvalues is exactly $\Lambda$. By the fundamental properties of the determinant, the coefficients of $M^\Lambda$ and its characteristic polynomial are in the same ring. Since the coefficients of the characteristic polynomial of $M^\Lambda$ are up to the sign the $\Sigma_i$ we have that the $\Sigma_i$ are in the same ring as the coefficients of $M^\Lambda$. 
\newline
Hence, we deduce that there exists $R\in\R[(X_{i,j})_{1\le i,j\le N}]$ such that for all $\Lambda\in\R^N$, $\mathcal D(\Lambda)=R((M^\Lambda_{i,j})_{1\le i,j\le N})$.
So we get that the spectrum of $M$ is simple if and only if $D(\Lambda^0)=R((M_{i,j})_{1\le i,j\le N})\ne 0$.
Let us recall a basic lemma which can be proved by induction on the dimension. 
\begin{lemma}
\label{lemma:polyalmostsur}
Let $X\in \R^d$ be a random variable with a density with respect to the Lebesgue measure of $\R^d$. Then, if $P\in \R[Z_1,...,Z_d]$ is not null, we almost surely have $P(X)\ne 0$. 
\end{lemma}
\begin{proof} 
We sketch the proof. We prove the result by induction on $d$. 
\newline
If $d=1$, if we denote $\{x_1,...,x_k\}\in\R^k$ the real roots of $P$ (with possibly no real roots if $k=0$) then $P(X)=0$ if and only if $X\in \{x_1,...,x_k\}$. Since $X$ has a density with respect to the Lebesgue measure, $\mP(X\in \{x_1,...x_k\})=0$ proving the result. 
\newline
We suppose that the result is established for $d\ge 1$. We write $X=(X_1,...,X_{d+1})$. We consider $P\in \R[Z_1,...,Z_{d+1}]$ a non null polynomial. We can suppose that $P$ depends on all the $(Z_i)_{1\le i\le d+1}$. Indeed if it is not the case, we can for instance suppose that $P\in \R[Z_1,...,Z_{d}]$ and we can use the induction hypothesis to $P$ and $(X_1,...,X_d)$ which also has a density with respect to the Lebesgue measure in $\R^d$ giving the result.
We now suppose that $P$ depends on all the variables. Conditioning by $X_1$ yields \begin{equation}
\label{eq:polzero}
\mP(P(X)=0)=\E(\mP(P(X)=0|X_1)).
\end{equation}
Conditionally on $X_1$, $X$ has a density with respect to the Lebesgue measure of $\R^d$ and for all $\omega\in \Omega$, we can define the polynomial $Q_\omega(Z_2,...,Z_{d+1})$ of $\R[Z_2,...,Z_{d+1}]$ as $ Q_\omega(Z_2,...,Z_{d+1})=P(X_1(\omega),Z_2,...,Z_{d+1})$.
Applying the induction hypothesis for all $\omega\in \Omega$ gives that $\mP(P(X)=0|X_1)(\omega)=\mP(P(X_1,X_2,...,X_{d+1})=0|X_1)(\omega)$ is equal to 0 if $Q_\omega(Z_2,...,Z_{d+1})$ is not the null polynomial and is equal to 1 if $Q_\omega(Z_2,...,Z_{d+1})$ is the null polynomial. 
So we obtain $$\mP(P(X)=0|X_1)(\omega)=\mathds 1_{Q_\omega(Z_2,...,Z_{d+1})=0}=\mathds 1_{P(X_1(\omega),Z_2,...,Z_{d+1})=0}.$$
Taking the expectation and using \eqref{eq:polzero} yields $$\mP(P(X)=0)=\mP(P(X_1(\omega),Z_2,...,Z_{d+1})=0).$$
Writing $P\in \R[Z_1,...,Z_{d+1}]$ as an element of $\R[Z_1][Z_2,...,Z_{d+1}]$, as $P$ is not the null polynomial, we can find $k\ge 1$ non null polynomials of $\R[Z_1]$ denoted $\widetilde P_1,...,\widetilde P_k$ such that $P(x_1,Z_2,...,Z_{d+1})=0$ if and only if $\widetilde  P_1(x_1)=\widetilde P_2(x_1)=...=\widetilde P_k(x_1)=0$. So we have $$\mP(P(X)=0)=\mP(P(X_1(\omega),Z_2,...,Z_{d+1})=0)=
\mP(\widetilde P_1(X_1(\omega))=...
=\widetilde P_k(X_1(\omega))=0).$$
Using the case $d=1$ we obtain the result.
\end{proof}
Thanks to this lemma we deduce that almost surely the spectrum of $M$ is simple.
\newline
Similarly, the determinant of a matrix is a polynomial in its entries and so Lemma \ref{lemma:polyalmostsur} gives that $\det(M)\ne 0$ almost surely.
\end{proof}
Let us mention that in the previous proof we do not use the explicit law of the entries of the matrix. The only important points are the independence of the entries and the fact that the entries have a density with respect to the Lebesgue measure.

\subsubsection{Ginibre ensemble}
Let $N\ge 1$ be a positive integer.
\begin{definition}
A $N$ complex Ginibre random matrix is a random variable with value in $M_N(\C)$ such that for all $(\Re(M_{i,j}),\Im(M_{i,j}))_{1\le i,j\le N}$ is a family of independent Gaussian random variable of law $\mathcal N(0,1/2)$ (for all $1\le i,j\le N$, $M_{i,j}$ is said to be a standard complex Gaussian random variable).
\newline
For a $N$ complex Ginibre matrix, we denote $(\lambda_1,...,\lambda_N)\in\C^N$ the vector of its complex eigenvalues viewed as an exchangeable vector of $\C^N$.
\end{definition}
Now we shall show that a $N$ complex Ginibre random matrix has a density with respect to the Lebesgue measure of $\C^{N^2}$. For all $k\ge 0$, we define the Lebesgue measure on $\C^k$, which shall be written $dz_1...dz_k$, as the Lebesgue measure on $\R^{2k}$ with the identification of $\C$ as $\R^2$. More exactly it means that for suitable $f$ defined on $\C^k$: $$\int_{\C^k} f(z_1,...,z_k) dz_1...dz_k=\int_{\R^{2k}}f(x_1,y_1,x_2,y_2,...,x_k,y_k)dx_1dy_1dx_2dy_2...dx_kdy_k,$$ where $z_j=x_j+iy_j$ and $dz_j:=dx_j dy_j$ for all $1\le j\le k$.
\begin{prop}
\label{prop :law ginibre matr}
Let $M$ be a $N$ complex Ginibre random matrix. The law of $M$ has a density with respect to the Lebesgue measure of $\C^{N^2}$ which is proportional to $$M\mapsto\exp\left(-\sum_{i,j=1}^N|M_{i,j}|^2\right)=\exp(-\Tr(M M^*)),$$ with $M^*=\overline M^T$ the transconjugate of $M$.  
\end{prop}
\begin{proof}
By definition of a $N$ complex Ginibre random matrix, for all $1\le i,j\le N$, $\Re(M_{i,j})$ and $\Im(M_{i,j})$ are real Gaussian random variables with densities with respect to the real Lebesgue measure proportional to $x\mapsto\exp(-x^2)$. By independence, we get that for all $1\le i,j\le N$, $M_{i,j}$ (as an element of $\C$ viewed as $\R^2$) has a density with respect to the complex Lebesgue measure proportional to $$z=(x,y)\mapsto\exp(-(x^2+y^2))dxdy=\exp(-|z|^2)dz.$$ By independence of $(M_{i,j})_{1\le i,j\le N}$ we have the result.
\end{proof}
Let us state some basics properties of Ginibre random matrices.
First, the law of a $N$ Complex Ginibre matrix is invariant by unitary change of basis. We first state a counterpart of Lemma \ref{lemma: det 1} whose proof is totally similar. 
\begin{lemma}
\label{lemma: det changevariablegin}
Let $(U,V)\in\mathcal U_N(\C)^2$ and $\psi_{V}:\mathcal M_N(\C)\to\mathcal M_N(\C)$ defined by $\psi_{V}(X)=V X$ and $\widetilde{\psi_{U}}:\mathcal M_N(\C)\to\mathcal M_N(\C)$ defined by $\widetilde {\psi_U}(X)=XU$. Then we have $|\det(\psi_V)|=|\det(\widetilde {\psi_U)}|=1$.
\end{lemma}
\begin{prop}
\label{prop: inv ginibre law}
If $M$ is a $N$ complex Ginibre random matrix and $U,V\in\mathcal U_N(\C)$ are deterministic matrices, then $UMV$ has the law of a $N$ complex Ginibre matrix.
\end{prop}
\begin{proof}
Using Proposition \ref{prop :law ginibre matr} and doing a change of variables using Lemma \ref{lemma: det changevariablegin}, we get that $UMV$ has a density with respect to the complex Lebesgue measure which is proportional to 
\begin{align*}
M\mapsto \exp{(-\Tr(U^{-1}MV^{-1}(U^{-1}MV^{-1})^*))}&=\exp{(-\Tr(U^{-1}MM^*(U^{-1})^*))}\\
&=\exp{(-\Tr(MM^*))}
\end{align*}
where we used that $UU^*=I_N$, $VV^*=I_N$ and the properties of the trace.
\end{proof}
A direct computation gives us a formal interpretation of a $N$ complex Ginibre matrix when $N$ is large. 
\begin{prop}
\label{prop:orthogonality N large}
Let $(M_{i,j})_{ (i,j)\in \N_{\ge 1}^2}$ be a family of independent standard complex Gaussian random variables. 
For all $N\ge 1$, let $M_N=(M_{i,j})_{1\le i,j\le N}$ be a $N$ complex Ginibre. Then almost surely $M_N/\sqrt{N}$ has orthonormal rows/columns when $N$ goes to $+\infty$. 
\end{prop}

\begin{proof}
The result is an immediate application of the law of large number since for all $1\le i,j,k,l\le N$ with $(i,j)\ne (k,l)$ we have that $\E(|M_{i,j}|^2)=1$ and $\E(M_{i,j}\overline{M_{k,l}})=0$.
\end{proof}

Formally, Proposition \ref{prop:orthogonality N large} explains that given a $N$ complex Ginibre matrix $M_N$ when $N$ is large, $M_N/\sqrt{N}$ is like a unitary matrix whose law is invariant by left and right multiplication by a unitary matrix of $\mathcal U_N(\C)$ by Proposition \ref{prop: inv ginibre law}. This shall be linked with the notion of Haar measure on $\mathcal U_N(\C)$ as we shall see in next section. Moreover, since the spectrum of a unitary matrix is located on the circle we can formally expect that the spectrum of $M_N/\sqrt{N}$ will be located near the circle when $N$ becomes large. This shall be proved in next section proving the so-called circular law.   
Finally, as for the Wigner case, one can show that the spectrum of a $N$ complex Ginibre matrix is simple almost surely.
\begin{prop}
\label{prop: spectre simple}
Let $M$ be a $N$ complex Ginibre matrix. Then almost surely the spectrum of $M$ is simple (all the eigenvalues are distinct) and $\det(M)\ne 0$. Hence almost surely, $M$ is diagonalizable.
\end{prop}
\begin{proof}
The exact same proof as for complex Wigner matrices can be done. 
\end{proof}

\subsection{Law of the eigenvalues and eigenvectors}
Let $N\ge 1$ be a positive integer. In this section we shall explicitly give the law of the eigenvalues of a $N$ complex Wigner and Ginibre matrix. We first give the details for the Wigner case and then give the counterpart results for the Ginibre case.

\subsubsection{Gaussian unitary ensemble case}
\begin{definition}
A function $f:M_N(\C)\to \C$ is said to only depend on the spectrum if for all $M\in M_N(\C)$ and $U\in GL_N(\C)$ we have $f(UMU^{-1})=f(M)$.
\end{definition}
To give the law of the eigenvalues, the difficult point is that we can not associate canonically to a matrix its eigenvalues and eigenvectors in order to do properly a change of variable. It is actually possible by doing the so-called Weyl integration formula and integrating in quotient spaces. See Section 4.1 of \cite{Alicelivre} or Section 3.1 of \cite{compactgroup}. We will not do this method.
\newline 
To overcome this problem, in this section we shall find a law on eigenvectors and eigenvalues that matches with the law of a complex Wigner matrix. 
\paragraph{Eigenvectors}
Let us recall a general result about the notion of Haar measures on compact groups. 
\begin{definition}
Let $G$ be a compact metrizable group. A measure is said to be a Haar measure on $G$ if and only if for all $g\in G$, for all $A$ measurable $\mu(gA)=\mu(A)$ (this property is called left invariance). 
\end{definition}
\begin{Remark}
The Haar measure could be more generally define for a topological group and generalizes the Lebesgue measure in $\R$.
\end{Remark}
The following theorem gives the existence of a unique Haar probability measure.
\begin{theorem}[Haar]
Let $G$ be a metrizable compact group. All Haar measures on $G$ are finite and are proportional to a unique Haar probability measure which shall be denoted $\mu_G$ which is also right invariant.
\end{theorem}
The existence and uniqueness of a Haar measure for a compact group can be obtained using the so-called Kakutani fixed point theorem. We shall not detail this general proof that does not explain how to construct explicitly a such Haar measure. 
However, in the particular case of $\mathcal U_N(\C)$ one can construct explicitly a Haar measure.
\begin{prop}
There exists a Haar probability measure on $\mathcal U_N(\C)$.
\end{prop}

\begin{proof}
This construction is based on the Gram-Schmidt algorithm. Starting from a matrix $M\in GL_n(\C)$, we write $C_1,...,C_N$ its column. We construct an orthonormal basis associated to the the basis $(C_1,...,C_N)$ of $\C^N$. We first define $\widetilde {C_1}:=C_1/||C_1||_2$, we then define $$\widetilde {C_2}:=\cfrac{C_2-\cfrac{1}{||C_1||_2^2}\langle C_1|C_2 \rangle \,C_1}{|| C_2-\cfrac{1}{||C_1||_2^2}\langle C_1|C_2 \rangle \,C_1||_2},$$ and so on. We construct an orthonormal basis of $\C^N$ $(\widetilde {C_1},...,\widetilde {C_N})$. We define the Gram-Schmidt map $\mathcal G : GL_N(\C)\to \mathcal U_N(\C)$ by $\mathcal G(M)=(\widetilde {C_1},...,\widetilde {C_N})$ if $M=(C_1,...,C_N)$.
An important property of the Gram-Schmidt algorithm is that it commutes with the multiplication by a matrix of $\mathcal U_N(\C)$. More exactly, let $U\in \mathcal U_N(\C)$ and $M=(C_1,...,C_N)\in GL_N(\C)$. Then $\mathcal G(UM)=U\mathcal G(M)$. Indeed let $UM=(UC_1,...UC_N)$. Then using that $U\in U_N(\C)$ preserves the norm and the scalar product, we have 
\begin{align*}
\widetilde {UC_1}&=\cfrac{UC_1}{||UC_1||_2}=U \,\cfrac{C_1}{||C_1||_2}=U \widetilde {C_1}\\
\widetilde {UC_2}&=\cfrac{UC_2-\cfrac{1}{||UC_1||_2^2}\langle UC_1|UC_2 \rangle \,UC_1}{|| UC_2-\cfrac{1}{||UC_1||_2^2}\langle UC_1|UC_2 \rangle \,UC_1||_2}=U\, \cfrac{C_2-\cfrac{1}{||C_1||_2^2}\langle C_1|C_2 \rangle \,C_1}{|| C_2-\cfrac{1}{||C_1||_2^2}\langle C_1|C_2 \rangle \,C_1||_2}=U \widetilde {C_2},
\end{align*}
and so on. So we obtain that for all $1\le k\le N$, $\widetilde {UC_k}=U\widetilde {C_k}$.
\newline
We can now construct a random variable with law a Haar measure on $\mathcal U_N(\C)$. Consider a $N$ complex Ginibre random matrix $M_N$. By Proposition \ref{prop: inv ginibre law}, we know that the law of $M_N$ is invariant by multiplication by a deterministic matrices of $\mathcal U_N(\C)$. 
We consider the random variable $\mathcal G(M_N)$ which is defined almost surely by Proposition \ref{prop: spectre simple}. Then, for $U\in \mathcal U_N(\C)$ we have: $$U \mathcal G(M_N)=\mathcal G(UM_N)\overset{\text{Law}}{=}\mathcal G(M_N),$$ using the fact that the multiplication by $U$ and $\mathcal G$ commute. Hence the law of $\mathcal G(M_N)$ is a Haar probability measure on $\mathcal U_N(\C)$.
\end{proof}

\begin{Remark}
Another more geometric point of view would be to sample uniformly on the sphere a vector, then sample an orthogonal vector to this one on the unit sphere and so on to create $N$ orthogonal vectors. We refer to the first chapter of \cite{compactgroup} for this construction.
\end{Remark}
Since $\mathcal U_N(\C)$ is a compact group one can show some general properties of its Haar probability measures.
\begin{lemma}
Let $\mu$ be a Haar probability measure of $G$ a compact metrizable group. Then for all $A$ measurable subset and for all $g\in G$, we have: 
\newline
$\bullet$ $\mu(A^{-1})=\mu(A)$ with $A^{-1}=\{g^{-1},\, g\in A\}$,
\newline
$\bullet$ $\mu(Ag)=\mu(A)$,
\newline
$\bullet$ If $\nu$ is another Haar probability measure, then $\mu=\nu$.
\end{lemma}

\begin{proof}
$\bullet$ We directly compute using the left invariance of $\mu$ and the Fubini-Tonelli theorem: 
\begin{equation}
\label{eq haar calcul}
\begin{split}
\mu(A)&=\int_G \mu(gA) d\mu(g)\\
&=\int_G\int_G \mathds {1}_{gA}(z)d\mu(z)d\mu(g)\\
&=\int_G\int_G \mathds {1}_{zA^{-1}}(g)d\mu(g)d\mu(z)\\
&=\int_G \mu(zA^{-1})d\mu(z)\\
&=\mu(A^{-1}).
\end{split}
\end{equation}

$\bullet$ Using the previous computation and the left invariance of $\mu$ yields $$\mu(Ag)=\mu((Ag)^{-1})=\mu(g^{-1}A^{-1})=\mu(A^{-1})=\mu(A).$$

$\bullet$ We do the same computation as in \eqref{eq haar calcul} using this time $\nu$ we get that for every $B$ measurable:
\begin{equation}
\label{eq haar calcul 2}
\begin{split}
\mu(B)&=\int_G \mu(gB) d\nu(g)\\
&=\int_G\int_G \mathds {1}_{gB}(z)d\mu(z)d\nu(g)\\
&=\int_G\int_G \mathds {1}_{zB^{-1}}(g)d\nu(g)d\mu(z)\\
&=\int_G \nu(zB^{-1})d\mu(z)\\
&=\nu(B^{-1})=\nu(B).
\end{split}
\end{equation}
\end{proof}

As a consequence of the existence of the Haar measure on $\mathcal U_N(\C)$, we can give the law of the eigenvectors of a $N$ complex Wigner matrix matrix.

\begin{prop}
\label{prop trigo unitaire diagonale}
There exists a measure $\mu_{\Lambda}$ on $\R^N$ such that if $\Lambda$ is a diagonal matrix such that $(\Lambda_{1,1},..., \Lambda_{N,N})$ has the law $\mu_\Lambda$ and $U$ is a random variable independent of $\Lambda$ whose law is  $\mu_{\mathcal U_N(\C)}$ then $U\Lambda U^*$ is a $N$ complex Wigner matrix.
\end{prop}

\begin{proof}
Let $X$ be a $N$ complex Wigner matrix. By the spectral theorem we can find $W \in \mathcal U_N(\C)$ and $\Lambda$ a real diagonal matrix which are random variables and such that $X= W\Lambda W^*$. Let $V$ be a matrix distributed according to $\mu_{\mathcal U_N(\C)}$ independent of $W$ and $\Lambda$. Then the couple $(\Lambda,U:=\Lambda, VW)$ satisfies the properties of the proposition. 
\newline
Indeed we first prove that $U\Lambda U^*\equalaw X$. Let $f$ be a bounded continuous function, then we have: 
\begin{align*}
\E(f(U\Lambda U^*))&=\E(f(V(W\Lambda W^*)V^*))\\
&\hspace{-0.7cm}\overset{\text{Independence}}{=}\int_{\mathcal U_N(\C)}\E(f(\nu W\Lambda W^*\nu^*))d\mu_{\mathcal U_N(\C)}(\nu)\\
&=\int_{\mathcal U_N(\C)}\E(f(\nu X\nu^*))d\mu_{\mathcal U_N(\C)}(\nu)\\
&=\int_{\mathcal U_N(\C)}\E(f(X))d\mu_{\mathcal U_N(\C)}(\nu)=\E(f(X)),
\end{align*}
where we used the invariance of the law of $X$ by a unitary change of basis. It proves the first point. 
\newline
Secondly we must look at the the law of $(\Lambda,U)$. Let $f,g$ two bounded continuous functions, then we have: 
\begin{align*}
\E(f(U)g(\Lambda))&=\E(f(VW)g(\Lambda))\\
&=\E[\E(f(VW)g(\Lambda)|W,\Lambda)]\\
&\hspace{-0.7cm}\overset{\text{Independence}}{=}\E\left[\int_{\mathcal U_N(\C)}f(\nu W)g(\Lambda)d\mu_{\mathcal U_N(\C)}(\nu)\right]\\
&\hspace{-0.2cm}\overset{\text{Haar}}{=}\E\left[g(\Lambda)\int_{\mathcal U_N(\C)}f(\nu)d\mu_{\mathcal U_N(\C)}(\nu)\right]\\
&=\int_{\mathcal U_N(\C)}f(\nu)d\mu_{\mathcal U_N(\C)}(\nu)\E(g(\Lambda)).
\end{align*}
It yields that $\Lambda$ and $U$ are independent and that the law of $U$ is the Haar measure in $\mathcal U_N(\C)$.
\end{proof}

\paragraph{Eigenvalues} We consider $(U,\Lambda)$ as in Proposition \ref{prop trigo unitaire diagonale}. The goal is to find the law $\mu_\Lambda$ of $\Lambda$. 
\newline
\tab Fix $U_0\in \mathcal U_N(\C)$, $\Lambda_0$ a diagonal matrix and $M_0\in\mathcal H_N(\C)$ such that $M_0=U_0\Lambda_0 U_0^*$. We suppose that the spectrum of $M_0$ is simple (which shall be almost surely true for a $N$ complex Wigner matrix by Proposition \ref{prop: spectrum simple}). 
\newline
To find the law of the eigenvalues we shall use a change of variables in the neighbourhood of $M_0$. 
\newline
Let us give some notations and basic linear algebra results to do this change of variables. $\mathcal H_N(\C)$ is a $\R$ vectorial space of dimension $N^2$.
$\R^N$ is a $\R$ vectorial space of dimension $N$. In this section, let $\mathcal A_{N,0}$ be the set of anti hermitian matrices of $M_N(\C)$ with null diagonal, it is a $\R$ vectorial space of dimension $N(N-1)$. 
\newline
We define $\phi: \R^N\times A_{N,0}\to \mathcal H_N(\C)$  by: $$\phi(\Lambda, A)=U_0\exp(A)(\Lambda_0+\text{Diag}(\Lambda))\exp(-A)U_0^*, $$ where $\text{Diag}(X)$ is the diagonal matrix such that $\text{Diag}(X)_{i,i}=X_i$ for all $1\le i\le N$.
\newline
We have that $\phi(0,0)=M_0$. We shall compute the differential at (0,0) of $\phi$ and then use implicit function theorem to inverse locally around this point the application $\phi$. We first observe that the dimension as a $\R$ vector space of $\R^N\times A_{N,0}$ is $N^2$ which is the dimension of $\mathcal H_N(\C)$ as a $\R$ vector space which makes sense.
\begin{prop}
$\phi$ is differentiable at the point $0:=(0,0)$ and we have: $$d\phi_{0}(\Lambda, A)=U_0(\emph{Diag}(\Lambda)+A\Lambda_0-\Lambda_0A)U_0^*.$$
\end{prop}
\begin{proof}
Let $\Lambda$ and $A$  small elements of  $\R^N$ and  $A_{N,0}$. We have: 
\begin{align*}
\phi(\Lambda,A)&=U_0\exp(A)(\Lambda_0+\text{Diag}(\Lambda))\exp(-A)U_0^*\\
&=U_0(I_N+A)(\Lambda_0+\text{Diag}(\Lambda))(I_N-A)U_0^*+o((\Lambda,A))\\
&=U_0(\Lambda_0+\text{Diag}(\Lambda)+A\Lambda_0)(I_N-A)U_0^*+o((\Lambda,A))\\
&=U_0(\Lambda_0+\text{Diag}(\Lambda)+A\Lambda_0-\Lambda_0A)U_0^*+o((\Lambda,A))\\
&=\phi(0,0) + U_0(\text{Diag}(\Lambda)+A\Lambda_0-\Lambda_0A)U_0^*+o((\Lambda,A))
\end{align*}
Hence, $\phi$ is differentiable at $(0,0)$ and we have $d\phi_{0}(\Lambda, A)=U_0(\text{Diag}(\Lambda)+A\Lambda_0-\Lambda_0A)U_0^*$.
\end{proof}
Now we need to compute the determinant of the linear map  $d\phi_0$ and prove that this determinant is non null so that by the implicit function theorem we can inverse this map locally around $0$.
\newline
Let us remark that $d\phi_0$ is the composition of two linear maps: $\psi_{U_0}:=\mathcal H_N(\C)\to\mathcal H_N(\C)$ defined by $\psi_{U_0}(X)=U_0 XU_0^*$ and $\Psi_{\Lambda_0}:= \R^N\times A_{N,0}\to \mathcal H_N(\C)$ defined by $\gamma(\Lambda,A)=\text{Diag}(\Lambda)+A\Lambda_0-\Lambda_0A$. The notations $\psi_{U_0}$ and $\Psi_{\Lambda_0}$ will only be used in this section. 
\newline
We already proved that $|\det(\psi_{U_0})|=1$ in Lemma \ref{lemma: det 1}.
It remains to compute the determinant of $\Psi_{U_0}$.
\begin{lemma}
\label{lemma det 2}
The determinant of $\Psi_{\Lambda_0}$ is: $$\det(\Psi_{\Lambda_0})=\prod_{k<j}(\lambda_0(j)-\lambda_0(k))^2=\Delta_N(\Lambda_0)^2,$$
with $\Delta_N(\Lambda_0)=\prod_{k<j}(\lambda_0(j)-\lambda_0(k))$ the Vandermonde determinant.
\end{lemma}
\begin{proof}
This proof is a basic linear algebra argument. Let $(e_k)_{1\le k\le N}$ the canonical basis of $\R^N$, we consider $((\Delta_{k,j}-\Delta_{j,k})_{1\le k<j\le N},  (i(\Delta_{k,j}+\Delta_{j,k}))_{1\le k<j\le N})$ as a basis of $A_{N,0}$ and $((\Delta_{k,k})_{k\le N}, (\Delta_{k,j}+\Delta_{j,k})_{1\le k<j\le N},  (i(\Delta_{k,j}-\Delta_{j,k}))_{1\le k<j\le N})$ a basis of $\mathcal H_N(\C)$.
\newline
First, for all $1\le k\le N$ we have that $\Psi_{\Lambda_0}((e_k,0))=\Delta_{k,k}$.
\newline
Moreover we can compute for $1\le k<j\le N$: $$\Psi_{\Lambda_0}(0, \Delta_{k,j}-\Delta_{j,k})=(\Delta_{k,j}-\Delta_{j,k})\Lambda_0-\Lambda_0(\Delta_{k,j}-\Delta_{j,k})=(\Lambda_0(j)-\Lambda_0(k))(\Delta_{k,j}+\Delta_{j,k}),$$ and $$\Psi_{\Lambda_0}(0, i(\Delta_{k,j}+\Delta_{j,k}))=(i(\Delta_{k,j}+\Delta_{j,k}))\Lambda_0-\Lambda_0(i(\Delta_{k,j}+\Delta_{j,k}))=i(\Delta_{k,j}-\Delta_{j,k})(\Lambda_0(j)-\Lambda_0(k)).$$
In these basis $\Psi_{\Lambda_0}$ is diagonal and we can directly compute its determinant: $$\det(\Psi_{\Lambda_0})=\prod_{k<j}(\lambda_0(j)-\lambda_0(k))^2.$$
\end{proof}
We can now compute the Jacobian of $\phi$ near $0$.
\begin{prop}
\label{prop changement de variable}
$\phi$ is $C^1$ local diffeomorphism around $0$ and $|\det(d\phi_0)|=\Delta_N(\Lambda_0)^2$.
\end{prop}
\begin{proof}
This is an immediate consequence of the fact $d\phi_0=\psi_{U_0}\circ\Psi_{\Lambda_0}$ and Lemma \ref{lemma: det 1} and Lemma \ref{lemma det 2}. Since $\Delta_N(\Lambda_0)\ne 0$ by hypothesis on $M_0$ (simple spectrum), the implicit function theorem yields that $\phi$ is locally invertible around $0$.
\end{proof}

Thanks to this change of variable we can compute the law of the eigenvalues. 
\begin{theorem}
\label{thm:law of spectrum}
The law of the exchangeable spectrum in $\R^N$ of a $N$ complex Wigner matrix has a density with respect to the Lebesgue measure which is proportional to $$\Lambda=(\lambda_1,...,\lambda_N)\mapsto \Delta_N(\Lambda)^2\exp\left(-\cfrac{\Tr(\Lambda^2)}{2}\right)=\prod_{1\le j<k\le N}(\lambda_k-\lambda_j)^2\exp\left(-\cfrac{\sum_{k=1}^N\lambda_k^2}{2}\right).$$
\end{theorem}

\begin{Remark}
More explicitly, the previous theorem said that there exists a constant of normalization $c_N$ such that for every symmetric, bounded or positive measurable functions $F$ we have $$\E(F(\Lambda_1,...,\Lambda_N))=c_N\int_{\R^N}F(x_1,...,x_N)\prod_{i<j}(x_i-x_j)^2\exp\left(-\cfrac{\sum_{i=1}^N|x_i|^2}{2}\right)dx_1...dx_N,$$ if $\Lambda=(\Lambda_1,...,\Lambda_N)$ is the spectrum of a $N$ complex Wigner matrix.
\end{Remark}

\begin{proof}
As explained at the beginning of this section the goal is to find the law $\mu_\Lambda$ on $\R^N$ that matches with Proposition \ref{prop trigo unitaire diagonale}. We look for $\mu_\Lambda$ such that $\mu_\Lambda$ has a density $g_\Lambda$ with respect to the Lebesgue measure on $\R^N$. This density shall be a  symmetric function since we look at the eigenvalues as an exchangeable vector of $\R^N$ and shall vanish when at least two of its arguments are equal since the spectrum of a $N$ complex Wigner matrix is almost surely simple. 
\newline
Let $f$ be a non negative measurable function that only depends on the spectrum of a matrix. We can find $\varepsilon>0$ such that $\phi_\varepsilon: \R^N\cap B(0,\varepsilon)\times A_{N,0}\cap B(0,\varepsilon)\to \mathcal H_N(\C)$ defined by $\phi_\varepsilon(\Lambda,A)=\phi(\Lambda,A)$ is a $C^1$ diffeomorphism onto its image by Proposition \ref{prop changement de variable}.
\newline
We recall by Proposition \ref{prop trigo unitaire diagonale} that $U$ is Haar distributed and independent of $\Lambda$ and is such that $X=U\Lambda U^*$ is a $N$ complex Wigner matrix.
Now we compute: 
\begin{equation}
\begin{split}
\E\left[f(\Lambda)\mathds{1}_{||\Lambda-\Lambda_0||\le\varepsilon}\right]&=\E\left[f(\Lambda)\mathds{1}_{||\Lambda-\Lambda_0||\le\varepsilon}\right]\cfrac{\E(\mathds{1}_{U\in\{U_0\exp(A)\,|\, A\in A_{N,0}\cap B(0,\varepsilon)\}})}{\mP(U\in\{U_0\exp(A)\,|\, A\in A_{N,0}\cap B(0,\varepsilon)\})}\\
&\hspace{-0.7cm}\overset{\text{Independance}}{=}c_\varepsilon\E(f(X)\mathds{1}_{X\in Im(\phi_\varepsilon)}),
\end{split}
\end{equation}
with $c_\varepsilon^{-1}:=\mP(U\in\{U_0\exp(A)\,|\, A\in A_{N,0}\cap B(0,\varepsilon)\})\overset{\text{Haar}}{=}\mP(U\in\{\exp(A)\,|\, A\in A_{N,0}\cap B(0,\varepsilon)\})$.
\newline
Now we use the law of $X$ to compute this term. It gives: $$\E\left[f(\Lambda)\mathds{1}_{||\Lambda-\Lambda_0||\le\varepsilon}\right]=\frac{c_\varepsilon}{Z_N}\int_{\mathcal H_N(\C)}f(X)\mathds{1}_{X\in Im(\phi_\varepsilon)}\exp(-\Tr(X^2)/2)dX, $$ with $Z_N$ the constant of normalization of the law of a $N$ complex Wigner matrix.
\newline
Now we can do the change of variable to have: 
$$\E\left[f(\Lambda)\mathds{1}_{||\Lambda-\Lambda_0||\le\varepsilon}\right]=c_\varepsilon'\int_{\R^N\cap B(0,\varepsilon)\times A_{N,0}\cap B(0,\varepsilon)} f(\Lambda+\Lambda_0)e^{-\Tr(\Lambda+\Lambda_0)^2/2}|\det(\phi_{\Lambda,A})|d\Lambda d\leb(A),$$ with $c'_\varepsilon:=c_\varepsilon/Z_N$.
\newline
To find the exact expression of $g_\Lambda$ we shall let $\varepsilon$ goes to $0$. 
\newline
On one side we have by definition of $g_\Lambda$: 
\begin{equation}
\label{calc:densieigenvalues1}
\E\left[f(\Lambda)\mathds{1}_{||\Lambda-\Lambda_0||\le\varepsilon}\right]=\int_{||\Lambda-\Lambda_0||\le\varepsilon}f(\Lambda)g_\Lambda(\Lambda)d\leb(\Lambda)\underset{\varepsilon\to 0}{\sim} \widetilde {C_N}\, \varepsilon^N f(\Lambda_0)g_\Lambda(\Lambda_0),
\end{equation}
with $\widetilde {C_N}$ a constant that only depend on $N$.
\newline
On the other hand, we have: 
\begin{equation}
\label{calc:densieigenvalues2}
\begin{split}
c_\varepsilon'\int_{\R^N\cap B(0,\varepsilon)\times A_{N,0}\cap B(0,\varepsilon)} &f(\Lambda+\Lambda_0)e^{-\Tr(\Lambda+\Lambda_0)^2/2}|\det(\phi_{\Lambda,A})|d\Lambda d\leb(A)\underset{\varepsilon\to 0}{\sim}\\
&\hspace{-1cm}c_\varepsilon'\,f(\Lambda_0)\exp\left(-\cfrac{\Tr(\Lambda_0)^2}{2}\right)|\det(\phi_{0,0})|\widetilde {C_N'}\,\varepsilon^N\leb(A_{N,0}\cap B(0,\varepsilon))=:\\
&\hspace{3cm}\widetilde {C'_{N,\varepsilon}}\,\varepsilon^N \,f(\Lambda_0)\exp\left(-\cfrac{\Tr(\Lambda_0)^2}{2}\right)\Delta_N(\Lambda_0)^2,
\end{split}
\end{equation}
where $\widetilde {C'_{N,\varepsilon}}$ is a constant that only depend on $N$ and $\varepsilon$.
\newline
Comparing \eqref{calc:densieigenvalues1} and \eqref{calc:densieigenvalues2}, since the result holds for every non positive measurable functions $f$, we deduce that $\widetilde {C'_{N,\varepsilon}}$ necessarily converges when $\varepsilon$ goes to 0 to a certain quantity that only depends on $N$ and we have that the density $g_\Lambda$ is proportional to $$\Lambda_0\mapsto \exp\left(-\cfrac{\Tr(\Lambda_0)^2}{2}\right)\Delta_N(\Lambda_0)^2.$$

\end{proof}

\subsubsection{Ginibre case}

We first recall the Schür decomposition that can be viewed as the counterpart for matrices of $M_N(\C)$ of the spectral theorem for Hermitian matrices. 
\begin{prop}[Schür decomposition]
Let $M\in M_N(\C)$, there exists $U\in \mathcal U_N(\C)$, $N\in M_N(\C)$ a nilpotent upper triangular matrix and $D\in D_N(\C)$ a diagonal matrix such that $M=U(D+N)U^*$. Moreover we have that $\Tr(MM^*)=\Tr(DD^*)+\Tr(NN^*)$.
\end{prop}
To prove the theorem we can triangularize a matrix in $M_N(\C)$ in a certain base and then orthogonalize this base using the Gram Schmidt algorithm.
\newline
Thanks to the Schür decomposition we can do a similar proof to compute the density of the eigenvalues of a $N$ complex Ginibre matrix as we did for a $N$ complex Wigner matrix.

\begin{theorem}
\label{Thmginibrelaw}
The law of the spectrum in $\C^N$ of a random matrix distributed according to a $N$ complex Ginibre matrix has a density with respect to the complex Lebesgue measure proportional to $$(z_1,...,z_N)\mapsto\exp(-\sum_{i=1}^N|z_i|^2)\prod_{i<j}|z_i-z_j|^2.$$
\end{theorem}

\subsection{Determinental structure}
Viewing the eigenvalues of $N$ complex Ginbire or Wigner matrices as point processes in $\C$ or $\R$, we shall compute the correlation functions of these point processes. These computations are based on the remark that both densities of the eigenvalues of the two models can be written in a so-called determinental form that make the computations easier using an orthogonal family of polynomials. For the links between random matrices and point processes see for instance \cite{Alicelivre,hough2009zeros,compactgroup,Mehta}.
\newline
 We first detail completely the Ginibre case and then state the same property giving some proofs for the Wigner case.
\subsubsection{Ginibre case}
We recall that the Vandermonde determinant of $(\mu_1,...,\mu_N)$ is given by $\Delta_N(\mu_1,...,\mu_N)=\prod_{k<j}(\mu_j-\mu_k)=\det[\mu_j^{k-1}]_{1\le j,k\le N}$.
\newline
\tab Let $c_N$ be the constant of normalization for the density of the eigenvalues of a matrix distributed according to a $N$ complex Ginibre matrix given in Theorem \ref{Thmginibrelaw}. We first compute its value: 
\begin{align*}
c_N^{-1}&=\int_{\C^N}\exp(-\sum_{i=1}^N|z_i|^2)\prod_{k<j}|z_k-z_j|^2 dz_1...dz_N\\
&=\int_{\C^N}\exp(-\sum_{i=1}^N|z_k|^2)|\Delta_N(z_1,...,z_N)|^2 dz_1...dz_N\\
&=\int_{\C^N}\exp(-\sum_{i=1}^N|z_k|^2)|\det[z_j^{k-1}]_{1\le j,k\le N}|^2 dz_1...dz_N.
\end{align*}
Using the multi linearity of the determinant and a family of $N$ monic polynomial $(P_k)_{0\le k\le N-1}$ such that for all $0\le k\le N-1$ the degree of $P_k$ is $k$, we can write: 
\begin{equation}
\label{comp: computation constant}
\begin{split}
c_N^{-1}&=\int_{\C^N}\exp(-\sum_{i=1}^N|z_i|^2)|\Delta_N(z_1,...,z_N)|^2 dz_1...dz_N\\
&=\int_{\C^N}\exp(-\sum_{i=1}^N|z_i|^2)|\det[P_{k-1}(z_j)]_{1\le j,k\le N}|^2 dz_1...dz_N\\
&=\int_{\C^N}\exp(-\sum_{i=1}^N|z_i|^2)\det[P_{k-1}(z_j)]_{1\le j,k\le N} \,\,\overline{\det[P_{k-1}(z_j)]_{1\le j,k\le N}}dz_1...dz_N\\
&=\sum_{\sigma,\sigma'\in\mathfrak{S}_N}(-1)^{\varepsilon(\sigma)+\varepsilon(\sigma')}\int_{\C^N}\prod_{i=1}^N\exp(-|z_i|^2)\prod_{i=1}^N P_{\sigma(i)-1}(z_i)\overline{P_{\sigma'(i)-1}(z_i)}\,dz_1...z_N \\
&\hspace{-0.3cm}\overset{\text{Fubini}}{=}\sum_{\sigma,\sigma'\in\mathfrak{S}_N}(-1)^{\varepsilon(\sigma)+\varepsilon(\sigma')}\prod_{i=1}^N\int_{\C}\exp(-|z|^2) P_{\sigma(i)-1}(z)\overline{P_{\sigma'(i)-1}(z)}\,dz.
\end{split}
\end{equation}
We see that a good idea would be to choose a family of polynomials that are orthogonal for the scalar product $\langle f|g\rangle=\int_{\C}\exp(-|z|^2)\overline{f(z)}g(z)dz$. 
A such family can be found explicitly.
\begin{lemma}
\label{orthogonality}
Let $\gamma(z):=\pi^{-1}\exp(-|z|^2)$. Then for all $k,l\in\N$, we have: $$\int_\C\cfrac{z^k}{\sqrt{k!}}\cfrac{\overline{z}^l}{\sqrt{l!}}\,\gamma(z)dz=\mathds{1}_{k=l}.$$
\end{lemma}

\begin{proof}
To do these computations, we use the polar coordinate. Indeed, we have:
\begin{align*}
\int_\C z^k\overline{z}^l\gamma(z)dz&=\cfrac{1}{\pi}\int_{\R^2}(x+iy)^k(x-iy)^l\exp(-x^2-y^2)dxdy\\
&=\cfrac{1}{\pi}\int_{\R^+}\int_{[0,2\pi]}(r(\cos(\theta)+i\sin(\theta)))^k(r(\cos(\theta)-i\sin(\theta)))^l\exp(-r^2)rdrd\theta\\
&=\cfrac{1}{\pi}\int_{\R^+}r^{k+l}r\exp(-r^2)dr\int_{[0,2\pi]}\exp(i\theta(k-l))d\theta
\end{align*}
We observe that the angular integral is null if $k\ne l$ and is equal to $2\pi$ if $k=l$. When $k=l$ the radial part can be computed without difficulty by integrating by parts and is equal to $$\int_{\R^+}r^{2k+1}\exp(-r^2)dr=\cfrac{k!}{2}.$$
It yields: $$\int_\C\cfrac{z^k}{\sqrt{k!}}\cfrac{\overline{z}^l}{\sqrt{l!}}\gamma(z)dz=\mathds{1}_{k=l}.$$
\end{proof}

\begin{Remark}
\label{remark:ortho}
For $j\ge 0$, if $\mu_j(z):=\sqrt{\gamma(z)}z^j/j!$, Lemma \ref{orthogonality} says that these functions are orthonormal in $L_2(\C)$ for the Lebesgue measure.
\end{Remark}
Now we can continue the computation \eqref{comp: computation constant} with the specific polynomials $P_j(X)=X^j$ for all $j\in\N$. 
We have: 
\begin{equation}
\label{calculsonstantenorm}
\begin{split}
c_N^{-1}&=\sum_{\sigma,\sigma'\in\mathfrak{S}_N}(-1)^{\varepsilon(\sigma)+\varepsilon(\sigma')}\prod_{i=1}^N\int_{\C}\exp(-|z|^2) P_{\sigma(i)-1}(z)\overline{P_{\sigma'(i)-1}(z)}dz\\
&=\sum_{\sigma\in\mathfrak{S}_N}\prod_{i=1}^N\int_{\C}\exp(-|z|^2) P_{\sigma(i)-1}(z)\overline{P_{\sigma(i)-1}(z)}dz\\
&=\sum_{\sigma\in\mathfrak{S}_N}\prod_{i=1}^N\pi (\sigma(i)-1)!=\pi^N N!\prod_{k=0}^{N-1}k!.
\end{split}
\end{equation}
Hence, we can now give the complete expression of the density of the eigenvalues of a $N$ complex Ginibre matrix.

\begin{theorem}
\label{thm : law spectrum ginibre}
The law of the spectrum in $\C^N$ of a random matrix distributed according to a $N$ complex Ginibre matrix has a density with respect to the complex Lebesgue measure which is equal to: 
\begin{equation}
\label{eq:laweigenvaluesconstant}
\phi_N(z_1,...,z_N):=\cfrac{\prod_{k=1}^N\gamma(z_k)}{\prod_{k=1}^N k!}\prod_{i<j}|z_i-z_j|^2,
\end{equation} 
with $\gamma(z):=\exp(-|z|^2)\pi^{-1}$.
\end{theorem}

We can use the Vandermonde formula to give another form of the formula \eqref{eq:laweigenvaluesconstant}.  
Indeed we have: 
\begin{equation}
\label{eq: vandermonde formula ginibre}
\begin{split}
\phi_N(z_1,...,z_N)&=\cfrac{\prod_{k=1}^N\gamma(z_k)}{\prod_{k=1}^N k!}\Delta_N(z_1,...,z_N)\Delta_N(\overline{z_1},...,\overline{z_N})\\
&=\cfrac{\prod_{k=1}^N\gamma(z_k)}{\prod_{k=1}^N k!}\det[z_i^{j-1}]_{1\le i,j\le N}\det[\overline {z_i^{j-1}}]_{1\le i,j\le N}\\
&=\cfrac{\prod_{k=1}^N\gamma(z_k)}{\prod_{k=1}^N k!}\det[z_i^{j-1}]_{1\le i,j\le N}\det[\overline {z_j^{i-1}}]_{1\le i,j\le N}\\
&=\cfrac{\prod_{k=1}^N\gamma(z_k)}{N!}\det\left[\cfrac{z_i^{j-1}}{\sqrt{(j-1)!}}\right]_{1\le i,j\le N}\det\left[\cfrac{\overline {z_j^{i-1}}}{{\sqrt{(i-1)!}}}\right]_{1\le i,j\le N}\\
&=\cfrac{1}{N!}\det[K_N(z_i,z_j)]_{1\le i,j\le N},
\end{split}
\end{equation}
with for all $z\in\C$ and $w\in \C$: $$K_N(z,w):=\sqrt{\gamma(z)\gamma(w)}\sum_{l=0}^{N-1}\cfrac{(z\overline w)^l}{l!}.$$
$K_N$ shall be called the kernel of the point process of the eigenvalues of a $N$ complex Ginibre matrix.
\newline
Let us give some definitions and basic properties about point processes. 
\begin{definition}
The set of locally finite part of $\C$ is the subset of the part of $\C$ defined by: $$P_{lf}^\C=\{Z\subset \C,\, |Z\cap K|<+\infty \,\forall K \text{ compact of }\C\}.$$
We endow $P_{lf}^\C$ with a sigma algebra $\mathcal F_{lf}^\C$ which is the smallest sigma algebra that makes measurable the maps $\phi_{A}: P_{lf}^\C\to (\N,\mathcal P(\N))$ such that $\phi_A(Z)=|Z\cap A|$ for $A$ compact subset of $\C$.
\end{definition}
We can now define a point process.
\begin{definition}
$Z$ is said to be a point process if it is a random variable with value in $(P_{lf}^\C,\mathcal F_{lf}^\C)$.
\end{definition}
\begin{Remark}
A function $Z$ defined on a probability space with value in $(P_{lf}^\C,\mathcal F_{lf}^\C)$ is a random variable if and only if $|Z\cap A|$ is a $(\N,\mathcal P(\N))$ valued random variable for all $A$ compact subset of $\C$. The law of $Z$ is determined by the law of $(|Z\cap A_1|,...,|Z\cap A_k|)$ for all $(A_i)_{i\le k}$ compact subsets of $\C$ for all $k\in\N$. By some standard arguments as the Dynkin lemma we can restrict the previous fact on disjoint compact subsets of $\C$.
\end{Remark}

Let us give a standard example of point process.

\begin{example}
\label{example poisson point process}
The Poisson point process of parameter $\lambda>0$ is the point process such that for all $k\ge 1$ for all $(A_i)_{1\le i\le k}$ disjoint compact subsets of $\C$ we have that: $(|Z\cap A_1|,...,|Z\cap A_k|)\equalaw \emph{Poisson}(\lambda \leb(A_1))\otimes...\otimes \emph{Poisson}(\lambda \leb(A_k)),$  where in this example $\emph{Poisson}(\nu)$ is a standard Poisson random variable of parameter $\nu>0$. 
\end{example}

\begin{definition}
\label{definition:correlation}
We say that the correlation functions of a point process $Z$ are $(\phi_k)_{k\ge 1}$ if for all $k\ge 1$, $\phi_k:\C^k\to\R^+$ is such that for all $B_1,...,B_k$ disjoint compact subsets of $\C$ we have: $$\E(|Z\cap B_1|...|Z\cap B_k|)=\int_{B_1\times...\times B_k} \phi_k(z_1,...,z_k)dz_1...dz_k.$$
\end{definition}
For the Example \ref{example poisson point process} we can compute explicitly the correlation functions.
\begin{example}
For a Poisson point process of parameter $\lambda>0$ the correlation functions are given by: 
$\phi_k(z_1,...,z_k)=\lambda^k$ for all $k\ge 1$.
\end{example}

\begin{Remark}
Pay attention to the fact that correlation functions are not sure to exist. For instance, let $Z=\{0\}$ then $\E(|Z\cap\{0\}|)=1\ne\int_{\{0\}}\phi_1(z)dz=0$.
\end{Remark}
When the point process is given by a random variable in $\C^N$ with a symmetric density with respect to the Lebesgue measure one can express the correlation functions of this point process. 
\begin{prop}
\label{prop: correlation density}
Let $Z=(X_1,...,X_N)$ be $N$ complex random variables such that the law of $(X_1,...X_N)$ has a density with respect to the Lebesgue measure $f(z_1,...,z_N)dz_1...dz_N$ that we suppose symmetric, then the correlation functions of $Z$ are given by $\psi_k=0$ if $k>N$ and if $k\le N$, $$\label{form:correlationcomplexe}\psi_k(z_1,...,z_k)=\cfrac{N!}{(N-k)!}\int_{\C^{N-k}}f(z_1,...,z_N)dz_{k+1}...dz_N.$$
\end{prop}

\begin{proof}
Let $B_1,..,B_k$ be disjoint compact subsets of $\C$, we compute the correlations functions: 
\begin{align*}
\E(|Z\cap B_1|,...,|Z\cap B_k|)&=\E\left[\sum_{i_1=1}^N\mathds{1}_{X_{i_1}\in B_{1}}...\sum_{i_k=1}^N\mathds{1}_{X_{i_k}\in B_{k}}\right]\\
&=\sum_{i_1,...,i_k=1}^N\E\left[\mathds{1}_{X_{i_1}\in B_{1}}...\mathds{1}_{X_{i_k}\in B_{k}}\right].
\end{align*}
Since the $B_i$ are disjoint subsets, an $X_k$ can not be in two different $B_i$. Hence, we can rewrite the sum and then use the symmetry of the law to have: 
\begin{equation}
\label{eq corr density}
\begin{split}
\E(|Z\cap B_1|...|Z\cap B_k|)&=\sum_{i_1,...,i_k \in \{1,...N\}^k \text{ disjoint}}\E\left[\mathds{1}_{X_{i_1}\in B_{1}}...\mathds{1}_{X_{i_k}\in B_{k}}\right]\\
&=\sum_{i_1,...,i_k \in \{1,...N\}^k \text{ disjoint}}\E\left[\mathds{1}_{X_{1}\in B_{1}}...\mathds{1}_{X_{k}\in B_{k}}\right].
\end{split}
\end{equation}
Hence, by \eqref{eq corr density} $\E(|Z\cap B_1|...|Z\cap B_k|)$ is equal to $0$ if $k>N$ (we could directly have said it because by definition of $Z$ if $k>N$ we necessarily have a set $B_i$ with no points of $Z$ and so almost surely $|Z\cap B_1|...|Z\cap B_k|=0$ and so its expectation is also $0$) and otherwise is equal to: $$\cfrac{N!}{(N-k)!}\int_{B_1\times...\times B_k}\left({\int_{\C^{N-k}}}f(z_1,...,z_N)dz_{k+1}...dz_N\right)dz_1,...dz_k. $$
It proves the result.
\end{proof}
We recall that in the case of $N$ complex Ginibre matrices we are interested in the point process of its eigenvalues $Z=\{\lambda_1,...,\lambda_N\}$ for which the density is the symmetric function obtained in Theorem \ref{thm : law spectrum ginibre} and that was written in \eqref{eq: vandermonde formula ginibre} as $$\phi_N(z_1,...,z_N)=\cfrac{1}{N!}\det[K_N(z_i,z_j)]_{1\le i,j\le N}.$$
Thanks to the orthogonality property of the monic polynomials, $K_N$ shall satisfied a semi group property making the computations of the correlation functions easier.

\begin{lemma}
\label{lemma:semigroupprop}
For all $N\ge 1$ and all $x,z\in\C$, we have: $$\int_{\C}K_N(z,z)dz=N,\, \,\int_{\C}K_N(x,y)K_N(y,z)dy=K_N(x,z).$$
\end{lemma}

\begin{proof}
This is a quite immediate consequence of the property of orthogonality of the monic polynomial for the measure $\gamma(z)dz$. 
\newline
Let for all $j\ge 0$, let $\mu_j(z)=\sqrt{\gamma(z)}z^j/j!$. These functions are orthonormal in $L_2(\C)$ for the Lebesgue measure by Lemma \ref{orthogonality}.
\newline
Hence we have that: $$\int_\C K_N(z,z)dz=\sum_{l=0}^{N-1}\int_\C|\mu_j(z)|^2dz=N$$ and: 
\begin{align*}
\int_{\C}K_N(x,y)K_N(y,z)dy&=\sum_{k,j=0}^{N-1}\int_{\C}\mu_j(x)\overline{\mu_j(y)}\mu_k(y)\overline{\mu_k(z)}dy\\
&=\sum_{k,j=0}^{N-1}\mu_j(x)\overline{\mu_k(z)}\mathds{1}_{j=k}\\
&=\sum_{k,j=0}^{N-1}\mu_k(x)\overline{\mu_k(z)}=K_N(x,z)
\end{align*}
\end{proof}
Thanks to this lemma we can now prove the following inductive formula that shall give the explicit value of correlation functions obtained in Proposition \ref{prop: correlation density}. 

\begin{prop}
\label{prop:fonctioncorredet}
For all $k<N$ and $z_1,...,z_k\in\C$ we have the following inductive formula: $$
\int_\C \det[K_N(z_i,z_j)]_{1\le i,j\le k+1}dz_{k+1}=(N-k)\det[K_N(z_i,z_j)]_{1\le i,j\le k}.$$
As a corollary, we have that for all $k\le N$, $$\int_{\C^{N-k}}\det[K_N(z_i,z_j)]_{1\le i,j\le N}dz_{k+1}...dz_N=(N-k)! \det[K_N(z_i,z_j)]_{1\le i,j\le k}.$$
\end{prop}
\begin{proof}
We expand the determinant along the last column and we write $A_{k+1}^{l,k+1}$ the matrix obtain from $[K_N(z_i,z_j)]_{1\le i,j\le k+1}$ by removing the line $l$ and the column $k+1$:
\begin{align*}
\det[K_N(z_i,z_j)]_{1\le i,j\le k+1}=\sum_{l=1}^{k+1}(-1)^{l+k+1}K_N(z_l,z_{k+1})\det (A_{k+1}^{l,k+1})
\end{align*}
We shall distinguish two cases. First when $l=k+1$ we have that $A_{k+1}^{k+1,k+1}=[K_N(z_i,z_j)]_{1\le i,j\le k}$ and so this term does not depend on $z_{k+1}$. Hence, when we integrate the term of the sum that corresponds to $l=k+1$ against $dz_{k+1}$, we only have $$\det[K_N(z_i,z_j)]_{1\le i,j\le k}\int_\C K_N(z_{k+1},z_{k+1})dz_{k+1}=N\det[K_N(z_i,z_j)]_{1\le i,j\le k},$$ thanks to Lemma \ref{lemma:semigroupprop}.
For all the other terms, we again expand the determinant along the last line to have that: $$ \det (A_{k+1}^{l,k+1})=\sum_{p=1}^{k}(-1)^{p+k}K_N(z_{k+1},z_p)\det[K_N(z_i,z_j)]_{1\le i,j\le k, i\ne l, j\ne p}.$$
We notice that in the sum the last determinant does not depend on $z_{k+1}$ so when we integrate with respect to $z_{k+1}$ the only important contribution in this expression is $K_N(z_{k+1},z_p)$.
\newline
So using Lemma \ref{lemma:semigroupprop} again we have that for all $l\le k$: 
\begin{align*}
&\int_\C (-1)^{l+k+1}K_N(z_l,z_{k+1})\det (A_{k+1}^{l,k+1})dz_{k+1}=\\
&\sum_{p=1}^k(-1)^{l+p-1}K_N(z_l,z_p)\det[K_N(z_i,z_j)]_{1\le i,j\le k, i\ne l, j\ne p}.
\end{align*}
We recognize that the last sum is: $$-\det[K_N(z_i,z_j)]_{1\le i,j\le k}$$
Since we sum all these terms for $l\le k$, we get that: 
$$\det[K_N(z_i,z_j)]_{1\le i,j\le k+1}=(N-k) \det[K_N(z_i,z_j)]_{1\le i,j\le k}.$$
The last point is clear by induction. 
\end{proof}

Let us also remark that for $k>N$ we have that: $$[K_N(z_i,z_j)]_{1\le i,j\le k}=[(\mu_l(z_i)]_{1\le i\le k,1\le l\le N}[\overline{\mu_i(z_l)}]_{1\le i\le k,1\le l\le N}$$ with the $\mu_j$ which are the orthogonal elements of $L_2(\C)$ introduced in Remark \ref{remark:ortho}. Since the matrix $[(\mu_l(z_i)]_{1\le i\le k,1\le l\le N}$ has $N$ columns the rank of $[K_N(z_i,z_j)]_{1\le i,j\le k}$ is smaller than $N$ and so strictly smaller than $k$. Hence, the determinant of $[K_N(z_i,z_j)]_{1\le i,j\le k}$ is equal to 0 if $k>N$. 
\newline
Let us summarize the results obtained for the correlation functions of the eigenvalues of $N$ complex Ginibre matrices.
\begin{definition}
A determinental point process of kernel $K:\,\C^2\to \C$ is a point process such that its correlation functions are of the form: $$\psi_k(z_1,...,z_k)=\det[K(z_i,z_j)]_{1\le i,j\le k},\, \text{for almost all  } (z_1,...,z_k)\in\C^k,\,\forall k\in\N_{\ge 1}.$$
\end{definition}

\begin{example}
A Poisson point process of parameter $\lambda>0$ is a determinental point process with kernel $K(z,z')=\lambda \mathds 1_{z=z'}$.
\end{example}

\begin{theorem}
\label{thm ginibre dpp}
The eigenvalues of a $N$ complex Ginibre matrix is a determinental point process of kernel $K_N$.
\end{theorem}

The next theorem gives the formulas obtained in this section and the notations we shall use in Section \ref{Section:circularlaw}.

\begin{theorem}
\label{thm: fonction correlation ginibre}
For all $N\ge 1$, the density with respect to the complex Lebesgue measure of the eigenvalues of a $N$ complex Ginibre matrix is given by: $$ \phi_N(z_1,...,z_N):=\cfrac{1}{N!}\det[K_N(z_i,z_j)]_{1\le i,j\le N},$$
with $K_N$ the kernel defined by: $$K_N(z,w):=\sqrt{\gamma(z)\gamma(w)}\sum_{l=0}^{N-1}\cfrac{(z\overline w)^l}{l!}, \,\,\forall(z,w)\in\C^2.$$
Moreover we have an expression of the marginal densities of this law defined for $1\le k\le N$ by: $$(z_1,...z_k)\in\C^k\mapsto\phi_{N,k}(z_1,...z_k):=\int_{\C^{N-k}}\phi_{N}(z_1,...,z_N)dz_{k+1}...dz_N=\cfrac{(N-k)!}{N!}\psi_k(z_1,...z_k),$$
with $\psi_k$ the correlation functions of the point process of the eigenvalues of a $N$ complex Ginibre matrix introduced in Definition \ref{definition:correlation}.
\newline
Indeed for all $1\le k\le N$, and for all $(z_1,...,z_k)\in\C^k$, we have: $$\phi_{N,k}(z_1,...z_k)=\cfrac{(N-k)!}{N!}\det[K_N(z_i,z_j)]_{1\le i,j\le k}.$$
In particular, we have the following explicit formulas for the marginal $\phi_{N,N}=\phi_N$ and for $k=1$, for all $z\in\C$, $$\phi_{N,1}(z)=\cfrac{\gamma(z)}{N}\sum_{l=0}^{N-1}\cfrac{|z|^{2l}}{l!}.$$
\end{theorem}

\subsubsection{Gaussian unitary ensemble case}
As for the Ginibre case, we introduce the weight $\gamma_H(x)=\exp(-x^2/2)$ that appears in the formula of the density of the eigenvalues in Theorem \ref{thm:law of spectrum}. As in the Ginibre case, the key point is to find a family of orthogonal polynomials for the weight $\gamma_H(x)$. These polynomials are called Hermite polynomials and can be deduced from the so-called Rodrigues formula \cite{koornwinder2013orthogonal}.
\begin{definition}[Hermite polynomials]
For all $k\ge 0$ we define the $k^{th}$ Hermite polynomial by $$H_k(x)=(-1)^n\exp(x^2/2)\left(\cfrac{d}{dx}\right)^k\left[\exp(-x^2/2)\right].$$
\end{definition}
Let us mention that a priori it is not clear that $H_k$ is a polynomial. The main important properties that we shall use about Hermite polynomials are given in the next proposition. 
\begin{prop}
\label{prop properties hermite}
The Hermite polynomials satisfy the following properties.
\begin{enumerate}
\item{For all $k\ge0$, $H_k$ is a monic polynomial of degree $k$ and has the same parity as $k$.}
\item{They satisfy the following orthogonal property: for all $m,n\in\N$ 
\begin{equation}
\label{ortho hermite}\int_\R H_m(x)H_n(x)\exp(-x^2/2)dx=\sqrt{2\pi}\,n!\mathds{1}_{m=n}.
\end{equation}}
\item{For all $k\ge 1$ we have the following inductive formula: 
\begin{equation}
\label{induction hermite}
H_{k+1}(X)=XH_k(X)-kH_{k-1}(X).
\end{equation}}
\item{For all $k\ge1$, we have $H_k'=kH_{k-1}$ and $H_k$ satisfies the following differential equation: 
\begin{equation}
\label{equadiff hermite}H_k''-xH_k'+kH_k=0.
\end{equation}}
\item{Hermite polynomials satisfy the so-called Christoffel–Darboux formula: for all $n\ge 1$, for all $x\ne y$
\begin{equation}
\label{critophel hermite}\sum_{k=0}^{n-1}\cfrac{H_k(x)H_k(y)}{k!}=\cfrac{H_n(x)H_{n-1}(y)-H_{n-1}(x)H_n(y)}{(x-y)(n-1)!}.
\end{equation}}
\end{enumerate}
\end{prop}
Let us mention that these properties are general for families of orthogonal polynomials.
\begin{proof}
We sketch the proof since the arguments are quite straightforward. The first point can be easily proved by induction. 
\newline
We detail the second point since it is the main reason for the use of Hermite polynomial in our context. 
\newline
First suppose that $m>n$ then $$\int_\R H_m(x)H_n(x)\exp(-x^2/2)dx=(-1)^m\int_\R \left(\cfrac{d}{dx}\right)^m\left[\exp(-x^2/2)\right]H_n(x)dx.$$
Doing $m$ integrations by parts yields $$(-1)^m\int_\R \left(\cfrac{d}{dx}\right)^m\left[\exp(-x^2/2)\right]H_n(x)dx=0$$ since the $m^{th}$ derivative of $H_n$ is 0 as $H_n$ is a polynomial of degree $n$ (let us also mention that at each integration by parts the bracket term vanishes since for all polynomial $P\in\R[X]$, $\exp(-x^2/2)P(x)\underset{|x|\to+\infty}{\longrightarrow}0$).  
\newline
In the case $m=n$, doing $n$ integrations by parts as before yields 
\begin{align*}
\int_\R H_n(x)H_n(x)\exp(-x^2/2)dx&=(-1)^n\int_\R \left(\cfrac{d}{dx}\right)^n\left[\exp(-x^2/2)\right]H_n(x)dx\\
&=(-1)^n (-1)^n\int_\R \exp(-x^2/2)\left(\cfrac{d}{dx}\right)^n\left[H_n(x)\right]dx\\
&=n!\int_\R \exp(-x^2/2)dx=\sqrt{2\pi}n!
\end{align*}
This gives the second point. 
\newline
For the third point we can notice that $\left[\exp(-x^2/2)\right]'=-x\exp(-x^2/2)$ and differentiating $n$ times this relation, using the Leibniz rule for the right term gives the result.
\newline
For the point $4$, by definition of $H_k$ we have $H_k'(X)=X H_k(X)-H_{k+1}(X)$ which gives the result using \eqref{induction hermite}. The formula \eqref{equadiff hermite} is a consequence of the fact that $H_k'=kH_{k-1}$ and the formula \eqref{induction hermite}.
\newline
Finally, the formula $\eqref{critophel hermite}$ can be directly proved by induction.
\end{proof} 

Thanks to these properties and especially the orthogonality of Hermite polynomials we can do the same computations as in \eqref{comp: computation constant} and \eqref{calculsonstantenorm}. We obtain the counterpart of Theorem \ref{thm : law spectrum ginibre}.

\begin{theorem}
\label{thm law eigenvalues guen}
The law of the spectrum in $\R^N$ of a random matrix distributed according to a $N$ complex Wigner matrix has a density with respect to the complex Lebesgue measure which is equal to: $$\phi_N^H(x_1,...,x_N):=\cfrac{\prod_{k=1}^N\gamma_H(x_k)}{(2\pi)^{N/2}\prod_{k=1}^N k!}\prod_{i<j}(x_i-x_j)^2,$$ with $\gamma_H(x):=\exp(-x^2/2)$.
\end{theorem}
As for the Ginibre case in \eqref{eq: vandermonde formula ginibre}, this expression is called determinental because thanks to the Vandermonde formula we can give another expression of the density.  
Indeed we have: 
\begin{equation}
\begin{split}
\phi_N^H(x_1,...,x_N)&=\cfrac{\prod_{k=1}^N\gamma_H(x_k)}{(2\pi)^{N/2}\prod_{k=1}^N k!}\Delta_N(x_1,...,x_N)\Delta_N(x_1,...,x_N)\\
&=\cfrac{\prod_{k=1}^N\gamma_H(x_k)}{(2\pi)^{N/2}\prod_{k=1}^N k!}\det[x_i^{j-1}]_{1\le i,j\le N}\det[x_i^{j-1}]_{1\le i,j\le N}\\
&=\cfrac{\prod_{k=1}^N\gamma_H(x_k)}{(2\pi)^{N/2}\prod_{k=1}^N k!}\det[x_i^{j-1}]_{1\le i,j\le N}\det[x_j^{i-1}]_{1\le i,j\le N}\\
&=\cfrac{\prod_{k=1}^N\gamma_H(x_k)}{(2\pi)^{N/2}\prod_{k=1}^N k!}\det[H_{j-1}(x_i)]_{1\le i,j\le N}\det[H_{i-1}(x_j)]_{1\le i,j\le N}\\
&=\cfrac{1}{N!}\det[K_N^H(x_i,x_j)]_{1\le i,j\le N},
\end{split}
\end{equation}
with for all $x\in\R$ and $y\in \R$: $$K_N^H(x,y):=\sqrt{\gamma_H(x)\gamma_H(y)}\sum_{l=0}^{N-1}\cfrac{H_{l}(x)H_l(y)}{\sqrt{2\pi}\, l!}.$$

We give the exact same definition of locally finite part in $\R$ that we did in $\C$ for the Ginibre case. We are not very exhausting since it has already been discussed in the part about the Ginibre case.
\begin{definition}
We call (real) locally finite part the subset of the part of $\R$: $$P_{lf}^{\R}=\{Z\subset \R,\, |Z\cap K|<+\infty \,\forall K \text{ compact of }\R\}.$$
We endow $P_{lf}^{\R}$ with a sigma algebra $\mathcal F_{lf}^{\R}$ which is the smallest sigma algebra that makes measurable the maps $\phi_{A}: P_{lf}\to (\N,\mathcal P(\N))$ such that $\phi_A(Z)=|Z\cap A|$ for $A$ compact subset of $\R$.
\end{definition}

\begin{definition}
$Z$ is said to be a (real) point process if it is a random variable with value in $(P_{lf}^{\R},\mathcal F_{lf}^{\R})$.
\end{definition}

\begin{definition}
\label{def corellation gue}
We say that the correlation functions of a point process are $\phi_k$ for $k\ge 1$ if $\phi_k:\R^k\to\R^+$ is such that for all $B_1,...,B_k$ disjoint compact subsets of $\R$ we have: $$\E(|Z\cap B_1|...|Z\cap B_k|)=\int_{B_1\times...\times B_k} \phi_k(x_1,...,x_k)dx_1...dx_k.$$
\end{definition}

The counterpart of Proposition \ref{prop: correlation density} is the following proposition.
\begin{prop}
Let $Z=(X_1,...,X_N)$ be $N$ real random variables such that the law of $(X_1,...,X_N)$ has a density with respect to the Lebesgue measure $f(x_1,...,x_N)dx_1...dx_N$ that we suppose symmetric, then the correlation functions of $Z$ are $\psi_k=0$ if $k>N$ and if $k\le N$, $$\label{form:correlationreel}\psi_k(x_1,...,x_k)=\cfrac{N!}{(N-k)!}\int_{\R^{N-k}}f(x_1,...,x_N)dx_{k+1}...dx_N.$$
\end{prop}
Doing the same computations as for the Ginibre case we can state the counterpart of Theorem \ref{thm ginibre dpp} and Theorem \ref{thm: fonction correlation ginibre}.
\begin{definition}
A (real) determinental point process of kernel $K:\,\R^2\to \R$ is a point process such that its correlation functions are of the form: $$\psi_k(x_1,...,x_k)=\det[K(x_i,x_j)]_{1\le i,j\le k},\,\text{for almost all  } (x_1,...x_k)\in\R^k,\,\forall k\in\N_{\ge 1}.$$
\end{definition}

\begin{theorem}
\label{gue dpp}
The eigenvalues of a $N$ complex Wigner matrix is a real determinental point process of kernel $K_N^H$.
\end{theorem}

\begin{theorem}
\label{thm: gue correlation function}
For all $N\ge 1$, the density with respect to the real Lebesgue measure of the eigenvalues of a $N$ complex Wigner matrix is given by: $$ \phi_N^H(x_1,...,x_N)=\cfrac{1}{N!}\det[K_N^H(x_i,x_j)]_{1\le i,j\le N},$$
with $K_N^H$ the kernel defined by: $$K_N^H(x,y)=\sqrt{\gamma_H(x)\gamma_H(y)}\sum_{l=0}^{N-1}\cfrac{H_{l}(x)H_l(y)}{\sqrt{2\pi}\, l!}, \,\,\forall(x,y)\in\R^2.$$
Moreover we have an expression of the marginal densities of this law defined for $1\le k\le N$, by: $$(x_1,...x_k)\in\R^k\mapsto\phi_{N,k}^H(x_1,...x_k):=\int_{\R^{N-k}}\phi_{N}^H(x_1,...,x_N)dx_{k+1}...dx_N=\cfrac{(N-k)!}{N!}\psi_k^H(x_1,...,x_k),$$
with $\psi_k^H$ the correlation functions of this point process introduced in Definition \ref{def corellation gue}.
\newline
Indeed for all $1\le k\le N$, and for all $(x_1,...,x_k)\in\R^k$, we have: $$\phi_{N,k}^H(x_1,...,x_k)=\cfrac{(N-k)!}{N!}\det[K_N^H(x_i,x_j)]_{1\le i,j\le k}.$$
In particular we have the following explicit formulas for the marginal $\phi_{N,N}^H=\phi_N$ and for $k=1$, for all $x\in\R$, $$\phi_{N,1}^H(x)=\cfrac{\gamma_H(x)}{N}\sum_{l=0}^{N-1}\cfrac{H_l(x)^2}{\sqrt{2\pi}\,l!}.$$
\end{theorem}

\section{Circular law}
\label{Section:circularlaw}
The goal of this section is to establish the circular law for the Ginibre ensemble which states that the empirical distribution of the eigenvalues of a $N$  complex Ginibre random matrix up to a re normalization converges to the uniform law on the unit disk of the complex plan. 
\newline
Historically, the circular law was obtained for Ginibre random matrices by Ginibre and Mehta \cite{Mehta}. It was expected a phenomenon of universality which means that the result should not depend on the distribution of the entries of the random matrix. There were a lot of works in order to generalize the circular law in this direction. A general proof was obtained for random matrices with independent and identically distributed complex random variables with few hypothesis on the moments by Tao and Vu in \cite{TaoVu1}. Tao and Vu published many others papers on this topic as \cite{TaoVu2}, \cite{TaoVu3}. The key tools are the hermitization method introduced by Girko, potential theory and good estimates for the singular values of random matrices. 
\newline
For readers interested in this topic, we refer to the review on the circular law of Bordenave and Chafaï \cite{bordenave2012around} for a very detailed work on these methods and for historic comments on the circular law. 

\subsection{Mean circular law}
First we shall prove the circular law in mean and then in the following section we shall explain why the convergence is actually almost sure. 
\begin{theorem}[Mean circular law] 
\label{thm mean ginibre}
Let $\Lambda_N=(\lambda_1^N,...,\lambda_N^N)\in\C^N$ be the random spectrum of a $N$ complex Ginibre matrix. We write: $$\mu_N=\cfrac{1}{N}\sum_{k=1}^N\delta_{\lambda_k^N/\sqrt{N}}$$ the empirical distribution associated to $\Lambda_N/\sqrt{N}$. Moreover, let $\mu_\infty$ be the uniform distribution on the unit disk given by: $$\mu_\infty(dz)=\cfrac{\mathds{1}_{|z|\le 1}}{\pi}\,dz.$$
Then for every $f:\C\to \R$ bounded continuous function we have: 
$$\E\left[\int_\C f d\mu_N\right]\underset{N\to+\infty}{\longrightarrow} \int_\C f d\mu_\infty.$$

\end{theorem}
Before doing the proof, let us recall that with the notations of Theorem \ref{thm mean ginibre}, for every $N\ge 1$, we can define a probability measure that shall be written $\E(\mu_N)$ defined by: $$\int_\C f d\E(\mu_N):=\E\left[\int_\C fd\mu_N\right],$$ for all functions $f$ continuous and bounded. The circular law in mean state that the sequence of probability measures $(\E(\mu_N))_{N\ge1}$ converges in law towards the uniform distribution on the circle.
\begin{proof}
We follow the arguments of \cite{Chafai1}.
Let $f:\C\to\R$ be a continuous compact supported function. Using the law of the eigenvalues obtained in Theorem \ref{thm: fonction correlation ginibre} and the symmetry of the density, we have: 
\begin{equation}
\label{eq:circularinmean}
\begin{split}
\E\left[\int_\C f d\mu_N\right]&=\cfrac{1}{N}\sum_{k=1}^N\int_{\C^N}f\left(\cfrac{z_k}{\sqrt{N}}\right)\phi_N(z_1,...,z_N)dz_1...dz_N\\
&=\cfrac{1}{N}\sum_{k=1}^N\int_{\C^N}f\left(\cfrac{z_k}{\sqrt{N}}\right)\phi_N(z_k,z_2,...z_1,...,z_N)dz_k...dz_1...dz_N\\
&=\cfrac{1}{N}\sum_{k=1}^N\int_{\C}f\left(\cfrac{z_k}{\sqrt{N}}\right)\left[\int_{\C^{N-1}}\phi_N(z_k,z_2,...z_1,...,z_N)dz_2...dz_1...dz_N
\right]dz_k\\
&=\cfrac{1}{N}\sum_{k=1}^N\int_{\C}f\left(\cfrac{z_k}{\sqrt{N}}\right) \phi_{N,1}(z_k)dz_k\\
&=\int_{\C}f\left(\cfrac{z}{\sqrt{N}}\right) \phi_{N,1}(z)dz\\
&=N\int_{\C}f(z)\phi_{N,1}(\sqrt{N}z)dz
\end{split}
\end{equation}
where in the last equality we did the change of variable $z\to\sqrt{N}z$ (pay attention that the Jacobian of this transformation is $N$ since we must view $\C$ as $\R^2$ in the integration).
Hence, to conclude it remains to prove that for all $K$ compact included in $\{z\in\C: |z|\ne 1\}$ we have: $$\lim_{N\to\infty}\sup_{z\in K}\left|N\phi_{N,1}(\sqrt{N}z)-\cfrac{\mathds{1}_{|z|\le 1}}{\pi}\right|=0.$$
We recall the explicit expression of $\phi_{N,1}$: $$\phi_{N,1}(z)=\frac{1}{\pi N}\exp(-|z|^2)e_N(|z|^2), $$ with $e_N(z)=\sum_{k=0}^{N-1}z^k/k!$ the truncation of the exponential series up to the $N$ term.
To finish we just need a lemma to estimate the previous quantity.
\begin{lemma}
\label{lemma: approx}
For every $N\in\N$, and $z\in\C$, $$\left|e_N(Nz)-\exp(Nz)\mathds{1}_{|z|\le 1}\right|\le r_N(z), $$ with $r_N(z)$ an error term of the form $$r_N(z)=\cfrac{e^N}{\sqrt{2\pi N}}\,|z|^N\left(\cfrac{N+1}{N(1-|z|)+1}\mathds{1}_{|z|\le 1}+\cfrac{N}{N(|z|-1)+1}\mathds{1}_{|z|>1}\right).$$
\end{lemma}

\begin{proof}
We consider two cases. First, for $|z|\le N$, we have: $$\left|e_N(z)-\exp(z)\mathds{1}_{|z|\le N}\right|\le\sum_{k\ge N}\cfrac{|z|^k}{k!}\le\cfrac{|z|^N}{N!}\sum_{k\ge 0}\cfrac{|z|^k}{(N+1)^k}=\cfrac{|z|^N}{N!}\cfrac{N+1}{N+1-|z|}.$$
Secondly if $|z|>N$, we get: 
$$|e_N(z)|\le\sum_{k=0}^{N-1}\cfrac{|z|^k}{k!}\le \cfrac{|z|^{N-1}}{(N-1)!}\sum_{k=0}^{N-1}\cfrac{(N-1)^k}{|z|^k}\le \cfrac{|z|^N}{(N-1)!}\cfrac{|z|}{N+1-|z|}.$$
To conclude we use the classical Stirling bound $\sqrt{2\pi N} N^N\le N!e^N$ in the two previous inequalities.
\end{proof}
Using Lemma \ref{lemma: approx} and the dominated convergence theorem in\eqref{eq:circularinmean} yields the result.
\end{proof}

\begin{Remark}
Let us mention a probabilistic point of view given in \cite{Chafai1} that gives the pointwise convergence of $N\phi_{N,1}(\sqrt N \bullet)$ to $\cfrac{1}{\pi}\left(\mathds{1}_{|z|<1}+\cfrac{1}{2}\mathds{1}_{|z|=1}\right)$ (but for the previous application it is not sufficient since we need to have a uniform convergence on every compact included in $\{z\in \C,\, |z|\ne 1\}$.) 
\newline
To prove this pointwise convergence, we only have to prove the result for $z=r>0$ by invariance by rotation of the formula. 
\newline 
Next, we remark that if $Y_1,...,Y_N$ are independent Poisson random variables of mean $r^2$, then $$\exp(-Nr^2)e_N(Nr^2)=\mP(Y_1+...+Y_N<N)=\mP\left(\cfrac{Y_1+...+Y_N}{N}<1\right), $$ since $Y_1+...+Y_N$ is a Poisson random variable of mean $Nr^2$. 
\newline
Now, using the strong law of large numbers we see that almost surely $\cfrac{Y_1+...+Y_N}{N}$ converges to $r^2$. We immediately deduce that if $r<1$ $$\lim_{N\to\infty}\exp(-Nr^2)e_N(Nr^2)=1$$ and if $r>1$, $$\lim_{N\to\infty}\exp(-Nr^2)e_N(Nr^2)=0.$$
Finally, for the case $r=1$ we use the central limit theorem to have that: $$\exp(-Nr^2)e_N(Nr^2)=\mP\left(\sqrt{N}\left(\cfrac{Y_1+...+Y_N}{N}-1\right)<0\right)\underset{N\to\infty}{\longrightarrow}\int_{\R^-}\cfrac{\exp(-x^2/2)}{\sqrt{2\pi}}dx=\cfrac{1}{2}.$$
\end{Remark}

\subsection{Strong circular law}
To prove an almost sure convergence thanks to the mean convergence we will first that the empirical mean is near its expectation almost surely. We follow the arguments of \cite{Chafai1}.
\newline
Let us also mention that the strong circular law can be viewed as a corollary of large deviations of Coulomb gases in dimension 2 that shall be mentioned 
in Section \ref{subsection: extenson} giving another proof of the strong circular law stated in this section.

\begin{lemma}
\label{lemma:strongcircularlaw}
For all $f$ continuous bounded function from $\C$ to $\R$, we can find a constant $C>0$ that only depends on $f$ such that for all $N\ge 1$: $$\E\left[ \left(\int_\C f d\mu_N -\E\left(\int_\C f d\mu_N\right)\right)^4\right]\le \cfrac{C}{N^2}.$$
\end{lemma}
We shall not give the proof. The idea of the proof is based on the same arguments as the proof of the law of large numbers with hypothesis that the fourth moment is bounded.
The proof is quite painful. We refer to \cite{Hwang} for a complete proof of this point.

\begin{corollary}
\label{coro:cvps}
For all $f$ continuous bounded function from $\C$ to $\R$, almost surely we have: $$\lim_N\left(\int_\C f d\mu_N -\E\left(\int_\C f d\mu_N\right)\right)=0.$$
\end{corollary}

\begin{proof}
With the notations of Lemma \ref{lemma:strongcircularlaw}, we get that: $$\sum_{N\ge 1} \E\left[ \left(\int_\C f d\mu_N -\E\left(\int_\C f d\mu_N\right)\right)^4\right] <\infty.$$
Thanks to Fubini theorem, we get that: $$ \E\left[\sum_{N\ge 1} \left(\int_\C f d\mu_N -\E\left(\int_\C f d\mu_N\right)\right)^4\right]< \infty.$$
So almost surely: $$\sum_{N\ge 1} \left(\int_\C f d\mu_N -\E\left(\int_\C f d\mu_N\right)\right)^4<\infty.$$ It implies directly that almost surely $$\lim_N\left(\int_\C f d\mu_N -\E\left(\int_\C f d\mu_N\right)\right)=0.$$
\end{proof}

\begin{Remark}
We shall see in Section \ref{section:wignerthm} that we can obtain this result by a concentration inequality in the Wigner case. We shall explain at this moment why we can not do these methods for the Ginibre case.
\end{Remark}

We can now state the strong circular law.
\begin{theorem}
With the notations of the mean circular law, we almost surely have the following convergence: 
$$\mu_N\overunderset{\mathcal L}{N\to+\infty}{\longrightarrow}\mu_\infty.$$
\end{theorem}
\begin{proof}
Let $f$ be a continuous bounded function from $\C$ to $\R$. On the one hand, thanks to the mean circular law we have that: $$\E\left[\int_\C f d\mu_N\right]\underset{N\to+\infty}{\longrightarrow} \int_\C f d\mu_\infty.$$ Thanks to Corollary \ref{coro:cvps} we almost surely have that: $$\lim_{N\to\infty}\left(\int_\C f d\mu_N-
\E\left[\int_\C f d\mu_N\right]\right)=0.$$
So we deduce that for all $f$ continuous and bounded functions from $\C$ to $\R$, almost surely: $$\int_\C f d\mu_N\underset{N\to+\infty}{\longrightarrow} \int_\C f d\mu_\infty.$$ 
Then, using a countable sequence of functions dense in the set of continuous and compactly supported functions from $\C$ to $\R$, standard arguments in measure theory (see for instance \cite{billingsley2013convergence}) give that almost surely for all $f$ continuous bounded functions from $\C$ to $\R$, $$\int_\C f d\mu_N\underset{N\to+\infty}{\longrightarrow} \int_\C f d\mu_\infty.$$
Hence, we get that almost surely: $$\mu_N\overunderset{\mathcal L}{N\to+\infty}{\longrightarrow}\mu_\infty.$$
\end{proof}

\subsection{Localisation of the spectrum}
First we shall give the law of the modulus of eigenvalues of a $N$ complex Ginibre random matrix. 
\begin{prop}
\label{prop:loimodule}
Let $\Lambda_N=(\lambda_1,...,\lambda_N)$ be the exchangeable vector of eigenvalues of a $N$ complex Ginibre random matrix. Then we have the following equality in law: $$(|\lambda_1|,...,|\lambda_N|)\equalaw(Z_{\sigma(1)},...,Z_{\sigma(N)}),$$ with $Z_1,...,Z_N$ is a family of independent random variables such that for all $k\ge 1$ the law of $Z_k^2$ is $\emph{Gamma}(k,1)$ and $\sigma$ is a random variable whose law is uniform in the group $\mathfrak{S}_N$ independent of $(Z_i)_{1\le i\le N}$.
More explicitly we have that for all $F:\R^N_+\to\R$ bounded symmetric function: $$\E(F(|\lambda_1|,...,|\lambda_N|))=\E(F(Z_{1},...,Z_{N})).$$
\end{prop}
\begin{Remark}
We recall that a random variable has a $\emph{Gamma}(k,\lambda)$ distribution if its density with respect to Lebesgue measure is given by $x\mapsto \lambda^k x^{k-1}\exp(-\lambda x)\mathds 1_{x\ge 0}/\Gamma(k)$ with $\Gamma(x)$ the Gamma function of Euler.
\end{Remark}
\begin{proof}
We follow the arguments of \cite{Chafai1} that can also be found in \cite{kostlan1992spectra}.
Using the law of the eigenvalues of $N$ complex Ginibre matrix we shall go in polar coordinate and integrate with respect to the angles to have the distribution of the modulus of the eigenvalues. 
\newline
In polar coordinates we can write the density of the eigenvalues as follow: $$(r_1,...r_N,\theta_1,...,\theta_N)\mapsto \cfrac{\exp(-\sum_{i=1}^Nr_i^2)}{\pi^N \prod_{i=1}^N k!} \prod_{j<k}|r_j\exp(i\theta_j)-r_k\exp(i\theta_k)|^2.$$
To integrate this we shall linearise the product. 
\newline
To do so we interpret the product as the product of the determinant of two matrices: 
\begin{align*}
\prod_{j<k}|r_j\exp(i\theta_j)-r_k\exp(i\theta&_k)|^2=\prod_{j<k}\left(r_j\exp(i\theta_j)-r_k\exp(i\theta_k)\right) \prod_{j<k}\left(r_j\exp(-i\theta_j)-r_k\exp(-i\theta_k)\right)\\
&=\det[r_j^{k-1}\exp(i(k-1)\theta_j)]_{1\le j,k\le N}\det[r_j^{k-1}\exp(-i(k-1)\theta_j)]_{1\le j,k\le N}\\
&=\sum_{\sigma,\sigma'\in\mathfrak{S}_N}(-1)^{\varepsilon(\sigma)+\varepsilon(\sigma')}\prod_{j=1}^N r_j^{\sigma(j)+\sigma'(j)-2}\exp(i(\sigma(j)-\sigma'(j))\theta_j).
\end{align*}
Integrating this term with respect to the angles, the only term that shall not vanish is when $\sigma=\sigma'$. 
Hence it gives: $$\int_{[0,2\pi]^N} \prod_{j<k}|r_j\exp(i\theta_j)-r_k\exp(i\theta_k)|^2 d\theta_1...
d\theta_N=(2\pi)^N\sum_{\sigma\in\mathfrak{S}_N}\prod_{j=1}^N r_j^{2\sigma(j)-2}.$$
Now let us recall a linear algebra notation. Let $A\in M_N(C)$, we call permanent of $A$ the following quantity: $$\text{per}(A)=\sum_{\sigma\in\mathfrak{S}_N}\prod_{j=1}^N A_{j,\sigma(j)}.$$
Thanks to this notation we can write a compact formula for the density of the modulus of the eigenvalues: 
\begin{align*}
\int_{[0,2\pi]^N}\cfrac{e^{-\sum_{i=1}^Nr_i^2}}{\pi^N \prod_{i=1}^N k!} \prod_{j<k}|r_j\exp(i\theta_j)-r_k\exp(i\theta_k)|^2 d\theta_1...d\theta_N&=\cfrac{2^N\exp(-\sum_{i=1}^Nr_i^2)}{ \prod_{i=1}^N k!}\text{per}\left[r_j^{2k}\right]_{1\le j,k\le N}\\
&=\cfrac{1}{N!}\,\text{per}\left[\cfrac{2}{(k-1)!}\,r_j^{2k}\,\exp(-r_j^2)\right]_{1\le j,k\le N}.
\end{align*}
We recognise inside the permanent the expression of the expected density. Indeed, if $Z_k$ has a density with respect to Lebesgue measure given by $g_k(r)=\cfrac{2}{(k-1)!}\,r^{2k}\exp(-r^2)\mathds{1}_{r\ge0}$ then an easy computation show that $Z_k^2$ has a $\text{Gamma}(k,1)$ law. 
\newline Finally we just need to notice that if $X_1,...,X_N$ are independent random variables with there respective densities given by $(f_k)_{1\le k\le N}$ then $(x_1,...x_N)\mapsto \cfrac{1}{N!}\,\text{per}[f_k(x_j)]_{1\le j,k\le N}$ is the density of the random variable $(X_{\sigma(1)},...,X_{\sigma(N)})$ with $\sigma$ a uniform permutation in $\mathfrak{S}_N$ independent of the $(X_i)_{1\le i\le N}$ . 
It proves the result.
\end{proof}

We first mention an application of this result to study the hole probability of the point process of the eigenvalues of a $N$ complex Ginibre matrices with $N$ large, that is the probability that a disk of radius $r$ contains no points of the point process when $r$ becomes large. 
\newline
We shall define the infinite Ginibre ensemble as the point process with kernel $K_\infty$ obtained as the limit of the kernel of $N$ complex Ginibre matrices: $$\forall (z,w)\in\C^2, \, \lim_{N\to+\infty}K_N(z,w)=:K_\infty(z,w).$$ A direct computation prove that for all $(z,w)\in\C^2$ $$K_\infty(z,w)=\cfrac{1}{\pi}\exp\left(z\overline w-\frac{1}{2}|z|^2-\frac{1}{2}|w|^2\right),$$ and the convergence is uniform on every compact of $\C^2$. We refer to \cite{adhikari2017hole,hough2009zeros} for more details about the infinite Ginibre ensemble.

\begin{definition}[Infinite Ginibre ensemble] We define the infinite Ginibre ensemble $ \Lambda_\infty$ as the point process in $\C$ with kernel $$K_\infty(z,w)=\cfrac{1}{\pi}\exp\left(z\overline w-\frac{1}{2}|z|^2-\frac{1}{2}|w|^2\right).$$
\end{definition}
We can also define the hole probability of a point process. 
\begin{definition}[Hole probability of a point process]
Let $Z$ be a point process in $\C$. For $r>0$, we define $n(r):=|Z\cap \overline {B(0,r)}|$ as the number of points of the point process $Z$ in $\overline {B(0,r)}$. We call hole probability of the point process the probability that the point process has no points in $\overline {B(0,r)}$ which is $\mP(n(r)=0)$.
\end{definition}
We give an easy example to understand this quantity. 
\begin{example}
Let $Z$ be a Poisson point process on $\C$ of parameter $\lambda>0$ introduced in Example \ref{example poisson point process}. By definition of a Poisson point process, $n(r)=|Z\cap \overline {B(0,r)}|$ is a Poisson random variable of parameter $\pi r^2\lambda$. So the hole probability is given by  $\mP(n(r)=0)=\exp(-\pi r^2\lambda)$.
\end{example}
We give another example that is an immediate consequence of Proposition \ref{prop:loimodule}.
\begin{example}
\label{example hole proba ginibre}
For $N\ge 1$, let $\Lambda_N=(\lambda_1,...,\lambda_N)$ be the set of eigenvalues of a $N$ complex Ginibre matrix. The hole probability of $\Lambda_N$ is given by $$\mP(n(r)=0)=\mP(\forall i\in[1,N],\, |\lambda_i|> R)=\mP(\forall i\in[1,N], Z_i> R), $$
where $Z_i$ are independent random variables such that for all $1\le k\le N$, $Z_k^2$ has a $\emph{Gamma}(k,1)$ distribution. By independence, we get $$\mP(n(r)=0)=\prod_{i=1}^N \mP(Z_k^2> r^2).$$
Let $(E_i)_{i\ge 1}$ be a family of iid exponential random variables of parameter $1$ (that can also be viewed as $\emph{Gamma}(1,1)$ random variables), then by the stability property of $\emph{Gamma}$ distributions, for all $k\ge 1$, we know that the law of $E_1+...+E_k$ is $\emph{Gamma}(k,1)$. Hence the hole probability can be written $$\mP(n(r)=0)=\prod_{i=1}^N \mP(E_1+...+E_k> r^2),$$ where $(E_i)_{i\ge 1}$ be a family of iid exponential law of parameter $1$. 
\end{example}

As a consequence of the uniform convergence on every compact of $K_N$ toward $K_\infty$ we have the following result
\begin{lemma}
\label{lemmaholeprobaginibre}
Let $\Lambda_\infty$ be the infinite Ginibre ensemble, then the hole probability of $\Lambda_\infty$ is given by $$\mP(n(r)=0)=\lim_{N\to+\infty}\mP( |\Lambda_N\cap\overline {B(0,r)}|=0)=\prod_{k=1}^{+\infty} \mP(Z_k^2> r^2)=\prod_{k=1}^{+\infty} \mP(E_1+...+E_k> r^2), $$
where $(E_k)_{k\ge 1}$ is a family of independent exponential law of parameter $1$ and $(Z_k)_{k\ge 1}$ is a family of independent random variables whose respective law are $\emph{Gamma}(k,1)$ .
\end{lemma}
\begin{proof}
We give the main idea but not get into details. Since $K_N$ converges uniformly to $K_\infty$ on every compacts, we have that 
\begin{equation}
\label{equation infinte kernel gini}
\mP(n(r)=0)=\mP( |\Lambda_\infty\cap\overline {B(0,r)}|)=\lim_{N\to+\infty}\mP( |\Lambda_N\cap\overline {B(0,r)}|).
\end{equation}
This can be obtained by using the so-called Fredholm determinant. These determinants naturally appear when we look at the probability that a set does not contain any eigenvalues of a $N$ complex Ginibre matrix  (more generally when we have a determinental point process). The Fredholm determinant theory is really useful when we have to pass to the limit when $N$ becomes large to study the limit of the probability that a set does not contain any eigenvalues of a $N$ complex Ginibre matrix. We refer to the Section 3.4 of \cite{Alicelivre} for a short introduction to Fredholm determinant and its use for these questions in determinental models of random matrices. Using \eqref{equation infinte kernel gini} and Example $\ref{example hole proba ginibre}$ give the result.
\end{proof}
The goal of the next theorem is to estimate for the infinite Ginibre ensemble how fast the hole probability goes to 0 when $r$ becomes large.
The following result can be viewed as large deviations principle and actually we shall use in the proof the Cramer theorem that is recall in the introduction of Section \ref{sec:largedev}.
\begin{theorem}
Let $\Lambda_\infty$ be the infinite Ginibre point process. Then $$\lim_{r\to+\infty}\cfrac{1}{r^4}\log\left(\mP(n(r)=0)\right)=-\cfrac{1}{4}.$$
\end{theorem}

\begin{proof}
We give the main arguments of the proof. A complete proof can be found in Section 7.2 of \cite{hough2009zeros}. 
From Lemma \ref{lemmaholeprobaginibre}, we have an explicit expression for the hole probability \begin{equation}
\label{eq1holeproba}
\mP(n(r)=0)=\prod_{k=1}^{+\infty} \mP(Z_k^2> r^2)=\prod_{k=1}^{+\infty} \mP(E_1+...+E_k> r^2), 
\end{equation}
where $(E_k)_{k\ge 1}$ is a family of independent exponential law of parameter $1$. As explained before this expression come from the fact that $E_1+...+E_k$ has a $\text{Gamma}(k,1)$ distribution for all $k\ge 1$.
\newline
We first lower bound this probability using the Markov inequality. Indeed, we obtain by independence of the $(E_i)_{i\ge 1}$ that for $k\ge 1$, for all $\theta<1$,  
\begin{equation}
\label{eqmarkovbound}
\begin{split}
\mP(E_1+...+E_k> r^2)\le \exp(-\theta r^2)\exp(\theta(E_1+...+E_k))&\le \exp(-\theta r^2)\exp(k\theta(E_1))\\
&=\exp(-\theta r^2)(1-\theta)^{-k}.
\end{split}
\end{equation}
We see that the optimal $\theta$ to optimize the previous bound is $\theta=1-\frac{k}{r^2}$. We now do as if $r^2$ is an integer (if not we should consider the integer part of $r^2$ instead which does not change the estimate we shall do). We can upper bound the hole probability using \eqref{eqmarkovbound}:
\begin{equation}
\label{eq2holeproba}
\begin{split}
\mP(n(r)=0)&=\prod_{k=1}^{+\infty} \mP(E_1+...+E_k> r^2)\\
&\le \prod_{k=1}^{r^2} \mP(E_1+...+E_k> r^2)\\
&\le \prod_{k=1}^{r^2} \exp\left(-\left(1-\cfrac{k}{r^2}\right)r^2-k\log\left(\cfrac{k}{r^2}\right)\right)\\
&\le \exp\left(-\cfrac{r^2}{2}\,\left(r^2-1\right)-r^4 \left(\int_0^1 x\log(x)dx\right) +o(r^4)\right)\\
&=\exp\left(-\cfrac{r^4}{4}+o(r^4)\right),
\end{split}
\end{equation}
where we used that the Riemann sum $\sum_{k=1}^{r^2}\frac{k}{r^2}\log\left(\frac{k}{r^2}\right)=r^2(\int_0^1 x\log(x)dx +o(1))$.
For the lower bound, we start by noticing that for all $k\ge 1$, for all $\lambda>0$, $$\mP(\text{Gamma}(k,1)>\lambda)=\mP(\text{Poisson}(\lambda)<k)$$ where $\text{Poisson}(\lambda)$ is the standard Poisson law of parameter $\lambda$.
Using this remark we obtain that for all $k\ge 1$, $$\mP(Z_k^2> r^2)=\mP(\text{Poisson}(r^2)<k-1)\ge \exp(-r^2)\cfrac{r^{2(k-1)}}{(k-1)!}.$$
Using this bound, we lower bound the terms for $k\le r^2$
\begin{equation}
\label{ineq3markovbound}
\begin{split}
\prod_{k=1}^{r^2} \mP(Z_k^2> r^2)&\ge \prod_{k=1}^{r^2}\exp(-r^2)\cfrac{r^{2(k-1)}}{(k-1)!}\\
&=\exp\left(-r^4+\sum_{k<r^2}\left[k\log(r^2)-\log(k!)\right]\right)
\\&=\exp\left(-r^4+\sum_{k<r^2}\left[(r^2-k)\log\left(\cfrac{k}{r^2}\right)\right]\right).
\end{split}
\end{equation}
Again using a Riemann sum, we obtain that \begin{equation}
\label{ineq4markovbound}
\prod_{k=1}^{r^2} \mP(Z_k^2> r^2)\ge \exp\left(-\cfrac{r^4}{4}+o(r^4)\right).
\end{equation}
For the other terms of the product, we can first notice that since $\mP(\text{Poisson}(\lambda)>\lambda)\underset{\lambda\to+\infty}{\longrightarrow} 1/2$, for $r$ large enough, for all $k>r^2$ we get 
\begin{equation}
\begin{split}
\mP(Z_k^2>r^2)=\mP(E_1+...+E_k> r^2)&\ge \mP(E_1+...+E_{r^2}> r^2)\\
&=\mP(Z_{r^2}^2>r^2)=\mP(\text{Poisson}(r^2)>r^2-1)\ge \cfrac{1}{4}.
\end{split}
\end{equation}
This gives that for $r$ large enough 
\begin{equation}
\label{ineqhole6}
\prod_{k=r^2+1}^{2r^2}\mP(E_1+...+E_k> r^2)\ge \exp(-r^2\log(4)).
\end{equation}
Finally, using the Cramer theorem, there exists $c>0$ such that for all $k\ge 1$, $$\mP\left(E_1+...+E_k<\cfrac{k}{2} \right)\le \exp(-ck).$$
Hence we obtain that 
\begin{equation}
\begin{split}
\sum_{k=2r^2+1}^{+\infty}\mP\left(Z_k^2<r^2 \right)=\sum_{k=2r^2+1}^{+\infty}\mP\left(E_1+...+E_k<r^2 \right)&\le \sum_{k=2r^2+1}^{+\infty}\mP\left(E_1+...+E_k<\cfrac{k}{2} \right)\\
&\le \sum_{k=2r^2+1}^{+\infty}\exp(-ck).
\end{split}
\end{equation}
Taking $r$ large enough we get $$\sum_{k=2r^2+1}^{+\infty}\mP\left(Z_k^2<r^2 \right)<\cfrac{1}{2}\,.$$
By independence of the $(Z_i)_{i\ge 1}$, we obtain that for $r$ large enough: 
\begin{equation}
\label{ineq5holeproba}
\begin{split}
\prod_{k=2r^2+1}^{+\infty}\mP\left(Z_k^2>r^2 \right)=\mP\left(\bigcap_{k>2r^2}Z_k^2>r^2\right)=1-\mP\left(\bigcup_{k>2r^2}Z_k^2\le r^2\right)&\ge 1-\sum_{k=2r^2+1}^{+\infty}\mP\left(Z_k^2\le r^2 \right)\\ &\ge \cfrac{1}{2}\,.
\end{split}
\end{equation}
Using \eqref{eq2holeproba}, \eqref{ineqhole6} \eqref{ineq5holeproba} we have that for $r$ large enough $$\mP(n(r)=0)=\prod_{k=1}^{+\infty}\mP\left(Z_k^2>r^2 \right) \ge \exp\left(-\cfrac{r^4}{4}+o(r^4)\right).$$

\end{proof}

Let us also mention that this result has also been proved using a potential theory approach and proving a large deviations principle for the hole probability of $N$ complex Ginibre matrix in \cite{adhikari2017hole}. Finally, let us also say that determinental point processes appear when we study the roots of power series with random coefficients. Similar results about the hole probability of such determinental point processes can be obtained. See Section 7 of \cite{hough2009zeros} for instance.
\newline
\tab We now mention some results about the spectral radius of $N$ complex Ginibre matrices. 
\begin{theorem}
Let $\Lambda_N=(\lambda_1,...,\lambda_N)$ be the exchangeable vector of eigenvalues of a $N$ complex Ginibre random matrix. Then we have the following convergence in probability: $$\rho_N:=\max_{1\le k\le N}\cfrac{\lambda_k}{\sqrt{N}}\overunderset{\mP}{N\to +\infty}{\longrightarrow} 1.$$
\end{theorem}

\begin{proof}
We just here give the principal idea based on Proposition \ref{prop:loimodule}. We use the same notations as in the Proposition \ref{prop:loimodule}. By the additive property of Gamma distribution, for all $k\ge 1$ we write $Z_k^2\equalaw E_1+...+E_k$ with $(E_i)_{i\in\Z}$ a family of iid random variable with exponential law of parameter $1$. Then by independence, for all $r>0$ we have $$\mP(\rho_N\le r)=\prod_{1\le k\le N}\mP\left(\cfrac{E_1+...+E_k}{N}\le r^2\right).$$
By the strong law of large numbers we know that if  $r<1$ each term in the product converges to 0 and if $r>1$ each term converges to 1. From that we can deduce that $\mP(\rho_N\le r)$ converges to 0 if $r<1$ and 1 if $r>1$. So $\rho_N$ converges to $0$ in law and so in probability since $0$ is a constant random variable.
\end{proof}

Actually we can prove that the convergence is not only in probability but also almost surely and we can compute the fluctuation of this convergence, see \cite{rider2003limit}.

\begin{theorem}
Let $\Lambda_N=(\lambda_1,...,\lambda_N)$ be the exchangeable vector of eigenvalues of a $N$ complex Ginibre random matrix. Then we have almost surely: $$\rho_N=\max_{1\le k\le N}\cfrac{\lambda_k}{\sqrt{N}}\underset{N\to +\infty}{\longrightarrow} 1.$$
Moreover, if we let $\kappa_N=\log(N/2\pi)-2\log(\log(N))$, then we have $$\sqrt{4N\kappa_N}\left(\rho_N-1-\sqrt{\cfrac{\kappa_N}{4N}}\right)\overunderset{\mathcal L}{N\to+\infty}{\longrightarrow} \text{Gumbel},$$ where Gumbel is the law of $-\log(X)$ with $X$ an exponential random variable of parameter 1.
\end{theorem}

\section{The Wigner theorem}
\label{section:wignerthm}
\subsection{The mean Wigner theorem}

\begin{theorem}[Wigner theorem in mean] 
\label{thm mean wigner}
Let $(M_N)_{N\ge 1}$ be a sequence of random matrices such that for every $N\ge 1$, $M_N$ is a $N$ complex Wigner matrix.
For every $N\ge 1$, let $\Lambda^N=(\lambda_1^N,...,\lambda_N^N)\in\R^N$ be the random spectrum of the $N$ complex Wigner matrix $M_N$. We write: $$\mu_N=\cfrac{1}{N}\sum_{k=1}^N\delta_{\lambda_k^N/\sqrt{N}}$$ the empirical distribution associated to the eigenvalues of $M_N/\sqrt{N}$. 
\newline
Then for every bounded and continuous function $f:\R \to\R$ we have: 
$$\E\left[\int_\R f d\mu_N\right]\underset{N\to+\infty}{\longrightarrow} \int_\R f d\sigma,$$
where $\sigma$ is the so-called Wigner semicircle distribution introduced in Appendix \ref{appendixsection} which is the measure on $\R$ whose density with respect to the Lebesgue measure is given by: $$\sigma(x):=\cfrac{\sqrt{4-x^2}}{2\pi}\, \mathds{1}_{[-2,2]}(x).$$
\end{theorem}
Before doing the proof, let us recall that with the notations of Theorem \ref{thm mean wigner}, for every $N\ge 1$, we can define a probability measure that shall be written $\E(\mu_N)$ defined by: $$\int_\R f d\E(\mu_N):=\E\left[\int_\R fd\mu_N\right],$$ for all functions $f$ continuous and bounded. The Wigner theorem in mean state that the sequence of probability measures $(\E(\mu_N))_{N\ge1}$ converges in law toward the semicircular distribution $\sigma$.
\begin{proof}
We follow in this proof the arguments given in \cite{Alicelivre} in Section 3.3 which are up to our knowledge originally from \cite{haagerup2003random}.
\newline
Let $f:\R\to\R$ be a continuous and bounded function. Using the density of the eigenvalues of a $N$ Wigner matrix obtained in Theorem \ref{thm: gue correlation function} and the symmetry of the distribution, we have: 
\begin{equation}
\label{eq: calcul mean wigner}
\begin{split}
\E\left[\int_\R f d\mu_N\right]&=\cfrac{1}{N}\sum_{k=1}^N\int_{\R^N}f\left(\cfrac{x_k}{\sqrt{N}}\right)\phi_N^H(x_1,...,x_N)dx_1...dx_N\\
&=\cfrac{1}{N}\sum_{k=1}^N\int_{\R^N}f\left(\cfrac{x_k}{\sqrt{N}}\right)\phi_N^H(x_k,x_2,...,x_1,...,x_N)dx_k...dx_1...dx_N\\
&=\cfrac{1}{N}\sum_{k=1}^N\int_{\R}f\left(\cfrac{x_k}{\sqrt{N}}\right)\left[\int_{\R^{N-1}}\phi_N^H(x_k,x_2,...x_1,...,x_N)dx_2...dx_1...dx_N
\right]dx_k\\
&=\cfrac{1}{N}\sum_{k=1}^N\int_{\R}f\left(\cfrac{x_k}{\sqrt{N}}\right) \phi_{N,1}^H(x_k)dx_k\\
&=\int_{\R}f\left(\cfrac{x_k}{\sqrt{N}}\right) \phi_{N,1}^H(x)dx\\
&=\int_{\R}f\left(\cfrac{x_k}{\sqrt{N}}\right)\,\cfrac{K_N^H(x,x)}{N}dx.
\end{split}
\end{equation}
We now use properties of the Hermite polynomials stated in Proposition \ref{prop properties hermite} to obtain another expression of $K_N^H(x,x)$.
In this proof only, we define for $k\ge 0$, $$\mathcal H_k(x):=\cfrac{\sqrt{\gamma_H(x)}H_k(x)}{\sqrt{\sqrt{2\pi}k!}} $$ such that we can rewrite $K_N^H$ as $$K_N^H(x,y)=\sum_{l=0}^{N-1}\mathcal H_l(x)\mathcal H_l(y).$$
Using the Christoffel–Darboux formula for $x\ne y$, we get $$K_N^H(x,y)=\sqrt{N}\,\cfrac{\mathcal H_N(x)\mathcal H_{N-1}(y)-\mathcal H_N(y)\mathcal H_{N-1}(x)}{x-y}.$$
Letting $y$ goes to $x$ yields that for all $x\in\R$, $$K_N^H(x,x)=\sqrt{N}\,\left[\mathcal H_N'(x)\mathcal H_{N-1}(x)-\mathcal H_N(x)\mathcal H_{N-1}'(x)\right].$$ 
Using the formula \eqref{induction hermite} or \eqref{equadiff hermite} we obtain that for all $N\ge 0$, $$\mathcal H_N''(x)+\left(n+\cfrac{1}{2}-\cfrac{x^2}{4}\right)\mathcal H_N(x)=0.$$
This gives the identity 
\begin{equation}
\label{deriative noyau}
\cfrac{d}{dx}K_N^H(x,x)=-\sqrt{N}\mathcal H_N(x) \mathcal H_{N-1}(x).
\end{equation} 
Thanks to these identities we shall do an integration by parts in \eqref{eq: calcul mean wigner}. 
\newline
Since the measure $\sigma$ is compactly supported we only need to prove the convergence of the expectation of the moments of the empirical measure toward the moments of the semicircular distribution.
\newline
To do so, we shall compute for every $N\ge 1$ the moment generating function of $\E[\mu_N]$. Given a measure $\mu\in\mes$ we introduce its Laplace transform defined on $\R$ as $$\mathcal L_\mu(s):=\int_{\R}\exp(sx)\mu(dx).$$
The Laplace transform is also called the moment generating function since for $\mu\in\mes$ $$
\mathcal L_\mu(s)=\sum_{n\ge 0}\int_\R x^n\mu(dx)\,\cfrac{s^n}{n!},$$ when the sum converges absolutely.
\newline
For $N\ge 1$, using \eqref{eq: calcul mean wigner}, for all $s\in\R^*$, 
\begin{equation}
\label{eq:calculmeanwigner2}
\begin{split}
\mathcal L_{\E[\mu_N]}(s)&=\E\left[\int_\R\exp(st)\mu_N(dt)\right]\\
&=\int_\R\exp\left(\cfrac{sx}{\sqrt{N}}\right) \cfrac{K_N^H(x,x)}{N}\,dx \\
&=\cfrac{1}{s}\int_\R\exp\left(\cfrac{sx}{\sqrt{N}}\right)\mathcal H_N(x)\mathcal H_{N-1}(x) \,dx,
\end{split}
\end{equation}
where we integrated by parts for the last identity and use \eqref{deriative noyau}.
Using the definition of $\mathcal H_N$, we compute for $s\in\R^*$ and $N\ge 1$,
\begin{equation}
\begin{split}
\label{eq:meanwigner3}
\cfrac{1}{s}\int_\R\exp\left(\cfrac{sx}{\sqrt{N}}\right)&\mathcal H_N(x)\mathcal H_{N-1}(x) \,dx=\cfrac{\sqrt{N}}{N!\sqrt{2\pi}s}\int_\R\exp\left(-\cfrac{x^2}{2}+\cfrac{sx}{\sqrt{N}}\right) H_N(x)H_{N-1}(x)\,dx\\
&=\cfrac{\sqrt{N}\exp\left(\cfrac{s^2}{2N}\right)}{N!\sqrt{2\pi}s}\int_\R\exp\left(-\cfrac{x^2}{2}\right) H_N\left(x+\cfrac{s}{\sqrt{N}}\right)H_{N-1}\left(x+\cfrac{s}{\sqrt{N}}\right)dx.
\end{split}
\end{equation}
We finally want to use the orthogonality of the Hermite polynomials with respect to the weight $\gamma_H$ to compute this integral. To do so, let us recall that since $H_n'=nH_{n-1}$ for all $n\ge 1$, a Taylor expansion in $s$ gives that for all $n\ge 0$, for all $x\in\R$, for all $s\in\R$, 
$$H_{n}\left(x+\cfrac{s}{\sqrt{N}}\right)=\sum_{k=0}^{n}{n\choose k} H_{n-k}(x)\,\left(\cfrac{s}{\sqrt{N}}\right)^k=\sum_{k=0}^{n}{n\choose k} H_{k}(x)\,\left(\cfrac{s}{\sqrt{N}}\right)^{n-k}.$$
Using this formula to express $$H_N\left(x+\cfrac{s}{\sqrt{N}}\right)$$ and $$H_{N-1}\left(x+\cfrac{s}{\sqrt{N}}\right)$$ in \eqref{eq:meanwigner3} and using the orthogonality property of the Hermite polynomials, we obtain that for all $N\ge 1$, for all $s\in\R$, 
\begin{equation}
\label{eq:laplacetransfor}
\begin{split}
\mathcal L_{\E[\mu_N]}(s)&=\frac{\exp\left(\frac{s^2}{2N}\right)}{s}\sum_{k=0}^{N-1}\cfrac{k!}{N!}{N\choose k}{{N-1}\choose k}\left(\cfrac{s}{\sqrt{N}}\right)^{2N-1-2k}\\
&=\frac{\exp\left(\frac{s^2}{2N}\right)}{s}\sum_{k=0}^{N-1}\cfrac{(N-1-k)!}{N!}{N\choose {N-1-k}}{{N-1}\choose {N-1-k}}\left(\cfrac{s}{\sqrt{N}}\right)^{2k+1}\\
&=\exp\left(\frac{s^2}{2N}\right)\sum_{k=0}^{N-1}\cfrac{1}{k+1}{{2k} \choose k}\,\cfrac{(N-1)...(N-k)}{N^k}\,\cfrac{s^{2k}}{(2k)!}
\end{split}
\end{equation}
We recall that the moments of the semicircular law were computed in Appendix in Proposition \ref{annexe prop moment}.
Passing to the limit, we obtain that for all $s\in\R$, \begin{equation}
\mathcal L_{\E[\mu_N]}(s)\underset{N\to+\infty}{\longrightarrow}\sum_{k=0}^{+\infty}\cfrac{1}{k+1}{{2k} \choose k}\,\cfrac{s^{2k}}{(2k)!}=\sum_{k=0}^{+\infty}\int_\R x^k \sigma(dx)\,\cfrac{s^{k}}{k!}=\mathcal L_\sigma(s),
\end{equation}
This gives the result.
\end{proof}

\subsection{A concentration inequality}
\label{section:concentrationineq}
We first prove a concentration inequality for Lipschitz functions of Gaussian random variables. More generally for concentration inequalities we refer to the textbook \cite{ledoux2001concentration}.
\begin{theorem}
\label{Talagrand ineq}
Let $X_1,...,X_d$ be $d$ independent centered Gaussian random variable of variance smaller than $\sigma^2$. Let $F:\R^d\to \R$ which is $\lambda>0$ Lispchitz which means that for all $(x,y)\in\R^d$, $|F(x)-F(y)|\le \lambda||x-y||_2$. Then for all $t>0$ we have: 
\begin{equation}
\label{eq talagrand}
\mP(|F(X_1,...,X_d)-\E(F(X_1,...,X_d))|>t)\le 2\exp\left(-2\left(\cfrac{t}{\lambda\sigma\pi}\right)^2\right)
\end{equation}
\end{theorem}

\begin{proof}
We follow the arguments of Tao in \cite{Tao} (Theorem 2.1.12). Up to changing $F$ and the $(X_i)_{1\le i\le d}$ we can suppose that $\Var(X_i)=1$ for all $1\le i\le d$, $\lambda=1$ and $\E(F(X_1,...,X_d))=0$. Moreover using a density argument or the fact that a Lipschitz function is differentiable almost everywhere by Rademacher theorem, we can assume that $F$ is $C^1$ and so satisfies that $|\nabla F|\le 1$ in $\R^d$.
\newline
Fix $\lambda>0$. As usual for concentration inequalities, we shall try to bound $\E(\exp(\lambda F(X_1,...,X_d)))$.
To do so, let introduce $(Y_1,...,Y_d)$ an independent copy  of $(X_1,...,X_d)$. First by the Jensen inequality, we get that $$\E(\exp(-\lambda F(Y_1,...,Y_d)))\ge\exp(-\lambda\E(F(Y_1,...,Y_d)))=1.$$
We obtain 
\begin{equation}
\label{ineq concentration 1}
\E(\exp(\lambda F(X_1,...,X_d)))\le \E(\exp(\lambda(F(X_1,...,X_d)-F(Y_1,...,Y_d)))).
\end{equation}
Let $X=(X_1,...,X_d)$ and $Y=(Y_1,...,Y_d)$. $X$ and $Y$ are independent centered Gaussian vectors in $\R^d$ of covariance matrix $I_d$ the identity matrix of $M_d(\R)$. We shall construct a path to rely $X$ and $Y$ which is not the affine one.
We define $Z:[0,\pi/2]\to \R^d$ by $Z(s)=\sin(s)X+\cos(s)Y$. We get 
\begin{equation}
\label{eq concentration 2}
F(X)-F(Y)=F(Z(\pi/2))-F(Z(0))=\int_0^{\pi/2}\langle\nabla F(Z(s))|Z'(s)\rangle\,ds.
\end{equation}
For all $s\in[0,\pi/2]$, $Z(s)$ is a centred Gaussian vector in $\R^d$ with covariance matrix $I_d$. We remark that for all $s\in[0,\pi/2]$, $Z'(s)=\cos(s)X-\sin(s) Y$ is also a centred Gaussian vector in $\R^d$ with covariance matrix $I_d$. The remarkable property that justifies the choice of the path $Z$ is that for all $s\in[0,\pi/2]$, $Z(s)$ and $Z'(s)$ are independent. Indeed, fix $s\in[0,\pi/2]$, then $(Z(s),Z'(s))$ is a centred Gaussian vector in $\R^{2d}$ and for all $1\le i\ne j\le d$, 
\begin{align*}
\E[Z(s)_i Z'(s)_j]&=\E[(\sin(s)X_i+\cos(s)Y_i)(\cos(s)X_j-\sin(s) Y_j)]=0\\
\E[Z(s)_i Z'(s)_i]&=\E[(\sin(s)X_i+\cos(s)Y_i)(\cos(s)X_i-\sin(s) Y_i)]=\cos(s)\sin(s)-\cos(s)\sin(s)=0
\end{align*}
since $X$ and $Y$ are independent centred Gaussian vectors of covariance matrix $I_d$. Hence, the covariance matrix of the Gaussian vector $(Z(s),Z'(s))$ is $I_{2d}$. By standard properties of Gaussian vectors it implies that $Z(s)$ and $Z'(s)$ are independent.
\newline
Starting from \eqref{ineq concentration 1} we get:
\begin{equation}
\label{ineq concentration 3}
\begin{split}
\E(\exp(\lambda F(X_1,...,X_d)))&\le\E\left[\exp\left(\lambda\int_0^{\pi/2}\langle\nabla F(Z(s))|Z'(s)\rangle\,ds \right)\right]\\
&=\E\left[\exp\left(\lambda\frac{\pi}{2}\int_0^{\pi/2}\langle\nabla F(Z(s))|Z'(s)\rangle\,\frac{2}{\pi}ds \right)\right]\\
&\hspace{-0.2cm}\overset{Jensen}{\le}\E\left[ \int_0^{\pi/2} \exp\left(\lambda\frac{2}{\pi}\langle\nabla F(Z(s))|Z'(s)\rangle\right) \frac{\pi}{2}\,ds\right]\\
&=\cfrac{2}{\pi} \int_0^{\pi/2} \E\left[\exp\left(\lambda\frac{\pi}{2}\langle\nabla F(Z(s))|Z'(s)\rangle\right)\right] \,ds.
\end{split}
\end{equation}
For all $s\in[0,\pi/2]$, conditioning on $Z(s)$ and and using that $Z'(s)$ is a standard Gaussian vector independent of $Z(s)$, we obtain:
\begin{align*}
\E\left[\exp\left(\lambda\frac{\pi}{2}\langle\nabla F(Z(s))|Z'(s)\rangle\right)\right]&=\E\left[\E\left[\exp\left(\lambda\frac{\pi}{2}\langle\nabla F(Z(s))|Z'(s)\rangle\right)\right]\dis{\lvert} Z(s)\right]\\
&=\E\left[\exp\left(\left(\lambda\frac{\pi}{2}\right)^2\frac{||\nabla F(Z(s))||^2}{2}\right)\right]\\
&\le \exp\left(\cfrac{\lambda^2\pi^2}{8}\right)
\end{align*}
Using this estimate in \eqref{ineq concentration 3} we get: 
\begin{equation}
\label{ineq concentration 4}
\E(\exp(\lambda F(X_1,...,X_d)))\le\exp\left(\frac{\lambda^2\pi^2}{8}\right)
\end{equation}
Let $t>0$, using the Markov inequality we have that for all $\lambda>0$: 
\begin{align*}
\mP( F(X_1,...,X_d)>t)&=\mP(\exp(\lambda  F(X_1,...,X_d))> \exp(\lambda t))\\
&\le \E(\exp(\lambda  F(X_1,...,X_d)))\exp(-\lambda t)\\
&\le\exp\left(\frac{\lambda^2\pi^2}{8}\right)\exp(-t\lambda)
\end{align*}
We optimize the right term in $\lambda$ and choosing $\lambda=4t/\pi^2$ yields $$\mP( F(X_1,...,X_d)>t)\le \exp\left(-\frac{2\,t^2}{\pi^2}\right).$$
By symmetry we obtain the result.
\end{proof}

To use Theorem \ref{Talagrand ineq} in the context of random matrices the key lemma is the so-called Hoffman-Wielandt inequality. 
\begin{lemma}[Hoffman-Wielandt's lemma]
\label{lemma: hoffman wielandt}
Let $A,B\in \mathcal H_N(\C)$ and $\Lambda^A:=(\lambda_1^A,...,\lambda_N^A)\in\R^N$ and $\Lambda^B:=(\lambda_1^B,...,\lambda_N^B)\in\R^N$ be their respective ordered spectrum $\lambda_1^A\ge ...\ge \lambda_N^A$, $\lambda_1^B\ge ...\ge \lambda_N^B$. Then $$||\Lambda^A-\Lambda^B||_2\le ||A-B||_2,$$
where $||C||_2=\Tr(CC^*)$ for $C\in M_N(\C)$ is the canonical norm in $M_N(\C)$ and $||X||_2$ for $X\in\R^N$ is the canonical norm in $\R^N$.  
\end{lemma}
Before proving this result let us mention that if we do not order the eigenvalues the result does not hold, as we shall see it plays an important role in the proof. In a probabilistic point of view we can equivalently reformulate the Hoffman-Wieland inequality saying that the map defined from $(\mathcal H_N(\C),||.||_2)$ to $\mes$ by $$A\mapsto \cfrac{1}{N}\sum_{i=1}^N\delta_{\lambda^A_i},$$ where $(\lambda_1^A,...,\lambda_N^A)$ is the spectrum of $A$ is $1/\sqrt{N}$ Lipschitz if we endow $\mes$ with the so-called Wasserstein distance $W_2$. 
\newline
The Hoffman-Wielandt inequality is actually true for normal matrices but is not true for matrices in $M_N(\C)$ which is a problem in order to obtain for example concentration inequalities for Ginibre matrices. However, for matrices in $M_N(\C)$ a Hoffman-Wielandt inequality holds by substituting the eigenvalues of a matrix by its singular values. Combining it with the concentration inequality it gives good concentration inequalities for the empirical mean of singular values of Ginibre matrices.
\newline
\tab Let us give the proof of Lemma \ref{lemma: hoffman wielandt}.
\begin{proof}
We first compute: 
$$||\Lambda^A-\Lambda^B||_2^2=||\Lambda^A||_2^2-2\langle\Lambda^A|\Lambda^B\rangle+||\Lambda^B||_2^2.$$
Since $||\Lambda^A||_2^2=\sum_{i=1}^N(\lambda_i^A)^2=\Tr(A^2)=||A||_2^2$, to prove the Hoffman-Wielandt inequality it remains to show that: $$\Tr(AB)\le\sum_{i=1}^N\lambda_i^A\lambda_i^B.$$
First, let us notice that up to changing $A$ by $A-\lambda_N^A I_N$ and using that $\Tr(B)=\sum_{i=1}^N\lambda_i^B$, we can suppose that the spectrum of $A$ is in $\R^+$. Since now we suppose that $\lambda_N^A\in\R^+$.
\newline
Then, using the spectral theorem, we can find $U\in\mathcal U_N(\C)$ such that $A=U \diag(\Lambda^A)U^*$ where $\diag(X)$ for $X\in\C^N$ is the diagonal matrix of $M_N(\C)$ with $\diag(X)_{i,i}=X_i$ for $1\le i\le N$. By properties of the trace we have $$\Tr(AB)=\Tr(U \diag(\Lambda^A)U^*B)=\Tr(\diag(\Lambda^A)U^*BU)=
\Tr(\diag(\Lambda^A)B'),$$ where $B':=U^*BU\in \mathcal H_N(\C)$ and has the same spectrum as $B$. Thanks to this transformation we reduce to the case $A$ is a diagonal matrix. Write $B'=(B'_{i,j})_{1\le i,j\le N}$ to compute \begin{equation}
\label{eq: trace maj}\Tr(AB)=\Tr(\diag(\Lambda^A)B')=\sum_{i=1}^N \lambda_i^A B'_{i,i}.
\end{equation}
A usual method to obtain the type of inequalities as Hoffman-Wielandt is to use the hypothesis that $\Lambda^A$ is ordered by doing a summation by parts.
For all $0\le i\le N+1$, we introduce $s_i:=\sum_{k=1}^i  B'_{k,k}$ and $s_{0}=0$. \eqref{eq: trace maj} gives 

\begin{equation}
\label{eq trace 2}
\Tr(AB)=\sum_{i=1}^N \lambda_i^A B'_{i,i}=\sum_{i=1}^N \lambda_i^A (s_{i}-s_{i-1})=\sum_{i=1}^N (\lambda_i^A-\lambda_{i+1}^A) s_i,
\end{equation}
where $\lambda_{N+1}^A:=0$ by convention.
We notice that for all $1\le i\le N$, $\lambda_i^A-\lambda_{i+1}^A\ge 0$ because $\Lambda^A$ is ordered and $\lambda_N^A\ge 0$.
\newline
The only thing that it remains to prove is that for all $1\le i\le N$, $s_i\le \sum_{k=1}^i \lambda_k^B$. We shall prove these inequalities in Lemma \ref{lemma ineq trace}. Indeed assume that such inequalities have been obtained, \eqref{eq trace 2} yields $$\Tr(AB)=\sum_{i=1}^N (\lambda_i^A-\lambda_{i+1}^A) s_i\le \sum_{i=1}^N \left[(\lambda_i^A-\lambda_{i+1}^A) \sum_{k=1}^i \lambda_k^B\right]=\sum_{i=1}^N \lambda_i^A\lambda_i^B,$$ doing again a summation by parts for the last equality.
\end{proof}
\begin{lemma}
\label{lemma ineq trace}
Let $B=(b_{i,j})_{1\le i,j\le N}\in \mathcal H_N(\C)$ and $\lambda_1\ge ... \ge \lambda_N\in\R^N$ be its ordered spectrum. Then for all $1\le i\le N$ we have $$s_i:=\sum_{k=1}^ib_{k,k}\le \sum_{k=1}^i \lambda_k.$$
\end{lemma}
\begin{proof}
In this proof, let $(e_1,...,e_n)$ be the canonical basis of $\C^N$.
We prove the result by induction on $1\le i\le N$. For $i=1$, the result comes directly from the fact that $$b_{1,1}=\langle e_1|Be_1\rangle\le \sup_{||x||_2=1}\langle x|Bx\rangle=\lambda_1.$$
Assume that the result holds for $1\le i\le N-1$. Let us prove that $$s_{i+1}\le \sum_{k=1}^{i+1} \lambda_k.$$
Let $F_i=\SPAN(e_1,...,e_{i+1})$. By the Min-max theorem, we get that $\lambda_{i+1}\ge \inf_{x\in F_i,\, ||x||_2=1}\langle x|Bx\rangle$. So, we can find $v_{i+1}\in F_i$ with $||v_{i+1}||_2=1$ such that $\lambda_{i+1}\ge \langle v_{i+1}|Bv_{i+1}\rangle$.
We complete $v_{i+1}$ by $v_1,...,v_i$ in an orthonormal basis of $F_i$. Then $\mathcal B_i:=(v_1,...,v_i,v_{i+1},e_{i+2},...e_n)$ is an orthonormal basis of $\C^N$. Let $\widetilde {B_i}:=(\widetilde {b_{k,j}})_{1\le k,j\le N}$ be the matrix $B$ in the basis $\mathcal B_i$. Since $B$ and $\widetilde {B_i}$ are conjugate they have the same spectrum. We apply the induction to the matrix $\widetilde {B_i}$. This gives that $$\sum_{k=1}^i \widetilde {b_{k,k}}\le \sum_{k=1}^i\lambda_k.$$
By construction of $v_{i+1}$ we have that: 

\begin{equation}
\label{eq: ineq trace}
\sum_{k=1}^{i+1} \widetilde {b_{k,k}}=\sum_{k=1}^{i} \widetilde {b_{k,k}}+\langle v_{i+1}|Bv_{i+1}\rangle\le \sum_{k=1}^i\lambda_k+\lambda_{i+1}.
\end{equation}
Let us remark that by construction of the basis $\mathcal B_i$, for all $N\ge j>i+1$ we have $\widetilde b_{j,j}=b_{j,j}=\langle e_j|Be_j\rangle$ and since $B$ and $\widetilde B_i$ are conjugate they have the same trace. This leads to $\sum_{k=1}^{i+1} \widetilde {b_{k,k}}=\sum_{k=1}^{i+1} b_{k,k}$. \eqref{eq: ineq trace} gives the result.

\end{proof}

Let us recall that given a function $f:\R\to \C$ and $M\in\mathcal H_N(\C)$ that can be written $M=UDU^*$ with $U\in\mathcal U_N(\C)$ and $D$ a real diagonal matrix, we define $f(M)$ as $Uf(D)U^*$ and this choice does not depend of the choice of basis $U$.
A direct application of the Hoffman-Wielandt inequality is the following result.
\begin{corollary}
\label{coro lip}
Let $f: \R\to \C$ be a $\lambda$ Lipschitz function, then the map $X\mapsto \cfrac{1}{N}\Tr(f(X))$ is $\lambda/\sqrt{N}$ Lipschitz on $(\mathcal H_N(\C),||.||_2)$.
\end{corollary}

\begin{proof}
Let $X,Y\in\mathcal H_N(\C)$, we write $\Lambda^X$ and $\Lambda^Y$ their ordered spectrum as in Lemma $\ref{lemma: hoffman wielandt}$. We directly compute: 
\begin{align*}
\left|\cfrac{1}{N}\Tr(f(X))-\cfrac{1}{N}\Tr(f(Y))\right|^2&=\cfrac{1}{N^2}\left|\sum_{i=1}^N (f(\lambda_i^X)-f(\lambda_i^Y))\right|^2\\
&\hspace{-1cm}\overset{Cauchy-Schwarz}{\le} \cfrac{1}{N}\sum_{i=1}^N \left[\left|f(\lambda_i^X)-f(\lambda_i^Y)\right|^2\right]\\
&\le \cfrac{\lambda^2}{N}\sum_{i=1}^N \left[\left|\lambda_i^X-\lambda_i^Y\right|^2\right]\\
&\le \cfrac{\lambda^2}{N}\,||\Lambda^X-\Lambda^Y||_2^2\\
&\le \cfrac{\lambda^2}{N}\,||X-Y||_2^2,
\end{align*}
where we used the Hoffman-Wielandt inequality for the last inequality.
\end{proof}
Now using Theorem \ref{Talagrand ineq}, we shall prove a concentration inequality for the spectrum of Wigner matrices.
\begin{theorem}
\label{thm concentration}
Let $f: \R\to \C$ be a $\lambda$ Lipschitz function and $M_N$ be a $N$ complex Wigner matrix then for all $t>0$: 
$$\mP\left(\, \left|\cfrac{1}{N}\Tr f(M_N)-\E\left[\cfrac{1}{N}\Tr f(M_N)\right] \right|\ge t\right)\le 2\exp\left(-2N\,\cfrac{t^2}{\pi^2\lambda^2}\right).$$
\end{theorem}

\begin{proof}
We recall that by definition of a complex Wigner matrix, its entries are centered normal law of variance $1$ on the diagonal and otherwise the real and the imaginary part of its entries are independent centered normal law of variance $1/2$. By Corollary \ref{coro lip}, we can apply Theorem \ref{Talagrand ineq} with $F_f(X)=X\mapsto \cfrac{1}{N}\Tr(f(X))$ which is $\lambda/\sqrt{N}$ Lipschitz and $\sigma^2=1$ which gives the result.
\end{proof}

\begin{Remark}
The result of Theorem \ref{thm concentration} can also be obtained by the so-called logarithmic Sobolev inequalities (in short LSI). We refer to \cite{chafai2024logarithmic,guionnet2003lectures} for generalities on this topic and to \cite{Alicelivre} for its applications in random matrices. The concentration inequality obtained in  Theorem \ref{thm concentration} can be obtained by LSI in two different ways. We can first apply a LSI directly to the random matrix $M_N$ whose law satisfies a LSI from the tensorization property since all its entries are independent Gaussian random variables that satisfy a LSI. We can also rewrite $\Tr (f(M_N))$ in function of the eigenvalues of $M_N$ and use that the law of the eigenvalues of a $N$ complex Wigner random matrix obtained in Theorem \ref{thm law eigenvalues guen} satisfies a LSI by the Bakry-Émery criterion. 
\newline
Let us mention that none of these two approaches can be used to obtain a similar concentration inequality for Ginibre random matrices.
\end{Remark}

As a consequence of this concentration inequality we obtain a bound on the variance of $\cfrac{1}{N}\Tr f(M_N/\sqrt N)$ where $M_N$ is a $N$ complex Wigner matrix that shall be used in the following part to obtain the strong Wigner theorem from the mean one.

\begin{prop}
\label{prop: bound variance}
Let $f: \R\to \C$ be a $\lambda$ Lipschitz function and $M_N$ be a $N$ complex Wigner matrix then there exists $c_f>0$ a constant that only depend on $f$ such that 
$$\Var\left(\cfrac{1}{N}\Tr f\left(\cfrac{M_N}{\sqrt N}\right)\right)\le \cfrac{c_f}{N^2}.$$
\end{prop}

\begin{proof}
By definition of the variance, we have: $$\Var\left(\cfrac{1}{N}\Tr f\left(\cfrac{M_N}{\sqrt N}\right)\right)=\E\left[\left(\cfrac{1}{N}\Tr f\left(\cfrac{M_N}{\sqrt N}\right)-\E\left(\cfrac{1}{N}\Tr f\left(\cfrac{M_N}{\sqrt N}\right)\right)\right)^2\right].$$
Using the Fubini-Tonelli theorem yields: 
\begin{align*}
\E&\left[\left(\cfrac{1}{N}\Tr f\left(\cfrac{M_N}{\sqrt N}\right)-\E\left(\cfrac{1}{N}\Tr f\left(\cfrac{M_N}{\sqrt N}\right)\right)\right)^2\right]=\\
&\hspace{1cm}\int_0^{+\infty}\mP\left[\left(\cfrac{1}{N}\Tr f\left(\cfrac{M_N}{\sqrt N}\right)-\E\left(\cfrac{1}{N}\Tr f\left(\cfrac{M_N}{\sqrt N}\right)\right)\right)^2\ge t\right]dt=\\
&\hspace{3.5cm}\int_0^{+\infty}\mP\left[\left|\cfrac{1}{N}\Tr f\left(\cfrac{M_N}{\sqrt N}\right)-\E\left(\cfrac{1}{N}\Tr f\left(\cfrac{M_N}{\sqrt N}\right)\right)\right|\ge \sqrt t\right]dt.
\end{align*}
Using Theorem \ref{thm concentration} with the function $\tilde f(x):=f(x/\sqrt{N})$ which is $\lambda/\sqrt N$ Lipschitz we get:
\begin{align*}
\Var\left(\cfrac{1}{N}\Tr f\left(\cfrac{M_N}{\sqrt N}\right)\right)\le 2\int_0^{+\infty}\exp\left(-2N^2\cfrac{t}{\pi^2\,\lambda^2}\right)dt=\cfrac{\pi^2\lambda^2}{N^2}.
\end{align*}
\end{proof}

\subsection{The strong Wigner theorem}
We can now state the strong Wigner theorem.
 
\begin{theorem}[Almost sure Wigner theorem] 
\label{thm strong wigner}
Let $(M_N)_{N\ge 1}$ be a sequence of random matrices such that for every $N\ge 1$ $M_N$ is a $N$ complex Wigner matrix.
For every $N\ge 1$, let $\Lambda^N=(\lambda_1^N,...,\lambda_N^N)\in\R^N$ be the random spectrum of the $N$ complex Wigner matrix $M_N$. We write: $$\mu_N=\cfrac{1}{N}\sum_{k=1}^N\delta_{\lambda_k^N/\sqrt{N}}$$ the empirical distribution associated to the eigenvalues of $M_N/\sqrt{N}$. 
\newline
Then almost surely $$\mu_N\overunderset{\mathcal L}{N\to+\infty}{\longrightarrow} \sigma.$$
\end{theorem}

\begin{proof}
As usual for the convergence in law it suffices to show that for a countable dense subset of the compactly supported functions $\mathcal D$ we have that almost surely for all $f\in\mathcal D$, $$\int_\R f d\mu_N\underset{n\to+\infty}{\longrightarrow}\int_\R f d\sigma.$$
Since $\mathcal D$ is countable, it suffices to prove that for all $f\in\mathcal D$, almost surely, $$\int_\R f d\mu_N\underset{n\to+\infty}{\longrightarrow}\int_\R f d\sigma.$$
\newline
For the rest of the proof, we fix $\mathcal D$ a countable dense subset of the compactly supported functions and $f\in \mathcal D$.
By the Wigner theorem in mean we have that $$\E\left[\int_\R f d\mu_N\right]\underset{N\to+\infty}{\longrightarrow} \int_\R f d\sigma.$$ Now, we need to be check that almost surely the empirical mean will not be far from its expectation. We shall use the concentration inequality obtained in the previous section.
\newline
We notice that by definition of the spectrum, $$\int_\R f d\mu_N=\cfrac{1}{N}\Tr f\left(\cfrac{M_N}{\sqrt N}\right).$$
Moreover the bound of the variance obtained in Proposition \ref{prop: bound variance} and the Markov inequality give that for all $\varepsilon>0$ $$\mP\left[\left|\cfrac{1}{N}\Tr f\left(\cfrac{M_N}{\sqrt N}\right)-\E\left(\cfrac{1}{N}\Tr f\left(\cfrac{M_N}{\sqrt N}\right)\right)\right|\ge \varepsilon\right]\le \cfrac{\Var\left(\cfrac{1}{N}\Tr f\left(\cfrac{M_N}{\sqrt N}\right)\right)}{\varepsilon^2}\le\cfrac{c_f}{N^2\varepsilon^2}.$$
Using the Borel–Cantelli lemma yields that for all $\varepsilon>0$, almost surely there exists $N_0$ such that for all $N\ge N_0$, $$\left|\cfrac{1}{N}\Tr f\left(\cfrac{M_N}{\sqrt N}\right)-\E\left(\cfrac{1}{N}\Tr f\left(\cfrac{M_N}{\sqrt N}\right)\right)\right|\le \varepsilon.$$
Again using a countable dense sequence of positive real numbers, we obtain that almost surely for all $\varepsilon>0$, there exists $N_0$ such that for all $N\ge N_0$ we have: 
\begin{equation}
\label{eq almost sure wigner}
\left|\cfrac{1}{N}\Tr f\left(\cfrac{M_N}{\sqrt N}\right)-\E\left(\cfrac{1}{N}\Tr f\left(\cfrac{M_N}{\sqrt N}\right)\right)\right|=\left|\int_\R f d\mu_N-\E\left[\int_\R f d\mu_N\right]\right|\le \varepsilon.
\end{equation}
Since $\E\left[\int_\R f d\mu_N\right]\underset{N\to+\infty}{\longrightarrow} \int_\R f d\sigma,$ \eqref{eq almost sure wigner} yields that almost surely, $$\int_\R f d\mu_N \underset{N\to+\infty}{\longrightarrow} \int_\R f d\sigma.$$
This concludes the proof.
\end{proof}
\subsection{Almost sure convergence of the smallest and the largest eigenvalue}
In this section, given a matrix $M\in \mathcal H_N(\C)$ we denote $\lambda_{\max}(M)$ and $\lambda_{\text{min}}(M)$ the largest and the smallest real eigenvalue of $M$.
As a consequence of the almost sure convergence of the empirical mean of the re normalized eigenvalues of a $N$ complex Wigner matrices we obtain the first estimate on the behaviour of the largest eigenvalue of a $N$ complex Wigner matrix. 
\begin{prop}
Let $(M_N)_{N\ge 1}$ be a sequence of random matrices such that for every $N\ge 1$ $M_N$ is a $N$ complex Wigner matrix. Then, we almost surely have $$\liminf_N\lambda_{\max}\left(\frac{M_N}{\sqrt{N}}\right)\ge 2.$$
\end{prop}
\begin{proof}
By the strong Wigner theorem, we know that almost surely $$\mu_N:=\cfrac{1}{N}\sum_{k=1}^N\delta_{\lambda_k^N/\sqrt{N}}$$ converges weakly towards the semicircle distribution $\sigma$ with the same notations as in Theorem \ref{thm strong wigner}. Let $\omega\in \Omega$ such that $$\mu_N(\omega)\overunderset{\mathcal L}{N\to+\infty}{\longrightarrow} \sigma.$$
Assume by contradiction that $$\liminf_N\lambda_{\max}\left(\frac{M_N(\omega)}{\sqrt{N}}\right)< 2.$$
Then we can find $\varepsilon>0$ and a subsequence $(\phi(N))_{N\ge 1}$ (which both depend on $\omega$) such that for all $N\ge 1$, $$\lambda_{\max}\left(\frac{M_{\phi(N)}(\omega)}{\sqrt{\phi(N)}}\right)< 2-\varepsilon.$$
Let $f_\varepsilon$ be the continuous bounded function on $\R$ which is null on $(-\infty,2-\varepsilon)$, equal to $1$ on $(2,+\infty)$ and linear on $(2-\varepsilon,2)$. Then, by choice of $f_\varepsilon$ we notice that $\int_\R f_\varepsilon d\sigma>0$ and for all $N\ge 1$ $\int_\R f_\varepsilon d\mu_{\phi(N)}(\omega)=0$ since all the eigenvalues of $\frac{M_{\phi(N)}(\omega)}{\sqrt{\phi(N)}}$ are smaller than $2-\varepsilon$ where $f$ is null.
However by the definition of the weak convergence of measure we have $$\int_\R f_\varepsilon d\mu_{\phi(N)}(\omega)\underset{N\to+\infty}{\longrightarrow} \int_\R f_\varepsilon d\sigma,$$
which is not possible. This proves the result.
\end{proof}
We now prove the other bound.
\begin{prop}
Let $(M_N)_{N\ge 1}$ be a sequence of random matrices such that for every $N\ge 1$ $M_N$ is a $N$ complex Wigner matrix. Then, we almost surely have $$\limsup_N\lambda_{\max}\left(\frac{M_N}{\sqrt{N}}\right)\le 2.$$
\end{prop}
\begin{proof}
This proof is based on the arguments of \cite{haagerup2003random}. The idea is to obtain a Chernoff bound on the largest eigenvalue and then to use the Borel-Cantelli lemma. 
\newline
Fix $\varepsilon>0$. Now fix $t>0$. Using that $\exp$ is an increasing function we get:
\begin{equation}
\label{eq:largesteigenvaluemark}
\begin{split}
\mP\left(\lambda_{\max}\left(\frac{M_N}{\sqrt{N}}\right)\ge 2+\varepsilon\right)&=\mP\left(\exp\left(t\lambda_{\max}\left(\frac{M_N}{\sqrt{N}}\right)\right)\ge \exp(t(2+\varepsilon))\right)\\
&=\mP\left(\lambda_{\max}\left(\exp\left(t\,\frac{M_N}{\sqrt{N}}\right)\right)\ge \exp(t(2+\varepsilon))\right).
\end{split}
\end{equation}
Let us notice that since the eigenvalues of $\exp\left(t\,\frac{M_N}{\sqrt{N}}\right)$ are positive, we have $$\lambda_{\max}\left(\exp\left(t\,\frac{M_N}{\sqrt{N}}\right)\right)\le \Tr\left(\exp\left(t\,\frac{M_N}{\sqrt{N}}\right)\right).$$
So we can upper bound using the Markov inequality
\begin{equation}
\label{eq:largesteigenvaluemark2}
\begin{split}
\mP\left(\lambda_{\max}\left(\frac{M_N}{\sqrt{N}}\right)\ge 2+\varepsilon\right)&\le\mP\left(\Tr\left(\exp\left(t\,\frac{M_N}{\sqrt{N}}\right)\right)\ge \exp(t(2+\varepsilon))\right)\\
&\le \exp(-t(2+\varepsilon))\,\E\left[\Tr\left(\exp\left(t\,\frac{M_N}{\sqrt{N}}\right)\right)\right].
\end{split}
\end{equation}
Now let us recall that we computed the expectation of the right hand side in Theorem \ref{thm mean wigner} because it is the Laplace transform of the measure denoted $N\E(\mu_N)$ (with the notations of the mean Wigner theorem).
So \eqref{eq:laplacetransfor} gives 
\begin{equation}
\label{eq:largesteigenvaluemark3}
\begin{split}
\E\left[\Tr\left(\exp\left(t\,\frac{M_N}{\sqrt{N}}\right)\right)\right]&=N\exp\left(\frac{t^2}{2N}\right)\sum_{k=0}^{N-1}\cfrac{1}{k+1}{{2k} \choose k}\,\cfrac{(N-1)...(N-k)}{N^k}\,\cfrac{t^{2k}}{(2k)!}\\
&\le N\exp\left(\frac{t^2}{2N}\right)\sum_{k=0}^{+\infty}\cfrac{1}{k!(k+1)!}\,t^{2k}\\
&\le N\exp\left(\frac{t^2}{2N}\right)\left[\sum_{k=0}^{+\infty}\cfrac{t^k}{k!}\right]^2\\
&\le N\exp\left(\frac{t^2}{2N}+2t\right).
\end{split}
\end{equation}
Using \eqref{eq:largesteigenvaluemark3} in \eqref{eq:largesteigenvaluemark2} we obtain that for all $t>0$ 
$$\mP\left(\lambda_{\max}\left(\frac{M_N}{\sqrt{N}}\right)\ge 2+\varepsilon\right)\le N\exp\left(\frac{t^2}{2N}+2t-t(2+\varepsilon)\right)=N\exp\left(\frac{t^2}{2N}-t\varepsilon\right).$$
Optimizing in $t>0$ gives that for all $N\ge 1$ $$\mP\left(\lambda_{\max}\left(\frac{M_N}{\sqrt{N}}\right)\ge 2+\varepsilon\right)\le N\exp\left(-\frac{\varepsilon^2 N}{2}\right).$$
Since the right hand side is summable, the Borel-Cantelli lemma yields that for all $\varepsilon>0$ almost surely $\limsup_N\left(\lambda_{\max}\left(\frac{M_N}{\sqrt{N}}\right)\right)\le 2+\varepsilon$. Using a countable dense family of $\varepsilon>0$ we obtain that almost surely for all $\varepsilon>0$ $\limsup_N\left(\lambda_{\max}\left(\frac{M_N}{\sqrt{N}}\right)\right)\le 2+\varepsilon$ giving the result.
\end{proof}
Combining the two previous propositions and using a symmetric argument for the smallest eigenvalue of a $N$ complex Wigner matrix we get the following result.
\begin{theorem}
Let $(M_N)_{N\ge 1}$ be a sequence of random matrices such that for every $N\ge 1$ $M_N$ is a $N$ complex Wigner matrix. Then, we almost surely have $$\lim_{N\to+\infty}\lambda_{\max}\left(\frac{M_N}{\sqrt{N}}\right)= 2 \emph{  and  } \lim_{N\to+\infty}\lambda_{\min}\left(\frac{M_N}{\sqrt{N}}\right)= -2 .$$
\end{theorem}

\newpage
\part{Dynamical approach to random matrices: the Dyson Brownian motion}
\label{part:dyson}
In 1962 \cite{Dyson}, Dyson introduced continuous models that defined new types of Coulomb gases through the spectrum of random matrices. We shall recover from these new models for instance the Wigner theorem stated in Theorem \ref{thm mean wigner}. It gives a new approach for the study of the eigenvalues of random matrices. We shall first introduced some models of random matrices and explain how we can obtain from these models a dynamic on their eigenvalues. This will define a system of stochastic differential equations (in short SDE) for the study of the eigenvalues of random matrices. We study these eigenvalues as a system of particles in interaction through a pair correlation interaction. Then we will prove the convergence of the empirical mean of these particles to a deterministic process, solution of a partial differential equation. Finally, we shall give some properties of solutions of this partial differential equation. We will emphasize on how the study of the process of the eigenvalues viewed as a system of particles in interaction can help us to recover some result presented in the first section as the Wigner theorem or the law of the eigenvalues of a $N$ complex Wigner matrix. 
\newline
Let $(\Omega,\mathcal F,\mP)$ be a standard filtered probability space on which all the random variables shall be defined.
\section{Dynamic on the eigenvalues}

\subsection{Dyson Brownian motion}

\begin{definition}
\label{def:dyson}
Fix $\beta\in\{1,2\}$ and $N\ge 1$ a positive integer. Let $(B_{k,j}, \widetilde{B_{k,j}})_{1\le k\le j\le N}$ be a family of independent Brownian motions. We call $\beta$ Dyson Brownian motion of size $N\ge 1$, the random process which is valued in $\mathcal S_N(\R)$ for $\beta=1$ or $\mathcal H_N(\C)$ for $\beta=2$ defined by: 
\begin{equation}
\left\{
\begin{split}
D_N^\beta(t)&:=\cfrac{1}{\sqrt{N}}(D_{k,j}^\beta(t))_{1\le k,j\le N}\\
\forall 1\le k\le N,\, D_{k,k}^\beta(t)&:=\cfrac{\sqrt{2}}{\sqrt{\beta}}\,B_{k,k}(t)\\
\forall 1\le k\le j\le N,\, D_{k,j}^\beta(t)&:=\cfrac{1}{\sqrt{\beta}}\,\left[B_{k,j}(t)+i(\beta-1)\widetilde{B_{k,j}}\right]
\end{split}
\right.
\end{equation} 
\end{definition}

\begin{Remark}
For $\beta=1$, $\sqrt{N}\,D_N^1(1/2)$ is a $N$ real Wigner matrix as it was defined in Definition \ref{def gue,goe}. For $\beta=2$, $\sqrt{N}\,D_N^2(1)$ is a $N$ complex Wigner matrix as it was defined in Definition \ref{def gue,goe}. 
\end{Remark}

\begin{definition}
\label{def:wishart}
Fix $N\ge 1$ and $M\ge 1$ be two non negative integers. Let $(B_{k,j}, \widetilde{B_{k,j}})_{1\le k\le N,\,1\le j\le M}$ be a family of independent Brownian motions.  We call Wishart Brownian motion of size $N\times M$, the random process which is valued in $\mathcal H_N(\C)$ defined by: 
\begin{equation}
\left\{
\begin{split}
W_N(t)&=\frac{1}{N}A_tA_t^*\\ A_t&=\left(\cfrac{B_{k,j}(t)}{\sqrt{2}}+i\,\cfrac{\widetilde{B_{k,j}}(t)}{\sqrt{2}}\right)_{1\le k\le N, 1\le j\le M}.
\end{split}
\right.
\end{equation}
\end{definition}

\subsection{A differential approach of the spectrum of matrices}
We present the approach for symmetric matrices but the approach is also valid without any difference for hermitian matrices. This method is called in Physic the perturbation theory. Formally the idea is given a matrix whose spectrum and eigenvectors are known, could we get an approximation of the eigenvalues and eigenvectors of a small perturbation of the initial matrix. This method is for instance used in quantum mechanic for solving the Schrödinger equation as in the hydrogen atom case \cite{hirschfelder1964recent}.
\newline
We recall that $ \overset{\circ}{\mathcal S_N(\R)} $ is the set of $N\times N$ symmetric matrices whose characteristic polynomial has simple roots. 
$ \overset{\circ}{\mathcal S_N(\R)} $ is an open and dense subset of $\mathcal S_N(\R)$.
\newline
Moreover, for $A\in \mathcal S_N(\R)$ we order its eigenvalues $$\lambda_1(A)\le...\le\lambda_N(A), $$ and we shall write $(u_i(A))_{1\le i\le N}$ an orthonormal basis of eigenvectors associated respectively to the $(\lambda_i(A))_{1\le i\le N}$.
Let us recall that given $A\in\overset{\circ}{\mathcal S_N(\R)}$ we can find a neighbourhood $V_A$ of $A$ in $\overset{\circ}{\mathcal S_N(\R)}$ such that for all $1\le i\le N$ the map $\lambda_i$ defined from $\overset{\circ}{\mathcal S_N(\R)}$ to $\R$ by $\lambda_i(A)$ is the $i^{th}$ eigenvalue of $A$ is smooth on $V_A$. 
\newline
Indeed, this result is a classical application of the implicit function theorem by looking at the map $\phi: \overset{\circ}{\mathcal S_N(\R)}\times \R\to \R$ defined by $$\phi(A,\lambda)=\chi_A(\lambda)$$ and using that $\lambda$ is an eigenvalue of $A$ if and only if $\chi_A(\lambda)=0$ and for $A\in \overset{\circ}{S_N}$ and $\lambda$ one of its eigenvalue we have $\chi_A'(\lambda)\ne 0$. 
\newline
We can do the same for the normalised eigenvector (even if they are defined up to a sign) by writing that for $A\in \overset{\circ}{\mathcal S_N(\R)}$ and $\lambda_i(A)$ an eigenvalue of $A$, an eigenvector of $A$ for $\lambda_i(A)$ is defined by the equation $(A-\lambda_i(A))x=0$ on $\R^N-\{0\}$. 
\newline
So, if we consider a smooth map $\mathcal A$ from an open interval $I\subset\R$ to $\overset{\circ}{\mathcal S_N(\R)}$ then we can find a parametrization of the eigenvectors such that for all $1\le i\le N$, $t\mapsto\lambda_i(A(t))$ and $t\mapsto u_i(A(t))$ are smooth on $I$. 

\begin{example}
\label{exemple:perturbationtheory}
Let $A\in \overset{\circ}{\mathcal S_N(\R)}$ and $B\in \mathcal S_N(\R)$ then there exists a small open interval centred in 0, $I\subset \R$, such that $t\mapsto\lambda_i(A+tB)$ and $t\mapsto u_i(A+tB)$ are smooth on $I$.
\end{example}

Until the end of this section, we fix a smooth map $\mathcal A$ from $I\subset \R$ an open interval to $\overset{\circ}{\mathcal S_N(\R)}$. 

\begin{prop}[Hadamard]
\label{prop:hadamard}
For all $1\le i\le N$, the maps defined on $I$ by $\phi_i(t)=\lambda_i(\mathcal A(t))$ are smooth and we have the following formulas for their derivatives:
\begin{align}
&\phi_i'(t)=\langle u_i(\mathcal A(t))|\mathcal A'(t)u_i(\mathcal A(t))\rangle\\
&\phi_i''(t)=\langle u_i(\mathcal A(t))|\mathcal A''(t)u_i(\mathcal A(t))\rangle+2\dis\sum_{1\le k\ne i\le N}\cfrac{|\langle u_k(\mathcal A(t))|\mathcal A'(t)u_i(\mathcal A(t))\rangle|^2}{\lambda_i(\mathcal A(t))-\lambda_k(\mathcal A(t))}
\end{align}

 \end{prop}

\begin{proof}
We fix $1\le i\le N$ and $t\in I$. We shall just write $\mathcal A$ for $\mathcal A(t)$ and $\lambda_i$ for $\phi_i$ to lighten the notations in the proof. 
\newline
By definition, we have $\mathcal A u_i=\lambda_i u_i$ and $\langle u_i|u_i\rangle=1$. We take the derivative of these two equalities. We get:
\begin{equation}
\label{eq:perturbationderivative1}
\mathcal A'u_i+\mathcal Au_i'=\lambda_i'u_i+\lambda_iu_i' 
\end{equation}
\begin{equation}
\label{eq:perturbationderivative2}
\langle u_i'|u_i\rangle=0
\end{equation}
We shall take the scalar product against $u_i$ in \eqref{eq:perturbationderivative1}.
Using the fact that $u_i$ and $u_i'$ are orthogonal by \eqref{eq:perturbationderivative2}, the symmetry of $\mathcal A$ yields $$\langle u_i|\mathcal Au_i'\rangle=\langle \mathcal Au_i|u_i'\rangle=\lambda_i\langle u_i|u_i'\rangle=0$$ 
We get: 
\begin{equation}
\label{eq:hadamard1}
\langle u_i| \mathcal A'u_i\rangle=\lambda_i'.
\end{equation}
This gives the first derivative. For the second derivative, we take the derivative of the formula \eqref{eq:hadamard1}:
\begin{equation}
\label{eq:hadamard2}
\lambda_i''=\langle u_i|\mathcal A''u_i\rangle+2\langle u_i'|\mathcal A'u_i\rangle.
\end{equation}
It remains to express $u_i'$ in the basis $(u_k)_{1\le k\le N}$. We already know that $\langle u_i'|u_i\rangle =0$ and to compute $\langle u_k|u_i'\rangle$ for $k\ne i$, we take the scalar product against $u_k$ in the formula $\eqref{eq:perturbationderivative1}$. 
Using that $\mathcal A$ is symmetric, it gives: $$(\lambda_i-\lambda_k)\langle u_i'|u_k\rangle =\langle u_k|\mathcal A'u_i\rangle .$$
So we get $u_i'=\dis\sum_{1\le i\ne k\le N}\cfrac{\langle u_k|\mathcal A'u_i\rangle }{\lambda_i-\lambda_k}u_k$. 
\newline
Using formula \eqref{eq:hadamard2}, we have: $$\lambda_i''=\langle u_i|\mathcal A''u_i\rangle+2\dis\sum_{1\le k\ne i\le N}\cfrac{|\langle u_k|\mathcal A'u_i\rangle |^2}{\lambda_i-\lambda_k}.$$

\end{proof}

\begin{example}
\label{exemple:perturbationtheory2}
With the notations of Example \eqref{exemple:perturbationtheory}, $A(t)=A+tB$ we have $\phi_i'(0)=\langle u_i(A)|Bu_i(A)\rangle $ and $\phi_i''(0)=2\dis\sum_{k\ne i}\cfrac{|\langle u_k(A)|Bu_i(A)\rangle|^2}{\lambda_i-\lambda_k}.$
\end{example}

\subsection{Particles in mean field interaction}
\label{subsection:particlesinmeanfieldinteraction}
We first introduce a family of stochastic processes that could be viewed as a system of particles in interaction through a singular repulsive force and which are subject to agitations modelled by Brownian motions. 
\newline
In this section, fix $N\ge 1$ and $M\ge 1$ two positive integers.
\begin{definition}
\label{def:dysonparticles}
Fix $\beta\in\{1,2\} $ and $N$ distinct real numbers $x_1<...<x_N$. We call $\beta$ Dyson particles starting at $(x_i)_{1\le i\le N}$ the family of stochastic processes $(\lambda_.^{i,\beta})_{1\le i\le N}$ solution of the system of SDE:
 \begin{equation}
   \label{dysonreal}
   \left\{
 \begin{split}
  \forall 1\le i\le N,\, d\lambda_t^{i,\beta}&=\cfrac{1}{N}\dis\sum_{j\ne i}\cfrac{1}{\lambda_t^{i,\beta}-\lambda_t^{j,\beta}}\,dt+\sqrt{\cfrac{2}{\beta N}}\,dB_t^i\\
  \lambda_0^{i,\beta}&=x_i
  \end{split}
  \right.
 \end{equation} with $(B^i)_{1\le i\le N}$ a family of independent Brownian motions.

\end{definition}

The next theorem proved by Dyson in \cite{Dyson} explains that the stochastic process of the eigenvalues of a Dyson Brownian motion matrix of size $N$ introduced in Definition \ref{def:dyson} has the same law as a Dyson particles system. 

\begin{theorem}[Dyson]
\label{thm:dyson}
Fix $\beta\in\{1,2\}$. Let $(D_N^\beta(t))_{t\in\R^+}$ be a $\beta$ Dyson Brownian motion of size $N$. Let $(\lambda^{i,\beta}(t))_{1\le i\le N,t\in\R^+}$ be the family of ordered eigenvalues of $(D_N^\beta(t))_{t\in\R^+}$. Then $(\lambda^{i,\beta})_{1\le i\le N,}$ is a weak solution of \eqref{dysonreal}. 

\end{theorem}
\begin{Remark}
For the notion of weak and strong solutions of stochastic differential equations we refer to \cite{oksendal2013stochastic}.
\end{Remark}
\begin{proof}
We shall not give a complete proof, focusing on the formal ideas. We want to emphasize on the methodology to go from a random matrix model to a system of SDE satisfied by the eigenvalues of this random matrix model. Here a main difficulty that shall not be treated is to prove that the process of eigenvalues is actually a semi martingale process in order for instance to use the Itô formula. A rigorous proof can be found in Theorem 4.3.2 of \cite{Alicelivre}. Nonetheless, the proof presented in \cite{Alicelivre} goes backward and so this requires to know the a priori system of SDE we should look for. That is why we present this method in order to identify the system of SDE that is satisfied by the eigenvalues of random matrix stochastic process. We follow the arguments given in \cite{Tao}.
\newline
We recall that $X$ is a complex centred Gaussian random variable of variance $1$ if $\Re(X)$ and $\Im(X)$ are independent real centred Gaussian random variable of variance $1/2$ (in particular $\E[\,|X|^2\,]=\E(\Re(X)^2+\Im(X)^2)=1$).   
\newline
Fix $t>0,\,h> 0$ be non negative real numbers with $h$ a small quantity. We have the following equalities:
\begin{equation}
\label{equation:particlesfirststep}
\begin{split}
   \lambda^{i,\beta}(t+h)&=\lambda^{i,\beta}(D_N^\beta(t+h))\\
   &=\lambda^{i,\beta}(D_N^\beta(t)+D_N^\beta(t+h)-D_N^\beta(t))\\
   &=: \lambda^{i,\beta}(D_N^\beta(t)+\widetilde{D_N^\beta}(h))
\end{split}
\end{equation} 
with $\widetilde {D_N^\beta}$ whose law is a $\beta$ Dyson Brownian motion independent of $D_N^\beta$ by the stationary and independent property of a real Brownian motion. 
\newline
We set $A(t):=D_N^\beta(t)$ and using the scaling property of a real Brownian motion we can rewrite \eqref{equation:particlesfirststep} as: $$\lambda^{i,\beta}(t+h)=:\lambda^{i,\beta}\left(A(t)+\sqrt{\cfrac{2h}{\beta N}}\,G^\beta\right),$$
with $G^\beta$ a $N$ real Wigner matrix in the case $\beta=1$ or a $N$ complex Wigner matrix in the case $\beta=2$ independent of $A(t)$.  
\newline
Now we condition on $A(t)$, since law of $G^\beta$ is the same as the law of $G^\beta$ after we condition by $A(t)$ (by independence of $G^\beta$ and $A(t)$), we can formally continue the proof as if $A(t)$ is a deterministic matrix. 
\newline
More exactly we shall show that almost surely for all $A\in \overset{\circ}{\mathcal S_N(\R)}$
\begin{equation}
\label{eq:valeurspropreshadamard}
\forall 1\le i\le N,\,\,\lambda_i\left(A+\sqrt{\cfrac{2h}{\beta N}}\,G^\beta\right)= \lambda_i(A)+\sqrt{\cfrac{2}{\beta N}}\,B_i(h)+\cfrac{h}{N}\dis\sum_{j\ne i}\cfrac{1}{\lambda_i(A)-\lambda_j(A)}+...
\end{equation} with $(B_i)_{1\le i\le N}$ a family of independent real centred Gaussian random variables of variance $h$. 
\newline
Assume that we proved this result, we apply \eqref{eq:valeurspropreshadamard} with $A=A_N(t)$ (which is almost surely in $\overset{\circ}{\mathcal S_N(\R)}$ by Proposition \ref{prop: spectre simple}) to have the following result: $$\lambda^{i,\beta}(t+h)-\lambda^{i,\beta}(t)=\sqrt{\cfrac{2}{\beta N}}\,B_i(h)+\cfrac{h}{N}\dis\sum_{j\ne i}\cfrac{1}{\lambda^{i,\beta}(t)-\lambda^{j,\beta}(t)}+...$$ with $(B_i)_{1\le i\le N}$ a family of independent real centred Gaussian random variables of variance $h$. In the Itô language it gives $$d\lambda_t^{i,\beta}=\cfrac{1}{N}\dis\sum_{j\ne i}\cfrac{1}{\lambda_t^{i,\beta}-\lambda_t^{j,\beta}}\,dt+\sqrt{\cfrac{2}{\beta N}}\,dB_t^i.$$
\newline
So it remains to prove \eqref{eq:valeurspropreshadamard}.
The Hadamard formulas of Proposition \ref{prop:hadamard} with the Example \ref{exemple:perturbationtheory2} yields:
\begin{equation}
\label{eq:hadamardvaleurppropre2}
\lambda_i\left(A+\sqrt{\cfrac{2 h}{\beta N}}\,G^\beta\right)=\lambda_i(A)+\sqrt{\cfrac{2 h}{\beta N}}\,\langle u_i|G^\beta u_i\rangle+\cfrac{2h}{\beta N}\dis\sum_{j\ne i}\cfrac{|\langle u_j|G^\beta u_i\rangle|^2}{\lambda_i(A)-\lambda_j(A)}+...
\end{equation} with $(u_i)_{1\le i\le N}$ an orthonormal basis of $A$ associated to the $(\lambda_i(A))_{1\le i\le N}$ as in Proposition \ref{prop:hadamard}. Then, we notice that by invariance of the law by a unitary change of basis of a $N$ real or complex Wigner matrix (Proposition \ref{prop:invarianceloiunitaire}), $(N_{i}:=\langle u_i|G^\beta u_i\rangle)_{1\le i\le N}$ are independent real centred Gaussian random variables of variance $1$. Similarly, for all $1\le j\ne i\le N$, $\langle u_j|G^\beta u_i\rangle$ is a real Gaussian random variable of variance $1/2$ in the case $\beta=1$ or a complex Gaussian random variable of variance $1$ in the case $\beta=2$. So for all $1\le j\ne i\le N$, we can do the following approximation thanks to the "nice" concentration property of Gaussian random variables $|\langle u_j|G^\beta u_i\rangle |^2\sim\E(|\langle u_j|G^\beta u_i\rangle |^2)=\beta/2$. 
Hence, we can rewrite \eqref{eq:hadamardvaleurppropre2} as:  $$\lambda_i\left(A+\sqrt{\cfrac{2h}{\beta N}}\,G^\beta\right)=\lambda_i(A)+\sqrt{\cfrac{2h}{\beta N}}\,N_i+\cfrac{h}{N}\dis\sum_{j\ne i}\cfrac{1}{\lambda_i(A)-\lambda_j(A)}+...$$ with $(N_i)_{1\le i\le N}$ independent real centred Gaussian random variables of variance $1$. Finally, defining for all $1\le i\le N$, $(B_i(h))_{h\ge0}:=(\sqrt{h}N_i)$ we obtain \eqref{eq:valeurspropreshadamard}. 
\end{proof}

\begin{definition}
\label{def:particlesWishart}
Fix $N$ distinct real numbers $x_1<...<x_N$. We call Wishart particles starting at $(x_i)_{1\le i\le N}$ the family of stochastic processes $(\lambda_.^{i})_{1\le i\le N}$ solution of the system of SDE:
\begin{equation}
\label{sdewishart}
\left\{
 \begin{split}
\forall 1\le i\le N,\, d\lambda_t^i&=\,\left(\cfrac{1}{N}\dis\sum_{j\ne i}\cfrac{\lambda_t^i+\lambda_t^j}{\lambda_t^i-\lambda_t^j}+\cfrac{M}{N}\right)\,dt+\,\sqrt{\cfrac{2\lambda_t^i}{N}}\,dB_t^i\\
\lambda_0^i&=x_i
 \end{split}
 \right.
 \end{equation} with $(B^i)_{1\le i\le N}$ a family of independent Brownian motions.
\end{definition}
We can state a similar result with the Dyson case.
\begin{theorem}
Let $(W_N(t))_{t\in\R^+}$ be a Wishart Brownian motion of size $N\times M$. Let $(\lambda^i(t))_{1\le i\le N,t\in\R^+}$ be the family of eigenvalues of $(W_N(t))_{t\in\R^+}$. Then $(\lambda^i)_{1\le i\le N}$ is a weak solution of \eqref{sdewishart}. 

\end{theorem}

\begin{proof}
We shall focus on how to obtain the counterpart of \eqref{eq:valeurspropreshadamard} in the Wishart case. We detail the case $N=M$ for simplicity but there is no difference when dealing with the general case.
\newline
By Definition \ref{def:wishart}, we can write $W_N(t)=\frac{1}{N}A_tA_t^*$ with $A_t$ introduced in the definition of a Wishart Brownian motion. 
\newline
As for the Dyson case we define $A_{t+h}-A_t=:\sqrt{h}G$ where $G$ has the law of a $N$ complex Ginibre matrix. We compute:

\begin{equation}
\label{eq:wisharteigenvalues}
\begin{split}
W_N(t+h)&=\cfrac{1}{N}\,A_{t+h}A_{t+h}^* \\
&= \cfrac{1}{N}\,\left(A_t+\sqrt{h}G\right)\left(A_t+\sqrt{h}G\right)^*\\
&=W_N(t)+\cfrac{\sqrt{h}}{N}\,A_t G^*+\cfrac{\sqrt{h}}{N} GA_t^*+\cfrac{h}{N}GG^*.
\end{split}
\end{equation}
Conditioning as for the Dyson case by $A_N(t)$, formally we have to study the evolution of the spectrum of 
$$\mathcal W(h):=\cfrac{1}{N}B B^*+\cfrac{\sqrt{h}}{N}\,B G^*+\cfrac{\sqrt{h}}{N}\,GB^*+\cfrac{h}{N}\,GG^*$$ for any $B\in M_{N}(\mathbb \R)$ a deterministic matrix. 
\newline
To do so, we introduce: $$\widetilde{\mathcal W}(h):=\,\cfrac{1}{N}B B^*+\cfrac{h}{N}\,B G^*+\cfrac{h}{N}\, GB^*+\cfrac{h^2}{N}\,GG^*.$$
Using the Hadamard formulas of Proposition \ref{prop:hadamard} yields: 
\begin{equation}
\label{wishart:hadamard1}
\begin{split}
\lambda_i(\widetilde{\mathcal W}(h))=\lambda_i(\widetilde{\mathcal W}(0))+\cfrac{h}{N}\langle u_i(BB^*)&|[BG^*+GB^*] u_i(BB^*)\rangle+\cfrac{h^2}{N}\langle u_i(BB^*)|GG^*\, u_i(BB^*)\rangle\\
&+\cfrac{h^2}{N^2}\dis\sum_{i\ne k}\cfrac{|\langle u_k(BB^*)|[BG^*+GB^*]\,u_i(BB^*)\rangle|^2}{\lambda_i(\widetilde{\mathcal W}(0))-\lambda_k(\widetilde{\mathcal W}(0))}
\end{split}
\end{equation}
with the notations of Proposition \ref{prop:hadamard}. 
\newline
To continue the computations, it is not clear to understand how $B$ or $B^*$ acts on $u_i(BB^*)$. So we shall use the singular value decomposition of $B$ to construct $(u_i(BB^*))_{1\le i\le N}$. We construct $(u_i(BB^*))_{1\le i\le N}$ such that $B^*u_i(BB^*)=\sqrt{\lambda_i(BB^*)}v_i=\sqrt{N}\sqrt{\lambda_i(\widetilde{\mathcal W}(0))}v_i$ with $(v_i)_{1\le i\le N}$ an orthonormal basis of $\R^N$ (we can consider a such construction thanks to the singular value decomposition of $B$). 
\newline
It gives: 
\begin{equation}
\label{wishart:hadamard2}
\begin{split}
\langle u_i(BB^*)|[BG^*+GB^*]u_i(BB^*)\rangle
&=\sqrt{N}\sqrt{\lambda_i(\widetilde{\mathcal W}(0))}\left[\langle v_i|G^* u_i(BB^*)\rangle+\langle u_i(BB^*) |Gv_i\rangle\right] \\
&=: \sqrt{2N}\sqrt{\lambda_i(\widetilde{\mathcal W}(0))}N_i
\end{split}
\end{equation}
with $(N_i:=\sqrt{2}\,\Re(\langle u_i(BB^*)|G v_i\rangle))_{1\le i\le N}$ a family of independent centred real Gaussian random variables of variance $1$ by invariance of the law of a Ginibre matrix by multiplication by unitary matrices (Proposition \ref{prop: inv ginibre law}).
\newline 
Moreover, since 
\begin{equation}
\label{wishart:hadamard3}
\langle u_i(BB^*)|GG^* u_i(BB^*)\rangle=||G^*u_i(BB^*)||_2^2=\sum_{k=1}^N |\langle v_k| G^*u_i(BB^*)\rangle|^2
\end{equation} and for all $1\le i,k\le N$, $\langle v_k| G^*u_i(BB^*)\rangle$ is a complex Gaussian random variable of variance 1 (again by invariance of the law of a Ginibre matrix by right and left multiplication by a unitary matrix), we can do the following approximation at the first order: 
\begin{equation}
\label{wishart:hadamard4}
\begin{split}
\langle u_i(BB^*)|GG^* u_i(BB^*)\rangle&=\sum_{k=1}^N |\langle v_k| G^*u_i(BB^*)\rangle|^2\\
&\sim \sum_{k=1}^N \E[|\langle v_k| G^*u_i(BB^*)\rangle|^2]\\
&=\sum_{k=1}^N 1=N
\end{split}
\end{equation}
\newline
Finally we also have for all $1\le k\ne i\le N$: 
\begin{equation}
\label{wishart:hadamard5}
\begin{split}
|\langle u_k(BB^*)&|[BG^*+GB^*]u_i(BB^*)\rangle|^2=\\
&N\left[\left|\sqrt{\lambda_k(\widetilde{\mathcal W}(0))}\langle v_k |G^*u_i(BB^*)\rangle+\sqrt{\lambda_i(\widetilde{\mathcal W}(0))}\langle u_k(BB^*)|Gv_i\rangle \right|^2\right]\\
\hspace{-2cm}&=:N\left[\lambda_k(\widetilde{\mathcal W}(0))|\widetilde{N_{i,k}}|^2+\lambda_i(\widetilde{\mathcal W}(0))|N_{k,i}|^2+2\sqrt{\lambda_i(\widetilde{\mathcal W}(0))\lambda_k(\widetilde{\mathcal W}(0))}\Re(\widetilde{N_{i,k}}N_{k,i})\right],
\end{split}
\end{equation} 
where $N_{k,i}:=\langle u_k(BB^*)|Gv_i\rangle$ and $\widetilde{N_{i,k}}:=\langle u_i(BB^*)|Gv_k\rangle$ are complex independent centred normal random variables of variance $1$ again by the invariance of the law of $G$ by left and right multiplication by unitary matrices. Doing the same approximation as before, we have: 
\begin{equation}
\label{wishart:hadamard6}
\begin{split}
N&\left[\lambda_k(\widetilde{\mathcal W}(0))|\widetilde{N_{i,k}}|^2+\lambda_i(\widetilde{\mathcal W}(0))|N_{k,i}|^2+2\sqrt{\lambda_i(\widetilde{\mathcal W}(0))\lambda_k(\widetilde{\mathcal W}(0))}\Re(\widetilde{N_{i,k}}N_{k,i})\right]\sim \\
&N\E\left[\lambda_k(\widetilde{\mathcal W}(0))|\widetilde{N_{i,k}}|^2+\lambda_i(\widetilde{\mathcal W}(0))|N_{k,i}|^2+2\sqrt{\lambda_i(\widetilde{\mathcal W}(0))\lambda_k(\widetilde{\mathcal W}(0))}\Re(\widetilde{N_{i,k}} N_{k,i})\right]\\
&=N\left[\lambda_k(\widetilde{\mathcal W}(0))\E(|\widetilde{N_{i,k}}|^2)+\lambda_i(\widetilde{\mathcal W}(0))\E(|N_{k,i}|^2)+2\sqrt{\lambda_i(\widetilde{\mathcal W}(0))\lambda_k(\widetilde{\mathcal W}(0))}\Re(\E(\widetilde{N_{i,k}}N_{k,i}))\right]\\
&=N\left[\lambda_k(\widetilde{\mathcal W}(0))+\lambda_i(\widetilde{\mathcal W}(0))\right]
\end{split}
\end{equation}
Using \eqref{wishart:hadamard2}, \eqref{wishart:hadamard4} and \eqref{wishart:hadamard6} in \eqref{wishart:hadamard1} yields: 
\begin{equation}
\lambda_i(\mathcal W(h))=\lambda_i(\mathcal W(0))+\sqrt{\cfrac{2\lambda_i(\mathcal W(0))}{N}}B_i(h)+h\left[1+\cfrac{1}{N}\dis\sum_{j\ne i}\cfrac{\lambda_i(\mathcal W(0))+\lambda_j(\mathcal W(0))}{\lambda_i(\mathcal W(0))-\lambda_j(\mathcal W(0))}\right]+...
\end{equation} where $(B_i(h):=\sqrt{h}N_i)_{1\le i\le N}$ is family of independent Gaussian random variables.
This corresponds to the Itô formulation of \eqref{eq:wisharteigenvalues}.
\end{proof}

This method is quite robust and can be used in many models as for instance if we consider a family of symmetric matrices with independent entries like the Dyson case but instead of having Brownian motions in the entries we have Ornstein-Uhlenbeck $(O_{i,j}(t))_{1\le i,j\le N}$ processes which are solutions of SDE $$dO_{i,j}=-\theta O_{i,j}dt+d(D_N^1)_{i,j}$$  with $D_N^1$ a Dyson Brownian motion of size $N$ and $\theta$ a parameter. Using the same approach one can show that the system of eigenvalues $(\lambda_i(t))_{1\le i\le N}$ satisfies the followings system of SDE: 
\begin{equation}
\label{eq:sdeornseteinuhlenbeck}
\forall 1\le i \le N, \,d\lambda_i(t)=\cfrac{1}{N}\dis\sum_{j\ne i}\cfrac{1}{\lambda_i(t)-\lambda_j(t)}\,dt+\cfrac{\sqrt{2}}{\sqrt{N}}\,dB_t^i-\theta \lambda_i(t) dt.
\end{equation}

\section{Study of the system of particles}
In this section we focus on general properties (existence, tightness of the empirical measure...) of solutions of Dyson Brownian motion SDE \eqref{dysonreal}. We could state similar results for other models like Wihsart or Ornstein-Uhlenbeck. We refer to \cite{bru1991wishart,graczyk2007moments} for the Wishart case and \cite{Chan,Shi} for the Ornstein-Uhlenbeck case. 
\newline
We shall follow the standard arguments to study the existence of a system of particles in interaction through a singular potential using a containment function. For instance, this method have been used for the Dyson Brownian case in \cite{Alicelivre}, for Coulomb gas in dimension 2 in \cite{bolley2018dynamics}, for the so-called generalized Dyson Brownian motion and Ornstein-Uhlenbeck models in \cite{li2013generalized,Shi} or for the cicular Dyson Brownian motion \cite{cepa2001brownian}.
\newline
We fix an integer $N\ge2$. 
\subsection{Notations}
We introduce the open subset on which we will work:  $$D=\R^N-\bigcup_{i\ne j}\{(x_1,...,x_N)\in\R^N: x_i=x_j\}$$ and its boundary $$\partial D=\{+\infty\}\cup_{i\ne j}\{(x_1,...,x_N)\in\R^N: x_i=x_j\}.$$
Let us write $\R_>^N$ the subset of $\R^N$: $\{(x_1,...,x_N)\in\R^N: x_1<x_2<...<x_N\}$.
We shall write $x=(x_1,...,x_N)$ for an element of $\R^N$. 
\newline
We say that $x$ is $\varepsilon$ near of $\partial D$ there exists $i$ such that $|x_i|>\varepsilon^{-1}$ or if there exists $i\ne j$ such that $|x_i-x_j|\le \varepsilon$.
\newline
We introduce an energy associated to the Dyson model. If $x\in D$, we define its energy by $$\mathcal E(x)=\cfrac{1}{N}\dis\sum_{i=1}^NV(x_i)+\cfrac{1}{2N^2}\dis\sum_{1\le i \ne j \le N} W(x_i-x_j)=:\mathcal E_V(x)+\mathcal E_W(x)$$ where $V$ is an external confinement potential which is in our case $V(x)=x^2$ as the kinetic energy that will ensure that the system of particle is spatially confined and $W(x)=-\log(|x|^2)$ is potential energy of interaction associated to our system since $-W'(x)=2/x$.
\newline
Let $\alpha_N$ and $\beta_N$ be two positive real numbers. 
In this section we shall focus on solutions of the following system of SDE:
\begin{equation} 
\label{eq:dysoncoeff}
\forall 1\le i\le N,\, d\lambda_t^i=-\cfrac{\alpha_N}{2N^2}\dis\sum_{j\ne i}W'(\lambda_t^i-\lambda_t^j)\,dt+2\sqrt{\cfrac{\alpha_N}{\beta_N}}\,dB_t^i
\end{equation}
 with $(B_i)_{1\le i\le N}$ a family of independent Brownian motions.
\newline
The cases of Dyson Brownian motion defined in Definition \ref{def:dyson} correspond to the particular cases $\alpha_N=N$ and $\beta_N=4N^2$ for $\beta=2$ and $\beta_N=2N^2$ for $\beta=1$.
\newline
The first goal is to study under which condition on $\alpha_N$ and $\beta_N$ the system is well posed starting from an initial data in $D$ which is not clear a priori since there is a singular term in the system. On the one side the singular term is a repulsive interaction that should help the system to be well defined but if the fluctuations due to the Brownian motion parts are too important collisions may appear.
\newline
For this aim, we introduce the following stopping times for a solution of \eqref{eq:dysoncoeff}: for all $\varepsilon>0$, $$T_\varepsilon=\inf\{t\ge0: \underset{1\le i\le N}{\max} |\lambda_t^i|\ge\varepsilon^{-1} \, \text{or}\, \underset{1\le i,j\le N}{\min}|\lambda_t^i-\lambda_t^j|\le \varepsilon\}$$ and $$T_{\partial D}=\,\underset{\varepsilon>0}{\sup}T_\varepsilon.$$

\subsection{Formulas}
The goal of this section is to give formulas to explain how the generator of the system of particles solution of \eqref{eq:dysoncoeff} acts on $\mathcal E$.
\newline
Firstly the generator associated to the system $(\lambda_t^i)_{1\le i\le N}$ solution of \eqref{eq:dysoncoeff} is $$Lf=2\cfrac{\alpha_N}{\beta_N}\Delta f-\cfrac{\alpha_N}{2}\nabla \mathcal E_W.\nabla f,$$ for $f$ a smooth function from $D$ to $\R$. It means that for a smooth function $f$ from $D$ to $\R$, the Itô formula yields: $$\E(f(\lambda_t^1,..,\lambda_t^N))=f(\lambda_0^i,...,\lambda_0^N)+\E\left(\dis\int_0^t Lf(\lambda_t^1,...,\lambda_t^N)dt\right).$$
For $x\in D$ and $1\le i,j\le N$, we have the following formulas: 
\begin{align}
\cfrac{\partial \mathcal E_V}{\partial x_i}(x)&=\cfrac{2x_i}{N}\\
\cfrac{\partial \mathcal E_W}{\partial x_i}(x)&=-\cfrac{2}{N^2}\dis\sum_{j\ne i}\cfrac{1}{x_i-x_j}\\
\nabla^2 \mathcal E_V(x)&=\cfrac{2}{N}I_N\\
\nabla^2\mathcal E_W(x)&=\cfrac{1}{N^2} A
\end{align}
with $A$ a matrix such that $A_{i,i}=2\dis\sum_{j\ne i}\cfrac{1}{(x_i-x_j)^2}$ and $A_{i,j}=-\cfrac{2}{(x_i-x_j)^2}$ if $i\ne j$.
\newline
We can now compute $L\mathcal E_V$ and after $L\mathcal E_W$ to finally compute  $L\mathcal E$.
\newline
We compute $$L\mathcal E_V(x)=\cfrac{2\alpha_N}{\beta_N}\times 2-\cfrac{\alpha_N}{2}\dis\sum_{1\le i\ne j\le N}\cfrac{-4x_i}{N^3}\cfrac{1}{x_i-x_j}=\cfrac{4\alpha_N}{\beta_N}+\cfrac{2\alpha_N}{N^3}\dis\sum_{1\le i\ne j\le N}\cfrac{x_i}{x_i-x_j}.$$
Remark that $$\dis\sum_{1\le i\ne j\le N}\cfrac{x_i}{x_i-x_j}=\dis\sum_{1\le i\ne j\le N}\cfrac{x_j}{x_j-x_i}$$
It yields: $$2\dis\sum_{1\le i\ne j\le N}\cfrac{x_i}{x_i-x_j}=\dis\sum_{1\le i\ne j\le N}\cfrac{x_i-x_j}{x_i-x_j}=N(N-1).$$
This gives 
\begin{equation}
\label{calculLHV}
L\mathcal E_V(x)=\cfrac{4\alpha_N}{\beta_N}+\cfrac{\alpha_N(N-1)}{N^2}.
\end{equation}
Now we compute $L\mathcal E_W(x)$.
Start with: 
\begin{align}
L\mathcal E_W(x)&=2\cfrac{\alpha_N}{\beta_N}\dis\sum_{i=1}^N\cfrac{2}{N^2}\dis\sum_{j\ne i}\cfrac{1}{(x_i-x_j)^2}-\cfrac{\alpha_N}{2}|\nabla \mathcal E_W(x)|^2\\
&=4\cfrac{\alpha_N}{N^2\beta_N}\dis\sum_{i=1}^N\dis\sum_{j\ne i}\cfrac{1}{(x_i-x_j)^2}-\cfrac{2\alpha_N}{N^4}\dis\sum_{i=1}^N\left(\dis\sum_{j\ne i}\cfrac{1}{x_i-x_j}\right)^2.
\end{align}
Notice that
\begin{equation}
\begin{split}
\dis\sum_{i=1}^N\left[\left(\dis\sum_{j\ne i}\cfrac{1}{x_i-x_j}\right)^2-\dis\sum_{j\ne i}\cfrac{1}{(x_i-x_j)^2}\right]&=\dis\sum_{j\ne i,k\ne i, j\ne k}\cfrac{1}{(x_i-x_k)(x_i-x_j)}\\
&=\dis\sum_{j\ne i,k\ne i, j\ne k}\cfrac{1}{x_j-x_k}\left(\cfrac{1}{x_i-x_j}-\cfrac{1}{x_i-x_k}\right)\\
&=-\dis\sum_{j\ne i,k\ne i, j\ne k}\cfrac{1}{(x_j-x_k)(x_j-x_i)}-\dis\sum_{j\ne i,k\ne i, j\ne k}\cfrac{1}{(x_k-x_j)(x_k-x_i)}\\
&=-2\dis\sum_{j\ne i,k\ne i, j\ne k}\cfrac{1}{(x_i-x_k)(x_i-x_j)}
\end{split}
\end{equation}
This gives: 
\begin{equation}
3\dis\sum_{j\ne i,k\ne i, j\ne k}\cfrac{1}{(x_i-x_k)(x_i-x_j)}=0\,\text{ and } \sum_{i=1}^N\left[\left(\dis\sum_{j\ne i}\cfrac{1}{x_i-x_j}\right)^2-\dis\sum_{j\ne i}\cfrac{1}{(x_i-x_j)^2}\right]=0
\label{cotlard}
\end{equation}

Hence, we have: 
\begin{equation}
\label{calculLHW}
L\mathcal E_W(x)=\left[\cfrac{4\alpha_N}{N^2\beta_N}-\cfrac{2\alpha_N}{N^4}\right]\dis\sum_{i=1}^N\dis\sum_{j\ne i}\cfrac{1}{(x_i-x_j)^2}.
\end{equation}
Computations \eqref{calculLHV} and \eqref{calculLHW} can be summarized into 
\begin{equation}
\label{calculLH}
L\mathcal E(x)=\cfrac{4\alpha_N}{\beta_N}+\cfrac{\alpha_N(N-1)}{N^2}+\left[4\cfrac{\alpha_N}{N^2\beta_N}-\cfrac{2\alpha_N}{N^4}\right]\dis\sum_{i=1}^N\dis\sum_{j\ne i}\cfrac{1}{(x_i-x_j)^2}.
\end{equation}
We obtain the following result.
\begin{prop}
\label{prop:upperboundenergy}
Assume that $\beta_N\ge 2N^2$. Then we have the following upper bound $$\sup_{x\in D} L\mathcal E(x)\le\cfrac{4\alpha_N}{\beta_N}+\cfrac{\alpha_N(N-1)}{N^2}.$$
\end{prop}

\subsection{Existence of solutions for the Dyson case}
\subsubsection{Existence thanks to a containment function}
Let state the theorem we will prove in this section.
\begin{theorem}
\label{thm:existenceparticule}
Assume that $\beta_N\ge 2N^2$. For any initial data $\lambda_0\in D$, there exists a unique strong solution of $~\eqref{eq:dysoncoeff}$ defined on $[0,+\infty($ and $T_{\partial D}=+\infty$ almost surely. 
\end{theorem}
This proof relies on the fact that $\mathcal E$ is a good containment function. We follow the arguments of \cite{bolley2018dynamics}.
\begin{lemma}
\label{lemma:exploenergy}
We have the following properties for $\mathcal E$: 
\begin{enumerate}
\item{\label{enumerate1}$\mathcal E\ge0$}
\item{\label{enumerate2}$\underset{x\to\partial D}{\lim}\mathcal E(x)=+\infty$}
\item{\label{enumerate3}$\exp(-\beta \mathcal E)$ is integrable on $D$ for any $\beta>0$.}
\end{enumerate}
\end{lemma}

\begin{proof}
Let $x$ in $D$. Start with: $$\cfrac{1}{2}\dis\sum_{1\le i\ne j\le N}|x_i-x_j|^2=\cfrac{1}{2}\dis\sum_{1\le i, j\le N}|x_i-x_j|^2=N\dis\sum_{1\le N}x_i^2-\left(\dis\sum_{1\le i\le N}x_i\right)^2\le N\dis\sum_{1\le i\le N}x_i^2=N|x|^2.$$
\newline
So we get that: $$2N^2\mathcal E(x)=N|x|^2+N|x|^2-\dis\sum_{1\le i\ne j\le N}\log|x_i-x_j|^2\ge N|x|^2+ \dis\sum_{1\le i\ne j\le N}\left(\cfrac{|x_i-x_j|^2}{2}-\log|x_i-x_j|^2\right).$$
Since for all $u>0$, $\cfrac{u}{2}-\log(u)\ge 1-\log(2)>1/4$, we can lower bound for $\mathcal E$ on $D$: 
\begin{equation}
\label{eq:lowerboundenergy}
\mathcal E(x)\ge\cfrac{|x|^2}{2N}+\cfrac{N(N-1)}{8N^2}\ge\cfrac{|x|^2}{2N}+\cfrac{1}{16}.
\end{equation}
It proves \ref{enumerate1} and \ref{enumerate3}.
\newline
We now prove \ref{enumerate2}. By definition of $\partial D$, we must prove that for all $R>0$ there exists $A>0$ and $\varepsilon>0$ such that $H(x)>R$ if there exists $1\le i\le N$ such that $|x_i|>A$ or if there exists $1\le i\le N$ and $1\le j\le N$ such that $|x_i-x_j|\le \varepsilon$. 
\newline
We fix $R>0$. First, notice that by $\eqref{eq:lowerboundenergy}$, we can find $A>0$ large enough such that if there exists $1\le i\le N$ such that $|x_i|>A$ then $\mathcal E(x)\ge R$. We now fix a such $A>0$. We can now assume furthermore that for all $1\le l\le N$, $|x_l|<A$
because otherwise as explained $\mathcal E(x)>R$. 
\newline
Let $\varepsilon>0$ and assume that there exists $1\le i\le N$ and $1\le j\le N$ such that $|x_i-x_j|\le \varepsilon$, then 
\begin{equation}
\label{eq:lowerboundH2}
\begin{split}
N^2\mathcal E(x)&\ge -2\log|x_i-x_j|-\dis\sum_{1\le k\ne j\le N, \{k,l\}\ne\{i,j\}}\log|x_k-x_l|\\
&\ge -2\log(\varepsilon)-\sum_{1\le k\ne j\le N, \{k,l\}\ne\{i,j\}}\log|x_k-x_l|.
\end{split}
\end{equation}
Using the inequality $|x-y|\le(1+|x|)(1+|y|)$ for all $(x,y)\in\R^2$, we get that for all $l\ne k$: $$-\log|x_k-x_l|\ge-\log(1+|x_k|)-\log(1+|x_k|)\ge-2\log(1+A).$$
Hence it yields: $$N^2\mathcal E(x)\ge-2\log(\varepsilon)-2N^2\log(1+A), $$ so that we can choose $\varepsilon>0$ small enough to be sure that $\mathcal E(x)$ is greater than $R$. This concludes the proof.
\end{proof}
We now prove Theorem \ref{thm:existenceparticule}.
\begin{proof}
This proof is quite standard when dealing with existence of SDE with a potential explosion in finite time. We regularize the singular part in the SDE, then construct a solution up to an explosion time and finally we prove that this solution is defined globally in time.
\newline
Let $\lambda_0\in D$ be an initial data and $\varepsilon>0 $ smaller than $\min_{1\le i\ne j\le N+1}(|\lambda_0^i-\lambda_0^j|)$ and such that $\varepsilon^{-1}>\max_{i}\lambda_0^i$. 
\newline
We consider a smooth approximation $W_\varepsilon$ of $W$ which coincides with $W$ on $\R-[-\varepsilon,\varepsilon]$ and we define as before $\mathcal E_\varepsilon=\mathcal E_V+\mathcal E_{W_\varepsilon}$. 
By existence and uniqueness of the solution to the SDE we can consider $(\lambda_t^\varepsilon)_{t\ge0}$ starting from $\lambda_0$ solution to the following problem: $$d\lambda_t^\varepsilon=2\sqrt{\cfrac{\alpha_N}{\beta_N}}dB_t-\cfrac{\alpha_N}{2}\nabla \mathcal E_{W_\varepsilon}(\lambda_t^\varepsilon)dt.$$
For $\varepsilon'<\varepsilon$, let $$T_{\varepsilon, \varepsilon'}=\inf\{s>0: \min_{1\le i\ne j\le N+1}|(\lambda_s^{\varepsilon'})^i-(\lambda_s^{\varepsilon'})^j|\le\varepsilon\text{ or }\max_{i}|(\lambda_s^{\varepsilon'})^i|\ge\varepsilon^{-1}\}$$ be a stopping time such that by definition we have that the processes $\lambda^\varepsilon$ and $\lambda^{\varepsilon'}$ coincide up to this time.
Hence, this allows us to define for all $\varepsilon$ as in the beginning of the proof, a stopping time $T_\varepsilon:=T_{\varepsilon,\varepsilon'}$ for $\varepsilon'<\varepsilon$ and $(\lambda_t)_{0\le t\le T_\varepsilon}:=(\lambda_t^{\varepsilon'})_{0\le t\le T_\varepsilon}$ for $\varepsilon'<\varepsilon$. We clearly have that if $\varepsilon'<\varepsilon$, $T_{\varepsilon'}\ge T_{\varepsilon}$ almost surely, and so that we can uniquely defined $\lambda$ until the stopping time $T_{\partial D}$. Moreover by definition of $W_\varepsilon$, we have that $\lambda$ satisfies the SDE with $W$ until the time $T_{\partial D}$ since it satisfies it for all $\varepsilon$ on $[0,T_\varepsilon]$ and $W$ and $W_\varepsilon$ coincide on $\R-[-\varepsilon,\varepsilon]$. So it remains to prove that $T_{\partial D}=+\infty$ almost surely to conclude. 
\newline
Let us define for all $R>0$ the stopping times: $$T'_R=\inf\{t\ge0\,:\, \mathcal E(\lambda_t)>R\}$$ and $$T'=\lim_{R\to\infty} T'_R.$$
By the previous lemma: $T'=+\infty\subset T_{\partial D}=+\infty$ because if $\lambda_t$ converges to $\partial D$ in a finite time, $\mathcal E(\lambda_t)$ would explode in finite time by Lemma \ref{lemma:exploenergy}. 
\newline
Let $R>0$, thanks to Lemma \ref{lemma:exploenergy} we can find $\varepsilon$ (that depends on $R$) such that $T_\varepsilon\ge T'_R$. 
\newline
Hence, we can use Itô's formula to have for all $t>0$: 
$$\E(\mathcal E_\varepsilon(\lambda^\varepsilon_{t\wedge T'_R}))-\mathcal E_\varepsilon(\lambda_0^\varepsilon)=\E\left(\int_0^{t\wedge T'_R}L\mathcal E_\varepsilon(\lambda_s^\varepsilon)ds\right).$$
Since $T_\varepsilon \ge T'_R$ it gives:
\begin{equation}
\label{eq:equationparticleito}
\E(\mathcal E(\lambda_{t\wedge T'_R}))-\mathcal E(\lambda_0)=\E\left(\int_0^{t\wedge T'_R}L\mathcal E(\lambda_s)ds\right).
\end{equation}
Now by Proposition \ref{prop:upperboundenergy}, $L\mathcal E$ is uniformly bounded on $D$ by a constant (which depends on $N$) that we will denote $K_N$. 
\newline 
So, we have the following upper bound for every $t>0$, $$\E\left(\int_0^{t\wedge T'_R}L\mathcal E(\lambda_s)ds\right)\le K_Nt.$$ 
As a consequence of the Itô formula $\eqref{eq:equationparticleito}$, for every $t>0$, $$\underset{R>0}{\sup}\,\E(\mathcal E(\lambda_{t\wedge T'_R}))<+\infty.$$
Thanks to that we can use a kind of Markov estimate. For every $t>0$ since $\mathcal E$ is non negative we have: $$R \mathds{1}_{T'_R\le t}\le \mathcal E(\lambda_{t\wedge T'_R}).$$
By taking the expectation for every $t>0$ we get: $$\mP(T'_R\le t)\le\cfrac{\underset{R>0}{\sup}\,\E(\mathcal E(\lambda_{t\wedge T'_R}))}{R}.$$
Hence, let $R\to+\infty$ to have that for all $t>0$, $$\mP(T'\le t)=0.$$ 
So $T'=+\infty$ almost surely which concludes the proof. 
\end{proof}

The criterion $\beta_N\ge 2N^2$ is the better that can be obtained. Indeed if $\beta_N<2N^2$ then Cépa and Lépingle proved that collisions occur almost surely \cite{cepa2001brownian,cepa1997diffusing}.
We sketch a simple argument which can be found in Lemma 2 of \cite{guillin2025quasi}. 

\begin{prop}
If $\beta_N<2N^2$ then collisions of \eqref{eq:dysoncoeff} occur almost surely. 
\end{prop}

\begin{proof}
Assume by contradiction that a strong solution $(\lambda_t^i)_{t\ge 0,\,1\le i\le N}$ of \eqref{eq:dysoncoeff} exists for every time and that $T_{\partial_D}=+\infty$ on a $\Omega'\subset\Omega$ of nonzero probability. 
\newline 
Up to replace $\lambda_t^i$ by $\frac{\sqrt{\beta_N}}{2\sqrt{\alpha_N}}\lambda_t^i$ we can consider that, instead of \eqref{eq:dysoncoeff}, the system of particles $(\lambda_t^i)_{t\ge 0,\,1\le i\le N}$ is solution of 

\begin{equation}
\label{eq:dysoncoeffbis}
\forall 1\le i\le N,\, d\lambda_t^i=\gamma_N\dis\sum_{j\ne i}\cfrac{1}{\lambda_t^i-\lambda_t^j}\,dt+dB_t^i,
\end{equation}
with $\gamma_N:=\frac{\beta_N}{4N^2}$. So, the hypothesis that $\beta_N<2N^2$ is replaced by $\gamma_N<\frac{1}{2}$.
\newline
By hypothesis for every time $t>0$ we have $(\lambda_t^i)_{1\le i\le N}\in\R^N_>$. Let $l\in\{1,N-1\}$ be an index and define $\sigma_{l,l+1}:=\inf\{t>0,\, \lambda_t^l=\lambda_t^{l+1}\}$. By hypothesis $\sigma_{l,l+1}$ is infinite on $\Omega'$. 
\newline
Applying the Itô formula to $S_t:=|\lambda_t^{l+1}-\lambda_t^l|^2$ we obtain: 
\begin{align*}
S_t=S_0+2t+2\int_0^t\sqrt{S_s}(dB_s^{l+1}-dB_s^l)+2\gamma_N\int_0^t \left(\sum_{j=1,\,j\ne l}^N\cfrac{\lambda_s^{l+1}-\lambda_s^l}{\lambda_s^{l+1}-\lambda_s^j}+\sum_{j=1,\,j\ne l+1}^N\cfrac{\lambda_s^{l}-\lambda_s^{l+1}}{\lambda_s^{l}-\lambda_s^j}\right)\,ds.
\end{align*} 
We remark that for every $s\ge 0$
\begin{align*}
\sum_{j=1,\,j\ne l}^N\cfrac{\lambda_s^{l+1}-\lambda_s^l}{\lambda_s^{l+1}-\lambda_s^j}+\sum_{j=1,\,j\ne l+1}^N\cfrac{\lambda_s^{l}-\lambda_s^{l+1}}{\lambda_s^{l}-\lambda_s^j}&=2+\sum_{j\notin\{l,l+1\}}\left(\cfrac{\lambda_s^{l+1}-\lambda_s^l}{\lambda_s^{l+1}-\lambda_s^j}+\cfrac{\lambda_s^{l}-\lambda_s^{l+1}}{\lambda_s^{l}-\lambda_s^j}\right)\\
&=2+\sum_{j\notin\{l,l+1\}}\cfrac{(\lambda_s^{l+1}-\lambda_s^j)^2}{(\lambda_s^{l+1}-\lambda_s^j)(\lambda_s^{j}-\lambda_s^l)}\\
&\le 2
\end{align*}
because $\lambda_s\in\R^N_>$.
It yields:
\begin{align*}
S_t\le S_0+2t(2\gamma_N+1)+2\sqrt{2}\int_0^t\sqrt{S_s}dW_s^l
\end{align*}
where $W_s^l:=(B_s^{l+1}-B_s^l)/\sqrt{2}$ is a standard Brownian motion.
By the comparison theorem for SDE of Ikeda and Watanabe \cite{ikeda1977comparison} we obtain that for all $t>0$ $$0\le S_t\le R_t$$ where $(R_t)_{t\ge 0}$ is the solution of the SDE $$dR_t=2(2\gamma_N+1)dt+2\sqrt{2}\sqrt{S_t}dW_t^l.$$
$(R_t)_{t\ge 0}$ is known as a squared Bessel process of parameter $2\gamma_N+1\in(1,2)$. Standard results about Bessel processes (see \cite{revuz2013continuous} for instance) prove that almost surely there exists $t>0$ such that $R_t=0$ and so $\sigma _{l,l+1}$ is finite which is a contradiction.
\end{proof}
\begin{Remark}
Cépa and Lépingle also proved that actually multiple collisions can not occur \cite{cepa2007no}. It means that in the regime $\beta_N<2N^2$ when there is a collision, it can not be strictly more than two particles that collide. Formally, this is linked with the fact that in dimension 1 two independent Brownian motions collide almost surely but the probability that strictly more than two independent Brownian motions collide at the same time is zero. The technique they used in order to study the multiple collisions is based on the properties of squared Bessel processes. Similar techniques have been used to study the existence and the collisions of more delicate systems of particles (typically when the singular term is attractive instead of repulsive as in the Dyson case). For instance, see the Keller-Seggel model for chemotaxis \cite{fournier2024collisions}.
\end{Remark}

In the particular case of the Dyson Brownian motion, it gives the following result.

\begin{theorem}
\label{thm:existenceparticuledyson}
Let $\beta\in\{1,2\}$. For any initial data $\lambda_0\in D$, there exists a unique strong solution of $~\eqref{dysonreal}$ defined on $[0,+\infty)$ and $T_{\partial D}=+\infty$ almost surely. 
\end{theorem}

\begin{Remark}
\label{remark:dysonextended}
Let us briefly explain how to give a sense to a solution to \eqref{dysonreal} when $\lambda_0$ is in $\R^N$ and not necessarily in $D$. The following argument is used in Theorem 4.3.2 in \cite{Alicelivre}. Fix $\beta\in\{1,2\}$
It is known that for all $t>0$, almost surely the spectrum of $D_N^\beta(t)$ introduced in Definition \ref{def:dyson} is simple by Proposition \ref{prop: spectre simple}.
For all $t\ge 0$, we denote $(\lambda_t^{j,\beta})_{1\le j\le N}$ the ordered spectrum of $D_N^\beta(t)$.
\newline
Let $(t_n)_{n\ge 0}$ be a sequence of positive real numbers that goes to 0 when $n$ goes to $\infty$. Almost surely for all $n\ge 0$, $(\lambda_{t_n}^i)_{1\le i\le N}\in\R^N_>$. By Theorem \ref{thm:existenceparticuledyson}, for all $n$ we can consider the unique strong solution to \eqref{dysonreal} starting from $(\lambda_{t_n}^{i,\beta})_{1\le i\le N}$, $(\tilde\lambda_t^{i,\beta})_{t\ge t_n,\, 1\le i\le N}$ whose law coincides with the law of the eigenvalues of $(D_N^\beta(t))_{t\ge t_n}$ by the Dyson Theorem \ref{thm:dyson} \cite{Dyson}. By uniqueness of the trajectory obtained in Theorem \ref{thm:existenceparticuledyson}, we can construct a solution $(\tilde\lambda_t)_{t>0}$ of the Dyson equation \eqref{dysonreal} whose law coincides with the law of the eigenvalues of $(D_N^\beta(t))_{t>0}$ and which is in $\R^N_>$ for all $t>0$. Moreover the Hoffman-Wielandt inequality stated in Lemma \ref{lemma: hoffman wielandt} implies that $\lambda_t^\beta$ converges to $\lambda_0$ when $t$ goes to 0 by continuity in 0 of $(D_N^\beta(t))_{t\ge 0}$. This proves existence and uniqueness in law of solutions of \eqref{dysonreal} starting from $\lambda_0\in D$.  
\newline
For solutions of \eqref{eq:dysoncoeff}, in general we can not interpret the particles as eigenvalues of random matrices to give a sense to a solution starting from $\lambda_0\in\R^N$. Nonetheless, we can define solutions using a comparison principle of the particles and more stochastic calculus. See Lemma 4.3.6 and Proposition 4.3.5 of \cite{Alicelivre}.
\end{Remark}
\subsubsection{Solution of the Fokker Planck equation of the Dyson system}
In this part, we follow \cite{Tao} to solve the Fokker Planck equation.
To lighten the notation, we take $\alpha_N$ and $\beta_N$ such that the Dyson equation in this case is
\begin{equation}
 \forall 1\le i\le N,\, d\lambda_t^i=\dis\sum_{j\ne i}\cfrac{1}{\lambda_t^i-\lambda_t^j}\,dt+\,dB_t^i
\label{dysonbis}
\end{equation}
\begin{Remark}
As explained in Section \ref{subsection:particlesinmeanfieldinteraction}, this system of SDE corresponds to the dynamic of the eigenvalues of $\sqrt{N}D_N^2(t)$ where $D_N^\beta$ is the Dyson Brownian motion of size $N\ge 1$ introduced in Definition \ref{def:dyson}.
\end{Remark}

\begin{prop}
Let $\lambda_0$ in $\R^N_>$, $(\lambda_t^i)_{t\in\R^+}$ be the solution of $~\eqref{dysonbis}$ starting from $\lambda_0$ and $f: \R^n\to R$ be a smooth function with bounded derivatives, then for any $t\ge 0$ we have: \begin{equation}
\label{eq:generateur}
\partial_t\E[f(\lambda_t^1,...,\lambda_t^N)]=\E[D^*f(\lambda_t^1,...,\lambda_t^N))]
\end{equation}
 where $D^*$ is the adjoint of the Dyson operator defined by: 
\begin{equation}
D\phi=\cfrac{1}{2}\sum _{i=1}^N\cfrac{\partial^2 \phi}{\partial \lambda_i^2}-\sum_{1\le i,j\le N, i\ne j}\cfrac{\partial}{\partial\lambda_i}\left(\cfrac{\phi}{\lambda_i-\lambda_j}\right)
\end{equation}
 and its adjoint is: 
 \begin{equation}
D^*\phi=\cfrac{1}{2}\sum _{i=1}^n\cfrac{\partial^2 \phi}{\partial \lambda_i^2}+\sum_{1\le i,j\le N, i\ne j}\cfrac{\partial\phi}{\partial\lambda_i}\cfrac{1}{\lambda_i-\lambda_j}.
\end{equation}
Moreover, let $\rho_t$ be the distribution of the ordered vector $(\lambda_t^1,...,\lambda_N^t)$, then $\rho$ satisfies the Fokker Planck equation in the sense of distribution on the subset of $\R_>^N$:\begin{equation}
\label{eq:edpdensité}
\partial_t\rho=D\rho.
\end{equation}
\end{prop} 

\begin{proof}
The proof is an immediate application of the Itô formula using \eqref{dysonbis}. Indeed, since the derivatives of $f$ are bounded, the local martingale part of the Itô formula is actually a real martingale with expectation 0 given \eqref{eq:generateur}. The Fokker Planck equation \eqref{eq:edpdensité} is obtained by integrations by parts as usual. 
\end{proof}

The Fokker Planck equation seems difficult to solve a priori. However, this equation is not far from the heat equation as we shall show. We shall give an intuition for this idea in Remark \ref{remark:conditionning}.
\newline
We recall that given $\lambda=(\lambda_1,...\lambda_N) \in\R^N$ we define the Vandermonde determinant of these real numbers the quantity: $$\Delta_N(\lambda)=\prod_{1\le i<j\le N} (\lambda_j-\lambda_i), $$ which is also given by 
$$\Delta_N(\lambda_1,...,\lambda_N)=\det[(\lambda_i^{j-1})_{1\le i,j\le N}].$$

\begin{Remark}
\label{rem:vanish}
We have that $\Delta_N(\lambda_1,...,\lambda_N)=0$ if and only if two of the $\lambda_i$ are equal. 
\end{Remark}
We can compute the derivatives of the Vandermonde determinant.
\begin{lemma}
\label{lemma:formuladerivativevander}
For $(\lambda_1,...,\lambda_N)\in \R_>^N$, we have the following derivatives: $$\cfrac{\partial \Delta_N}{\partial \lambda_i}(\lambda_1,...,\lambda_N)=\sum_{1\le j \le N, i\ne j}\cfrac{\Delta_N(\lambda_1,...,\lambda_N)}{\lambda_i-\lambda_j}$$
$$\Delta(\Delta_N)(\lambda_1,...,\lambda_N)=0$$ with $\Delta$ the Laplacian operator.
\end{lemma}

\begin{proof}
The first derivative is immediate by derivation. The second one is obtained by computing the derivatives and using the identities \eqref{cotlard}. Indeed, using the formula for the fist derivative yields:
\begin{align*}
&\Delta(\Delta_N)(\lambda_1,...,\lambda_N)=\sum_{i=1}^N \cfrac{\partial^2 \Delta_N}{\partial \lambda_i^2}(\lambda_1,...,\lambda_N)\\
&=\sum_{i=1}^N\left[\Delta_N(\lambda_1,...,\lambda_N)\,\sum_{1\le k \le N, i\ne k}\cfrac{1}{\lambda_i-\lambda_k}\,\sum_{1\le j \le N, i\ne j}\cfrac{1}{\lambda_i-\lambda_j}-\Delta_N(\lambda_1,...,\lambda_N)\sum_{1\le j \le N, i\ne j}\cfrac{1}{(\lambda_i-\lambda_j)^2}\right]\\
&=\Delta_N(\lambda_1,...,\lambda_N)\sum_{i=1}^N\left[\sum_{1\le k \le N, i\ne k}\cfrac{1}{\lambda_i-\lambda_k}\,\sum_{1\le j \le N, i\ne j}\cfrac{1}{\lambda_i-\lambda_j}-\sum_{1\le j \le N, i\ne j}\cfrac{1}{(\lambda_i-\lambda_j)^2}\right]\\
&\overset{\eqref{cotlard}}{=}0
\end{align*}
\end{proof}

The harmonic property of the Vandermonde determinant helps us to solve the Fokker Planck equation. The most important property to solve the Fokker Planck equation \eqref{eq:edpdensité} is the following one.

\begin{prop}
For $\rho$ a smooth solution on $\R_>^N$ of $$\partial_t\rho=D\rho,$$ then define $u$ by: \begin{equation}
\label{eq:changefunction}
\rho=\Delta_N u
\end{equation}
on $\R_>^N$. Then $u$ is a solution of the heat equation: 
\begin{equation}
\label{eq:edpheat}
\partial_tu=\cfrac{1}{2}\Delta u
\end{equation} on $\R_>^N$.
\end{prop}

\begin{proof}
The computation is quite straight forward.
Differentiating the equality $\rho=\Delta_N u$ yields: $$\partial_t\rho=\Delta_N\partial_tu,$$ $$\cfrac{1}{2}\Delta(\rho)=\cfrac{1}{2}\Delta(\Delta_N)u+\cfrac{1}{2}\Delta(u)\Delta_N+\sum_{i=1}^N\cfrac{\partial u}{\partial \lambda_i}\cfrac{\partial \Delta_N}{\partial \lambda_i},$$
$$\sum_{1\le i,j\le N, i\ne j}\cfrac{\partial}{\partial\lambda_i}\left(\cfrac{\rho}{\lambda_i-\lambda_j}\right)=\sum_{1\le i,j\le N, i\ne j}\cfrac{\partial}{\partial\lambda_i}\left(\cfrac{\Delta_N u}{\lambda_i-\lambda_j}\right)=\sum_{1\le i\le N}\cfrac{\partial}{\partial\lambda_i}\left(u\cfrac{\partial\Delta_N}{\partial \lambda_i}\right),$$ by using the expression of the first derivative of the Vandermonde determinant for the last equality. 
\newline
Finally using that $\rho$ satisfies the Fokker Planck equation \eqref{eq:edpdensité} and using the harmonic property of $\Delta_N$ on $\R_>^N$, we get: 
\begin{equation}
\label{eq:computationedp}
\Delta_N\partial_t u=\cfrac{1}{2}\Delta(u)\Delta_N+\sum_{i=1}^N\cfrac{\partial u}{\partial \lambda_i}\cfrac{\partial \Delta_N}{\partial \lambda_i}-\sum_{1\le i\le N}\cfrac{\partial}{\partial\lambda_i}\left(u\cfrac{\partial\Delta_N}{\partial \lambda_i}\right).
\end{equation}
However, again by the harmonic property of the Vandermonde determinant,  the last term of the equation is: 
\begin{equation}
\label{eq:computationheat}
u\sum_{1\le i\le N}\cfrac{\partial^2\Delta_N}{\partial \lambda_i^2}+\sum_{1\le i\le N}\cfrac{\partial u}{\partial\lambda_i}\cfrac{\partial\Delta_N}{\partial \lambda_i}=u\Delta(\Delta_N)+\sum_{1\le i\le N}\cfrac{\partial u}{\partial\lambda_i}\cfrac{\partial\Delta_N}{\partial \lambda_i}=\sum_{1\le i\le N}\cfrac{\partial u}{\partial\lambda_i}\cfrac{\partial\Delta_N}{\partial \lambda_i}.
\end{equation}
Using the relation \eqref{eq:computationheat} in \eqref{eq:computationedp} and simplifying  by $\Delta_N$ we have that $u$ satisfies on $\R_>^N$ $$\partial_t u=\cfrac{1}{2}\Delta u.$$
\end{proof}

\begin{Remark}
\label{remark:conditionning} The computation is quite immediate but it is very natural to ask ourselves what is the intuition behind this computation. An interpretation that shall be detailed in Section \ref{subsection:interpretationofthesystem} is that the law of the solution of the Dyson SDE \eqref{dysonbis} can be interpreted as the law of $N$ independent Brownian motions conditioned to not intersect. The density of $N$ independent Brownian motion is solution of the heat equation and $\Delta_N$ should be viewed as the natural quantity that vanishes exactly when two real numbers are equal as mentioned in Remark \ref{rem:vanish}. So, formally, \eqref{eq:changefunction} can be interpreted as a kind of Bayes rule.
\end{Remark}

We now look at the the eigenvalues of random matrices as an exchangeable vector of $\C^N$ (or $\R^N$ if the spectrum is real) as we did in Part \ref{part:ginibregaussian}.  We extend $\rho$ symmetrically across $\R_>^N$ on $\R^N$ and $u$ anti symmetrically across $\R^N$ (since $\Delta_N$ is antisymmetric). This extension of $u$ on $\R^N$ is solution of the heat equation \eqref{eq:edpheat} on $\R^N$ 
\newline
\tab This gives a heuristic to solve the Fokker Planck equation \eqref{eq:edpdensité} of the Dyson particles \eqref{dysonbis}. 
\newline
Indeed, assume that we start from an deterministic matrix $A_0\in \mathcal S_N(\R)$ with spectrum $\lambda_0=(\lambda_0^1,...,\lambda_0^N)$. First we assume that $\lambda_0\in \R_>^N$. So on $\R_>^N$ (before extended $\rho$ symmetrically on $\R^N$), $\rho(0,\lambda)=\delta_{\lambda_0}(\lambda)$ where $\delta_a(x)$ is the dirac measure at the point $a$.
\newline 
For $\sigma \in \mathfrak S_N$, we write $\sigma(\lambda_0)=(\lambda_0^{\sigma(1)},...,\lambda_0^{\sigma(N)})$ and $\varepsilon(\sigma)$ the sign of the permutation $\sigma$.
Hence $u(t,\lambda)$ satisfies the heat equation on $\R^N$ with initial condition: $$u(0,\lambda)=\cfrac{1}{N!\Delta_N(\lambda_0)}\sum_{\sigma\in \mathfrak S_N}\varepsilon(\sigma)\delta_{\sigma(\lambda_0)}(\lambda).$$
Using the fundamental solution of the linear heat equation, we directly obtain that: $$ u(t,\lambda=(\lambda^1,...,\lambda^N))=\cfrac{1}{N!(2\pi t)^{N/2}\Delta_N(\lambda_0)}\sum_{\sigma\in \mathfrak S_N}\varepsilon(\sigma)\exp\left(-\cfrac{|\lambda-\sigma(\lambda_0)|^2}{2t}\right).$$ 
We can write this expression with a more compact formula thanks to the determinant identity: $$u(t,\lambda=(\lambda^1,...,\lambda^N))=\cfrac{1}{N!(2\pi t)^{N/2}\Delta_N(\lambda_0)}\det\left[\left(\exp\left(-\cfrac{|\lambda^i-\lambda_0^j|^2}{2t}\right)\right)_{1\le i,j\le N}\right].$$
Using the formula \eqref{eq:changefunction}, it gives the following result for $\rho$ originally obtained by Johansson in \cite{johansson2001universality}.

\begin{theorem}[Johansson formula]
\label{thm:johansson}
Let $A_0 \in \overset{\circ}{ \mathcal S_N(\R)}$ and write $\lambda_0\in \R_>^N$ its spectrum. For $t>0$, let $\lambda_t\in \R_>^N$ be  the solution of $~\eqref{dysonbis}$ starting from $\lambda_0$. Then the density on $\R^N$ of $\lambda_t$ (viewed as an exchangeable vector) is: 
\begin{equation}
\label{formula:johansson}
\rho(t,\lambda=(\lambda^1,...,\lambda^N))=\cfrac{1}{N!(2\pi t)^{N/2}}\cfrac{\Delta_N(\lambda)}{\Delta_N(\lambda_0)}\det\left[\left(\exp\left(-\cfrac{|\lambda^i-\lambda_0^j|^2}{2t}\right)\right)_{1\le i,j\le N}\right].
\end{equation}
\end{theorem}

\begin{Remark}
More explicitly, with the notations of Theorem \ref{thm:johansson}, for every symmetric, bounded or positive measurable functions $F$ and for every $t>0$ we have 
\begin{equation}
\begin{split}
\E(F(\lambda_t^1,...,\lambda_t^N))=\cfrac{1}{N!(2\pi t)^{N/2}\Delta_N(\lambda_0)}\int_{\R^N}F(x_1,&...,x_N)\prod_{1\le i<j\le N}(x_j-x_i)\times\\
&\det\left[\left(\exp\left(-\cfrac{|x_i-\lambda_0^j|^2}{2t}\right)\right)_{1\le i,j\le N}\right]dx_1...dx_N\\
&\hspace{-5.5cm}=\cfrac{1}{(2\pi t)^{N/2}\Delta_N(\lambda_0)}\int_{\R^N_>}F(x_1,...,x_N)\prod_{1\le i<j\le N}(x_j-x_i)\times\\
&\det\left[\left(\exp\left(-\cfrac{|x_i-\lambda_0^j|^2}{2t}\right)\right)_{1\le i,j\le N}\right]dx_1...dx_N.
\end{split}
\end{equation}
\end{Remark}

Now, we give an argument for the existence of a solution starting from any $\lambda_0\in \R^N$. 
\newline
Let us consider the fundamental case $A_0=0$ as an initial matrix, so with initial spectrum $\lambda_0=0\in\R^N$.  
\newline
We consider a sequence of initial data $(\lambda_{0,n})_{n\in\N}\in(\R_>^N)^{\N}$ that converges to $0\in \R^N$ when $n$ goes to $+\infty$. 
\newline
It remains to notice that for all $\lambda\in \R^N$: 
\begin{equation}
\label{eq:passagelimitedet1}
\begin{split}
\det\left[\left(\exp\left(-\cfrac{|\lambda^i-\lambda_{0,n}^j|^2}{2t}\right)\right)_{1\le i,j\le N}\right]=\exp\left(-\cfrac{|\lambda|^2}{2t}\right)\exp\left(-\cfrac{|\lambda_{0,n}|^2}{2t}\right)\det\left[\left(\exp\left(\cfrac{\lambda^i\lambda_{0,n}^j}{t}\right)\right)_{1\le i,j\le N}\right].
\end{split}
\end{equation}
Using a Taylor expansion of the exponential, we have that for $n$ that goes to $+\infty$: 
\begin{equation}
\label{eq:passagelimitedet2}
\begin{split}
\det\left[\left(\exp\left(\cfrac{\lambda^i\lambda_{0,n}^j}{t}\right)\right)_{1\le i,j\le N}\right]&=\det\left[\left(\sum_{k=0}^{+\infty}\left(\cfrac{(\lambda^i)^k(\lambda_{0,n}^j)^k}{k!t^k}\right)\right)_{1\le i,j\le N}\right]\\
&=\det\left[\left(\sum_{k=0}^{N-1}\left(\cfrac{(\lambda^i)^k(\lambda_{0,n}^j)^k}{k!t^k}\right)\right)_{1\le i,j\le N}\right]+o(\Delta_N(\lambda_{0,n}))\\
&=\det\left[\left(\sum_{k=0}^{N-1}\left(\cfrac{(\lambda^i)^k}{k!t^k}\,\,(\lambda_{0,n}^j)^k\right)\right)_{1\le i,j\le N}\right]+o(\Delta_N(\lambda_{0,n}))\\
&=\det\left[\left(\cfrac{(\lambda^i)^{k-1}}{(k-1)!t^{k-1}}\right)_{1\le i,k\le N}\right]\det\left[\left((\lambda_{0,n}^i)^{k-1}\right)_{1\le k,i\le N}\right]+o(\Delta_N(\lambda_{0,n}))\\
&=\cfrac{1}{1!2!...(N-1)!}\,\Delta_N\left(\cfrac{\lambda}{t}\right)\Delta_N\left(\lambda_{0,n}\right)+o(\Delta_N(\lambda_{0,n})).
\end{split}
\end{equation}

Applying Theorem \ref{thm:johansson} with $\lambda_{0,n}\in \R_>^N$ as an initial data, we shall pass to the limit in the formula \eqref{formula:johansson} when $n$ goes to $+\infty$ and using \eqref{eq:passagelimitedet1} and \eqref{eq:passagelimitedet2} yields that the solution of the Fokker Planck equation \eqref{eq:edpdensité} with initial condition $\lambda_0=0\in\R^N$ is given by: 
\begin{equation}
\label{formula:johanssonbis}
\rho(t,\lambda)=\cfrac{1}{(2\pi t)^{N/2}\,1!2!...N!}\exp\left(-\cfrac{|\lambda|^2}{2t}\right)\Delta_N(\lambda)\Delta_N\left(\cfrac{\lambda}{t}\right).
\end{equation}
To justify the passage to the limit, we can use the Dyson theorem \ref{thm:dyson} to interpret $\rho(t,.)$ as  the law of the eigenvalues of $D_N^2$ and using the Hoffman-Wielandt Lemma \ref{lemma: hoffman wielandt} to pass to the limit.
\newline
Taking $t=1$ in \eqref{formula:johanssonbis}, we obtain that the density of the eigenvalues of a $N$ complex Wigner matrix is given by: 
$$\rho(1,\lambda)=\cfrac{1}{(2\pi)^{N/2}\,1!2!...N!}\exp\left(-\cfrac{|\lambda|^2}{2}\right)|\Delta_N(\lambda)|^2.$$
This gives another proof of the formula obtained in Theorem \ref{thm:law of spectrum} and Theorem \ref{thm law eigenvalues guen}.

\subsection{Interpretation of this system of particle}
\label{subsection:interpretationofthesystem}
We shall give a better physical interpretation of  $~\eqref{dysonbis}$ revealing the important role of the Vandermonde determinant. This section is quite formal. 
\newline
This interpretation is linked with the notion of Brownian motions conditioned to live in a Weyl Chamber. Here the Weyl Chamber is $\R_>^N$. We shall not detail the notion of Weyl chamber in this section, focusing on the particular Dyson case. We refer to \cite{grabiner1999brownian,hobson1996non} for the links between random matrices and Weyl Chambers and \cite{pinsky1985convergence} for general results about diffusion processes conditioned to remain in a certain domain. This part is inspired by the blog article of Mufan Li: An Unusally Clean Proof: Dyson Brownian Motion via Conditioning on Non-intersection.
\newline
First we shall briefly recall the notion of $h$-transform (or Doob transform) of a Markov process.
For a general introduction to $h$-transform, we refer to \cite{doob1984classical}.

\subsubsection{h-transform}
\begin{definition}
Let $(X(t))_{t\ge0}$ be a Markov process with value in $\R^d$. An event $A$ is said to be invariant if for all $x\in\R^d$, for all $s>0$ we have: $$\mP( (X(t))_{t\ge0}\in A|X(0)=x)=\mP(X(t+s))_{t\ge0}\in A|X(s)=x).$$
For a such process $X$ and event $A$, let us define: 
\begin{align*}
\mP_x(.)&=\mP(.|X(0)=x)\\
h(x)&=\mP_x(A)\\
\mP^t(x,dy)&=\mP_x(X(t))\\
\tilde\mP^t(x,dy)&=\mP_x(X(t)\in dy|A).
\end{align*}
\end{definition}

\begin{prop}
\label{prop:htransfo}
With the notations of the previous definition, we have that for all $x\in\R^d$ the measure $\tilde\mP^t(x,.)$ is absolutely continuous with respect to the measure $\mP^t(x,.)$ and the Radon-Nikodyn derivative is given by $$g(y)=h(y)/h(x).$$
\end{prop}

\begin{proof}
We give the intuition of the proof. Thanks to the Bayes formula and invariance of $A$ we have: 
$$\tilde\mP^t(x,dy)=\mP_x(X(t)\in dy|A)=\cfrac{\mP_x(A|X(t)\in dy)\mP_x(X(t)\in dy)}{\mP_x(A)}=\cfrac{h(y)}{h(x)}\mP^t(x,dy).$$

\end{proof}
For a Markov process we can define the generator by $$L[f](x)=\lim_{t\to 0}\cfrac{\E_x(f(X(t))-f(x)}{t},$$ when this quantity exists. 
\newline
We can express the generator $\tilde L$ under the probability $\tilde P$ of $X$. Indeed by Proposition \ref{prop:htransfo}, we get: 
\begin{equation}
\begin{split}
\tilde L[f](x)&=\lim_{t\to 0}\cfrac{\E_x(f(X(t))|A)-f(x)}{t}\\
&=\lim_{t\to 0}\cfrac{1}{t}\left(\int f(y)\tilde\mP^t(x,dy)-f(x)\right)\\
&=\lim_{t\to 0}\cfrac{1}{t}\left(\int f(y)\cfrac{h(y)}{h(x)}\mP^t(x,dy)-f(x)\right)\\
&=\cfrac{1}{h(x)}\lim_{t\to 0}\cfrac{1}{t}\left(\int f(y)h(y)\mP^t(x,dy)-f(x)h(x)\right)\\
&=\frac{1}{h(x)}L[fh](x).
\label{transform}
\end{split}
\end{equation}
An important example is when $X(t)$ is the solution of a SDE. Indeed, consider $X$ as the solution of the SDE: $$dX(t)=g(X(t))dt+dB(t),$$ with $g$ a function defined on an open subset $\Omega$ of $\R^d$ to $\R^d$ and smooth. By the Itô formula, the generator associated of $(X(t))_{t\ge 0}$ is $$L[f](x)=\langle g(x),\nabla f(x)\rangle+\cfrac{1}{2}\Delta f(x),$$ with $\langle. |.\rangle$ the canonical scalar product on $\R^d$. 
\newline
Now, assume that $h$ is harmonic on $\Omega$ which means $\Delta h(x)=0$ on $\Omega$.
Using the formula \eqref{transform} we have: \begin{equation}
\tilde L[f](x)=\langle g(x)+\nabla\log(h(x)),\nabla f(x)\rangle+\cfrac{1}{2}\Delta f(x).
\label{Ltransform}
\end{equation}
We recognise the generator of the solution of the SDE $$d\tilde X(t)=(g(\tilde X(t))+\nabla\log h(\tilde X(t)))dt+dB(t).$$
So, the conditioning by an invariant event transform the law of a process solution of a SDE into the law of another process, solution of another SDE whose difference is the drift term.

\subsubsection{Law of Brownian motions conditioned on the fact they do not intersect}

We want to prove the following result. 
\begin{theorem}
Let $(\lambda_t^i)_{t\ge 0, 1\le i \le N}$ be the solution of \eqref{dysonbis} starting from $\lambda_0\in \R_>^N$. Let $(B(t))_{t\ge0}$ be a Brownian motion in $\R^N$ starting from $\lambda_0$. Then, we have the following equality in law: $$(\lambda_t^i)_{t\ge 0, 1\le i \le N}=(B_i(t))_{t\ge 0, 1\le i\le N}|A, $$ with $A=\{\omega \in\Omega\,|\,\forall t>0,\, B(t)(\omega)\in D\}$.
\end{theorem}
Let us notice that we use a slight abuse of notation in the previous theorem. Indeed the event $A$ is of $0$ probability. But as for the Brownian bridge (for which we conditioned by the event $B_1=0$) the meaning of this conditional probability is that we shall construct a non increasing sequence of probability event $(A_c)_{c>0}$ which converges to $A$ but for all $c>0$, the probability of $A_c$ is not 0. 
\newline
By definition $B(t)\notin D$ is equivalent to $\Delta_N(B(t))=0$ and moreover we notice that $\Delta_N(\lambda)>0$ for $\lambda\in\R_>^N$. The harmonic property of the Vandermonde determinant shall be used to compute the $h$-transform of $(B(t))_{t\ge0}$.
\newline
For all $c>0$, let $$A_c=\{(\Delta_N(B(t)))_{t\ge 0}\,\, \text{hits c before 0} \},$$ and note that $$A=\lim_{c\to\infty}A_c.$$
\newline
Moreover, to compute $\mP_{\lambda_0}(A_c)$ we shall use the same idea as for the gambler ruin problem for the discrete random walk. 
\newline 
\begin{lemma}
Let $\tau=\inf\{t>0,\, \Delta_N(B(t))\,\, \emph{hits c or 0}\}$, then $(\Delta_N(B(t)))_{t \wedge \tau}$ is a martingale.
\end{lemma}

\begin{proof}
This is a consequence of the Itô formula. Since $\Delta(\Delta_N)=0$ by Lemma \ref{lemma:formuladerivativevander}, in the Itô formula we only have the local martingale term which is this case a real martingale using the expression of the first derivative of the Vandermonde determinant.
\end{proof}

The stopping time theorem yields that $$\E_{\lambda}(\Delta_N(B(\tau)))=\Delta_N(\lambda).$$
By definition of $\tau$, we have: $$\E_{\lambda}(\Delta_N(B(\tau)))=c\mP_{\lambda}(A_c).$$
Hence it gives the expression of the $h$-transform with respect to the event $A_c$: 
$$h_{c}(\lambda)=\mP_{\lambda}(A_c)=\cfrac{\Delta_N(\lambda)}{c}.$$
Now thanks to the formula \eqref{Ltransform}, for all $c>0$ under the probability $\mP$ conditioned by the event $A_c$, the law of $(B(t))_{t\ge 0}$ is the law of the solution of the SDE $$ d\lambda_t^c=\nabla\log h_c(\lambda_t^c)dt+dB(t).$$
Since $h_c(x)=\cfrac{\Delta_N(x)}{c}$, Lemma \ref{lemma:formuladerivativevander} yields: $$\cfrac{\partial \log h_c}{\partial x_i}\,(x)=\sum_{j\ne i}\cfrac{1}{x_i-x_j}.$$
Hence the law of the process $(\lambda_t^c)_{t\ge 0}$ is the same for all $c>0$ and so we can let $c$ goes to $+\infty$ so that the law of $(B(t))_{t\ge 0}$ conditioned by the event that $(B(t))_{t\ge 0}$ remains in $D$ is the law of the solution of $\eqref{dysonbis}$. 

\section{Mean field limit of the system of particles}
In this section, we shall deal with the Dyson case \eqref{dysonreal} and the Ornstein-Uhlenbeck case \eqref{eq:sdeornseteinuhlenbeck}. More precisely, we shall focus the case $\beta=1$ in the Dyson case to lighten the notations, but all the results also hold for the case $\beta=2$. For the Dyson case we refer to the Section 4.3 of \cite{Alicelivre} and for the Ornstein-Uhlenbeck case to \cite{Chan,Shi,li2013generalized}. 
\newline
The method we will use is really standard when dealing with system of particles in mean field interaction. The general strategy is to prove a compactness in a certain space (Section \ref{section:compactness}) and then to prove the uniqueness of the limit point (Section \ref{section:uniquenesslimit}). For people interested in this kind of questions, we recommend after the read of this section to see for instance \cite{allez2015random} for a similar approach for the study of the eigenvalues of random matrices in a non confining potential or \cite{tardy2022weak} for the the Keller-Segel in chemotaxis which is a much more difficult model than the Dyson model because of the difficulty to prove the uniqueness of the limit point.
\newline
In all this section, we consider a family $(B^i)_{i\in\N}$ of independent Brownian motions defined on a same probability space $(\Omega,\mathcal F,\mP)$. 
\subsection{Empirical distribution}
Given a solution $(\lambda^i)_{1\le i\le N}$ of the SDE system \eqref{dysonreal}, we define the empirical distribution of this system of particles as: $$\mu_N(t)=\cfrac{1}{N}\sum_{i=1}^N \delta_{\lambda_t^i}.$$
$\mu_N$ is a random variable and we shall specify in which space it is valued.
\newline
Fix $T>0$ to be a fix time. 
For all $\omega\in \Omega$, $\mu_N(\omega)$ is a continuous function from $[0,T]$ to $\mes$ endowing $\mes$ with the natural weak topology which is metrizable by certain metrics. We shall detail more precisely in Section \ref{Subsection:topoaspect} the topology we will consider.
\newline
For a measurable function $f:\mathbb X\to \R$ and a measure $\mu$ on a measurable set $\mathbb X$ we write $\langle \mu,f\rangle :=\int_\mathbb X fd\mu$.
\newline
We first give a key proposition that shall be used in all this section.
\begin{prop}
\label{prop:itoformulasystemparticles}
Let $f\in C^2([0,T]\times \R, \R)$ with all its derivatives bounded, then for all $t\in[0,T]$, we have: 
\begin{equation}
\begin{split}
\langle \mu_N(t),f(t,.)\rangle=&\langle\mu_N(0),f(0,.)\rangle+\int_{0}^t\langle\partial_sf(s,.),\mu_N(s)\rangle ds\\
&+\cfrac{1}{2}\int_0^t\int\int\cfrac{\partial_xf(s,x)-\partial_yf(s,y)}{x-y}d\mu_N(s)(x)d\mu_N(s)(y)ds\\
&+(2-1)\cfrac{1}{2N}\int_0^t\langle\mu_N(s),\partial_x^2f(s,.)\rangle ds+M_f^N(t)
\end{split}
\label{formule}
\end{equation}
with $(M_f^N(t))_{0\le t\le T}$ a martingale given by: 
$$M_f^N(t):=\cfrac{\sqrt{2}}{N^{3/2}}\sum_{i=1}^N\int_0^t\partial_xf(s,\lambda_t^i)dB_s^i.$$

\end{prop}

\begin{proof}
By definition of $\mu_N(t)$ we have: $$\langle \mu_N(t),f(t,.)\rangle =\cfrac{1}{N}\sum_{i=1}^Nf(t,\lambda_t^i).$$
Using that $(\lambda^i)_{1\le i\le N}$ are solution of \eqref{dysonreal} with $\beta=1$, the Itô formula yields:
\begin{equation}
\label{eq:itoformula1}
\begin{split}
\langle \mu_N(t),f(t,.)\rangle=&\langle\mu_N(0),f(0,.)\rangle+\int_{0}^t\langle\partial_sf(s,.),\mu_N(s)\rangle ds\\
&+\int_0^t\int\int_{x\ne y}\cfrac{\partial_xf(s,x)}{x-y}d\mu_N(s)(x)d\mu_N(s)(y)ds\\
&+2\,\cfrac{1}{2N}\int_0^t\langle\mu_N(s),\partial_x^2f(s,.)\rangle ds+M_f^N(t).
\end{split}
\end{equation}

We remark that for every $s\in[0,T]$ $$\int\int_{x\ne y}\cfrac{\partial_xf(s,x)}{x-y}d\mu_N(s)(x)d\mu_N(s)(y)=\cfrac{1}{2}\int\int_{x\ne y}\cfrac{\partial_xf(s,x)-\partial_yf(s,y)}{x-y}d\mu_N(s)(x)d\mu_N(s)(y).$$
The quantity $$\cfrac{\partial_xf(s,x)-\partial_yf(s,y)}{x-y}$$ can be extended to $\partial_x^2f(s,x)$ if $x=y$.
\newline
Hence we can rewrite for $s\in[0,T]$:
\begin{equation}
\label{eq:itoformula2}
\begin{split}
\int&\int_{x\ne y}\cfrac{\partial_xf(s,x)}{x-y}d\mu_N(s)(x)d\mu_N(s)(y)=\cfrac{1}{2}\int\int_{x\ne y}\cfrac{\partial_xf(s,x)-\partial_yf(s,y)}{x-y}d\mu_N(s)(x)d\mu_N(s)(y)\\
&=\cfrac{1}{2}\int\int\cfrac{\partial_xf(s,x)-\partial_yf(s,y)}{x-y}d\mu_N(s)(x)d\mu_N(s)(y)-\cfrac{1}{2}\int\int_{x=y}\partial_x^2f(s,x)d\mu_N(s)(x)d\mu_N(s)(y).
\end{split}
\end{equation}
We then compute the last term:
\begin{equation}
\begin{split}
\label{eq:itoformula3}
\int\int_{x=y}\partial_x^2f(s,x)d\mu_N(s)(x)d\mu_N(s)(y)&=\cfrac{1}{N^2}\sum_{i=1}^N\sum_{j=1}^N\mathds{1}_{\lambda_s^i=\lambda_s^j}\partial_x^2f(s,\lambda_s^i)\\
&=\cfrac{1}{N}\langle\mu_N(s),\partial_x^2f(s,.)\rangle
\end{split}
\end{equation}
So \eqref{eq:itoformula2} and \eqref{eq:itoformula3} gives: 
\begin{equation}
\label{eq:itoformula4}
\begin{split}
\int_0^t\int&\int_{x\ne y}\cfrac{\partial_xf(s,x)}{x-y}d\mu_N(s)(x)d\mu_N(s)(y)ds=\\
&\cfrac{1}{2}\int_0^t\int\int\cfrac{\partial_xf(s,x)-\partial_yf(s,y)}{x-y}d\mu_N(s)(x)d\mu_N(s)(y)ds-\int_0^t\cfrac{1}{2N}\langle\mu_N(s),\partial_x^2f(s,.)\rangle ds.
\end{split}
\end{equation}
Using \eqref{eq:itoformula4} in \eqref{eq:itoformula1} gives \eqref{formule}.
\newline 
Finally, we shall prove that $(M_f^N(t))_{0\le t\le T}$ is not just a local martingale but a martingale. 
\newline
Indeed we can compute its bracket: 
\begin{equation}
\label{eq:martingalebound}
\langle M_f^N(t),M_f^N(t)\rangle =\cfrac{2}{N^3}\int_0^t\sum_{i=1}^N(\partial_x f(s,\lambda_s^i))^2ds\le\cfrac{2t\sup_{s\in[0,T]}||\partial_xf(s,.)||_{\infty}^2}{N^2}.
\end{equation}
This implies that $\E[\langle M_f^N(T),M_f^N(T)\rangle]<+\infty$. So $(M_f^N(t))_{0\le t\le T}$ is a martingale.
\end{proof}

\subsection{Compact subsets of $C([0,T],\mes)$}
\subsubsection{Topological warm-up}
Consider $(Y,d_Y)$ a compact metric space. Let $X$ and $d$ and $d'$ two distances on $X$ such that $(X,d)$ and $(X,d')$ define the same topology. 
\newline
We recall that $C((Y,d_Y),(X,d))$ are the continuous functions from $Y$ to $X$ where $Y$ is endowed with the distance $d_Y$ and $X$ with the distance $d$. Let us first mention that $C((Y,d_Y),(X,d))$ and $C((Y,d_Y),(X,d'))$ are the same space. Indeed, since $(X,d)$ and $(X,d')$ define the same topology they have the same open subsets. As a function is continuous if and only if the preimage of an open subset of the arrival space is an open subset of the space of definition, $C((Y,d_Y),(X,d))$ and $C((Y,d_Y),(X,d'))$ are equal. We call this space $Z$. 
\newline
$Z$ can be endowed with two distances: $$\mathcal D(f,g)=\sup_{y\in Y}d(f(y),g(y))\text{ or }\mathcal D'(f,g)=\sup_{y\in Y}d'(f(y),g(y)).$$ We shall prove that $(Z,\mathcal D)$ and $(Z,\mathcal D')$ define the same topology. 
\begin{Remark}
Let us mention that the hypothesis on the compactness of $Y$ is necessary. Indeed, if we take $Y=\R^+$ with the natural distance, $X=\R^+$ again with the natural distance $d$ and with $d'(x,y)=|x^2-y^2|$ we can see that the sequence of functions $f_n(x):=x+1/n$ converges for the metric $\mathcal D$ to $f(x)=x$ but does not converge to $f$ in $\mathcal D'$.
\end{Remark}
Let us state a lemma that characterizes the convergence in $(Z,\mathcal D)$ when $Y$ is compact. 
\begin{lemma}
\label{lemma:convunifor}
Let $(f_n)_{n\in\N}$ and $f$ be functions in $Z$. Then we have the following equivalence:
$f_n\overunderset{\mathcal D}{n\to+\infty}{\longrightarrow} f$ if and only if for all sequences $(y_n)_{n\in\N}\in Y^\N$, $y\in Y$, $\phi :\N\to \N$ strictly increasing such that $y_{\phi(n)}\overunderset{d_Y}{n\to+\infty}{\longrightarrow}y$ and $\gamma:\N\to\N$ strictly increasing we have $d(f_{\gamma(n)}(y_{\phi(n)}),f(y))\underset{n\to+\infty}{\longrightarrow} 0$. 
\end{lemma}
\begin{proof}
Assume that $f_n\overunderset{\mathcal D}{n\to+\infty}{\longrightarrow} f$ and consider $(y_n)_{n\in\N}\in Y^\N$, $y\in Y$, $\phi :\N\to \N$ strictly increasing such that $y_{\phi(n)}\overunderset{d_Y}{n\to+\infty}{\longrightarrow}y$ and $\gamma:\N\to\N$ strictly increasing. 
Then it yields 
\begin{equation}
\begin{split}
d(f_{\gamma(n)}(y_{\phi(n)}),f(y))&\le d(f_{\gamma(n)}(y_{\phi(n)}),f(y_{\phi(n}))+d(f(y_{\phi(n)}),f(y))\\
&\le \mathcal D(f_{\gamma(n)},f)+ d(f(y_{\phi(n)}),f(y)).
\end{split}
\end{equation}
The right hand side goes to 0 by hypothesis and continuity of $f$.
\newline
Conversely, assume by contradiction that $f_n$ does not converges to $f$ in $\mathcal D$. There exists $\varepsilon>0$ and $\psi:\N\to\N$ strictly increasing such that for all $n\in\N$ $\mathcal D(f_{\psi(n)},f)>\varepsilon$. By definition, for all $n>0$ we can find $y_n\in Y$ such that \begin{equation}
\label{eq:ineqtopolo}
d(f_{\psi(n)}(y_n),f(y_n))>\varepsilon.
\end{equation}
By compactness of $(Y,d_Y)$, one can find $\phi:N\to \N$ strictly increasing and $y\in Y$ such that $y_{\phi(n)}\overunderset{d_Y}{n\to+\infty}{\longrightarrow}y$. 
Evaluating \eqref{eq:ineqtopolo} in $\phi(n)$ we obtain that for all $n\ge 0$:
\begin{equation}
\label{eq:ineqtopolo2}
d(f_{\psi\circ\phi(n)}(y_{\phi(n)}),f(y_{\phi(n)}))>\varepsilon.
\end{equation}
By continuity of $f$, for $n$ large enough we have $$d(f(y_{\phi(n)}),f(y))\le \frac{\varepsilon}{2}.$$ Combining it with \eqref{eq:ineqtopolo2} we have that for $n$ large enough $$d(f_{\psi\circ\phi(n})(y_{\phi(n)}),f(y_{\phi(n)}))>\frac{\varepsilon}{2}.$$ This contradicts the hypothesis with $\gamma(n)=\psi\circ\phi(n)$.
\end{proof}
As a corollary we obtain that $(Z,\mathcal D)$ and $(Z,\mathcal D')$ are the same topological spaces. 
\begin{corollary}
\label{coro:topoeq}
Let $(f_n)_{n\in\N}$ and $f$ in $Z$. We have  $f_n\overunderset{\mathcal D}{n\to+\infty}{\longrightarrow} f$ if and only if $f_n\overunderset{\mathcal D'}{n\to+\infty}{\longrightarrow} f$.
\end{corollary}
\begin{proof}
By Lemma \ref{lemma:convunifor} when we endow $Z$ with $\mathcal D$ we get $f_n\overunderset{\mathcal D}{n\to+\infty}{\longrightarrow} f$ if and only if for all sequences $(y_n)_{n\in\N}\in Y^\N$, $y\in Y$, $\phi :\N\to \N$ strictly increasing such that $y_{\phi(n)}\overunderset{d_Y}{n\to+\infty}{\longrightarrow}y$ and $\gamma:\N\to\N$ strictly increasing we have $d(f_{\gamma(n)}(y_{\phi(n)}),f(y))\underset{n\to+\infty}{\longrightarrow} 0$. Since $(X,d)$ and $(X,d')$ define the same topology we obtain that $f_n\overunderset{\mathcal D}{n\to+\infty}{\longrightarrow} f$ if and only if for all sequences $(y_n)_{n\in\N}\in Y^\N$, $y\in Y$, $\phi :\N\to \N$ strictly increasing such that $y_{\phi(n)}\overunderset{d_Y}{n\to+\infty}{\longrightarrow}y$ and $\gamma:\N\to\N$ strictly increasing we have $d'(f_{\gamma(n)}(y_{\phi(n)}),f(y))\underset{n\to+\infty}{\longrightarrow} 0$. Again by Lemma \ref{lemma:convunifor} when $Z$ is endowed with $\mathcal D'$ it yields $f_n\overunderset{\mathcal D}{n\to+\infty}{\longrightarrow} f$ if and only if $f_n\overunderset{\mathcal D'}{n\to+\infty}{\longrightarrow} f$.
\end{proof}
\subsubsection{Characterization of compact subsets of $C([0,T],\mes)$}
\label{Subsection:topoaspect}
This part can be skip for people non familiar with topology. The only result that shall be used after is Lemma \ref{lemma:topology} that gives some typical compact subsets of $C([0,T],\mes)$. 
\newline
As mentioned before, we consider $\mes$ as a topological space with the weak topology. Since now we fix a family $\mathfrak F:=(f_i)_{i\ge 1}$ of $C^2(\R,\R)$ functions such that for all $i\ge 1$, $||f_i||_{\infty}+||f_i'||_{\infty}+||f_i''||_{\infty }\le 1$ and $\text{span}(\mathcal F)$ is dense in the set of continuous compactly supported functions. Then the distance on $\mes$ defined by: $$d_\mathfrak F(\mu,\nu)=\sum_{n\ge 0} 2^{-n} \left|\int f_{n+1} d\mu-\int f_{n+1} d\nu\right|$$ metrizes the weak topology (convergence in law/in distribution) on $\mes$.
\newline
Now we endow $C([0,T],\mes)$ with the uniform convergence : $$d(f,g)=\sup_{0\le t\le T}d_{\mathfrak F}(f(t),g(t)).$$ As explained in Corollary \ref{coro:topoeq} we could have endowed $\mes$ with other distances that metrizes exactly the weak topology. We choose this distance to exhibit explicit compact subsets of $C([0,T],\mes)$.

\begin{lemma}
\label{lemma:topology}
Let $K$ be a compact subset of $\mes$, and $(C_i)_{i\ge 1}$ be a family of compact subsets of $C([0,T],\R)$. Then the set $$\mathcal K=C([0,T],\mes)\bigcap\{\forall t\in[0,T],\, \mu(t)\in K\}\bigcap_{i\ge 1}\{t\mapsto\langle \mu(t),f_i\rangle \in C_i\}$$
is a compact subset of $C([0,T],\mes)$.
\end{lemma}

\begin{proof}
We have to prove that $\mathcal K$ is relatively compact and closed. 
\newline
$\mathcal K$ is closed as an intersection of closed subsets. 
It remains to prove that $\mathcal K$ is relatively compact. Let $(\mu_n)_{n\in \N}$ be a sequence of elements of $\mathcal K$. By definition, for all $i\ge 1$ the sequence of functions $(t\to\langle\mu_n(t),f_i\rangle )_{n\in\N}$ is a sequence of functions in $C_i$. Hence by a diagonalization extraction argument, we can find a subsequence $\phi$ such that for all $i\ge 1$, $(t\to\langle\mu_{\phi(n)}(t),f_i\rangle)_{n\in\N}$ converges towards a function $g_i\in C_i$.
\newline
Now, let $(t_k)_{k\in \N}$ a dense countable sequence in $[0,T]$. By a same diagonalization extraction argument we can find sub sub sequence that shall be written $\phi$ such that for all $k\ge 0$ we have $\mu_{\phi(n)}(t_k)$ converges to $\mu(t_k)\in\mes$ and still for all $i\ge 1$, $(t\to\langle\mu_{\phi(n)}(t),f_i\rangle)_{n\in\N}$ converges towards $g_i\in C_i$. We necessarily have that 
\begin{equation}
\label{eq:topo}
\langle \mu(t_k),f_i\rangle=g_i(t_k)
\end{equation} for all $i\ge 1$ and $k\ge 0$. We then extend $\mu$ on $[0,T]$ by a continuity argument. 
Indeed, by construction, for all $k\ge 0$, $\mu(t_k)\in K$. Let $0\le t\le T$, and consider consider a subsequence, still denoted $t_k$ that converges to $t$. The sequence of measures $(\mu(t_k))_{k\ge 0}$ is relatively compact since for all $k\ge 0$, $\mu(t_k)\in K$. Moreover, let $\nu\in\mes$ be a accumulation point, passing to the limit in \eqref{eq:topo} yields $$\langle \nu, f_i\rangle =g_i(t)$$ for all $i\ge 1$ by continuity of $g$. Since the $(f_i)_{i\ge 1}$ are dense, these relations uniquely characterize $\nu$. So $(\mu(t_k))_{k\ge 0}$ converges to an element $\mu(t)\in\mes$ such that for all $i\ge 1$, $\langle \mu(t), f_i\rangle =g_i(t)$. This extend $\mu$ continuously as an element of $C([0,T],\mes)$.
Since, $(f_i)_{i\ge 1}$ is a dense sequence and $(t\to\langle\mu_{\phi(n)}(t),f_i\rangle)_{n\in\N}$ converges towards a function $\langle \mu(.),f_i\rangle \in C_i$ (uniformly on $[0,T]$) for all $i\ge 1$, $(\mu_{\phi(n)})_{n>0}$ converges to $\mu$ in $C([0,T],\mes)$.
\end{proof}

\subsection{Compactness of $(\mu_N)_{N\ge 1}$ }
\label{section:compactness}
The goal of this section is to prove that almost surely the sequence $(\mu_N)_{N>0}$ is relatively compact in $C([0,T],\mes)$. 
Now we suppose that for all $N\ge 1$, $(\lambda_0^{i,N})_{1\le i\le N}\in D$ is the initial data for the system of $N$ particles \eqref{dysonreal}.
We  assume that: $$C_0:=\underset{N\ge 0}{\sup}\cfrac{1}{N}\sum_{i=1}^N\log((\lambda_0^{i,N})^2+1)<+\infty.$$

\begin{prop}
\label{prop:almostsurelyprecompact}
The sequence $(\mu_N)_{N\ge 1}$ is almost surely relatively compact in $C([0,T],\mes)$.
\end{prop}

\begin{proof}
We follow the arguments of \cite{Alicelivre}.
We shall apply Lemma \ref{lemma:topology} to construct a compact subset of $C([0,T],\mes)$ such that for $N$ large enough $\mu_N$ is in this compact almost surely.
So, there is two parts in the proof. The first one will be to construct a compact subset of $\mes$ in which almost surely for all $t\le T$, $\mu_N(t)$ is and then compacts subsets of $C([0,T],\R)$ as in Lemma \ref{lemma:topology}. 
\newline
Let us consider the function $f(x):=\log(1+x^2)$ which is $C^2$ on $\R$  and  $||f'||_{\infty}\le 1$, $||f''||_{\infty}\le 2$. We also have that: $$\forall x,y \in \R, \, \left|\cfrac{f'(x)-f'(y)}{x-y}\right|\le ||f''||_{\infty}\le 2.$$
Hence using the Itô formula $\eqref{formule}$, we get the following bound: 
\begin{equation}
\label{eq:compactness1}
\forall N>0,\, \underset{t\le T}{\sup}|\langle\mu_N(t),f\rangle|\le |\langle\mu_N(0),f\rangle|+T\left(1+\cfrac{1}{N}\right)+\underset{t\le T}{\sup}|M_f^N(t)|.
\end{equation}
Moreover, by Proposition \ref{prop:itoformulasystemparticles}, $(M_f^N(t))_{0\le t\le T}$ is a martingale with its bracket bounded by: 
\begin{equation}
\label{eq:compactness2}
\sup_{0\le t\le T}\langle M_f^N(t),M_f^N(t)\rangle \le\cfrac{2T}{N^2},
\end{equation}
as obtained in \eqref{eq:martingalebound}.
Let $$K_M:=\{\mu\in\mes,\,\dis\int\log(1+x^2)d\mu(x)\le M\}$$ for a certain $M>0$ that will be specified later. 
\newline
$K_M$ is a compact subset of $\mes$. Indeed, $K_M$ is closed by the Portemanteau theorem and is relatively compact by the Prokhorov theorem since the Markov inequality yields that for every $\mu\in K_M$, for every $K>0$ $$\mu(\R-[-K,K])\le\cfrac{\dis\int\log(1+x^2)d\mu(x)}{\log(1+K^2)}\le \cfrac{M}{\log(1+K^2)}.$$
Now we shall prove that almost surely, for $N$ large enough, for all $0\le t\le T$, $\mu_N(t)\in K_M$ for a good $M$. 
We use the Borel-Cantelli Lemma. 
\newline
Indeed, we first estimate the martingale term  using \eqref{eq:compactness2}, the Markov inequality and the Burkholder–Davis–Gundy inequality: 
\begin{equation}
\label{eq:compactness3}
\mP(\underset{t\le T}{\sup}|M_f^N(t)|\ge\varepsilon)\le\cfrac{C\E(\langle M_f^N(T),M_f^N(T)\rangle)}{\varepsilon^2}\le\cfrac{2TC}{N^2\varepsilon^2},
\end{equation}
with $C>0$ a constant independent of $\varepsilon$ and $N$ given by the Burkholder–Davis–Gundy inequality.
Hence using \eqref{eq:compactness1}, we obtain that for all $N>0$ $$\underset{t\le T}{\sup}|\langle\mu_N(t),f\rangle|\le |\langle \mu_N(0),f\rangle |+T\left(1+\cfrac{1}{N}\right)+\underset{t\le T}{\sup}|M_f^N(t)|\le C_0+T\left(1+\cfrac{1}{N}\right)+\underset{t\le T}{\sup}|M_f^N(t)|.$$ So for $M>2T+C_0$, we get 
\begin{equation}
\label{eq:compactnes4}
\mP\left(\underset{t\le T}{\sup}|\langle\mu_N(t),f\rangle|\ge M\right)\le \mP\left(\underset{t\le T}{\sup}|M_f^N(t)|\ge M-2T-C_0\right) \le \cfrac{\widetilde C(T)}{(M-2T-C_0)^2N^2},
\end{equation} with $\widetilde C(T)$ a constant independent of $M$ and $N$.
\newline
Hence the Borel Cantelli lemma yields:
\begin{equation}
\label{eq:compactness5}\mP\left(\bigcup_K\bigcap_{N\ge K}\{\forall t\le T, \mu_N(t)\in K_M\}\right)=1. 
\end{equation}
This proves the first part.
\newline
Now we shall find good functions $(f_i)_{i\ge 0}$ and compact subsets $(C_i)_{i\ge 0}$ of $C([0,T])$ such that almost surely the sequence $(\mu_N)_{N>0}$ is in $$\bigcap_{i>0}\{t\mapsto\langle\mu(t),f_i\rangle\in C_i\}$$ for $N$ large enough.  
\newline
Let us recall that the sets of the form: \begin{equation}
\label{eq:compactness6}
C=\bigcap_{n}\{g\in C([0,T],\R), \, \underset{|t-s|\le\eta_n}{\sup}|g(t)-g(s)|\le\varepsilon_n,\,\underset{t\le T}{\sup} \,|g(t)|\le M\}
\end{equation} with $M>0$ a constant, $(\varepsilon_n)_{n\ge 0}$ and $(\eta_n)_{n\ge 0}$ two sequence that converge to 0 are compact subsets of $C([0,T],\R)$ by the Arzela-Ascoli theorem.
So, we shall look at the oscillation of $t\to\langle\mu_N(t),g\rangle$ with $g\in C^2(\R)$ and its derivatives bounded by $1$. 
\newline
Let $M>2$ and $\delta\in(0,1)$, thanks to \eqref{formule} we have that for all $0\le s,t\le T$, 
$$|\langle\mu_N(t),g\rangle-\langle\mu_N(s),g\rangle|\le ||g''||_{\infty}|s-t|+|M_g^N(t)-M_g^N(s)|.$$
Let $|s-t|\le \delta$ then we have: 
\begin{equation}
\label{eq:compactness7}
|\langle\mu_N(t),g\rangle-\langle\mu_N(s),g\rangle|\le \delta+|M_g^N(t)-M_g^N(s)|.
\end{equation}
Now, shall prove the following estimate: 
\begin{equation}
\label{eq:compactness8}
\mP\left(\underset{|t-s|\le\delta}{\sup} |\langle\mu_N(t),g\rangle-\langle\mu_N(s),g\rangle|\ge M\delta^{1/8}\right)\le \cfrac{a\delta^{1/2}}{N^4 M^4},
\end{equation} with $a>0$ a constant independent of $\delta$, $M$, $N$ and $g$.
\newline
By \eqref{eq:compactness7} we have:
\begin{equation}
\label{eq:compactness9}
\mP\left(\underset{|t-s|\le\delta}{\sup} |\langle\mu_N(t),g\rangle-\langle\mu_N(s),g\rangle|\ge M\delta^{1/8}\right)\le\mP\left(\underset{|t-s|\le\delta}{\sup}|M_g^N(t)-M_g^N(s)| \ge (M-1)\delta^{1/8}\right).
\end{equation}
It remains to estimate right hand term of \eqref{eq:compactness9}. By the same computation as for \eqref{eq:martingalebound} we have that for all $t\ge k\delta$, 
\begin{equation}
\label{eq:compactness10}
|\langle M_g^N(t)-M_g^N(k\delta),M_g^N(t)-M_g^N(k\delta)\rangle|^2\le 4 \cfrac{(t-k\delta)^2||g'||_\infty^2}{N^4}.
\end{equation}
The idea is to truncate $[0,T]$ in interval of length $\delta$. Let $J=\lfloor T/\delta\rfloor$ and now we bound the right hand term using as before the Markov inequality and the Burkholder–Davis–Gundy inequality: 
\begin{equation}
\label{eq:compactness11}
\begin{split}
&\hspace{-3cm}\mP\left(\underset{|t-s|\le\delta}{\sup}|M_g^N(t)-M_g^N(s)| \ge (M-1)\delta^{1/8}\right)\\
\le &\sum_{k=1}^{J+1}\mP\left(\underset{k\delta\le t\le (k+1)\delta}{\sup}|M_g^N(t)-M_g^N(k\delta)| \ge\cfrac{ (M-1)\delta^{1/8}}{3}\right)\\
\le &\sum_{k=1}^{J+1}\cfrac{3^4}{\delta^{1/2}(M-1)^4}\,\E\left(\underset{k\delta\le t\le (k+1)\delta}{\sup}|M_g^N(t)-M_g^N(k\delta)|^4\right)\\
\le &\sum_{k=1}^{J+1}\cfrac{3^4}{\delta^{1/2}(M-1)^4}\,\E\left(\left|\langle M_g^N((k+1)\delta)-M_g^N(k\delta)\rangle\right|^2\right)\\
\underset{\eqref{eq:compactness10}}{\le}& \cfrac{C\delta^2}{N^4\delta^{1/2}(M-1)^4}(J+1)||g'||_{\infty}^2\le \cfrac{a\delta^{1/2}}{N^4(M-1)^4}||g'||_\infty^2,
\end{split}
\end{equation}
with $a$ a constant independent of $N$, $M$, $\delta$ and $g$. This gives \eqref{eq:compactness8}
\newline
For $g$ a $C^2$ function with derivatives bounded by $1$ and $\varepsilon>0$ we define the subset of $C([0,T],\mes)$ $$C(g,\varepsilon):=\bigcap_{n\ge 1}\{\mu\in C([0,T],\mes), \underset{|t-s|\le n^{-4}}{\sup}|\langle\mu(t),g\rangle-\langle\mu(s),g\rangle|\le \cfrac{1}{\varepsilon\sqrt{n}}\}.$$ As a consequence of \eqref{eq:compactness8} we obtain 
\begin{equation}
\begin{split}
\label{eq:compactness12}
\mP\left(\mu_N\notin  C(g,\varepsilon)\right)&\le\sum_{n\ge 1}\mP\left(\mu_N \notin \{\mu\in C([0,T],\mes), \underset{|t-s|\le n^{-4}}{\sup}|\langle\mu(t),g\rangle-\langle\mu(s),g\rangle|\cfrac{1}{\varepsilon\sqrt{n}}\}\right)\\
&\underset{\eqref{eq:compactness8}}{\le} \sum_{n\ge 1}\cfrac{a\varepsilon^4}{N^4\,n^2}=:\cfrac{\tilde{a}\varepsilon^4}{N^4},
\end{split}
\end{equation} with $\tilde a$ a constant independent of $g$, $\varepsilon$ and $N$. 
We recall that the family of functions $\mathfrak F=(f_i)_{i\ge 1}$ is the family introduced at the beginning of Section \ref{Subsection:topoaspect} to define a metric on $\mes$. Set $\varepsilon_k:=1/k$. Using \eqref{eq:compactness12} we get: 
\begin{equation}
\label{eq:compactness13}
\mP\left(\mu_N\notin  \bigcap_k C(f_k,\varepsilon_k)\right)\le\cfrac{b}{N^4},
\end{equation} with $b>0$ a constant independent on $N$. This quantity is summable and so the Borel Cantelli lemma yields: 
\begin{equation}
\label{eq:compactness14}\mP\left(\bigcup_{K\ge 1}\bigcap_{N\ge K}\{\mu_N\in \bigcap_k C(f_k,\varepsilon_k)\}\right)=1.
\end{equation}
Set $$\mathcal K=K_M\bigcap_k C(f_k,\varepsilon_k)\subset C([0,T],\mes),$$ which is a compact subset of $C([0,T],\mes)$ by Lemma \ref{lemma:topology}. By \eqref{eq:compactness7} and  \eqref{eq:compactness14} we get that $$\mP\left(\bigcup_{K\ge 1}\bigcap_{N\ge K}\{\mu_N\in \mathcal K\}\right)=1.$$
Hence, almost surely the sequence $(\mu_N)_N$ is relatively compact in $C([0,T],\mes)$.
\end{proof}

\subsection{Convergence of $\mu_N$}
In this section we shall prove that almost surely the sequence $(\mu_N)_{N\ge 1}$ converges to an element of $C([0,T],\mes)$. We assume that $\mu_N(0)\in \mes$ converges weakly to a measure $\mu_0\in\mes$.
\newline
Since the sequence is relatively compact we shall just prove that there exists a unique limit for this sequence. 
Fix $\omega\in \Omega$.
\subsubsection{Characterization of the limit}

First, we can give a description of a limit point of $(\mu_N(\omega))_{N\ge 1}$.  
\begin{prop}
\label{prop:characlimit}
Let $\mu\in C([0,T],\mes)$ be a limit point of the sequence $(\mu_N(\omega))_{N\ge 1}$ in the space $C([0,T],\mes)$, then $\mu(0)=\mu_0$ and for every $f\in C^2([0,T]\times\R,\R)$ with its derivatives bounded, for every $t\ge 0$, 
\begin{equation}
\begin{split}
\label{eq:characlimit}
\int f(t,x)d\mu(t)(x)&=\int f(0,x)d\mu_0(x)+\int_0^t\int\partial_s f(s,x)d\mu(s)(x)ds\\
&+\cfrac{1}{2}\int_0^t\int\int\cfrac{\partial_xf(s,x)-\partial_yf(s,y)}{x-y}d\mu(s)(x)d\mu(s)(y)ds.
\end{split}
\end{equation}
\end{prop}

\begin{proof}
We pass to the limit in the Itô formula \eqref{formule}. The martingale term $\left(\underset{0\le t\le T}{\sup}|M_f^N(t)|\right)_{N\ge 1}$ goes to 0 almost surely using the Markov inequality and the Burkholder-Davis-Gundy inequality as we already did for the compactness of the sequence in Proposition \ref{prop:almostsurelyprecompact}.
\end{proof}

To identify the limit point we shall use the Stieltjes transform. We recall that we introduced this transformation and some standard properties in Appendix \ref{annexesectionstieljes}.

\begin{definition}
Let $\mu$ be a probability measure on $\R$. We define its Stieltjes transform as: $$S_\mu(z)=\int_{\R}\cfrac{1}{z-x}d\mu(x), \, z\in \C-\R.$$

\end{definition}

\begin{example}
\label{example:semicircle}
We computed in Proposition \ref{annexe prop unique stieljes sqrt} the Stieljes transform the semicircle distribution defined by: $$\sigma(dx):=\cfrac{1}{2\pi}\sqrt{(4-x^2)_+}dx.$$ It is equal to: $$S_\sigma(z)= \cfrac{z-\sqrt{z^2-4}}{2},$$ with $\sqrt{\,}$ a suitable determination of the square root. 
\end{example}

\begin{prop}
\label{prop:charastieljes}
Let $\mu$ be a limit point of $(\mu_N(\omega))_{N\ge 1}$ as in Proposition \ref{prop:characlimit}, then $\mu(0)=\mu_0$. Set $S_t:=S_{\mu(t)}$ for the Stieltjes transform of the measure $\mu(t)$ for every $t\ge 0$ then we have
\begin{equation}
S_t(z)=S_0(z)-\int_0^tS_s(z)\partial_zS_s(z)ds, \,\forall z\in \C-\R
\label{Steljes eq}
\end{equation}
\end{prop}

\begin{proof}
This is an immediate consequence of Proposition \ref{prop:characlimit}. Evaluating \eqref{eq:characlimit} with the test functions $f_z(t,x)=1/(z-x)$ for $z\in\C-\R$, we obtain:

\begin{align*}
S_t(z)&=S_0(z)
+\cfrac{1}{2}\int_0^t\int\int\cfrac{\cfrac{1}{(z-x)^2}-\cfrac{1}{(z-y)^2}}{x-y}d\mu(s)(x)d\mu(s)(y)ds\\
&=S_0(z)+\cfrac{1}{2}\int_0^t\int\int\cfrac{(z-y)^2-(z-x)^2}{(z-x)^2(z-y)^2(x-y)}d\mu(s)(x)d\mu(s)(y)ds\\
&=S_0(z)+\cfrac{1}{2}\int_0^t\int\int\cfrac{2z-x-y}{(z-x)^2(z-y)^2}d\mu(s)(x)d\mu(s)(y)ds\\
&=S_0(z)+\cfrac{1}{2}\int_0^t\int\int\cfrac{z-x+(z-y)}{(z-x)^2(z-y)^2}d\mu(s)(x)d\mu(s)(y)ds\\
&=S_0(z)+\int_0^t\int\int\cfrac{z-x}{(z-x)^2(z-y)^2}d\mu(s)(x)d\mu(s)(y)ds\\
&=S_0(z)+\int_0^t\int\int\cfrac{1}{(z-x)(z-y)^2}d\mu(s)(x)d\mu(s)(y)ds\\
&=S_0(z)-\int_0^tS_s(z)\partial_zS_s(z)ds
\end{align*} 
\end{proof}

\subsubsection{Uniqueness of the limit and convergence}
\label{section:uniquenesslimit}
We fix $\mu\in C([0,T],\mes)$ a limit point of $(\mu_N(\omega))_{N\ge 1}$ and we recall that we defined $S_t:=S_{\mu(t)}$ as the Stieljes transform of the measure $\mu(t)$ for every $t\ge 0$ which is solution of $\eqref{Steljes eq}$ and satisfies $\mu(0)=\mu_0$.
\newline
Equation \eqref{Steljes eq} can be written $$\partial_tS(t,z)+S(t,z)\partial_z S(t,z)=0.$$ This is a Burger's equation on the complex plane. Uniqueness of a solution of a Burger equation is not immediate and sometimes it is not the case. In this case we should prove the uniqueness of a solution of this Burger equation using that we look for solutions that are the Stieljes transform of a certain measure at any time. We shall use  the method of characteristic. 
Set $\mathbb{H}=\{z\in\C,\, \Im(z)>0\}$ and recall by Remark \eqref{annexe remark stiejes} that the Stieljes transform on $\mathbb H$ uniquely determined a real probability measure (see Proposition \eqref{annexe prop property stieljes}).
More exactly we shall show that there exists a unique solution of the system:
\begin{equation}
\left\{
\begin{split}
&S(0,z)=S_{\mu_0}(z), \forall z\in \mathbb H \\
&\partial_tS(t,z)+S(t,z)\partial_zS(t,z)=0, \forall t>0, \forall z\in\mathbb H\\
&\text{ for all }t\ge 0,\, S(t,.)\text{ is the Stieljes transform of a real probability measure}
\end{split}
\right.
\label{system Stieljes}
\end{equation}
Assume that the uniqueness of solution of $\eqref{system Stieljes}$ has been proved. Since the Stieljes transform on $\mathbb H$ of a real probability measure uniquely characterizes this measure, we deduce from Proposition \eqref{prop:charastieljes} $(\mu_N(\omega))_{N\ge 1}$ has a unique limit point.
\newline
We shall prove some general lemmas in order to apply the characteristic method as the existence of the characteristics as solutions of a ordinary differential equations and secondly the fact that we can go back up the characteristic that have been built.

\begin{lemma}
\label{lemma:existencechara}
For every $z\in \mathbb H$, for every $t>0$ there exists a unique $r\in \mathbb H$ (which depends on $z$ and $t$) such that: $$z=r+tS_{\mu_0}(r).$$
\end{lemma}

\begin{proof}
We follow the arguments of \cite{Shi}.
First of all, let us notice a general property of the Stieljes transform that shall be used throughout the proof: 
\begin{equation}
\label{eq:stieljesastuce}
\Im(S_\nu(z))=-\Im(z)\int_\R\cfrac{1}{|z-x|^2}d\nu(x)\,\text{ for }\nu\in \mes.
\end{equation}
Let us prove the uniqueness. Assume that we have $r_1$ and $r_2$ both in $\mathbb H$ that are different and such that $z=r_1+tS_{\mu_0}(r_1)$ and $z=r_2+tS_{\mu_0}(r_2)$.
If we subtract the two identities we have: $$0=r_1-r_2+tS_{\mu_0}(r_1)-tS_{\mu_0}(r_2).$$
So it gives $$r_1-r_2=t\dis\int_{\R}\left(\cfrac{1}{r_2-x}-\cfrac{1}{r_1-x}\right)d\mu_0(x)=t\int_{\R}\cfrac{r_1-r_2}{(r_2-x)(r_1-x)}d\mu_0(x).$$
Hence, we have the following relation: 
\begin{equation}
\label{eq:uniquenesscharac1}
\cfrac{1}{t}=\int_{\R}\cfrac{1}{(r_2-x)(r_1-x)}d\mu_0(x).
\end{equation}
Now we write $r_j=a_j+ib_j$ for $j=1,2$, $a_j\in\R$ and $b_j\in \R^{+*}$. Taking the imaginary and real part of \eqref{eq:uniquenesscharac1} yields: 
\begin{equation}
\begin{split}
\int_\R\cfrac{(x-a_1)(x-a_2)-b_1b_2}{|x-r_1|^2|x-r_2|^2}d\mu_0(x)=&\cfrac{1}{t}\\
\int_\R\cfrac{b_1(x-a_2)+b_2(x-a_1)}{|x-r_1|^2|x-r_2|^2}d\mu_0(x)=&0.
\end{split}
\label{sys}
\end{equation}
Thanks to the second equation of $\eqref{sys}$ we have: 
\begin{equation}
\int_{\R}\cfrac{x}{|x-r_1|^2|x-r_2|^2}d\mu_0(x)=\cfrac{a_2b_1+a_1b_2}{b_1+b_2}\int_\R\cfrac{1}{|x-r_1|^2|x-r_2|^2}d\mu_0(x) \label{first}.
\end{equation}
Then using that $$(x-a_1)(x-a_2)-b_1b_2=|x-r_2|^2+(a_2-a_1)(x-a_2)-b_2^2-b_1b_2,$$ the first equation of $\eqref{sys}$ gives that: 
\begin{equation}
\label{eq:uniquenesschara2}
\begin{split}
\cfrac{1}{t}=&\int_\R\cfrac{d\mu_0(x)}{|x-r_1|^2}+(a_2-a_1)\int_\R\cfrac{xd\mu_0(x)}{|x-r_1|^2|x-r_2|^2}\\
&-(a_2(a_2-a_1)+b_2^2+b_1b_2)\int_\R\cfrac{d\mu_0(x)}{|x-r_1|^2|x-r_2|^2}\\
&\underset{~\eqref{first}}{=}\int_\R\cfrac{d\mu_0(x)}{|x-r_1|^2}+\cfrac{(a_2-a_1)(a_2b_1+a_1b_2)}{b_1+b_2}\int_\R\cfrac{d\mu_0(x)}{|x-r_1|^2|x-r_2|^2}\\
&-(a_2(a_2-a_1)+b_2^2+b_1b_2)\int_\R\cfrac{d\mu_0(x)}{|x-r_1|^2|x-r_2|^2}.
\end{split}
\end{equation}
We can rewrite the last term of \eqref{eq:uniquenesschara2} by taking imaginary part of the relation $z=r_1+tS_{\mu_0}(r_1)$: 
$$\Im(z)=b_1-tb_1\int_\R\cfrac{1}{|r_1-x|^2}d\mu_0(x).$$
Using this identity in \eqref{eq:uniquenesschara2} gives: $$\cfrac{1}{t}=\cfrac{1}{t}-\cfrac{\Im(z)}{tb_1}-\cfrac{b_2}{b_1+b_2}((a_2-a_1)^2+(b_1+b_2)^2)\int_\R\cfrac{d\mu_0(x)}{|x-r_1|^2|x-r_2|^2}.$$
However the second term is clearly strictly smaller than $1/t$ since $z\in \mathbb H$ and $b_1$ and $b_2$ are positive which gives a contradiction.
\newline
\newline
Now let us prove the existence. The proof uses some similar ideas with the proof of the D'Alembert-Gauss theorem and the Liouville theorem for holomorphic and bounded functions on $\C$. 
\newline
Fix $z\in \mathbb H$. We remark that the function defined on $\mathbb H$ by $$\phi(r):=z-tS_{\mu_0}(r)$$ is holomorphic on $\mathbb H$ and takes value in $\mathbb H$ (since $\Im(\phi(r))=\Im(z)+\Im(r)\int_\R\frac{1}{|r-x|^2}\,d\mu_0(x)$ by \eqref{eq:stieljesastuce}).
\newline
To prove the existence in Lemma \ref{lemma:existencechara}, we have to prove $r\mapsto r-\phi(r)$ has a zero on $\mathbb H$.
Assume that is not the case. 
Then, the function $h$ defined on $\mathbb H$ by $$h(r):=\cfrac{1}{r-\phi(r)}$$ is holomorphic on $\mathbb H$. Let us prove that $h$ is bounded on $\mathbb H$. 
\newline
Indeed, first we have that if $r\in \mathbb H$ and $0<\Im(r)<\Im(z)/2$, then $\Im(r-\phi(r))\le -\Im(z)/2<0$ which gives that for $r\in \mathbb H$ with $0<\Im(r)<\Im(z)/2$, $$|h(r)|\le\cfrac{1}{|\Im(r-\phi(r))|}\le\cfrac{2}{\Im(z)}=:C_1.$$
Secondly, we see that for all $r\in \mathbb H$, $$|\phi(r)|\le |z|+t\cfrac{1}{\Im(r)}$$ and so in particular if $\Im(r)\ge \Im(z)/2$, we get: $$|\phi(r)|\le |z|+t\cfrac{2}{\Im(z)}.$$
Hence we can find constants $K, C_2$ such that for all $r\in\mathbb H$ such that $|r|\ge K$ and $\Im(r)\ge \Im(z)/2$ we have $|h(r)|\le C_2$. By continuity of $h$ on $\mathbb H$ we can find a constant $C_3$ such that for all $r\in\mathbb H$ such that $\Im(r)\ge \Im(z)/2$ and $|r|\le K$ we get $|h(r)|\le C_3$.
\newline
So for all $r\in \mathbb H$, $|h(r)|\le \max(C_1,C_2,C_3)$. It proves that $h$ is bounded on $\mathbb H$.
\newline
For all $\varepsilon>0$, we define $h_\varepsilon(r):=h(r+i\varepsilon)$ for $r\in\C$ such that $\Im(r)> -\varepsilon$.
\newline
From now on, fix $r\in\mathbb H$. The goal shall be to express the value of $h_\varepsilon(r)$ thanks to the value of $h_\varepsilon(x)$ for $x\in\R$. 
To do so, we use the residue theorem to the function $$g_{\varepsilon,r}(z'):=\cfrac{h_\varepsilon(z')}{(z'-\Re(r))^2+\Im(r)^2},$$ with the closed curved $\gamma_R=\gamma_{1,R}\vee\gamma_{2,R}$ such that $\gamma_{1,R}=[-R,R]$ and $\gamma_{2,R}$ is the semi circle centred in $0$ which joins $R$ and $-R$ in the upper half plane for $R>0$. For $R$ large enough, there is one residue inside $\gamma_{R}$ at $z'=r$ with value $\cfrac{h_\varepsilon(r)}{2i\Im(r)}.$ 
\newline
The residue theorem yields that $R>0$ large enough: 
\begin{equation}
\label{eq:residuetheorem}
\cfrac{\pi h_\varepsilon(r)}{\Im(r)}=\int_{\gamma_{1,R}}g_{\varepsilon,r}(z')dz'+\int_{\gamma_{2,R}}g_{\varepsilon,r}(z')dz'.
\end{equation} 
This $h$ is bounded on $\mathbb H$, thanks to the dominated convergence theorem we obtain $$\int_{\gamma_{2,R}}g_{\varepsilon,r}(z')dz'\underset{R\to+\infty}{\longrightarrow} 0.$$
Passing to the limit in \eqref{eq:residuetheorem} yields \begin{equation}
\label{eq:residuetheorem2}\pi h_\varepsilon(r)=\int_{\R}h_\varepsilon(x)\cfrac{\Im(r)}{(x-\Re(r))^2+\Im(r)^2}dx.
\end{equation}
Now we notice that for $x\in\R$, $\Im(x+i\varepsilon-\phi(x+i\varepsilon))\le\varepsilon-\Im(z)$. So if we choose $\varepsilon<\Im(z)/2$, then $\Im(x+i\varepsilon-\phi(x+i\varepsilon))<0$ for all $x\in\R$.
\newline
It gives that $\Im(h_\varepsilon(x))>0$ for all $x\in\mathbb R$.
We deduce from formula \eqref{eq:residuetheorem2} that $\Im(h_\varepsilon(r))>0$ for all $r\in \mathbb H$ if $\varepsilon$ is small enough. Let $\varepsilon$ goes to 0 to obtain $\Im(h(r))\ge 0$ for all $r\in\ \mathbb H$. Hence, it gives that for all $r\in\mathbb H$, $\Im(\phi(r)-r)\ge 0$.
\newline
However for $r\in\mathbb H$, we can compute:        $$0\le\Im(\phi(r)-r)=\Im(z)+t\,\Im(r)\int_\R\cfrac{1}{|r-x|^2}\,d\mu_0(x)-\Im(r),$$ and thanks to: $$\int_\R\cfrac{1}{|r-x|^2}\,d\mu_0(x)\le\cfrac{1}{|\Im(r)|^2}, $$ we obtain for $r\in\mathbb H$ that $$0\le \Im(\phi(r)-r)\le\Im(z)+t\,\cfrac{1}{\Im(r)}-\Im(r)$$
Let $\Im(r)$ goes to $+\infty$ to have a contradiction.
\end{proof}

A second important lemma will be the study of a ordinary differential equation to construct the characteristic.

\begin{lemma}
\label{lemma:edocharac}
Let $z\in\mathbb H$ and $S(.,.)$ be a solution of \eqref{system Stieljes}. We can construct a solution of the ordinary differential equation in $\C$: 
\begin{equation}
\left\{
\begin{split}
z(0)=&z\\
z'(t)=&S((t,z(t)),
\end{split}
\right.
\label{edo}
\end{equation} 
which will exist up to the time $z(t)$ hits the real axis.
\end{lemma}
\begin{proof}
For $a> 0$, let $I_a:=\{z\in\C,\, \Im(z)> a\}$. We fix $a$ such that $0<a<\Im(z)$. We notice that $(t,z)\to S(t,z)$ is uniformly in time Lipschitz for the second variable on $\R^+ \times I_a$. Indeed, for all $t\ge 0$, we consider $\gamma(t)\in\mes$ such that $S(t,.)$ is the Stieljes transform of $\gamma(t)$ (by definition of a solution of \eqref{system Stieljes}). For $t\ge 0$ and $z,z'$ in $I_a$ we have: $$S_t(z)-S_t(z')=\int_\R\cfrac{1}{z-x}-\cfrac{1}{z'-x}\,d\gamma(t)(x)=\int_\R\cfrac{z'-z}{(z-x)(z'-x)}\,d\gamma(t)(x),$$ and for $x\in\R$, $$a^2\le|\Im(z)||\Im(z')|\le|z-x||z'-x|$$
which gives that $$|S_t(z)-S_t(z')|\le \cfrac{1}{a^2}|z-z'|.$$
Hence, by the Cauchy-Lipschitz theorem, for all $a>0$ we have a solution of the ordinary differential equation $\eqref{edo}$, $z_a$ which exists until the time $\tau_a$ which is the moment when $z_a$ hits $\{z\in\C,\, \Im(z)= a\}$. Thanks to the uniqueness part of the Cauchy-Lipschitz theorem if $a<b$, $z_a$ and $z_b$ coincide until $\tau_b$. This allow us to define a solution of $\eqref{edo}$; $z_0$ until the moment when this solution hits $\{z\in\C,\, \Im(z)=0\}$ (which is possibly infinite).
\end{proof}

Now let us prove the existence and uniqueness of a solution of \eqref{system Stieljes}.
 
\begin{prop}
\label{prop:uniquenesschara}
There exists a unique solution $S$ of $\eqref{system Stieljes}$.
\end{prop}

\begin{proof}
The existence of a solution of this system is obtained by Proposition \ref{prop:charastieljes} taking any almost sure limit point of $(\mu_N)_{N\ge 1}$ which is almost surely relatively compact by Proposition \ref{prop:almostsurelyprecompact}.
\newline
As said before we shall use the method of characteristic to prove the uniqueness. Indeed, by Lemma \ref{lemma:edocharac} let $z_r(t)$ be the solution of $\eqref{edo}$ starting from $r\in \mathbb H$ and defined until a time $\tau(r)>0$ (since $r\in\mathbb H$) which is the moment when $z_r(t)$ its the real axis. By definition of $z_r(t)$ the derivative of $t\to S(t,z_r(t))$ is 0 since $S$ is solution of $\eqref{system Stieljes}$. So, for all $r\in \mathbb H$, for all $ t<\tau(r)$, 
\begin{equation}
\left\{
\begin{split}
S(t,z_r(t))&=S_{\mu_0}(r)\\
z_r'(t)&=S(t,z_r(t))=S_{\mu_0}(r)
\end{split}
\right.
\end{equation}
Hence we deduce that for all $z\in\mathbb H$,  for all $t<\tau(z)$, 
\begin{equation}
\left\{
\begin{split}
z_r(t)&=r+tS_{\mu_0}(r),\\
S(t,z_r(t))&=S(t,r+tS_{\mu_0}(r))=S_{\mu_0}(r).
\end{split}
\right.
\label{charac}
\end{equation}
Now we have an explicit expression of $z_r(t)$ and so we can have an explicit expression of $\tau(r)$ which is solution of: $$\Im(z_r(t))=\Im(r)-t\Im(r)\int_\R\cfrac{1}{|r-x|^2}\,d\mu_0(x)=0.$$
Hence, we find  $$\tau(r)=\cfrac{1}{\dis\int_\R\cfrac{1}{|r-x|^2}\,d\mu_0(x)}.$$
Now fix $t>0$ and $z\in\mathbb H$. By Lemma \ref{lemma:existencechara}, we can find a unique $r\in\mathbb H$ such that $z=r+tS_{\mu_0}(r)$. So, we can write $S(t,z)=S(t,r+tS_{\mu_0}(r))$. Since $z\in\mathbb H$ we necessarily have that $t<\tau(r)$ because if not we would have that: $$\Im(z)=\Im(r+tS_{\mu_0}(r))=\Im(r)-t\Im(r)\int_\R\cfrac{1}{|r-x|^2}\,d\mu_0(x)=\Im(r)\left(1-\cfrac{t}{\tau(r)}\right)\le 0.$$
So using \eqref{charac} we get: $$S(t,z)=S(t,r+tS_{\mu_0}(r))=S(t,z_r(t))=S_{\mu_0}(r).$$
This proved that the value of $S(t,z)$ for every $t\ge 0$ and $z\in\mathbb H$ is uniquely determined by $S_{\mu_0}$ proving the uniqueness of a solution of \eqref{system Stieljes}.
\end{proof}

Finally, we conclude this section by stating a theorem for the convergence of the system of particles which summarizes the results obtained in this section. 

\begin{theorem}
\label{thm:systemofparticlesconvergence}
For all $N\ge 1$, let $(\lambda_t^{i,N})_{1\le i\le N}$ be the solution of $\eqref{dysonreal}$ starting from $\lambda_0^N\in \R^N_>$. Let $(\mu_N)_{N\ge 1}\in \left(C([0,T],\mes)\right)^{\N_\ge 1}$ be the sequence of empirical distributions associated to the $(\lambda_t^{i,N})_{1\le i\le N}$. Assume that $(\mu_N(0))_{N\ge 1}\in (\mes)^{\N}$ converges weakly to a measure $\mu_0\in\mes$ and that $$C_0:=\underset{N\ge 0}{\sup}\cfrac{1}{N}\sum_{i=1}^N\log((\lambda_0^{i,N})^2+1)<+\infty.$$
Then almost surely $(\mu_N)_{N\ge 1}\in \left(C([0,T],\mes)\right)^\N$ converges to the deterministic measure valued process $\mu\in C([0,T],\mes)$ whose Stieljes transform $S_t(z)=S_{\mu(t)}(z)$ is the unique solution of \begin{equation}
\label{eq:edpstieljes}
S_t(z)=S_{\mu_0}(z)-\int_0^tS_s(z)\partial_zS_s(z)ds, \,\forall t\ge 0,\, \forall z\in \C-\R
\end{equation}
\end{theorem}

\begin{proof}
By Proposition \ref{prop:almostsurelyprecompact} $$(\mu_N)_{N\ge 1}\in \left(C([0,T],\mes)\right)^\N$$ is almost surely relatively compact in $\left(C([0,T],\mes)\right)$. 
\newline 
Let $\omega \in \Omega$ such that $(\mu_N(\omega))_{N\ge 1}$ is relatively compact in $(C([0,T],\mes)$. Consider $\nu(\omega)\in C([0,T],\mes)$ a limit point of $(\mu_N(\omega))_{N\ge 1}$. By Proposition \ref{prop:charastieljes} and Proposition \ref{prop:uniquenesschara} $\nu(\omega)$ is uniquely determined as the unique measure valued process $\nu\in C([0,T],\mes)$ whose Stieljes transform $S_t(z):=S_{\nu(t)}(z)$ is the unique solution of \eqref{eq:edpstieljes}.
\newline
So, the sequence $(\mu_N(\omega))_{N\ge 1}$ has a unique limit point and so converges towards it. 
\end{proof}

\subsubsection{Application: the Wigner theorem}
In this section, for all $N\ge 0$ we consider the unique strong solution of \eqref{dysonreal}, $(\lambda_t^{i,1})_{t\in\R^+, \,1\le i\le N}$ starting from $(\lambda_0^{i,N})_{1\le i\le N}=0\in\R^N$ as it was explained in Remark \ref{remark:dysonextended}. We shall use the notations of Theorem \ref{thm:systemofparticlesconvergence} as for instance that $S_t$ is the Stieljes transform of $\mu(t)$ the almost sure limit point of $(\mu_N)_{N\ge 1}$.
\newline
We recall that by the Dyson Theorem for all $N\ge 1$, if we denote $(\lambda_t^{i,1,\text{Dyson}})_{t\in\R^+, \,1\le i\le N}$ the spectrum of $(D_N^1(t))_{t\in\R^+}$, then $(\lambda_t^{i,1})_{t\in\R^+, \,1\le i\le N}$ and $(\lambda_t^{i,1,\text{Dyson}})_{t\in\R^+, \,1\le i\le N}$ have the same law.

\begin{lemma}
\label{lemma1:wignerdynamic}
The Stieljes transform of $\mu$ satisfies the following identity: for all $t\ge 0$, for all $z\in \C-\R$,: 
$$S_t(z)=t^{-1/2}S_1(zt^{-1/2}).$$
\end{lemma}

\begin{proof}
Let $t>0$. We denote $S_t^N(z)$ the Stieljes transform of $\mu_N(t)$ (let us mention that this Stieljes transform is a random varaible since $\mu_N$ is). Since $\mu_N(t)$ converges almost surely towards $\mu(t)$ in $\mes$(which is deterministic), for all $z\in\C-\R$ the dominated convergence theorem yields 
\begin{equation}
\label{eq:stieljesidentity1}
\E(S_t^N(z))\underset{N\to+\infty}{\longrightarrow}\E(S_t(z))=S_t(z).
\end{equation}
We denote $\mu_N^{\text{Dyson}}(t)$ the empirical distribution of the Dyson particles $(\lambda_t^{i,1,\text{Dyson}})_{t\in\R^+, \,1\le i\le N}$ and $S_t^{N,\text{Dyson}}$ the Stieljes transform of $\mu_N^{\text{Dyson}}(t)$. By the Dyson theorem, for all $N\ge 1$, for all $z\in\C-\R$,
\begin{equation}
\label{eq:stieljesidentity2}
\E(S_t^N(z))=\E(S_t^{N,\text{Dyson}}(z))
\end{equation}
By the scaling property of the Brownian motion: $B(t)\overset{\text{Law}}{=}\sqrt{t}B(1)$, $D_N(t)$ and $\sqrt{t}D_N(1)$ have the same law and so $(\lambda_t^{i,1,\text{Dyson}})_{1\le i\le N}$ and $(\sqrt{t}\lambda_1^{i,1,\text{Dyson}})_{1\le i\le N}$ also. 
This gives that for all $z\in\C-\R$:
\begin{equation}
\begin{split}
\label{eq:stieljesidentity3}
\E(S_t^{N,\text{Dyson}}(z))&=\E\left[\cfrac{1}{N}\sum_{i=1}^N\cfrac{1}{z-\lambda_t^{i,1,\text{Dyson}}}\right]
=\E\left[\cfrac{1}{N}\sum_{i=1}^N\cfrac{1}{z-\sqrt{t}\lambda_1^{i,1,\text{Dyson}}}\right]\\
&=t^{-1/2}\E(S_1^{N,\text{Dyson}}(zt^{-1/2}))
\end{split}
\end{equation}
Using \eqref{eq:stieljesidentity2} we obtain: \begin{equation}
\label{eq:stieljesidentity4}
\E(S_t^N(z))=t^{-1/2}\E(S_1^{N}(zt^{-1/2}))
\end{equation}
Passing to the limit in \eqref{eq:stieljesidentity4} and using \eqref{eq:stieljesidentity1} we get that for all $z\in \C-\R$ $$S_t(z)=t^{-1/2}S_1(zt^{-1/2}).$$
\end{proof}
Thanks to this lemma, to compute the Stieljes tranform at any time of $\mu$ it suffices to compute it at the the time $t=1$. 
\begin{lemma}
\label{lemma2:Wignerdynamic}
The Stieljes transform of $\mu$ at time $t=1$ is given by $S_1=S_\sigma$ where $\sigma$ is the semicircle distribution introduced in Example \ref{example:semicircle}. As a consequence $\mu(1)=\sigma$.
\end{lemma}

\begin{proof}
Using the identity obtained in Lemma \ref{lemma1:wignerdynamic} and Proposition \ref{prop:charastieljes}, for all $z\in\C-\R$, for all $t>0$, we get
\begin{equation}
\label{eq:stieljesequation1}
S_1(z)=\sqrt{t}S_t(\sqrt{t}z)=\cfrac{1}{z}-\sqrt{t}\int_0^t s^{-1/2}S_1(zt^{1/2}s^{-1/2})s^{-1}S_1'(zt^{1/2}s^{-1/2})ds.\end{equation}
Doing the change of variable $t^{1/2}s^{-1/2}=u$ in the integral yields 
\begin{equation}
\label{eq:stieljesequation2}
\begin{split}
\sqrt{t}\int_0^t s^{-1/2}S_1(zt^{1/2}s^{-1/2})s^{-1}S_1'(zt^{1/2}s^{-1/2})ds&=\int_{1}^{+\infty}S_1(zu)S_1'(zu)u^3\,\cfrac{2}{u^3}\,du\\
&=2\int_{1}^{+\infty}S_1(zu)S_1'(zu) du.
\end{split}
\end{equation}
However we have 
\begin{equation}
\label{eq:stieljesequation3}
\int_{1}^{+\infty}S_1(zu)S_1'(zu) du=\cfrac{1}{2z}\int_{1}^{+\infty}\partial_u(\left(S_1(zu)\right)^2)du=-\cfrac{1}{2z}S_1(z)^2.
\end{equation}
So \eqref{eq:stieljesequation1} and $\eqref{eq:stieljesequation3}$ give $$S_1(z)=\cfrac{1}{z}+\cfrac{1}{z}\,S_1(z)^2.$$
We recognize the functional equation satisfied by $S_\sigma$. By Proposition \ref{annexe prop unique stieljes sqrt} we deduce that $S_1=S_\sigma$.

\end{proof}

Let state the Wigner theorem in the real and complex case.

\begin{theorem}[Wigner]
For all $N\ge 1$, let $X_N$ be a real Wigner matrix of size $N$ and $\mu_N^\text{Wigner}$ be the spectral measure of $X_N/\sqrt{N}$: $$\mu_N^\text{Wigner}=\cfrac{1}{N}\sum_{i=1}^N\delta_{\frac{\lambda_{i}^N}{\sqrt{N}}},$$ with $(\lambda_i^N)_{1\le i\le N}$ the eigenvalues of $X_N$. Then for all $f:\R\to \R$ continuous and bounded: $$\E\left[\int fd\mu_N^\text{Wigner}\right]\underset{N\to+\infty}\longrightarrow\int fd\sigma_1,$$
where $\sigma_1(dx)$ is given by $$\sigma_1(dx):=\cfrac{1}{\pi}\,\sqrt{(2-x^2)_+}\,dx.$$
\end{theorem}
\begin{proof}
We use the notations of Lemma \ref{lemma1:wignerdynamic}.
\newline
We recall that for all $N\ge 1$, $X_N/\sqrt{N}$ and $D_N^1(1/2)$ have the same law. Moreover, by Lemma \ref{lemma1:wignerdynamic} and Lemma \ref{lemma2:Wignerdynamic}, we have $\mu(1/2)=\sigma_1$. Let $f:\R\to\R$ be a continuous bounded function. Then for all $N\ge 1$
\begin{equation}
\label{eq:wignerdyn1}
\E\left[\int fd\mu_N^\text{Wigner}\right]=\E\left[\int fd\mu_N^\text{Dyson}(1/2)\right].
\end{equation}
By the Dyson theorem for every $N\ge 1$ we also have: 
\begin{equation}
\label{eq:wignerdyn2}
\E\left[\int fd\mu_N^\text{Dyson}(1/2)\right]=\E\left[\int fd\mu_N(1/2)\right].
\end{equation}
Moreover, by Theorem \ref{thm:systemofparticlesconvergence}, the dominated convergence theorem yields: 
\begin{equation}
\label{eq:wignerdyn3}
\E\left[\int fd\mu_N(1/2)\right]\underset{N\to+\infty}{\longrightarrow}\E\left[\int fd\mu(1/2)\right]=\int fd\mu(1/2)=\int fd\sigma_1.
\end{equation}
\eqref{eq:wignerdyn1}, \eqref{eq:wignerdyn2} and \eqref{eq:wignerdyn3} gives the result.
\end{proof}

\begin{theorem}[Wigner]
For all $N\ge 1$, let $X_N$ be a complex Wigner matrix of size $N$ and $\mu_N^\text{Wigner}$ be the spectral measure of $X_N/\sqrt{N}$: $$\mu_N^\text{Wigner}=\cfrac{1}{N}\sum_{i=1}^N\delta_{\frac{\lambda_{i}^N}{\sqrt{N}}},$$ with $(\lambda_i^N)_{1\le i\le N}$ the eigenvalues of $X^N$. Then for all $f:\R\to \R$ continuous and bounded: $$\E\left[\int fd\mu_N^\text{Wigner}\right]\underset{N\to+\infty}\longrightarrow\int fd\sigma,$$
where $\sigma(dx)$ is the semicircle law.
\end{theorem}
\begin{proof}
We sketch the proof because it is similar to the real case. In this section we detailed the case $\beta=1$ to lighten the notations but all the computations are also valid for $\beta=2$. We notice that when we pass to the limit in Proposition \ref{prop:characlimit}, the coefficient in front the Brownian motions in the system of particles \eqref{dysonreal} does not play any role. Indeed this coefficient appears in the martingale term and in the term with second derivatives in the Itô formula in Proposition \ref{prop:itoformulasystemparticles} which disappear when we pass to the limit when $N$ goes to $+\infty$. Hence in the real case, the exact same Theorem \ref{thm:systemofparticlesconvergence} holds with the same equation at the limit for the measure valued process $\mu\in C([0,T],\mes)$. Since in the complex case for all $N\ge 1$, $X_N/\sqrt{N}$ and $D_N^2(1)$ have the same law, the exact same proof yields that for all $f:\R\to \R$ continuous and bounded: 
$$\E\left[\int fd\mu_N^\text{Wigner}\right]\underset{N\to+\infty}\longrightarrow\int f d\mu(1)=\int fd\sigma.$$
\end{proof}

\begin{Remark}
An interpretation of the equation \eqref{eq:characlimit} comes from the theory of free probability. The free probability were introduced by Voiculescu \cite{voiculescu2002free} and appears when we deal with non commutative random variables as random matrices. We choose in these notes to not use the language of free probability but the free probability represents an important topic in the study of random matrices as it gives amazing results and interpretations for some results that can be obtained in random matrix theory. We refer to \cite{Tao,Alicelivre,guionnet2010uses} for general introductions on this topic. The Dyson Brownian motion can be considered as the free Brownian motion \cite{biane1997free2,biane2001free3} and $\eqref{eq:characlimit}$ can be in this sense considered as a free version of the heat equation. As we saw in this section, the semicircle distribution can be viewed as the fundamental solution of this partial differential equation which means the solution with initial condition $\delta_0(dx)$. As for the heat equation, thanks to the free probability interpretation we can give an "explicit" formula of the solution starting from $\mu_0\in\mes$. The solution is given by what is called the free convolution of $\mu_0$ and the semicircle distribution. Thanks to this representation of the solutions one can study some analytic properties of solutions of \eqref{eq:characlimit} as the regularity. We refer to \cite{biane1997free} for such results. 
\end{Remark}

\subsection{Case of Ornstein-Uhlenbeck process}
\label{section:Orn}
In this section, we consider the particles which are solutions of the Ornstein Uhlenbeck equation \eqref{eq:sdeornseteinuhlenbeck}. First we explain the convergence of the system of particles to obtain the counterpart of Theorem \ref{thm:systemofparticlesconvergence} for the ornstein-Uhlenbeck case. We shall not give many proof but just explain how to adapt the proof of the Dyson case (that corresponds to the case $\theta=0$). Then, we study the long time behaviour of solutions of the Ornstein-Uhlenbeck case. We shall follow the arguments of \cite{Shi,Chan} for this case.

\subsubsection{Convergence of the system of particles}
For all $N\ge 1$, we consider the family of $N$ particles solution of the Ornstein-Uhlenbeck system of stochastic differential equations.
\begin{equation}
\left\{
\begin{split}
d\lambda_i(t)&=\cfrac{1}{N}\dis\sum_{j\ne i}\cfrac{1}{\lambda_i(t)-\lambda_j(t)}\,dt+\sqrt{\cfrac{2}{N}}\,dB_t^i-\theta \lambda_i(t) dt\,  \,\forall 1\le i \le N\\
\lambda_i(0)&=\lambda_0^i\,\, \forall 1\le i \le N
\end{split}
\right.
\label{Orns}
\end{equation}
with $\lambda_0\in D$ an initial data.
One can prove a counterpart of Theorem \eqref{thm:existenceparticuledyson} that is the existence of a strong solution for this system of SDE and such that the particles never collide for $t>0$ almost surely (see Lemma 1 of \cite{Shi}). 
As for the Dyson case, for all $N\ge 1$, we denote $\mu_N\in C([0,T],\mes)$ the empirical distribution of a solution $(\lambda_i)_{1\le i\le N}$ of \eqref{Orns} which means that for all $N$, for all $t\ge 0$ we have $$\mu_N(t)=\cfrac{1}{N}\sum_{i=1}^N\delta_{\lambda_i(t)}.$$
We can prove that almost surely the sequence $(\mu_N)_{N\ge 1}$ is relatively compact (see Section 3 of \cite{Chan,Shi}).
We can state the counterpart of Proposition \eqref{prop:characlimit}.
\begin{prop}
\label{prop:limitornsetein}
Fix $\omega\in \Omega$. Suppose that $(\mu_N(0))_{N\ge 1}$ converges in $\mes$ towards $\mu_0\in\mes$. Let $\mu\in C([0,T],\mes)$ be a limit point of the sequence $(\mu_N(\omega))_{N\ge 1}$ in the space $C([0,T],\mes)$, then $\mu(0)=\mu_0$ and for every $f\in C^2(\R,\R)$ with its derivatives bounded, for every $t\ge 0$, 
\begin{equation}
\begin{split}
\int_\R f(x)d\mu(t)(x)&=\int_\R f(x)d\mu_0(x)-\theta\int_0^t \int_\R xf'(x)d\mu(s)(x)ds\\
&+\cfrac{1}{2}\int_0^t\int\int\cfrac{f'(x)-f'(y)}{x-y}d\mu(s)(x)d\mu(s)(y)ds
\end{split}
\label{FormulaOrn}
\end{equation}
\end{prop}
We introduce the Stieljes transform $S_t(z)=S_{\mu(t)}(z)$ of $\mu(t)$ where $\mu$ is a limit point as in Proposition \ref{prop:limitornsetein}. The counterpart of Proposition \ref{prop:charastieljes} is the following result.

\begin{prop}
Let $\mu$ a limit point as in Proposition \ref{prop:limitornsetein}, then $\mu(0)=\mu_0$ and $(S_t)_{t\ge 0}$ satisfies: 
\begin{equation}
S_t(z)=S_0(z)-\int_0^tS_s(z)\partial_zS_s(z)ds+\int_0^t [\theta z\partial_z S_s(z)+\theta S_s(z)]ds, \,\forall z\in \C-\R.
\label{Steljes eq Orn}
\end{equation}
\end{prop}

The counterpart of \eqref{system Stieljes} is 
\begin{equation}
\left\{
\begin{split}
&S(0,z)=S_{\mu_0}(z), \forall z\in \mathbb H \\
&\partial_tS(t,z)+(S(t,z)-\theta z)\partial_zS(t,z)=\theta S(t,z), \forall t>0, \forall z\in\mathbb H\\
&\text{ for all }t\ge 0,\, S(t,.)\text{ is the Stieljes transform of a real probability measure}
\end{split}
\right.
\label{system Stieljes Orn}
\end{equation}

To prove that there is a unique solution for this system we again use the characteristic method. We follow the arguments of the Dyson case. We consider the solution of the ordinary differential equation
\begin{equation}
\left\{
\begin{split}
\label{edoornsetein}
&\partial_tz_r(t)=S(t,z_r(t))-\theta z_r(t)\\
&z(0)=r\in\mathbb H
\end{split}
\right.
\end{equation}
Using that $S(.,.)$ is a solution of \eqref{system Stieljes Orn}, we obtain:
$$\partial_t^2z_r(t)=\theta^2 z_r(t).$$ Solving this equation yields: 
\begin{equation}
\label{soledo}
z_r(t)=r\exp(-\theta t)+\cfrac{1}{\theta}S_0(r)\sinh(\theta t)
\end{equation} (if $\theta$ converges to 0 it gives that $z(t)=r+S_0(r)t$ which is the result obtained in Dyson case). In particular for every $t\ge 0$, \begin{equation}
\label{eq:stilejescharac}
S(t,z_r(t))=z_r'(t)+\theta z_{r}(t)=\exp(\theta t)S_0(r).
\end{equation}
Hence to prove the uniqueness of the solution of $\eqref{Steljes eq Orn}$ we prove that for all $t>0$, for all $z\in\mathbb H$ there exists a unique $z_0(t,z)\in\mathbb H$ such that: 
\begin{equation}
\label{charaornsetin}
z_0(t,z)\exp(-\theta t)+\cfrac{1}{\theta}S_0(z_0(t,z))\sinh(\theta t)=z.
\end{equation} This result is a consequence of Lemma \ref{lemma:existencechara}.
\newline
Thanks to \eqref{soledo} and \eqref{charaornsetin} we can follow the exact same proof as for Proposition \ref{prop:uniquenesschara} to obtain that for all $t\ge 0$ and $z\in\mathbb H$, $S(t,z)$ is uniquely determined by $S_{\mu_0}$ giving the uniqueness of a solution of \eqref{system Stieljes Orn}. We refer to \cite{Shi} for a detailed proof.
\newline
We now state a theorem for the convergence of the particles associated to an Ornstein-Uhlenbeck process.
\begin{theorem}
\label{thm:limitornsetin}
For all $N\ge 1$, let $(\lambda_t^{i,N})_{1\le i\le N}$ be the solution of $\eqref{dysonreal}$ starting from $\lambda_0^N\in \R^N_>$. Let $(\mu_N)_{N\ge 1}\in \left(C([0,T],\mes)\right)^{\N_\ge 1}$ be the sequence of empirical distributions associated to the $(\lambda_t^{i,N})_{1\le i\le N}$. Assume that $(\mu_N(0))_{N\ge 1}\in (\mes)^{\N}$ converges weakly to a measure $\mu_0\in\mes$ and that $$C_0:=\underset{N\ge 0}{\sup}\cfrac{1}{N}\sum_{i=1}^N\log((\lambda_0^{i,N})^2+1)<+\infty.$$
Then almost surely $(\mu_N)_{N\ge 1}\in \left(C([0,T],\mes)\right)^\N$ converges to the element $\mu^\theta\in C([0,T],\mes)$ whose Stieljes transform $S_t(z):=S_{\mu^\theta(t)}(z)$ is the unique solution of $$S_t(z)=S_{\mu_0}(z)-\int_0^tS_s(z)\partial_zS_s(z)ds+\int_0^t \theta z\partial_z S_s(z)+\theta S_s(z)ds, \,\forall z\in \C-\R.$$
\end{theorem}
\subsubsection{Long time behaviour and stationary solutions}

In this section we consider $\theta>0$ and let $\mu^\theta\in C([0,T],\mes)$ be the almost sure limit of the system of particles for the Ornsetin Uhlenbeck process obtained in Theorem \ref{thm:limitornsetin}. We want to look at the convergence of the process of measures $(\mu^\theta(t))_{t\ge 0}$ when $t$ goes to $\infty$. Actually we shall prove that it convergences to a semi-circle distribution.

\begin{theorem}
\label{thm:longtimebhaviou}
The measure valued process $(\mu_\theta(t))_{t\ge 0}$ converges weakly towards the measure $\sigma_\theta\in\mes$ with density: $$\sigma_\theta(x)=\cfrac{\theta}{\pi}\sqrt{\cfrac{2}{\theta}-x^2}\,\,\mathds{1}_{|x|\le\sqrt{2/\theta}}.$$
\end{theorem}
We shall prove that the Stieljes transform of $\mu^\theta(t)$ converges to the Stieljes transform of $\sigma_\theta$ when $t$ goes to $\infty$. This implies the result by Proposition \ref{annexe prop property stieljes}.
\begin{proof}
We first recall the expression of $S_t(z)$, the Stieljes transform of $\mu_\theta(t)$ for $z\in\mathbb H$ obtained thanks to the characteristic method. Using the notations of \eqref{soledo}, \eqref{charaornsetin} and \eqref{eq:stilejescharac} we have: 
\begin{equation}
\begin{split}
S(t,z)&=S\left(t,z_0(t,z)\exp(-\theta t)+\cfrac{1}{\theta} S_0(z_0(t,z))\sinh(\theta t)\right)\\
&=S(t,z_{z_0(t,z)}(t))\\
&=\exp(\theta t)S_0(z_0(t,z)).
\end{split}
\end{equation} 
So for all $z\in\mathbb H$, $\forall t\ge 0$, we have: 
\begin{equation}
S_t(z)=\exp(\theta t)S_0(z_0(t,z)),
\label{explicit}
\end{equation}
with $z_0(t,z)$ the unique element of $\mathbb H$ such that: 
\begin{equation}
z_0(t,z)\exp(-\theta t)+\cfrac{1}{\theta}S_0(z_0(t,z))\sinh(\theta t)=z.
\label{z_0}
\end{equation}
To look at the convergence of $S(t,z)$ when $t$ goes to $\infty$, we shall look at the behaviour of $z_0(t,z)$.
\newline
We fix $z\in\mathbb H$ and write $\exp(-\theta t)z_0(t,z):=a(t)+ib(t)$. Now let us study $b$ and $a$.
\newline
By taking the imaginary part of $\eqref{z_0}$, we have: 
\begin{equation}
\label{eq:convtempslong1}
b(t)-\cfrac{\sinh(\theta t)b(t)\exp(\theta t)}{\theta}\int_\R\cfrac{1}{|z_0(t,z)-x|^2}d\mu_0(x)=\Im(z).
\end{equation}
From \eqref{eq:convtempslong1}, we deduce that $b(t)\ge \Im(z)$ and also using that $$\int_\R\cfrac{1}{|z_0(t,z)-x|^2}d\mu_0(x)\le \cfrac{1}{|\Im(z_0(t,z))|^2}=\cfrac{1}{\exp(2\theta t)b(t)^2}, $$ we get $$b(t)-\cfrac{1}{b(t)\theta}\le \Im(z).$$ Solving this quadratic inequality we obtain that for all $t\ge 0$, 
\begin{equation}
\label{eq:convtempslong2}
\Im(z)\le b(t)\le\cfrac{\Im(z)+\sqrt{\Im(z)^2+\frac{4}{\theta}}}{2}.
\end{equation}
We shall now prove that $\sup_{t>0}|a(t)|<+\infty$. Indeed if it is not the case we can find a sequence $t_n$ that goes to $+\infty$ such that $|a(t_n)|$ goes to $+\infty$ when $n$ goes to $+\infty$. 
\newline
Now we take the real part of $\eqref{z_0}$ to have that: 
\begin{equation}
\label{eq:convtempslong3}
\Re(z)=a(t_n)+\cfrac{1}{\theta}\,\,\Re\left(\sinh(\theta t_n)\int_\R\cfrac{1}{z_0(t_n,z)-x}d\mu_0(x)\right). 
\end{equation}
We remark that: 
\begin{equation}
\label{eq:convtempslong4}
\sinh(\theta t_n)\cfrac{1}{z_0(t_n,z)-x}=\cfrac{1}{\cfrac{z_0(t_n,z)}{\sinh(\theta t_n)}-\cfrac{x}{\sinh(\theta t_n)}},
\end{equation} and 
\begin{equation}
\label{eq:convtempslong5}
\Re\left(\cfrac{z_0(t_n,z)}{\sinh(\theta t_n)}\right)=a(t_n)\cfrac{\exp(\theta t_n)}{\sinh(\theta t_n)},
\end{equation} which converges in absolute value to $+\infty$ by hypothesis on $a(t_n)$. As a consequence of \eqref{eq:convtempslong4} and \eqref{eq:convtempslong5} we obtain: 
\begin{equation}
\label{eq:convtempslong6}
\begin{split}
\left|\sinh(\theta t_n)\cfrac{1}{z_0(t_n,z)-x}\right|&=\cfrac{1}{\left|\cfrac{z_0(t_n,z)}{\sinh(\theta t_n)}-\cfrac{x}{\sinh(\theta t_n)}\right|}\\
&\le \cfrac{1}{\left|\Re\left(\cfrac{z_0(t_n,z)}{\sinh(\theta t_n)}-\cfrac{x}{\sinh(\theta t_n)}\right)\right|}\\
&=\cfrac{1}{\left|\Re\left(\cfrac{z_0(t_n,z)}{\sinh(\theta t_n)}\right)-\cfrac{x}{\sinh(\theta t_n)}\right|}\underset{n\to+\infty}{\longrightarrow} 0
\end{split}
\end{equation}
Moreover we also have: $$\left| \sinh(\theta t_n)\,\cfrac{1}{z_0(t_n,z)-x}\right|\le |\sinh(\theta t_n)|\cfrac{1}{\Im(z_0(t_n,z))}=\cfrac{|\sinh(\theta t_n)|}{\exp(\theta t_n)b(t_n)}\le \cfrac{1}{b(t_n)}\le\cfrac{1}{\Im(z)}, $$ by bound \eqref{eq:convtempslong2} on $b(t)$ for all $t\ge 0$.
The dominated convergence theorem yields: \begin{equation}
\label{eq:convtempslong7}
\sinh(\theta t_n)\int_\R\cfrac{1}{z_0(t_n,z)-x}d\mu_0(x)\underset{n\to\infty}{\longrightarrow} 0, \end{equation}
which is in contradiction with the equality \eqref{eq:convtempslong3}.
\newline
At this point we obtained $t\ge 0$, $$\Im(z)\le b(t)\le\cfrac{\Im(z)+\sqrt{\Im(z)^2+4/\theta}}{2}\,\,,\,\,\, \sup_{t\ge 0}|a(t)|<+\infty.$$ 
We now want to prove that $(a(t))_{t\ge 0}$, $(b(t))_{t\ge 0}$ converges when $t$ goes to $\infty$.
Since they are bounded we only have to prove that there is only one limit point for $((a(t),b(t))_{t\ge 0}$. So let $(a,b)\in\ \R\times\R^{+*}$ be a limit point of $((a(t),b(t))_{t\ge 0}$ and a sequence $t_n$ that goes to $+\infty$ such that $\exp(-\theta t_n)z_0(t_n,z)=a(t_n)+ib(t_n)$ converges to $w:=a+ib\in\mathbb H$.
Using a similar dominated convergence theorem as for \eqref{eq:convtempslong7}, passing to the limit in \begin{equation}
\begin{split}
z&=z_0(t_n,z)\exp(-\theta t_n)+\cfrac{1}{\theta}S_0(z_0(t_n,z))\sinh(\theta t_n)\\
&=z_0(t_n,z)\exp(-\theta t_n)+\cfrac{\sinh(\theta t_n)}{\theta}\,\int_\R \cfrac{1}{z_0(t_n,z)-x}\, d\mu_0(x)
\end{split}
\end{equation}
we obtain $$z=w+\cfrac{1}{2\theta w}.$$ 
The only solution in $\mathbb H$ is: $$w:=\cfrac{z+\sqrt{z^2-\cfrac{2}{\theta}}}{2}=\cfrac{1}{\theta}\cfrac{1}{z-\sqrt{z^2-\cfrac{2}{\theta}}}.$$  
Since $(a(t),b(t))_{t\ge 0}$ has only one limit point, this proves that $\exp(-\theta t)z_0(t,z)$ converges to $w$ when $t$ goes to $+\infty$.
\newline
Using $\eqref{explicit}$ and again dominated convergence theorem we get $$S_t(z)=\exp(\theta t)\int_\R\cfrac{1}{z_0(t,z)-x}d\mu_0(x)\underset{t\to+\infty}{\longrightarrow} \cfrac{1}{w}=\theta \left(z-\sqrt{z^2-\cfrac{2}{\theta}}\right)=\int_\R\cfrac{1}{z-x}\,d\sigma_\theta(x).$$
\newline
Hence, for all $z\in\mathbb C-\R$, $S_{\mu^\theta(t)}(z)\underset{t\to\infty}{\longrightarrow} S_{\sigma_\theta}(z).$
It proves that $(\mu^\theta(t))_{t\ge 0}$ converges weakly to the measure $\sigma_\theta\in\mes$.
\end{proof}
We shall now link the long time behaviour and the notion of stationary solutions.
\begin{definition}
$\nu \in \mes$ is said to be a stationary solution of $\eqref{FormulaOrn}$ if and only if for every $f\in C^2(\R,\R)$ with its derivatives bounded, 
\begin{equation}
\begin{split}
0&=-\theta \int_\R xf'(x)d\nu(x)+\cfrac{1}{2}\int_\R\int_\R\cfrac{f'(x)-f'(y)}{x-y}d\nu(x)d\nu(y)
\end{split}
\label{FormulaOrnstat}
\end{equation}
\end{definition}

\begin{prop}
$\nu\in\mes$ is a stationary solution if and only if for all $x\in supp(\nu)$, $$\theta x=\int_\R \cfrac{d\nu(y)}{x-y}.$$
\end{prop}

\begin{proof}
By definition of a stationary measure we have that $\mu\in\mes$ is a stationary measure if and only if for every $f\in C^2(\R,\R)$ with its derivatives bounded, 
\begin{equation}
\begin{split}
0= \int_\R \left(-\theta x+\int_\R\cfrac{d\nu(y)}{x-y}\right)f'(x)d\nu(x).
\end{split}
\end{equation}
Since this identity should hold for all $f\in C^2(\R,\R)$ with its derivatives bounded it gives the result.
\end{proof}

For $\nu\in\mes$, let us specify the definition of the term $$\int_\R\cfrac{d\nu(y)}{x-y}$$ for $x\in supp(\nu)$ which is a priori not defined. This term must be viewed as a distribution.
\newline
We recall that the principal value of $1/x$: $P.V.(1/x)$ is a distribution defined by: $$P.V.(1/x)(\phi)=\lim_{\varepsilon \to 0}\int_{|y|\ge \varepsilon}\cfrac{\phi(y)}{y}\, dy,$$ for $\phi$ a smooth function with compact support.
\begin{definition}
\label{def:hilbtransfo}
Let $\nu\in\mes$, we define the Hilbert transform of $\nu$ as the distribution $$\mathcal H(\nu)=P.V.(1/x)\ast \nu.$$
We can express the Hilbert transform of $\mu$ as
 $$\mathcal H(\nu)(x)=\lim_{\varepsilon \to 0}\int_{|x-y|\ge \varepsilon}\cfrac{\nu(dy)}{x-y}$$ if this quantity is well defined. This quantity is linked with the Stieljes transform of $\nu$ by the following formula: $$\mathcal H(\nu)(x)=\lim_{\varepsilon\to 0}\cfrac{S_\nu(x+i\varepsilon)+S_\nu(x-i\varepsilon)}{2}=\lim_{\varepsilon\to 0}\Re(S_\nu(x+i\varepsilon)).$$

\end{definition}

With this definition we can reformulate the equation for a stationary solution $\nu\in\mes$ writing that $\nu\in\mes$ is a stationary solution if and only if for all $x\in supp(\nu)$: 
\begin{equation}
\label{equationhilberttransfoorns}
\mathcal H(\nu)(x)=\theta x.
\end{equation}

\begin{example}
We can compute the Hilbert transform of the Wigner semi circle law using its Stieljes transform. For all $x\in [-\sqrt{2/\theta},\sqrt{2/\theta}]$: $$\mathcal H(\sigma_\theta)(x)=\theta x,$$ which proves that $\sigma_\theta$ is a stationary solution. This is coherent with Theorem \ref{thm:longtimebhaviou}. In other words, a solution of \eqref{FormulaOrn} converges in long time to a stationary solution of the equation.
\end{example}
To conclude this section let us mention an heuristic to solve \eqref{equationhilberttransfoorns}. We focus on the case $\theta=1/2$ to lighten the notations.
\newline
Let $\nu\in\mes$ with a density with respect to Lebesgue measure such that for all $x\in supp(\mu)$
$$x=2\mathcal H(\nu)(x).$$
We can rewrite this equality as saying that for all $x\in\R$ we have
\begin{equation}
\label{eq:solvestieljes}
(x-2\mathcal H(\nu)(x))\nu(x)=0
\end{equation}
Using the properties of the Stieljes transform (see Proposition \ref{annexe prop property stieljes}), for all $x\in\R$, we have $$\lim_{\varepsilon
\to 0} \Im(S_\nu(x+i\varepsilon))=-\pi\nu(x),\,\, \lim_{\varepsilon
\to 0} \Re(S_\nu(x+i\varepsilon))=\mathcal H(\nu)(x).$$
We can rewrite \eqref{eq:solvestieljes} as \begin{equation}
\label{eq:solvestieljes2}
\lim_{\varepsilon\to 0}\left[(x-2\Re(S_\nu(x+i\varepsilon)))\Im(S_\nu(x+i\varepsilon))\right]=0
\end{equation} for all $x\in\R$.   
This gives that for all $x\in\R$: 
\begin{equation}
\label{eq:solvestieljes3}
\lim_{\varepsilon\to 0}\Im\left[(x+i\varepsilon) S_\nu(x+i\varepsilon)-S_\nu(x+i\varepsilon)^2\right]=0.
\end{equation}
Let $f(z):=zS_\mu(z)-S_\mu(z)^2$. By the properties of the Stieljes transform, $f$ is holomorphic on $\C-\R$, satisfies $\overline{f(z)}=f(\overline z)$ for all $z\in\C$ and for all $x\in\R$, $\Im\left[\lim_{\varepsilon\to 0}f(x+i\varepsilon)\right]=0$. This allows us to extend $f$ as an analytic function on $\C$. Moreover, since for all $z\in\C-\R$, we have $$S_\mu(z)=\int_\R\cfrac{1}{z-t}\,d\nu(t)=\cfrac{1}{z}\int_\R\cfrac{1}{1-\frac{t}{z}}\,d\nu(t)\underset{z\to+\infty}{\sim}\cfrac{1}{z},$$ we obtain $f(z)\underset{z\to+\infty}{\longrightarrow}1$.
The Liouville theorem gives that $f$ is constant equal to $1$. So for all $z\in\C$, $$zS_\mu(z)-S_\mu(z)^2=1.$$
We recover the functional equation that characterizes the Stieljes transform of the semicircle distribution obtained in Proposition \ref{annexe prop unique stieljes sqrt}.
\newpage
\part{Large deviations}
\label{sec:largedev}
The goal of large deviations is to understand how small is the probability that rare events append. The most usual example is the deviations of the empirical mean of iid random variables from its mean. By the law of large numbers we know that if $(X_i)_{i\in\N_{\ge 1}}$ are iid real random variables in $L^1$ then $(X_1+...+X_n)/n$ converges almost surely towards $\E(X_1)$. The large deviations theorem associated to the law of large numbers is the Cramer theorem.
\begin{theorem}[Cramer \cite{Demb}]
Let $(X_i)_{i\in\N_{\ge 1}}$ be a family of real random variables and $\Lambda(t)=\log(\E[\exp(tX_1)])$ be the logarithmic moment generating function of $X_1$. Assume that for all $t\in\R$, $\Lambda(t)<+\infty$. Let $\Lambda^*(x)=\sup_{t\in\R}(tx-\Lambda(t))$ be the Legendre transform of $\Lambda$. Then for all $x>\E(X_1)$ we have: $$\lim_{n\to\infty}\cfrac{1}{n}\log\left[\mP\left(\sum_{i=1}^nX_i\ge nx\right)\right]=-\Lambda^*(x).$$
\end{theorem}
In the language of large deviations (the complete definitions of large deviations shall be given in the Section \ref{subsection def grande dev}), we can reformulate the Cramer theorem saying that the family $(\mP_n)_{n\in \N_{\ge 1}}$ where for all $n\ge 1$, $\mP_n$ is the law of $\sum_{i=1}^n X_i/n$, satisfies a large deviations principle at speed $n$ and with good rate function $\Lambda^*$. 
\newline
Another important theorem in large deviations theory is the Sanov theorem \cite{csiszar2006simple,Demb}. Again by the law of large numbers if $(X_i)_{i\in\N_{\ge 1}}$ are iid real random variables with law $\mathcal L$ in $L^1$ then $(\delta_{X_1}+...+\delta_{X_n})/n$ converges almost surely in $\mes$ endowed with the weak convergence towards $\mathcal L$.
The Sanov theorem states that for $(X_i)_{i\in\N_{\ge 1}}$ a family of independent random variables of law $\mathcal L\in\mes$, defining $(\mP_n)_{n\in\N_{\ge 1}}$ the family of law of $(\delta_{X_1}+...+\delta_{X_n})/n$ (viewed as random variables in $\mes$), then this sequence of probabilities satisfies a large deviations principle at speed $n$ with good rate function $D(P||\mathcal L):=\int\log(dP/d\mathcal L)dP$ if $P<<\mathcal L$ and $+\infty$ otherwise. The function $D$ is called the Kullback-Leibler divergence and appears in statistics and also in information theory where it is called relative entropy.
\newline
The goal of this section is to present the theory of large deviations through the large deviations of the empirical mean of eigenvalues of random matrices. The theorem is similar with the Sanov theorem but here the eigenvalues are not at all independent and identically distributed. The good rate function shall be linked with the so-called free entropy of Voiculescu \cite{voiculescu2002free} instead of the entropy.
In 1997, G.Ben Arous and A.Guionnet proved this large deviations principle for the eigenvalues of so-called $\beta$-ensembles as the $GUE_N$ and $GOE_N$ \cite{Alice1}. Then it was extended to other models of random matrices as Ginibre ensemble \cite{arous1998large} and the circular ensemble \cite{hiai2000large}. The method they used is actually quite general and has been generalised to Coulomb and Riesz gases in any dimensions \cite{chafai2014first,serfaty2024lectures,serfaty2015coulomb,chafai2018concentration}. 

\section{Notations and motivations}
\subsection{Recalls and notations}
First we recall some general notations and classical results about random matrices that have been seen in the two first parts. 
\begin{definition}
\label{def:measureemppirique}
For a finite family of real numbers $(\lambda_i)_{1\le i\le N }$ we write $\overline{\mu_N}$ its empirical distribution, which the is a probability measure on $\R$ defined by $$\overline{\mu_N}=\dis\frac{1}{N}\dis\sum_{i=1}^N\delta_{\lambda_i}.$$

Furthermore if a matrix $M\in M_N(\C)$ with real eigenvalues $(\lambda_i)_{1\le i\le N }$, we define its spectral measure as the empirical distribution of its eigenvalues. 
\end{definition}

\begin{definition}
For $\beta>0$, let $\beta$ be the semicircular distribution defined as the probability measure on $\R$ with density $$\sigma_\beta(x):=\dis\frac{1}{\beta\pi}\mathds{1}_{[-\sqrt{2\beta},\sqrt{2\beta}]}\sqrt{2\beta-x^2}.$$ 
\end{definition}

\begin{definition}
The space of probability measure on $\R$ will be denoted $\mes$. We endow it with its usual weak topology which is metrizable by the Kantorovich-Rubinstein distance. 
\end{definition}
We do not explicit the Kantorovich-Rubinstein distance for now. The important point is to take a distance that metrizes the weak topology of $\mes$.
\newline
Let us recall that we defined the Gaussian Unitary Ensemble and the Gaussian Orthogonal Ensemble in Definition \ref{def gue,goe}. The Wigner theorem which can be viewed as the law of large numbers for the eigenvalues of random matrices (see Theorem \ref{thm strong wigner}).

\begin{theorem}[Wigner]If $X_N$ is a $N$ complex Wigner matrix (resp. $N$ real Wigner matrix) then the spectral measure of $X_N/\sqrt{N}$ almost surely converges to $\sigma_2$ (resp. $\sigma_1$). 
\end{theorem}

In Part \ref{part:ginibregaussian}, we computed the distribution of the eigenvalues of a $N$ complex Wigner matrix (see the Theorem \ref{thm:law of spectrum}). A similar result holds for a matrix which is a $N$ real Wigner matrix. 
\begin{prop}
\label{prop:law goe}
Let $X_N$ be a $N$ real Wigner matrix then the probability distribution of the eigenvalues of $X_N/\sqrt{N}$ is $$Q^N_1(d\lambda_1,...,d\lambda_N):=\dis\cfrac{1}{Z_1^N}\prod_{1\le i<j\le N}|\lambda_i-\lambda_j|\exp{\left(-\dis\frac{1}{2}N\sum_{i=1}^N\lambda_i^2\right)}\prod_{i=1}^Nd\lambda_i$$ with $Z_1^N$ a normalizing constant.
\end{prop}

\begin{prop}
\label{prop:law gue}
Let $X_N$ be a $N$ complex Wigner matrix then the probability distribution of the eigenvalues of $X_N/\sqrt{N}$ is $$Q^N_2(d\lambda_1,...,d\lambda_N):=\dis\cfrac{1}{Z_2^N}\dis\prod_{1\le i<j\le N}|\lambda_i-\lambda_j|^2\exp{\left(-\dis\frac{1}{2}N\dis\sum_{i=1}^N\lambda_i^2\right)}\dis\prod_{i=1}^Nd\lambda_i$$ with $Z_2^N$ a normalizing constant.
\end{prop}
These two distributions can be viewed as particular cases of so-called $\beta$ Hermite ensembles.
\begin{definition}
\label{def:measure beta}
Let $\beta>0$, we define $$Q^N_\beta(d\lambda_1,...,d\lambda_N):=\dis\cfrac{1}{Z_\beta^N}\dis\prod_{1\le i<j\le N}|\lambda_i-\lambda_j|^{\beta}\exp{\left(-\dis\frac{1}{2}N\dis\sum_{i=1}^N\lambda_i^2\right)}\dis\prod_{i=1}^Nd\lambda_i$$ with $Z_\beta^N$ a normalizing constant  to obtain a probability measure on $\R^N$. The measure $Q^N_\beta$ defined the $\beta$ Hermite ensemble.
\newline
We also define $$\overline{Q^N_\beta}(d\lambda_1,...,d\lambda_N):=\dis\prod_{1\le i<j\le N}|\lambda_i-\lambda_j|^{\beta}\exp{\left(-\dis\frac{1}{2}N\dis\sum_{i=1}^N\lambda_i^2\right)}\dis\prod_{i=1}^Nd\lambda_i,$$ the measure without normalization.
\end{definition}
\begin{Remark}
It is quite natural to ask ourselves if for all $\beta>0$ we can find a law on random matrices such that if a random matrix is distributed with this law, then its eigenvalues are distributed like the $\beta$ Hermite ensemble. This is the case with $\beta=1$ and $\beta=2$ by the Proposition \ref{prop:law goe} and the Proposition \ref{prop:law gue}. The answer is true for all $\beta>0$ with random tridiagonal hermitian matrices \cite{dumitriu2002matrix}.
\end{Remark}
Now we extend $Q^N_\beta$ on $\mes$ to view $Q^N_\beta$ as a probability measure on $\mes$ and not only $\R^N$ in view of proving a large deviations principle in $\mes$ as for the Sanov theorem.
\begin{definition}
\label{def:extend}
For all $N\ge 1$, let $i_N: \R^N\to \mes$ be the map defined by $i_N(\lambda_1,...,\lambda_n):=\overline{\mu_N}$. For all $N\ge 1$, for all $\beta>0$, we extend $Q^N_\beta$ as a probability measure on $\mes$ by considering the push forward measure $i_N\sharp Q^N_\beta$ defined by $$(i_N\sharp Q^N_\beta)(A)=Q^N_\beta(i_N^{-1}(A)),$$ for $A$ a measurable subset of $\mes$. By abuse of notation we still call $Q^N_\beta$ these measures on $\mes$.
\end{definition}

To avoid confusions this means in practice that for a measurable function $F :\mes\to \R^+$ we have $$\int_{\mes} F(\mu)Q^N_\beta(d\mu)=\int_{\R^N} F\left(\cfrac{1}{N}\sum_{i=1}^N\delta_{\lambda_i}\right)Q^N_\beta(d\lambda_1,...,d\lambda_N)=\int_{\R^N} F\left(\overline{\mu_N}\right)Q^N_\beta(d\lambda_1,...,d\lambda_N),$$
where the notation $\overline{\mu_N}$ is introduced in Definition \ref{def:measureemppirique}.
In particular if $F$ is just the indicator of a ball $B(\mu,\varepsilon)$ we get $$Q_\beta^N(B(\mu,\varepsilon))=\int_{\R^N} \mathds 1_{\overline{\mu_N}\in B(\mu,\varepsilon)}(\lambda_1,...,\lambda_N)Q^N_\beta(d\lambda_1,...,d\lambda_N)
=Q_\beta^N(\overline{\mu_N}\in B(\mu,\varepsilon)),$$
where the $Q_\beta^N$ on the left refers to the measure on $\mes$ and the $Q_\beta^N$ on the right refers to the measure on $\R^N$.
\subsection{Motivation}
The Wigner theorem could be viewed as a law of large numbers. We want to obtain as the Cramer or the Sanov theorem an estimation of the probability that the empirical mean of the eigenvalues of random matrices is different from its almost sure limit. 
\newline
As explained at the beginning of this section the goal of large deviations is to find a "good" function $I$ on $\mes$ such that $$\forall \mu\in\mes,\, \mP(\,\overline{\mu_N}\cong \mu)\underset{N\to+\infty}{\cong}\exp(-N^2I(\mu)),$$ where $\overline{\mu_N}$ is the spectral measure of a random matrix in the $GUE_N$ for instance.
We can naturally expect a term in $N^2$ because a random matrix has $N^2$ entries and we want that $I(\mu)>0$ for every $\mu\in\mes$ different from $\sigma_2$ and $I(\sigma_2)=0$ in the case of the $GUE_N$. Hence it shall imply an exponential decrease of the probability for the spectral measure to be far from $\sigma_2$ in this case.
\subsection{Few definitions of large deviations}
\label{subsection def grande dev}
We give some standard definitions of the theory of large deviations, for a more general approach of this topic see the textbook \cite{Demb} or for instance \cite{touchette2011basic}.
\begin{definition}
Let $X$ be a metric space. A function $I:X\to ]-\infty,\infty]$ is a lower semi continuous (lsc) function if for all $M\in\R$, $\{x,I(x)\le M\}$ is a closed subset of $X$.
\end{definition}
We have a sequential characterization for the lsc functions.
\begin{prop}
\label{prop: lcs equi}
Let $X$ be a metric space. A function $I:X\to ]-\infty,\infty]$ is a lsc function if and only if for all sequences $(x_n)_{n\in\N}\in X^\N$ that converge to $x\in X$ we have $\lim \inf_n I(x_n)\ge I(x)$.
\end{prop}
\begin{proof}
For the direct sense, we suppose that $I$ is lsc and we consider a sequence $(x_n)_{n\in\N}$ that converges to an $x\in X$. Assume by contradiction that $\lim \inf_n I(x_n)< I(x)$. We can find $\varepsilon>0$ and a subsequence of $(x_n)_{n\in\N}$ that we call $(x_{\phi(n)})_{n\in\N}$ such that for $n$ large enough $I(x_{\phi(n)})\le I(x)-\varepsilon$. Since $\{y,I(y)\le I(x)-\varepsilon\}$ is closed and $(x_n)_{n\in\N}$ converges to $x$, we deduce that $x\in\{y,I(y)\le I(x)-\varepsilon\}$. So $I(x)\le I(x)-\varepsilon$ which is a contradiction.
\newline
Conversely, let $M\in\R$ and $(x_n)_{n\in\N}$ a sequence that converges to $x\in X $ such that for all $n\in\N$, $x_n\in\{y,I(y)\le M\}$. It yields $M\ge \lim\inf_nI(x_n)\ge I(x)$. So $x\in\{y,I(y)\le M\}$. This proves the result.
\end{proof}

\begin{lemma}
\label{lemma: minim lsc}
Let $X$ be a compact metric space and $I:X\to ]-\infty,\infty]$ a lsc function. Then $I$ reaches its infimum.
\end{lemma}

\begin{proof}
Let $m=\inf_X I(x)$ the infimum of $I$. By definition of the infimum, we can find a sequence $(x_n)\in X^\N$ such that $(I(x_n))_{n\in\N}$ converges to $m$. Since $X$ is compact, we can find a subsequence $(x_{\phi(n)})_{n\in\N}$ that converges to a $x\in X$. By Proposition \ref{prop: lcs equi}, $m\ge I(x)$. Hence by definition of the infimum, it gives that $m=I(x)$ and so $I$ reaches his infimum in $x$. 
\end{proof}

\begin{definition}
Let $X$ be a metric space. A rate function is a lsc function $I: X\to ]-\infty,+\infty]$. A rate function $I$ is called a good rate function if its sub-level sets $\{x,I(x)\le M\}$ are compact.
\end{definition}

\begin{lemma}
\label{coro ferme}
Let $X$ be a metric space and $I:X\to [0,\infty]$ a good rate function that has a unique minimizer $x^*$ such that $I(x^*)=0$. Let $F$ be a closed subset of $X$ that does not contain $x^*$. Then $\inf_F I>0$.
\end{lemma}

\begin{proof}
Assume by contradiction that $\inf_F I=0$. We can find a sequence $x_n\in F$ such that $I(x_n)$ converges to 0. Since for all $M\in\R$, $\{I\le M\}$ is a compact subset of $X$, we can find a sub sequence $x_{\phi(n)}$ that converges to an element $x_0\in F$ since $F$ is closed. Since $I$ is lower semi continuous, Proposition $\ref{prop: lcs equi}$ yields $$0=\liminf_n I(x_{\phi(n)})\ge I(x_0).$$ By uniqueness of the minimizer of $I$, we get that $x_0=x^*$ which contradicts that $ x^*\notin F$.
\end{proof}

We can now give a definition for a large deviations principle.
\begin{definition}
\label{def: large dev}
Let $(P_N)_{N\in\N}$ be a sequence of Borel probability measures on a metric space $X$ and $(a_N)_{N\in\N}$ be a sequence of positive real numbers diverging to $+\infty$. Let also $I$ be a good rate function on $X$. The sequence $(P_N)_{N\in\N}$ is said to satisfy a large deviations principle (LDP) at speed $a_N$ with good rate function $I$ if for every Borel set $E\subset X$ the following inequalities hold: $$-\inf_{x\in\overset{\circ} {E}}I(x)\le \underset{N\to +\infty}{\underline{\lim}}\cfrac{1}{a_N}\log(P_N(\overset{\circ} {E}))\le \underset{N\to +\infty}{\overline{\lim}}\cfrac{1}{a_N}\log(P_N(\overline {E}))\le -\inf_{x\in\overline {E}}I(x).$$
\end{definition}

\begin{prop}
\label{prop: inffunctionrate}
Let $(P_N)_{N\in\N}$ be a sequence of Borel probability measures on a metric space $X$ that satisfies a large deviations principle with good rate function $I$, then $\inf_X I=0$. 
\end{prop}

\begin{proof}
Taking $E$ equal to $X$ in Definition $\ref{def: large dev}$ gives: 
$$-\inf_{X}I\le 0\le 0\le -\inf_X I,$$
since $P_N(X)=1$ for all $N\in\N$. This gives the result.
\end{proof}
We finally justify how we usually go from a LDP to an almost sure convergence using the Borel-Cantelli Lemma.
\begin{prop}[From large deviations to almost sure convergence]
\label{prop:almostsureconvlargedev}
Let $(X_N)_{N\in\N_{\ge 1}}$ be a family of random variables defined on a same probability space $(\Omega,\mathcal F,\mP)$ with value in a metric space $X$. For all $N\ge 1$, let $P_N$ be the law of $X_N$. Assume that $(P_N)_{N\in\N_{\ge 1}}$ satisfies a LDP with good rate function $I$ at speed $N^s$ with $s>0$. First, $I$ reaches its infimum. Assume that this infimum is reached in a unique $x^*\in X$. Then, we have the following convergence $$X_N\overunderset{\emph{Almost surely}}{N\to+\infty}{\longrightarrow}x^*.$$
\end{prop}

\begin{proof}
By Proposition \ref{prop: inffunctionrate}, $\inf_X I=0$. Let $K=\{x\in X,\, I(x)\le 1\}$. Since $\inf_X I=0$ and $I$ is a good rate function, $K$ is a non empty compact subset of $X$. Moreover as $\inf_X I=\inf_K I$, Lemma \ref{lemma: minim lsc} yields that $I$ reaches its infimum. We now assume that this infimum is uniquely reached in $x^*\in X$. 
\newline
The LDP gives that for all $\varepsilon>0$ \begin{equation}
\label{eq:limsupalmostsure}
\underset{N\to +\infty}{\overline{\lim}}\cfrac{1}{N^s}\log(P_N(X\setminus B(x^*,\varepsilon)))\le -\inf_{x\in X\setminus B(x^*,\varepsilon)}I(x),
\end{equation}
with $B(x^*,\varepsilon)$ the open ball of center $x^*$ and of radius $\varepsilon>0$ and $X\setminus B(x^*,\varepsilon)$ its complement in $X$.
\newline
By Lemma \ref{coro ferme}, for all $\varepsilon>0$ we have $0<\inf_{x\in X\setminus B(x^*,\varepsilon)}I(x)=:2c_\varepsilon$.
We deduce from \eqref{eq:limsupalmostsure} that for all $\varepsilon>0$ there exists $N(\varepsilon)$ such that for all $N\ge N(\varepsilon)$
$$\mP(X_N\notin B(x^*,\varepsilon))\le \exp(-N^s c_\varepsilon)).$$
The Borel-Cantelli lemma yields that for all $\varepsilon>0$, for almost all $\omega\in\Omega$ there exists $N^*(\omega,\varepsilon)$ such that for all $N\ge N^*(\omega,\varepsilon)$ $X_N(\omega)\in B(x^*,\varepsilon)$. 
Using a countable dense sequence of $\varepsilon$, we deduce that for almost all $\omega\in \Omega$, for all $\varepsilon>0$, there exists $N^*(\omega,\varepsilon)$ such that for all $N\ge N^*(\omega,\varepsilon)$ $X_N(\omega)\in B(x^*,\varepsilon)$. 
\newline
This proves that $X_N$ converges almost surely towards $x^*$.
\end{proof}
\section{Free entropy and rate function}
In this section we define the function that shall be the good rate function for the large deviations principle. In this section $\beta$ is a fixed positive real number.

\subsection{General study of the rate function}

\begin{definition}
\label{def:enegrygrandedev}
Let $H_\beta$ the function defined on $\mes$ by: $$H_\beta(\mu):=\cfrac{1}{2}\,\int_{\R^2} \left[\cfrac{x^2}{2}+\cfrac{y^2}{2}-\beta\log(|x-y|)\right]\mu(dx)\mu(dy).$$
We write $f_\beta(x,y):=\dis\cfrac{1}{2}(x^2+y^2)-\beta\log(|x-y|)$ if $x\ne y$ and $f_\beta(x,x)=+\infty$ such that $H_\beta(\mu)=\frac{1}{2}\int_{\R^2}f_\beta(x,y)\mu(dx)\mu(dy)$.
\end{definition}
Let us explain formally why this function is a priori the good function to look at as a good rate function for the large deviations principle of the sequence of probability $(Q_\beta^N)_N$. Indeed we can write $$Q_\beta^N(\lambda_1,...,\lambda_N)=\exp(-N^2 H_\beta^N(\lambda_1,...,\lambda_N)),$$ where 
\begin{equation}
\label{Energie}
H_\beta^N(\lambda_1,...,\lambda_N)=-\cfrac{\beta}{2N^2}\sum_{1\le i\ne j\le N}\log(|\lambda_i-\lambda_j|)+\cfrac{1}{2N}\sum_{i=1}^N \lambda_i^2+\cfrac{\log(Z_\beta^N)}{N^2}.
\end{equation}
Formally, let us suppose that the empirical mean $\sum_{i=1}^N\delta_{\lambda_i}/N$ converges to $\mu$ and that $\cfrac{\log(Z_\beta^N)}{N^2}$ converges to a finite quantity $Z$, then $H_\beta^N(\lambda_1,...,\lambda_N)$ converges to $H_\beta(\mu)+Z$. So formally, we get that $$Q_\beta^N(\lambda_1,...,\lambda_N)\approx\exp(-N^2 H_\beta(\mu)).$$ This makes a link with Coulomb and Riesz gases \cite{serfaty2024lectures} as the formula \eqref{Energie} corresponds to particles in interaction through a logarithmic potential that corresponds in dimension 1 to a Riesz gas and an external potential $V(x)=x^2/2$. We shall give some definitions of Riesz and Coulomb gases in Section \ref{subsection: extenson}.
\newline
\newline
\tab We first prove that $H_\beta(\mu)>-\infty$ for all $\mu\in\mes$.
\begin{lemma}
\label{lemmaHfinite}
$H_\beta$ is defined on $\mes$ and takes values in $\left[\dis\cfrac{\beta}{4}\,\left(1-\log{(2\beta)}\right),+\infty\right]$.
\end{lemma}
\begin{proof}
We shall first justify that 
\begin{equation}
\label{eq:ineq1}
-\beta\log(|x-y|)+\cfrac{x^2}{2}+\cfrac{y^2}{2}\ge \dis\frac{\beta}{2}\left(1-\log(2\beta\right)).
\end{equation}
We write $$-\beta\log(|x-y|)+\cfrac{x^2}{2}+\cfrac{y^2}{2}=\dis\cfrac{1}{2}\left((x^2+y^2)-\beta\log(x^2+y^2)\right)+\beta\log\left(\dis\frac{\sqrt{x^2+y^2}}{|y-x|}\right).$$
Since we have that for all $t\in\R^+$, $t-\beta\log(t)\ge \beta(1-\log(\beta))$ by an easy study of function and for all $(x,y)\in\R^2$, $\dis\frac{\sqrt{x^2+y^2}}{|y-x|}\ge\dis\frac{1}{\sqrt{2}}$. 
\newline
Integrating \eqref{eq:ineq1} against $\mu(dx)\mu(dy)$ gives the result.
\end{proof}
\begin{Remark}
\label{Remark:L^1}
Let us also points out that since $f_\beta$ is bounded from below by \eqref{eq:ineq1} we have that for $\mu\in\mes$, $H_\beta(\mu)$ is finite if and only if $f_\beta\in L^1(\mu\otimes \mu)$.
\end{Remark}
When $H_\beta$ is finite we can rewrite it in another form.
\begin{lemma}
\label{lemma: rewrite}
For all $\mu\in\mes$ such that $H_\beta(\mu)<+\infty$ we have $\int_\R x^2 \mu(dx)<+\infty$ and $\int_{\R^2}|\log|x-y||\mu(dx)\mu(dy)<+\infty$. So if $H_\beta(\mu)<+\infty$ we can rewrite $H_\beta(\mu)$ as $$H_\beta(\mu)=-\frac{\beta}{2}\int_{\R^2}\log|x-y|\mu(dx)\mu(dy)+\int_\R\frac{x^2}{2}\,\mu(dx)=:-\frac{\beta}{2}\,\Sigma(\mu)+\int_\R\frac{x^2}{2}\,\mu(dx).$$
\end{lemma}

\begin{proof}
Since the quadratic term dominates the logarithmic term we shall justify that we can find a constant $C_\beta$ such that for all $x,y\in\R^2$ 
\begin{equation}
\label{eq:rewrite}
-\beta\log(|x-y|)+\cfrac{x^2}{4}+\cfrac{y^2}{4}\ge C_\beta.
\end{equation}
Doing the same computations as in Lemma \ref{lemmaHfinite} yields $$-\beta\log(|x-y|)+\cfrac{x^2}{4}+\cfrac{y^2}{4}=\dis\cfrac{1}{4}\left((x^2+y^2)-2\beta\log(x^2+y^2)\right)+\beta\log\left(\dis\frac{\sqrt{x^2+y^2}}{|y-x|}\right).$$
Since we have that for all $t\in\R^+$, $t-2\beta\log(t)\ge 2\beta(1-\log(2\beta))$ by an easy study of function and for all $(x,y)\in\R^2$, $\dis\frac{\sqrt{x^2+y^2}}{|y-x|}\ge\dis\frac{1}{\sqrt{2}}$ we obtain \eqref{eq:rewrite}.
\newline
Let $\mu\in\mes$ such that $H_\beta(\mu)<+\infty$.
From \eqref{eq:rewrite} we have $$+\infty>H_\beta(\mu)\ge \frac{1}{2}\int_{\R^2}\left[C_\beta+\frac{x^2}{4}+\frac{y^2}{4}\right]\mu(dx)\mu(dy)=\cfrac{C_\beta}{2}+\int_\R \cfrac{x^2}{4}\,\mu(dx).$$
So we get that $\int_\R x^2\,\mu(dx)<+\infty$.
Moreover, again by \eqref{eq:rewrite} for all $(x,y)\in\R^2$ we have
$$C_\beta-\frac{x^2}{4}-\frac{y^2}{4} \le-\beta\log|x-y| \le f_\beta(x,y).$$
Since the two extremal terms of the inequalities are in $L^1(\mu\otimes\mu)$, this implies that $\int_{\R^2} |\log|x-y||\mu(dx)\mu(dy)<+\infty$.
\newline
For $\mu$ is a probability measure such that $H_\beta(\mu)<+\infty$ we get $$H_\beta(\mu)=\dis\cfrac{1}{2}\dis\int_{\R^2}f_\beta(x,y)\mu(dx)\mu(dy)=\dis\cfrac{1}{2}\left(\dis\int_\R \cfrac{x^2}{2}\,\mu(dx)+\dis\int_\R \cfrac{y^2}{2}\,\mu(dy)\right)-\cfrac{\beta}{2}\,\Sigma(\mu).$$
\end{proof}
\begin{Remark}
For $\mu\in\mes$, the free entropy of $\mu$ is defined as $$\Sigma(\mu)=\displaystyle\int_{\R^2}\log|x-y|\mu(dx)\mu(dy).$$
This notion of free entropy was introduced by Voiculescu in his work on free probability \cite{voiculescu2002free,voiculescu1999analogues}.
\end{Remark}

Let us state some properties on the function $H_\beta$.
\begin{prop}
\label{prop:optim}
The following properties are satisfied:
\begin{enumerate}
\item{$H_\beta$ is defined on $\mes$ and takes values in $\left[\dis\cfrac{\beta}{4}\,\left(1-\log{(2\beta)}\right),+\infty\right]$.}
\item{$H_\beta(\mu)$ is infinite if one of the conditions is satisfied:
\begin{enumerate}
\item{$\dis\int_\R x^2\mu(dx)=+\infty$.}
\item{There exists a set $A$ with positive mass for $\mu$ but null logarithmic capacity, which means $$\mu(A)>0,\,\,\, \emph{cap}(A):=\exp{\left(-\inf_{\nu\in\mathcal P(A)}\dis\int\log{(1/|x-y|)}\nu(dx)\nu(dy)\right)}=0.$$}
\end{enumerate}}
\item{If $\mu$ has an atom, then $H_\beta(\mu)=+\infty$. If $\mu$ has no atom, $$H_\beta(\mu)=\dis\int_{x<y}f_\beta(x,y)\mu(dx)\mu(dy).$$}
\item{For all $M\in\R$, $\{H_\beta\le M\}$ is a compact subset of $\mes$.}
\item{$H_\beta$ is convex on $\mes$.}
\item{$H_\beta$ is actually strict convex on $\mes$.}
\end{enumerate}
\end{prop}

\begin{proof}
In this proof let $\Delta=\{(x,x),\, x\in\R\}$ be the diagonal of $\R^2$.
\begin{enumerate}
\item{This point corresponds to Lemma \ref{lemmaHfinite}. To fix the notations for this proof, \eqref{eq:ineq1} yields
\begin{equation}
\label{ineq:f_beta}
\forall(x,y)\in\R^2,\, f_\beta(x,y)\ge\dis\frac{\beta}{2}\left(1-\log(2\beta\right))=:m_\beta.
\end{equation}
Integrating against $\mu(dx)\mu(dy)$ gives the result.}
\item{If the first condition holds, $H_\beta(\mu)=+\infty$ by Lemma \ref{lemma: rewrite}. 
\newline 
Then, we can notice using the inequality \eqref{ineq:f_beta} that for all Borelian subsets $B$ of $\R^2$ $$H_\beta(\mu)\ge\dis\int_B f_\beta(x,y)\mu(dx)\mu(dy)+(1-\mu\otimes\mu(B))m_\beta.$$
But if $B= A\times A$, we also have $$ \dis\int_B f_\beta(x,y)\mu(dx)\mu(dy)\ge-\beta\dis\int_{A\times A}\log(|x-y|)\mu(dx)\mu(dy)\ge-\beta\mu(A)^2\log(\text{cap}(A)).$$
So under the second condition we also have $H_\beta(\mu)=+\infty$.}
\item{If $\mu$ has an atom $a$, then $\mu(\{a\})>0$ and $\text{cap}(\{a\})=0$, then $H_\beta(\mu)=+\infty$. Assume that $\mu$ has no atom, then by Fubini-Tonelli $$\mu\times\mu(\Delta)=\dis\int_\R\int_\R \mathds{1}_{x=y}\mu(dx)\mu(dy)=\dis\int_\R \mu(\{y\}) \mu(dy)=0.$$
Since $f_\beta$ is symmetric we have the result.}
\item{Firstly, let us define for $M>0$,  $f_\beta^M(x,y)=\min(f_\beta(x,y),M)$ in order to truncate $f_\beta$ and to have a bounded function (since $f_\beta$ is bounded from below by $m_\beta$).
\newline
We naturally associate to $f_\beta^M$ the quantity:
\begin{equation}
\label{def H_m}
H_\beta^M(\mu):=\dis\cfrac{1}{2}\dis\int_{\R^2}f_\beta^M(x,y)\mu(dx)\mu(dy).
\end{equation}
Since $f_\beta^M$ is bounded, $\mu\mapsto H_\beta^M(\mu)$ is continuous on $\mes$. 
\newline
Moreover, as $(f_\beta^M)_{M>0}$ increases, converges pointwise to $f_\beta$ and for all $M>0$ $f_\beta^M$ is bounded from below by $m_\beta$ we can apply monotone convergence theorem to have that for every $\mu\in\mes$, $H_\beta^M(\mu)$ converges to $H_\beta(\mu)$.
\newline
Of course it does not imply that $H_\beta$ is continuous but at least $H_\beta$ is lower semi continuous. Indeed, since for all $\mu\in\mes$ $H_\beta^M(\mu)$ increases to $H_\beta(\mu)$, it gives that for all $L\in\R$ $\{H_\beta\le L\}=\bigcap_{M>0} \{H_\beta^M\le L\}$. Since for all $M>0$ $H_\beta^M$ is continuous on $\mes$, for all $L\in\R$, $\{H_\beta\le L\}$ is an intersection of closed subsets of $\mes$.
\newline
It remains to prove that for all $L\in\R$, $\{H_\beta\le L\}$ is a relatively compact subset of $\mes$. Fix $L\in\R$. Let $(\mu_n)_{n\in\N}\in\mes^\N$ be a family of probability measures such that for all for all $n\in\N$, $H_\beta(\mu_n)\le L$. For a constant $C>0$ fixed we can find a compact subset $K\subset \R$ (which depend on $C$) such that $$\min_{(K\times K)^c}\left(\cfrac{x^2}{2}+\cfrac{y^2}{2}-\beta\log|x-y|\right)\ge C.$$
Using Lemma \ref{lemma: rewrite} yields 
\begin{align*}
L\ge H_\beta(\mu_n)&\ge \cfrac{m_\beta}{2}\mu_n(K)^2+\cfrac{C}{2}\,\,(\mu_n\otimes\mu_n)((K\times K)^c)\\
&\ge -\cfrac{|m_\beta|}{2}+\cfrac{C}{2}\mu_n(K^c)
\end{align*}
Hence for all $\varepsilon>0$, we can find $K$ a compact subset of $\R$ such that for all $n\in\N$, $\mu_n(K^c)\le \varepsilon$ by choosing $C$ large enough. The Prokohorov theorem implies that for all $L\in\R$, $\{H_\beta\le L\}$ is a relatively compact subset of $\mes$. }

\item{In order to prove that $H_\beta$ is convex we only have to prove that $\Sigma$ is concave. Let state a lemma to have another expression of $\Sigma$.
\begin{lemma}
Let $a,b>0$ then $$\dis\int_{\R^+}\dis\cfrac{\exp(-at)-\exp(-bt)}{t}\,dt=\log(b/a).$$

\end{lemma}
\begin{proof}
Since the problem of this integral is essentially at 0 let us focus on this part. 
$$\dis\int_{\R^+}\dis\cfrac{\exp(-at)-\exp(-bt)}{t}\,dt=\limite\dis\int_{\varepsilon}^{+\infty}\dis\cfrac{\exp(-at)-\exp(-bt)}{t}.$$
As the two parts of the integral are converging we can cut this integral in two parts, we have: $$\dis\int_{\varepsilon}^{+\infty}\dis\cfrac{\exp(-at)-\exp(-bt)}{t}\,dt=\dis\int_{\varepsilon}^{+\infty}\dis\cfrac{\exp(-at)}{t}\,dt-\dis\int_{\varepsilon}^{+\infty}\dis\cfrac{\exp(-bt)}{t}\,dt.$$
Now, on both integral we do the change of variables $t'=at$ and $t'=bt$.
Hence $$\dis\int_{\varepsilon}^{+\infty}\dis\cfrac{\exp(-at)-\exp(-bt)}{t}\,dt=\int_{a\varepsilon}^{+\infty}\dis\cfrac{\exp(-t)}{t}\,dt-\dis\int_{b\varepsilon}^{+\infty}\dis\cfrac{\exp(-t)}{t}\,dt =\dis\int_{a\varepsilon}^{b\varepsilon}\dis\cfrac{\exp(-t)}{t}\,dt$$\newline
Now by continuity of $\exp$ at 0 we have: $$\limite \dis\int_{a\varepsilon}^{b\varepsilon}\dis\cfrac{\exp(-t)}{t}\,dt=\limite\dis\int_{a\varepsilon}^{b\varepsilon}\dis\cfrac{1}{t}\,dt=\log(b/a).$$
It gives the result.
\end{proof}
Thanks to this formula we can write: $$\forall (x,y)\in\R^2,\,\log|x-y|^2=\dis\int_{\R^+}\cfrac{\exp\left(-\frac{t}{2}\right)-\exp\left(-\frac{t|x-y|^2}{2}\right)}{t}\,dt$$
Then the change of variables $t'=1/t$, yields
\begin{equation}
\forall (x,y)\in\R^2,\,\log|x-y|=\dis\int_{\R^+}\cfrac{\exp\left(-\frac{1}{2t}\right)-\exp\left(-\frac{|x-y|^2}{2t}\right)}{2t}\,dt
\label{eq1}
\end{equation}
We shall use $\eqref{eq1}$ to have another expression of $\Sigma$ and view it as a limit of concave functions. 
\newline
Indeed for all $T>1$, $\mu\in\mes$, we define 
$$\Sigma_T(\mu)=\dis\int_{\R^2}\dis\int_{1/T}^{T}\cfrac{\exp\left(-\frac{1}{2t}\right)-\exp\left(-\frac{|x-y|^2}{2t}\right)}{2t}dt \mu(dx)\mu(dy).$$
We can cut this integral in two parts: one part with the $(x,y)$ such that $|x-y|\le 1$ and the complement of this part.
\newline
On these parts the sign of the function we integrate is constant and so the Fubini-Tonelli theorem implies that that for all $\mu\in \mes$ $$\underset{T\to+\infty}{\lim} \Sigma_T(\mu)=\Sigma(\mu).$$
Hence, we only have to prove that for all $T>1$, $\Sigma_T$ is concave. Fix $T>1$.
\newline
We can first isolate the dependence in $\mu$ of $\Sigma_T$: $$\Sigma_T(\mu)=\dis\int_{1/T}^T\cfrac{\exp\left(-\frac{1}{2t}\right)}{2t}\,dt -\dis\int_{1/T}^T\cfrac{1}{2t}\dis\int_{\R^2}\exp\left(-\cfrac{|x-y|^2}{2t}\right)\mu(dx)\mu(dy)dt.$$
So to prove that $\Sigma_T$ is concave we only have to prove that $$\mu\mapsto \dis\int_{1/T}^T\cfrac{1}{2t}\dis\int_{\R^2}\exp\left(-\cfrac{|x-y|^2}{2t}\right)\mu(dx)\mu(dy)dt$$ is convex on $\mes$. 
The key argument is to express $\Sigma_T$ using the Fourier transform of $\mu$.
\begin{lemma}
\label{lemma: sorte fourier}
Let $\nu$ be a finite signed measure. Then for all $t>0$, 
\begin{equation}
\begin{split}
\dis\int_{\R^2}\exp\left(-\cfrac{(x-y)^2}{2t}\right)&d\nu(x)d\nu(y)\\
&=\sqrt{\cfrac{t}{2\pi}}\dis\int_{\R}\left|\dis\int_{\R}\exp(i\lambda x)d\nu(x)\right|^2\exp(-t\lambda^2/2)d\lambda.
\end{split}
\label{eq2}
\end{equation}
\end{lemma}
\begin{proof}
Let us recall the Fourier transform of the Gaussian $$\forall z\in \C,\, \dis\int_\R\exp(z\lambda)\exp(-\lambda^2)d\lambda=\sqrt{\pi}\exp(z^2/4).$$
Using this formula with $z=\sqrt{2}i(x-y)/\sqrt{t}$, applying the Fubini theorem and doing an affine change of variables gives the result.
\end{proof}
With this lemma the convexity of $\Sigma_T$ easily follows since $$\mu\mapsto\left|\dis\int_{\R}\exp(i\lambda x)d\mu(x)\right|^2$$ is convex on $\mes$. 
\newline
Indeed if $\mu\in\mes$, we have: $$\left|\dis\int_{\R}\exp(i\lambda x)d\mu(x)\right|^2=\left(\dis\int_{\R}\cos(\lambda x)d\mu(x)\right)^2+\left(\dis\int_{\R}\sin(\lambda x)d\mu(x)\right)^2 $$ and the two terms are convex using the convexity of $t\mapsto t^2$.
\newline
It concludes about the concavity of $\Sigma$ and so the convexity of $H_\beta$.}
\item{About the strict convexity we shall use that $\Sigma$ is a quadratic form on $\mes$. Let state a general lemma.
\begin{lemma}
\label{lemma: quadratic form}
Let $E$ be a vector space, $\mathcal C$ be a convex subset of $E$ and $q$ a quadratic form defined on $E$ with polar form $\phi$. Assume that for all $(x,y)\in \mathcal C$, such that $x\ne y$ we have $q(x-y)<0$.
\newline 
Then $q$ is strictly concave on $\mathcal C$.
\end{lemma}
\begin{proof}
Let $x\ne y$ both in $\mathcal C$ an and $t\in (0,1)$, then we have $$q((1-t)x+ty)=(1-t)^2q(x)+t^2q(y)+2t(1-t)\phi(x,y).$$
Moreover, since $q(x-y)<0$ and $q(x-y)=q(x)+q(y)-2\phi(x,y)<0$ it gives $$q(x)+q(y)<2\phi(x,y).$$
Hence as $t\in(0,1)$, with the previous statements we have 
\begin{align*}
q((1-t)x+ty)&>(1-t)^2q(x)+t^2q(y)+t(1-t)(q(x)+q(y))\\
&=(1-t)q(x)+tq(y).
\end{align*}
\end{proof}
We now apply this lemma with $E=\{\mu \, \text{finite signed measure such that: }\, \Sigma(\mu)>-\infty\}$, $\mathcal C=\mes\cap E$ and $q(\mu)=\Sigma(\mu)$.
\newline
By definition, $\Sigma$ is a quadratic form on $E$ with $\phi(\mu,\nu)=\dis\int_{\R^2}\log |x-y|d\mu(x)d\nu(y)$ its polar form. 
\newline
Moreover if $(\mu,\nu)\in \mathcal C^2$ we can rewrite $\Sigma(\mu,\nu)$ using \eqref{eq1}: $$\Sigma(\mu-\nu)=-\dis\int_{\R^+}\cfrac{1}{2t}\dis\int_{\R^2}\exp\left(-\cfrac{|x-y|^2}{2t}\right)d(\mu-\nu)(x)d(\mu-\nu)(y)dt.$$
By $\eqref{eq2}$, for $(\mu,\nu)\in \mathcal C^2$ we have $$\Sigma(\mu-\nu)=\dis\int_{\R^+}\cfrac{1}{2\sqrt{2\pi t}}\dis\int_{\R}\left|\dis\int_{\R}\exp(i\lambda x)d(\mu-\nu)(x)\right|^2\exp(-t\lambda^2/2)d\lambda \,dt.$$
\newline
This quantity is non positive and is equal to $0$ if and only if for every $\lambda\in\R$ we have $\dis\int_{\R}\exp(i\lambda x)d(\mu-\nu)(x)=0$ (by continuity in $\lambda$) if and only if for every $\lambda\in\R$, $\dis\int_{\R}\exp(i\lambda x)d\mu(x)=\dis\int_{\R}\exp(i\lambda x)d\nu(x)$. 
\newline
Since the characteristic function characterizes a probability measure, $\Sigma(\mu-\nu)\ge 0$ with equality if and only if $\mu=\nu$.
\newline
Hence, the assumptions of Lemma \ref{lemma: quadratic form} are checked. $\Sigma$ is strictly concave of $\mathcal C$. So $H_\beta$ is strictly convex on $\mathcal C$.
}

\end{enumerate}
\end{proof}
 
\begin{Remark}
\label{remark : fourier repr}
Let us explain the deep link between the formula of Lemma \ref{lemma: sorte fourier} and Fourier representation. This link can be used for more general Coulomb and Riesz gases as we will explain in the Section \ref{subsection: extenson} \cite{serfaty2024lectures}. Since the Fourier transform of $x\to -\log|x|$ is $\xi\to 1/|\xi|$, if $\mu=\mu_1-\mu_2$ where $\mu_1$ and $\mu_2$ are probability measures on $\R$ we get \begin{equation}
\label{formula fourier}
\int_{\R^2}-\log|x-y|\mu(dx)\mu(dy)=\int_\R\cfrac{|\mathcal F(\mu)(\xi)|^2}{|\xi|}d\xi,
\end{equation} where $\mathcal F(\mu)$ is the Fourier transform of $\mu$ (this formula can be obtained by first assuming that $\mu$ has a density in the Schwartz space and then using a density argument). Thanks to the Fourier representation \eqref{formula fourier}, if $\mu$ is the difference of two probabilities measures $$-\Sigma(\mu)=||\mu||_{\dot{H}^{-1/2}},$$ where $||.||_{\dot{H}^{-1/2}}$ is the usual Sobolev semi norm on $H^{-1/2}$. This can be directly used to prove Proposition \ref{prop:optim} thanks to Lemma \ref{lemma: quadratic form}.
\end{Remark}
The notion of logarithmic capacity is linked with the notion of capacity in electrostatic and play an important role to have a condition that ensures $\inf_{\mu\in\mes}H_\beta(\mu)<+\infty$ (see for instance \cite{landkof1972foundations,saff2013logarithmic} for a more detailed introduction to the notion of capacity). 
The only basic properties about the logarithmic capacity that shall be used is that a set of zero logarithmic capacity has also zero Lebesgue measure and for a borelian $E$, $\text{cap}(E)=\sup_{K\subset E, \,K \text{ is compact} }\text{cap}(K)$. 
\newline
\tab Let state a lemma to prove that $\inf_{\mu\in\mes}H_\beta(\mu)$ is finite. 
\begin{lemma}
\label{lemma: finite}
The quantity $\inf_{\mu\in\mes}H_\beta(\mu)$ is finite.
\end{lemma} 

\begin{proof}
We follow the arguments of Lemma 2.4 of \cite{serfaty2024lectures}. Let $\varepsilon>0$ and let us define $\mathcal V_\varepsilon:=\{x\,|, V(x)\le \varepsilon^{-1}\}$ where $V(x)=x^2/2$ is the external potential that appears in the potential $H_\beta$. $\mathcal V_\varepsilon$ is a compact subset of $\R$ with positive Lebesgue measure and so has a positive logarithmic capacity. By definition of the logarithmic capacity, we can find a measure $\nu_\varepsilon$ with a support included in $\mathcal V_\varepsilon$ such that $$-\int\log(|x-y|)\nu_\varepsilon(dx)\nu_\varepsilon(dy)<+\infty.$$
By definition of $\mathcal V_\varepsilon$, $\int Vd\nu_\varepsilon<+\infty$. So $H_\beta(\nu_\varepsilon)<+\infty$.
\end{proof}
\begin{Remark}
One can also directly prove by computations that $H_\beta(\mu)$ is finite if $\mu$ is compactly supported with a bounded density with respect to the Lebesgue measure or for instance if $\mu(dx)=\exp(-x^2/2)/\sqrt{2\pi}\,dx$.
\end{Remark}

Thanks to this lemma we can now define a good rate function. 
\begin{definition}
Let $I_\beta :\mes\to[0,+\infty]$ the function defined by: $$I_\beta(\mu)=H_\beta(\mu)-\inf_{\mu\in\mes}H_\beta(\mu).$$
\end{definition}
As a consequence of Proposition \ref{prop:optim} we get the following result.
\begin{theorem}
$I_\beta$ is a good rate function. 
\end{theorem}

\subsection{Equilibrium measure.}

We first introduce the notion of equilibrium measure. 
\begin{prop}
\label{prop:equili measure}
$I_\beta$ admits a minimum on $\mes$ and this minimum is reached in a unique element $\mu_\beta \in\mes$ which is called the equilibrium measure of $I_\beta$. 
\end{prop}

\begin{proof}
By Lemma \ref{lemma: finite}, we can consider $\nu\in\mes$ such that $I_\beta(\nu)<+\infty$. Since $\{\mu\in\mes\,|\, I_\beta(\mu)\le I_\beta(\nu)\}$ is compact and $I_\beta$ is lsc by Proposition \ref{prop:optim}, $I_\beta$ reaches its infimum by Lemma \ref{lemma: minim lsc}.
Moreover, by strict convexity of $I_\beta$, this minimum is reached in a unique element.
\end{proof}

We shall use a general result of Frostman that characterizes the unique minimizer of $I_\beta$ (called the Frostman equilibrium measure) as solution of an Euler-Lagrange equation. The following theorem is general as we will mention in Section \ref{subsection: extenson} but we only give a proof for $I_\beta$ (but the proof would be similar for Coulomb and Riesz gases, see for instance \cite{serfaty2024lectures}).
Intuitively the equilibrium measure corresponds to the distribution of the particles that minimizes the energy of the system. Here $V(x)=x^2/2$ is the external potential and satisfies $\lim_{|x|\to+\infty} (V(x)-\log|x|)=+\infty$. This formally forces at equilibrium the particles to be not to far from the origin because if not the energy of the system would be to high. This explains why the support of the equilibrium measure shall be compact as we will see.

\begin{theorem}[Frostman]
\label{thm: frostman}
$\,$
\begin{enumerate}
\item{There exists a unique minimizer of $I_\beta$: $\mu_\beta\in\mes$ which satisfies $$I_\beta(\mu_\beta)=0.$$}
\item{The support $\mathcal S_\beta$ of $\mu_\beta$ is compact and with positive logarithmic capacity. }
\item{We can find $A$ with null logarithmic capacity such that for for all $x\notin A$, we have:\begin{equation}
\label{thm: frost 1}
\beta\dis\int_\R\log|x-y|d\mu_\beta(y)\le\cfrac{1}{2}\,x^2-2\inf H_\beta+\cfrac{1}{2}\dis\int_\R y^2d\mu_\beta(y).
\end{equation}}
\item{For all $x\in\mathcal S_\beta$ we have: \begin{equation}
\label{thm: frost 2}
\beta\dis\int_\R\log|x-y|d\mu_\beta(y)=\cfrac{1}{2}\,x^2-2\inf H_\beta+\cfrac{1}{2}\dis\int_\R y^2d\mu_\beta(y).
\end{equation}}
\end{enumerate}
\end{theorem}

\begin{proof}$\,$
\begin{enumerate}
\item{This is Proposition \ref{prop:equili measure}.}
\item{Let $K$ be a compact such that $f_\beta(x,y)\ge 2(H_\beta(\mu_\beta)+1)$ for $(x,y)\notin K\times K$ and $\mu_\beta(K)>0$.
\newline
Assume by contradiction that $\mu_\beta(K)<1$ and let us define the conditional probability $$\mu_{\beta,K}:=\cfrac{\mu_\beta|_K}{\mu_\beta(K)}.$$ The goal is to show that $H_\beta(\mu_{\beta,K})<H_\beta(\mu_\beta)$ to have a contradiction.
We compute $H_\beta(\mu_\beta)$ by distinguishing $(x,y)$ that are in $K\times K$ and not in: 
\begin{equation}
\label{eq: compact support}
\begin{split}
H_\beta(\mu_\beta)&=\cfrac{1}{2}\int_{\R\times \R}f_\beta(x,y) \mu_\beta(dx)\mu_\beta(dy)\\
&=\cfrac{1}{2}\int_{K\times K}f_\beta(x,y) \mu_\beta(dx)\mu_\beta(dy)+\cfrac{1}{2}\int_{(K\times K)^c}f_\beta(x,y) \mu_\beta(dx)\mu_\beta(dy)\\
&=\mu_\beta(K)^2 H_\beta(\mu_{\beta,K})+ (H_\beta(\mu_\beta)+1)\mu_\beta\otimes\mu_\beta((K\times K)^c)\\
&\ge\mu_\beta(K)^2 H_\beta(\mu_{\beta,K})+ (H_\beta(\mu_\beta)+1)(1-\mu_\beta(K)^2)
\end{split}
\end{equation} 
It yields that $$H_\beta(\mu_{\beta,K})\le \cfrac{H_\beta(\mu_{\beta})}{\mu_\beta(K)^2}+\cfrac{\mu_\beta(K)^2-1}{\mu_\beta(K)^2}\left(H_\beta(\mu_\beta)+1\right)<H_\beta(\mu_\beta), $$ since $\mu_\beta(K)<1$. This is a contradiction with the minimality of $\mu_\beta$. Hence, the support of $\mu_\beta$ is compact. Moreover by Proposition \ref{prop:optim}, since $H_\beta(\mu_\beta)$ is finite by Lemma \ref{lemma: finite}, the logarithmic capacity of the support of $\mu$ is positive.}
\item{Now we shall prove the two last points. We use the standard method of variations to obtain the Euler-Lagrange equation satisfies by the minimizer of $H_\beta$. 
\newline
Let $\mu\in \mes$ such that $H_\beta(\mu)<+\infty$. Since $\mu_\beta$ minimizes $H_\beta$, for $t\in[0,1]$, we have: $$H_\beta((1-t)\mu_\beta+t\mu)\ge H_\beta(\mu_\beta).$$
Using the expression of $H_\beta$, we get $$H_\beta(\mu_\beta)+t\left[\int f_\beta(x,y)\mu_\beta(dx)(\mu-\mu_\beta)(dy)\right]+O(t^2)\ge H_\beta(\mu_\beta).$$
Simplifying the previous inequality, dividing by $t>0$ and letting $t\to 0$ gives: 
\begin{equation}
\label{eq: euler lag 1}
\begin{split}
\int f_\beta(x,y)\mu_\beta(dx)(\mu-\mu_\beta)(dy)&\ge 0,\\
\int f_\beta(x,y)\mu_\beta(dx)\mu(dy)&\ge \int f_\beta(x,y)\mu_\beta(dx)\mu_\beta(dy)=2H_\beta(\mu_\beta)=2\inf H_\beta
\end{split}
\end{equation}
Set $$h_\beta(x):=-\beta\int\log(|x-y|)\mu_\beta(dy)+\cfrac{x^2}{2}$$ and use the expression of $f_\beta$ to obtain: 
\begin{equation}
\label{eq: euler lag 2}
\begin{split}
\int h_\beta(x)\mu(dx)\ge 2\inf H_\beta-\cfrac{1}{2}\int x^2\mu_\beta(dx)=2H_\beta(\mu_\beta)-\cfrac{1}{2}\int x^2\mu_\beta(dx)=:c_\beta
\end{split}
\end{equation}
Hence for all $\mu$ such that $H_\beta(\mu)<+\infty$ $$\int h_\beta(x)\mu(dx)\ge c_\beta.$$
Let us notice that \eqref{thm: frost 1} is equivalent to $h_\beta(x)\ge c_\beta$.
\newline
Formally if we could chose dirac measures $\mu(dx)=\delta_y(dx)$ we would have \eqref{thm: frost 1} for every $y\in\R$. However we cannot use such measures since $H_\beta(\delta_y(dx))=+\infty$.
We prove \eqref{thm: frost 1} by contradiction. Assume that there exists a set $E$ of positive capacity such that $\eqref{thm: frost 1}$ is false. By properties of the capacity of a set, we can suppose that $E$ is in fact compact. We can find a probability measure $\nu$ supported in $E$ such that $$-\int\log(|x-y|)\nu(dx)\nu(dy)<+\infty.$$ 
Since $\nu$ is supported on a compact, it implies that $H_\beta(\nu)<+\infty$ and \eqref{eq: euler lag 2} gives $\int h_\beta(x)\nu(dx)\ge c_\beta$ which is in contradiction with the facts that $\eqref{thm: frost 1}$ does not hold on $E$. This proves the third point of Theorem \ref{thm: frostman}.  
\newline
Now we prove \eqref{thm: frost 2}. Since $H_\beta(\mu_\beta)<+\infty$, $\mu_\beta$ does not charge sets of null logarithmic capacity. Hence, we deduce that on $\mathcal S_\beta$ the inequality \eqref{thm: frost 1} is true: 
\begin{equation}
\label{eq: frostm egal}
h_\beta(x)\ge c_\beta, \, \mu_\beta \text{-a.e}
\end{equation}
Let us notice that by definition of $h_\beta$ and $c_\beta$, $\int h_\beta \mu_\beta(dx)=c_\beta$. The inequality in \eqref{eq: frostm egal} is actually an equality on $\mathcal S_\beta$.
}

\end{enumerate}
\end{proof}

This theorem gives us the following important characterization of the minimizer of $I_\beta$.
\begin{prop}
\label{prop:euler lagrange}
Let $\mu\in\mes$. Then $\mu$ minimizes $I_\beta$ if and only if for $\mu$-almost all $x$ we have: 
\begin{equation}
\label{eq:euler lagrange}
\beta\dis\int_\R\log|x-y|d\mu(y)=\cfrac{1}{2}\,x^2-2\inf H_\beta+\cfrac{1}{2}\dis\int_\R y^2d\mu(y).
\end{equation}
\end{prop}

\begin{proof}
The direct sense is a consequence of Theorem \ref{thm: frostman}.
\newline
For the reverse, integrating \eqref{eq:euler lagrange} with respect to $\mu$ yields: $$\beta \Sigma(\mu)=\dis\int_\R \,x^2d\mu(x)-2\inf H_\beta.$$
Hence by definition of $H_\beta$, we have: $H_\beta(\mu)=\inf H_\beta$.

\end{proof}

We will use this equation to compute explicitly the equilibrium measure of $I_\beta$ in Section \ref{subsection: cons}. Formally, if we differentiate \eqref{thm: frost 2} on $\mathcal S_\beta$, we obtain that for all $x\in\mathcal S_\beta$ $$\beta\int_\R\cfrac{1}{x-y}\,d\mu_\beta(y)=x.$$
This is exactly the equation of a stationary measure  obtained in \eqref{equationhilberttransfoorns} when studying the partial differential equation \eqref{FormulaOrn} of the Ornstein-Uhlenbeck process in Section \ref{section:Orn}. We justified why the stationary measure of a such process should be the semicircle distribution. In Section \ref{subsection: cons} we shall give another proof using the large deviations principle.

\section{Large deviations}
\label{section: large dev}
In this section we present the large deviations of the family of measures $(Q_\beta^N)_{N\in\N}$ for all $\beta>0$.
We fix $\beta>0$.
\subsection{Large deviations principle of the normalizing factor}
\label{subsetionlargedevnorm}
In Statistical mechanics, the study of the normalizing constant of the system is a first important step as it contains many information of the system and its called the partition function of the system. 
\newline
\tab We first study the behaviour when $N$ goes large of the normalizing factor.
\begin{prop}
\label{prop: large dev constant}
We have: $$ \underset{N\to\infty}{\lim} \cfrac{1}{N^2}\log Z_\beta^N=\cfrac{\beta}{4} \log\left(\cfrac{\beta}{2}\right)-\cfrac{3}{8}=:F(\beta).$$
\end{prop}

\begin{proof}
To do this, let us state an important result of Selberg \cite{forrester2008importance} about the value of $Z_\beta^N$. This result is very important to compute explicitly the normalizing constant of the density of eigenvalues of random matrices. A proof can be found in \cite{Alicelivre} in Section 2.5.3. Let us mention that we computed $Z_\beta^N$ in the case $\beta=2$ which corresponds to the complex Wigner case in Theorem \ref{thm law eigenvalues guen}.
\begin{lemma}[Admitted]
\label{lemma: selberg}
For all $\beta>0$, and $N>0$, we have the following equality: 
$$Z_\beta^N=(2\pi)^{N/2}N^{-N/2-\beta N(N-1)/4}\left[\Gamma(1+\beta/2)\right]^{-N}\dis\prod_{j=1}^N\Gamma(1+\beta \,j/2),$$ with $\Gamma$ the Euler function.
\end{lemma}
Using Lemma \ref{lemma: selberg} we have: $$\cfrac{1}{N^2}\log (Z_\beta^N)\underset{N\to\infty}{=}o(1)-\cfrac{\beta}{4}\log (N)+o(1)+\cfrac{1}{N^2}\dis\sum_{j=1}^N\log\left(\Gamma\left(1+\beta\, j/2\right)\right).$$
Moreover, Stirling's formula yields: 
$$\log(\Gamma(x))\underset{x\to\infty}{=}x\log(x/e)+o(x).$$
\newline
Hence, we obtain: $\dis\sum_{j=1}^N\log\left(\Gamma\left(1+\beta\, j/2\right)\right)\underset{N\to\infty}{=}\dis\sum_{j=1}^N\cfrac{\beta j}{2}\log{\cfrac{\beta j}{2e}}+o(N^2)$.
\newline
Now let us notice that: $$\cfrac{1}{N^2}\dis\sum_{j=1}^N\cfrac{\beta j}{2}\log{\cfrac{\beta j}{2e}}-\cfrac{\beta}{4}\log N\underset{N\to\infty}{=}\cfrac{\beta}{2N}\dis\sum_{j=1}^N\cfrac{ j}{N}\log{\cfrac{\beta j}{2eN}}+o(1).$$
But the last sum is a Riemann sum. It gives:
\newline
 $$\cfrac{\beta}{2N}\dis\sum_{j=1}^N\cfrac{ j}{N}\log{\cfrac{\beta j}{2eN}}\underset{N\to\infty}{=}\cfrac{\beta}{2}\dis\int_0^1x\log{\cfrac{\beta x}{2e}}+o(1)=F(\beta)+o(1).$$
\newline
Hence we have: $$\cfrac{1}{N^2}\dis\sum_{j=1}^N\cfrac{\beta j}{2}\log{\cfrac{\beta j}{2e}}\underset{N\to\infty}{=}\cfrac{\beta}{4}\log N+F(\beta)+o(1).$$
This gives $$\cfrac{1}{N^2}\log (Z_\beta^N)\underset{N\to\infty}{=}F(\beta)+o(1).$$
\end{proof}

\begin{Remark}
\label{Rem:largedevconstant}
Let us mention that in the general case of Coulomb or Riesz gases if for instance we consider another external potential than $V(x)=x^2/2$ we cannot compute the exact value of the normalizing factor. However, we are still able to prove large deviations as we shall see in next sections and the large deviations of the normalizing factor. Indeed, similar computations as we shall do to prove the upper bound \eqref{upboun} and the lower bound \eqref{lowboun} can be done to show that the normalizing factor satisfies a large deviation principles. We refer to Theorem 3.3 of \cite{serfaty2024lectures} or Theorem 2.6.1 \cite{Alicelivre}. We choose to not present this approach seems the computations should have been very similar with what we shall do and because in the particular case of $\beta$ ensembles of this section we think it is worth mentioning the Selberg integral formula (which was computed in previous sections for the case $\beta=2$). We shall mention in Remark \ref{Remark:largdevconvlow} and Remark \ref{Rem:largedevconstantup} how we can obtain a such large deviation principles for the normalizing factor without computing the exact value of $Z_\beta^N$.
\end{Remark}

\subsection{Weak large deviations principle}
\label{subsecion: weak large dev}
The main goal of this section is to prove a so-called weak large deviations principle, that means a deviations principle for compact and open subsets of $\mes$. 
\newline
Let $\widehat {I_\beta}=H_\beta+F(\beta)$.
\begin{theorem}
\label{thm:weak large}
$(Q_\beta^N)_{N\in\N}$ satisfies a weak deviations principle with good rate function $\widehat {I_\beta}$. More exactly we have:
\newline
\begin{enumerate}
\item{For any open subsets $O$ of $\mes$, 
\begin{equation}
\label{eq:openlargedev}
\underset{N\to\infty}{\underline{\lim}}\cfrac{1}{N^2}\log Q_\beta^N (\overline{\mu_N}\in O)\ge-\inf_O \widehat {I_\beta}.
\end{equation}}
\item{For any compact subsets $K$ of $\mes$, \begin{equation}
\label{eq:compactlargedev}\underset{N\to\infty}{\overline{\lim}}\cfrac{1}{N^2}\log Q_\beta^N (\overline{\mu_N}\in K)\le-\inf_K \widehat {I_\beta}.
\end{equation}
}
\end{enumerate}
\end{theorem}
We first state an elementary lemma that plays a key role in large deviations theory as we will see.  

\begin{lemma}
\label{lemma: limsup sum}
Let $(u_{j,N})_{j,N\in\N^2}$ be a family of non negative real numbers and $k\in\N_{\ge 1}$, then $$\underset{N\to\infty}{\overline{\lim}}\cfrac{1}{N^2}\log\left(\sum_{j=1}^k u_{j,N}\right)\le \underset{1\le j\le k}{\max}\underset{N\to\infty}{\overline{\lim}}\cfrac{1}{N^2}\log(u_{j,N}).$$
\end{lemma}

\begin{proof}
The proof is just a use of the definition of the superior limit. Fix $\varepsilon>0$. By definition of the superior limit, for all $j$, we can find $N(j)\in\N$ such that for all $N\ge N(j)$, $$\cfrac{1}{N^2}\log(u_{j,N})\le \underset{N\to\infty}{\overline{\lim}}\cfrac{1}{N^2}\log(u_{j,N})+\varepsilon.$$
Let $\widetilde N=\underset{1\le j\le k} {\max}N(j)$. For all $1\le j\le k$, for all $N\ge \widetilde N$, we have $$\cfrac{1}{N^2}\log(u_{j,N})\le \underset{N\to\infty}{\overline{\lim}}\cfrac{1}{N^2}\log(u_{j,N})+\varepsilon,$$
which implies that for all $1\le j\le k$, for all $N\ge \widetilde N$,  $$u_{j,N}\le \exp\left(N^2\left(\underset{N\to\infty}{\overline{\lim}}\cfrac{1}{N^2}\log(u_{j,N})+\varepsilon\right)\right)\le\exp\left(N^2\left(\underset{1\le j\le k}{\max}\underset{N\to\infty}{\overline{\lim}}\cfrac{1}{N^2}\log(u_{j,N})+\varepsilon\right)\right).$$
Summing these inequalities and taking the logarithm yields $$ \cfrac{1}{N^2}\log\left(\sum_{j=1}^k u_{j,N}\right)\le \cfrac{\log(k)}{N^2}+\underset{1\le j\le k}{\max}\underset{N\to\infty}{\overline{\lim}}\cfrac{1}{N^2}\log(u_{j,N})+\varepsilon,$$
for all $N\ge \widetilde N$. Taking the superior limit and letting $\varepsilon$ goes to 0 gives the result.
\end{proof}

Thanks to this lemma we can reduce the weak deviations principle to a "local" large deviations principle. 

\begin{lemma}
\label{lemma:largedevponct}
If for all $\mu\in\mes$ we have: 
\begin{equation}
\label{weak large dev 1}
-\widehat {I_\beta}(\mu)=\lim_{\delta\to 0}\underset{N\to\infty}{\underline{\lim}}\cfrac{1}{N^2}\log Q_\beta^N (\overline{\mu_N}\in B(\mu,\delta))
\end{equation} and 
\begin{equation}
\label{weak large dev 2}
-\widehat {I_\beta}(\mu)=\lim_{\delta\to 0}\underset{N\to\infty}{\overline{\lim}}\cfrac{1}{N^2}\log Q_\beta^N (\overline{\mu_N}\in B(\mu,\delta)),\end{equation}
then $(Q_\beta^N)_{N\in\N}$ satisfies a weak deviations principle with rate function $\widehat {I_\beta}$.
\end{lemma}
\begin{proof}
Let us prove that the hypothesis \eqref{weak large dev 2} implies the weak deviations principle for compact subsets. 
\newline
Let $K$ be a compact of $\mes$. Fix $\eta>0$. By hypothesis \eqref{weak large dev 2}, for all $\mu\in K$ we can find $\delta_\mu>0$ such that: $\underset{N\to\infty}{\overline{\lim}}\cfrac{1}{N^2}\log \dis\sum_{i=1}^{l_\delta}Q_\beta^N (\overline{\mu_N}\in B(\mu,\delta_\mu))\le-I_\beta(\mu)+\eta\le -\inf_K I_\beta+\eta$. Since $K$ is compact, we can find a finite number $l_\eta$ of $\mu\in K$ such that $K\subset \dis\bigcup_{i=1}^{l_\eta}B(\mu_i,\delta_{\mu_i})$. By Lemma \ref{lemma: limsup sum} we have: 
\begin{align}
\underset{N\to\infty}{\overline{\lim}}\cfrac{1}{N^2}\log Q_\beta^N (\overline{\mu_N}\in K)&\le \underset{N\to\infty}{\overline{\lim}}\cfrac{1}{N^2}\log \dis\sum_{i=1}^{l_\eta}Q_\beta^N (\overline{\mu_N}\in B(\mu_i,\delta_{\mu_i})) \\
&\le \max_{i\le l_\eta}\left(\underset{N\to\infty}{\overline{\lim}}\cfrac{1}{N^2}\log Q_\beta^N (\overline{\mu_N}\in B(\mu_i,\delta_{\mu_i}))\right)
\end{align}
Hence, thanks to the choice $\delta_\mu$ we have: $$\underset{N\to\infty}{\overline{\lim}}\cfrac{1}{N^2}\log Q_\beta^N (\overline{\mu_N}\in K)\le -\inf_K \widehat {I_\beta}+\eta.$$
Let $\eta$ converges to 0 to conclude. 
\newline
\newline
Let us prove that the hypothesis \eqref{weak large dev 1} implies the weak deviations principle for open subsets of $\mes$.
\newline
Let $O$ be an open subset of $\mes$. Fix $\varepsilon>0$, we take $\mu\in O$ such that $\widehat {I_\beta}(\mu)\le \inf_O \widehat {I_\beta}+\varepsilon$. Hence for all $\delta>0$ small enough we have $B(\mu,\delta)\subset O$ and so $$\underset{N\to\infty}{\underline{\lim}}\cfrac{1}{N^2}\log Q_\beta^N (\overline{\mu_N}\in O)\ge\underset{N\to\infty}{\underline{\lim}}\cfrac{1}{N^2}\log Q_\beta^N (\overline{\mu_N}\in B(\mu,\delta)).$$
Letting $\delta$ converges to 0 and using the hypothesis \eqref{weak large dev 2} yields: $$\underset{N\to\infty}{\underline{\lim}}\cfrac{1}{N^2}\log Q_\beta^N (\overline{\mu_N}\in O)\ge-\widehat {I_\beta}(\mu)\ge-\inf_O \widehat {I_\beta} -\varepsilon.$$
Let $\varepsilon$ converges to 0 to conclude. 
\end{proof}
\begin{Remark}
\label{rem: density large dev}
As we shall explain after in the proof of the lower bound, we only have to prove the lower bound \eqref{lowboun} for "regular" measures. See for instance Corollary 2.7 and Theorem 1.1 of \cite{chafai2014first}.
\end{Remark}

Thanks to the large deviations principle for the normalizing constant stated in Proposition \ref{prop: large dev constant}, to prove Theorem \ref{thm:weak large} it remains to show that for every $\mu\in\mes$:
\begin{equation}
\label{lowboun}
-H_\beta(\mu)\le\lim_{\delta\to 0}\underset{N\to\infty}{\underline{\lim}}\cfrac{1}{N^2}\log\overline{Q^N_\beta}(\overline{\mu_N}\in B(\mu,\delta)) 
\end{equation}
\begin{equation}
\label{upboun}
\lim_{\delta\to 0}\underset{N\to\infty}{\overline{\lim}}\cfrac{1}{N^2}\log \overline{Q^N_\beta}(\overline{\mu_N}\in B(\mu,\delta))\le-H_\beta(\mu),
\end{equation}
where we recall that $\overline{Q^N_\beta}$ was defined in the Definition \ref{def:measure beta} as the non normalize version of $Q^N_\beta$.

\subsubsection{Upper bound \eqref{upboun}}
By definition we have:
$$\overline{Q^N_\beta}(\overline{\mu_N}\in B(\mu,\delta))=\dis\int_{\overline{\mu_N}\in B(\mu,\delta)} \dis\prod_{1\le i<j\le N}|\lambda_i-\lambda_j|^{\beta}\exp{\left(-\dis\frac{1}{2}N\dis\sum_{i=1}^N\lambda_i^2\right)}\dis\prod_{i=1}^Nd\lambda_i.$$
We want to make appear $H_\beta(\mu)$ in the previous expression. We recall that $$H_\beta(\mu)=\dis\cfrac{1}{2}\dis\int_{\R^2}f_\beta(x,y)\mu(dx)\mu(dy).$$
Hence, we write 
\begin{equation}
\label{equation: formule up bound}
\begin{split}
&\overline{Q^N_\beta}(\overline{\mu_N}\in B(\mu,\delta))=\dis\int_{\overline{\mu_N}\in B(\mu,\delta)} \dis\prod_{\lambda_i<\lambda_j}|\lambda_i-\lambda_j|^{\beta}\exp{\left(-\dis\frac{1}{2}\dis\sum_{\lambda_i<\lambda_j}(\lambda_i^2+\lambda_j^2)-\dis\frac{1}{2}\dis\sum_{i=1}^N\lambda_i^2\right)}\dis\prod_{i=1}^Nd\lambda_i \\
&=\dis\int_{\overline{\mu_N}\in B(\mu,\delta)}\exp{\left(-N^2\dis\int_{x<y}f_\beta(x,y)\overline{\mu_N}(dx)\overline{\mu_N}(dy)\right)}\exp{\left(-\dis\frac{1}{2}\dis\sum_{i=1}^N\lambda_i^2\right)}\dis\prod_{i=1}^Nd\lambda_i.
\end{split}
\end{equation}
Let $(\lambda_1^N,...,\lambda_N^N)$ be different real numbers then: $$\dis\int_{\Delta} \overline{\mu_N}(dx)\overline{\mu_N}(dy)=\dis\int_{\R}\dis\int_\R\mathds{1}_{x=y} \overline{\mu_N}(dx)\overline{\mu_N}(dy)=\cfrac{1}{N}\dis\int_{\R}|\{i;\, \lambda_i^N=y\}| \overline{\mu_N}(dy).$$
Since the $(\lambda_1^N,...,\lambda_N^N)$ are different real numbers, for all $y$ the cardinal is either 1 if $y$ is one of these numbers or 0 otherwise. Hence: $$\overline{\mu_N}\otimes\overline{\mu_N}(\Delta)=\dis\int_{\Delta} \overline{\mu_N}(dx)\overline{\mu_N}(dy)=\cfrac{1}{N}.$$
The technical point is now that since $H_\beta$ is not continuous we must pay attention to the passage to the limit. So the idea is to look at $H_\beta^M$ for a certain $M$ which is continuous and which converges pointwise to $H_\beta$ which was defined during the proof of Proposition \ref{prop:optim} by the formula \eqref{def H_m}.
Let us notice that:  $$H_\beta^M(\overline{\mu_N})=\dis\cfrac{1}{2}\dis\int_{\R^2}f_\beta^M(x,y)\overline{\mu_N}(dx)\overline{\mu_N}(dy)=\dis\int_{x<y}f_\beta^M(x,y)\overline{\mu_N}(dx)\overline{\mu_N}(dy)+\cfrac{1}{2}\dis\int_{\Delta}f_\beta^M(x,y)\overline{\mu_N}(dx)\overline{\mu_N}(dy).$$
So if $\overline{\mu_N}\in B(\mu,\delta)$ then: \begin{equation}
\label{eq: up bound bound}
\inf_{\nu\in B(\mu,\delta)}H_\beta^M(\nu)\le\dis\int_{x<y}f_\beta(x,y)\overline{\mu_N}(dx)\overline{\mu_N}(dy)+\cfrac{M}{2N}.
\end{equation}
We can now bound \eqref{equation: formule up bound} using \eqref{eq: up bound bound}. It gives: $$\overline{Q^N_\beta}(\overline{\mu_N}\in B(\mu,\delta))\le\exp{\left(-N^2\inf_{\nu\in B(\mu,\delta)}H_\beta^M(\nu)+\cfrac{MN}{2}\right)}\dis\int_{\R^N}\exp{\left(-\dis\frac{1}{2}\dis\sum_{i=1}^N\lambda_i^2\right)}\dis\prod_{i=1}^Nd\lambda_i.$$
Hence we have the following upper bound: $$\overline{Q^N_\beta}(\overline{\mu_N}\in B(\mu,\delta))\le\exp{\left(-N^2\inf_{\nu\in B(\mu,\delta)}H_\beta^M(\nu)+\cfrac{MN}{2}\right)}(2\pi)^{N/2}.$$
So, we get that for any $M>0$, $$\underset{N\to\infty}{\overline{\lim}}\cfrac{1}{N^2}\log \overline{Q^N_\beta}(\overline{\mu_N}\in B(\mu,\delta))\le-\inf_{\nu\in B(\mu,\delta)}H_\beta^M(\nu).$$
Since, for any $M>0$, $H_\beta^M$ is continuous: $$\lim_{\delta\to 0}\underset{N\to\infty}{\overline{\lim}}\cfrac{1}{N^2}\log \overline{Q^N_\beta}(\overline{\mu_N}\in B(\mu,\delta))\le-H_\beta^M(\mu).$$
Let $M$ goes to $+\infty$: $$\lim_{\delta\to 0}\underset{N\to\infty}{\overline{\lim}}\cfrac{1}{N^2}\log \overline{Q^N_\beta}(\overline{\mu_N}\in B(\mu,\delta))\le-H_\beta(\mu).$$

\subsubsection{Lower bound \eqref{lowboun}}
We will only prove the result for a family of "regular" measures as mentioned in Remark \ref{rem: density large dev}. 
We consider measures $\mu\in \mes$ such that $\mu$ is absolutely continuous with respect to Lebesgue measure (we still write $\mu$ its density), such that $\int_\R x^2d\mu$, $\int_\R \log(\mu)d\mu$ are integrable and for which there exists $C\in\R$ such that for all $x\in\R$
\begin{equation}
\label{eq: hypothesis density}
\cfrac{x^2}{2}+\log(\mu(x))\ge C.
\end{equation}
In this section we fix a such measure $\mu\in \mes$ and a constant $C$.
We start as for the proof of the upper bound: 
\begin{equation}
\label{equation: formule lower bound}
\begin{split}
&\overline{Q^N_\beta}(\overline{\mu_N}\in B(\mu,\delta))=\dis\int_{\{\overline{\mu_N}\in B(\mu,\delta)\}} \dis\prod_{1\le i<j\le N}|\lambda_i-\lambda_j|^{\beta}\exp{\left(-\dis\frac{1}{2}N\dis\sum_{i=1}^N\lambda_i^2\right)}\dis\prod_{i=1}^Nd\lambda_i\\
&=\dis\int_{\{\overline{\mu_N}\in B(\mu,\delta)\} } \dis\prod_{\lambda_i<\lambda_j}|\lambda_i-\lambda_j|^{\beta}\exp{\left(-\dis\frac{1}{2}\dis\sum_{\lambda_i<\lambda_j}(\lambda_i^2+\lambda_j^2)-\dis\frac{1}{2}\dis\sum_{i=1}^N\lambda_i^2\right)}\dis\prod_{i=1}^Nd\lambda_i \\
&=\dis\int_{\{\overline{\mu_N}\in B(\mu,\delta)\} }\exp{\left(-N^2\dis\int_{x<y}f_\beta(x,y)\overline{\mu_N}(dx)\overline{\mu_N}(dy)\right)}\exp{\left(-\dis\frac{1}{2}\dis\sum_{i=1}^N\lambda_i^2\right)}\dis\prod_{i=1}^Nd\lambda_i\\
&=\dis\int_{\{\overline{\mu_N}\in B(\mu,\delta)\} }\exp{\left(-N^2\dis\int_{x<y}f_\beta(x,y)\overline{\mu_N}(dx)\overline{\mu_N}(dy)\right)}\exp{\left(-\dis\frac{1}{2}\dis\sum_{i=1}^N\lambda_i^2-\sum_{i=1}^N\log(\mu(\lambda_i))\right)}\dis\prod_{i=1}^N\mu(d\lambda_i).
\end{split}
\end{equation}
Taking the logarithm and using the Jensen inequality yields: 
\begin{equation}
\label{equation: formule lower bound 2}
\begin{split}
&\log\left(\overline{Q^N_\beta}(\overline{\mu_N}\in B(\mu,\delta))\right)\ge\\
&\dis\int_{\{\overline{\mu_N}\in B(\mu,\delta)\} }\log\left[\exp{\left(-N^2\dis\int_{x<y}f_\beta(x,y)\overline{\mu_N}(dx)\overline{\mu_N}(dy)\right)}\exp{\left(-\dis\frac{1}{2}\dis\sum_{i=1}^N\lambda_i^2-\sum_{i=1}^N\log(\mu(\lambda_i))\right)}\right]\dis\prod_{i=1}^N\mu(d\lambda_i)\\
&=\dis\int_{\{\overline{\mu_N}\in B(\mu,\delta)\}}-N^2\dis\int_{x<y}f_\beta(x,y)\overline{\mu_N}(dx)\overline{\mu_N}(dy)-\dis\frac{1}{2}\dis\sum_{i=1}^N\lambda_i^2-\sum_{i=1}^N\log(\mu(\lambda_i))\dis\prod_{i=1}^N\mu(d\lambda_i)\\
&=:I_{1,N}+I_{2,N},
\end{split}
\end{equation}
where $I_{1,N}$ and $I_{2,N}$ are defined by: 
\begin{equation}
\left\{
\begin{split}
I_{1,N}&=\dis\int_{\R^N}-N^2\dis\int_{x<y}f_\beta(x,y)\overline{\mu_N}(dx)\overline{\mu_N}(dy)-\dis\frac{1}{2}\dis\sum_{i=1}^N\lambda_i^2-\sum_{i=1}^N\log(\mu(\lambda_i))\dis\prod_{i=1}^N\mu(d\lambda_i)\\
I_{2,N}&=\dis\int_{\{\overline{\mu_N}\notin B(\mu,\delta)\}}N^2\dis\int_{x<y}f_\beta(x,y)\overline{\mu_N}(dx)\overline{\mu_N}(dy)+\dis\frac{1}{2}\dis\sum_{i=1}^N\lambda_i^2+\sum_{i=1}^N\log(\mu(\lambda_i))\dis\prod_{i=1}^N\mu(d\lambda_i)
\end{split}
\right.
\end{equation}
We first look at $I_{2,N}$. Using \eqref{ineq:f_beta} and \eqref{eq: hypothesis density} we get: 
\begin{equation}
\label{equation: I}
I_{2,N}\ge \left(N^2 m_\beta\cfrac{N(N-1)}{2N^2}+CN\right) \int_{\{\overline{\mu_N}\notin B(\mu,\delta)\}}\prod_{i=1}^N \mu(d\lambda_i). 
\end{equation}
By the law of large numbers, we have $$\underset{N\to\infty}{\lim}\int_{\{\overline{\mu_N}\notin B(\mu,\delta)\}}\prod_{i=1}^N \mu(d\lambda_i)=0.$$
So we obtain that:
\begin{equation}
\label{equation: I 2}
\cfrac{I_{2,N}}{N^2}\ge o_N(1). 
\end{equation}
We now focus on $I_{1,N}$. Using that $\mu\in\mes$ and the symmetry of $f_\beta$ yields
\begin{equation}
\begin{split}
I_{1,N}&=\dis\int_{\R^N}-N^2\dis\int_{x<y}f_\beta(x,y)\overline{\mu_N}(dx)\overline{\mu_N}(dy)-\dis\frac{1}{2}\dis\sum_{i=1}^N\lambda_i^2-\sum_{i=1}^N\log(\mu(\lambda_i))\dis\prod_{i=1}^N\mu(d\lambda_i)\\
&=-\cfrac{1}{2}\int_{\R^N}\sum_{i\ne j}f_\beta(\lambda_i,\lambda_j)\dis\prod_{i=1}^N\mu(d\lambda_i)-\dis\frac{N}{2}\int_\R\lambda^2\mu(d\lambda)-N\int_\R\log(\mu(\lambda))\mu(d\lambda)\\
&=-\cfrac{1}{2}\sum_{i\ne j}\int_{\R^2}f_\beta(x,y)\mu(dx)\mu(dy)-\dis\frac{N}{2}\int_\R\lambda^2\mu(d\lambda)-N\int_\R\log(\mu(\lambda))\mu(d\lambda)\\
&=-\cfrac{N(N-1)}{2}\int_{\R^2}f_\beta(x,y)\mu(dx)\mu(dy)-\dis\frac{N}{2}\int_\R\lambda^2\mu(d\lambda)-N\int_\R\log(\mu(\lambda))\mu(d\lambda)\\
&=-N(N-1)H_\beta(\mu)-\dis\frac{N}{2}\int_\R\lambda^2\mu(d\lambda)-N\int_\R\log(\mu(\lambda))\mu(d\lambda)
\end{split}
\end{equation}
Hence we get 
\begin{equation}
\label{equation I1}
\begin{split}
\underset{N\to\infty}{\lim}\cfrac{I_{1,N}}{N^2}=-H_\beta(\mu).
\end{split}
\end{equation}
Using \eqref{equation: formule lower bound 2}, \eqref{equation: I 2} and \eqref{equation I1}, for all $\delta >0$, $$-H_\beta(\mu)\le\underset{N\to\infty}{\underline{\lim}}\cfrac{1}{N^2}\log\overline{Q^N_\beta}(\overline{\mu_N}\in B(\mu,\delta)).$$
Let $\delta$ goes to $0$ to have the lower bound: $$-H_\beta(\mu)\le\lim_{\delta\to 0}\underset{N\to\infty}{\underline{\lim}}\cfrac{1}{N^2}\log\overline{Q^N_\beta}(\overline{\mu_N}\in B(\mu,\delta)).$$
Approximating the measures with a density with respect to the Lebesgue measure by "regular" measures defined in this part, one can prove that for all $O$ open subset of $\mes$, we have the lower bound $$\underset{N\to\infty}{\underline{\lim}}\cfrac{1}{N^2}\ \log(\overline{Q^N_\beta}(\overline {\mu_N} \in O))\ge -\inf\{H_\beta(\mu),\, \mu\in O\text{ is absolutely continuous}\}.$$ This result is similar with the one obtain in Lemma \ref{lemma:largedevponct} and can be found in \cite{chafai2014first} (see Corollary 2.7).
Then, thanks to the properties of $H_\beta$, we can extend the previous result for open subsets to obtain the large deviations principle for open subsets \eqref{eq:openlargedev}. See the hypothesis (H$4$) and the end of the proof of Theorem 1 of \cite{chafai2014first} to have a complete proof of this point.
\begin{Remark}
\label{Remark:largdevconvlow}
In view of Remark \ref{Rem:largedevconstant}, let us mention how this computation can be used to obtain a part of the large deviations principle of the normalizing constant factor without knowing its exact value. 
\newline
Since $Z_\beta^N$ is the normalizing factor, we have that for all $\mu \in\mes$ $$\underset{N\to\infty}{\underline{\lim}}\cfrac{1}{N^2}\, \log Z_\beta^N\ge \lim_{\delta\to 0}\underset{N\to\infty}{\underline{\lim}}\cfrac{1}{N^2}\log\overline{Q^N_\beta}(\overline{\mu_N}\in B(\mu,\delta))\ge -H_\beta(\mu).$$
This gives that $$\underset{N\to\infty}{\underline{\lim}}\cfrac{1}{N^2}\, \log Z_\beta^N\ge -\inf_{\mes} H_\beta(\mu).$$
\end{Remark}
\subsection{Exponential tightness and complete large deviations}
\label{subsection: expo tightness}
First, we shall see why the exponential tightness of a family of measures is important to go from large deviations on compact sets to large deviations on closed sets.
\begin{definition}
The family $(Q_\beta^N)_{N\in\N}$  is said to be exponentially tight if for every $L$ there exists a compact subset $K_L$ of $\mes$ such that $\underset{N\to\infty}{\overline{\lim}}\cfrac{1}{N^2}\log Q_\beta^N (\overline{\mu_N}\notin K_L)\le-L$.
\end{definition}

\begin{prop}
Assume $(Q_\beta^N)_{N\in\N}$ is exponentially tight, then $(Q_\beta^N)_{N\in\N}$ satisfies a large deviations principle for closed sets, that means: for any closed sets $F$ of $\mes$ we have: $$\underset{N\to\infty}{\overline{\lim}}\cfrac{1}{N^2}\log Q_\beta^N (\overline{\mu_N}\in F)\le-\inf_F I_\beta.$$
\end{prop}

\begin{proof}
Let $F$ be a closed subset of $\mes$ and fix $L>0$ and $K_L$ such that: $${\overline{\lim}}\cfrac{1}{N^2}\log Q_\beta^N (\overline{\mu_N}\notin K_L)\le-L. $$ Then we have that: $$\underset{N\to\infty}{\overline{\lim}}\cfrac{1}{N^2}\log Q_\beta^N (\overline{\mu_N}\in F)\le \underset{N\to\infty}{\overline{\lim}}\cfrac{1}{N^2}\log (Q_\beta^N (\overline{\mu_N}\in F\cap K_L)+Q_\beta^N (\overline{\mu_N}\notin K_L)).$$
Using Lemma \ref{lemma: limsup sum} and the compactness of $F\cap K_L$, the large deviations principle for compact subsets gives: $$\underset{N\to\infty}{\overline{\lim}}\cfrac{1}{N^2}\log Q_\beta^N (\overline{\mu_N}\in F)\le\max\{-\inf_{F\cap K_L}I_\beta;-L\}\le\max\{-\inf_{F}I_\beta;-L\}.$$
Now let $L$ tends to $+\infty$ to obtain that: $$\underset{N\to\infty}{\overline{\lim}}\cfrac{1}{N^2}\log Q_\beta^N (\overline{\mu_N}\in F)\le-\inf_F I_\beta.$$
\end{proof}

Hence to prove the complete large deviations principle it remains to prove the exponential tightness property.

\begin{prop}
The sequence $(Q_\beta^N)_{N\in\N}$ is exponentially tight.
\end{prop}

\begin{proof}
Let $E\subset \mes$ be a measurable set, doing the same computation as in the upper bound proof to obtain \ref{equation: formule up bound} we have: 
\begin{equation}
\label{eq: calcul exp tightness}
\begin{split}
Q_\beta^N(\overline \mu_N\in E)&=\cfrac{1}{Z_\beta^N}\dis\int_{\overline{\mu_N}\in E}\exp{\left(-N^2\dis\int_{x<y}f_\beta(x,y)\overline{\mu_N}(dx)\overline{\mu_N}(dy)\right)}\exp{\left(-\dis\frac{1}{2}\dis\sum_{i=1}^N\lambda_i^2\right)}\dis\prod_{i=1}^Nd\lambda_i\\
&=\cfrac{1}{Z_\beta^N}\dis\int_{\overline{\mu_N}\in E}\exp{\left(-\cfrac{N^2}{2}\dis\int_{x\ne y}f_\beta(x,y)\overline{\mu_N}(dx)\overline{\mu_N}(dy)\right)}\exp{\left(-\dis\frac{1}{2}\dis\sum_{i=1}^N\lambda_i^2\right)}\dis\prod_{i=1}^Nd\lambda_i.
\end{split}
\end{equation}
As a consequence of the large deviations principle for the partition function obtained in Proposition \ref{prop: large dev constant}, we can find a constant $C\in\R$ such that for all $N\ge 1$, 
\begin{equation}
\label{ineq: zbeta}
\log(Z_\beta^N)\ge -C N^2.
\end{equation}
Moreover, as for the inequality \ref{ineq:f_beta}, we can find a constant $\widetilde {m_\beta}>-\infty$ such that for all $(x,y)\in\R^2$, 
\begin{equation}
\label{eq: febeta min bis}
f_\beta(x,y)\ge \cfrac{x^2}{4} +\cfrac{y^2}{4}+\widetilde{ m_\beta}.
\end{equation}
\newline
For all $L>0$, let $K_L:=\{\mu\in\mes\,|\, \int_\R x^2 \mu(dx)\le L\}$. For all $L>0$, $K_M$ is closed as a consequence of the Portemanteau theorem and is relatively compact. Indeed, the Markov inequality implies that for all $\mu\in K_L$, for all $\varepsilon>0$, $$\mu(\{|x|\ge \varepsilon\})\le \cfrac{\int_\R x^2 \mu(dx)}{\varepsilon^2}\le \cfrac{L}{\varepsilon^2}, $$ giving the relatively compactness of $K_L$ by the Prokhorov theorem.
Using \eqref{eq: febeta min bis} and \eqref{ineq: zbeta} in \eqref{eq: calcul exp tightness} yields
\begin{equation}
\label{eq: exp tightness1}
\begin{split}
Q_\beta^N(\overline \mu_N\notin K_L)&\le e^{CN^2-\frac{\widetilde {m_\beta} N (N-1)}{2}}\dis\int_{\overline{\mu_N}\notin K_L}e^{-\frac{N^2}{4}\int_{x\ne y}x^2\,\overline{\mu_N}(dx)\overline{\mu_N}(dy)}e^{-\frac{1}{2}\sum_{i=1}^N\lambda_i^2}\dis\prod_{i=1}^Nd\lambda_i \\
&=e^{CN^2-\frac{\widetilde {m_\beta} N (N-1)}{2}}\dis\int_{\overline{\mu_N}\notin K_L}e^{-\frac{N(N-1)}{4}\int_{\R}x^2\,\overline{\mu_N}(dx)}e^{-\frac{1}{2}\sum_{i=1}^N\lambda_i^2}\dis\prod_{i=1}^Nd\lambda_i \\
&\le e^{CN^2-\frac{\widetilde {m_\beta}N(N-1)}{2}-\frac{N(N-1)L}{4}} \left(\int_\R e^{-\frac{x^2}{2}}dx\right)^N
\end{split}
\end{equation}
We deduce that for all $L\in\R$, 
\begin{equation}
\label{eq: exp tightness2}
\underset{N\to\infty}{\overline{\lim}}\cfrac{1}{N^2}\log Q_\beta^N (\overline{\mu_N}\notin K_L)\le C-\cfrac{\widetilde {m_\beta}}{2}-\cfrac{L}{4}.
\end{equation}
Since $L$ is arbitrary, $(Q_\beta^N)_{N\in\N}$ is exponentially tight.
\end{proof}
\begin{Remark}
\label{Rem:largedevconstantup}
If we look at the previous computations, we proved in particular an exponential tightness for the family $\overline{Q^N_\beta}$ which means that for all $L>0$ we can find a compact $K_L$ of $\mes$ such that \begin{equation}
\label{exp:tightness}
\underset{N\to\infty}{\overline{\lim}} \cfrac{1}{N^2}\,\log \overline{Q^N_\beta} (\overline{\mu_N}\notin K_L) \le-L.
\end{equation}
Combining this property with the weak large deviations principle on compact subsets obtained as a consequence of \eqref{upboun}, we have:
\begin{equation}
\label{exp:tight2}
\underset{N\to\infty}{\overline{\lim}} \cfrac{1}{N^2}\,\log Z_{\beta}^N =\underset{N\to\infty}{\overline{\lim}} \cfrac{1}{N^2}\,\log \overline{Q^N_\beta}(\mes) \le -\inf_{\mu\in\mes} H_\beta(\mu).
\end{equation}
Indeed, for $L>0$ fix a compact $K_L$ as in \eqref{exp:tightness}. Then using Lemma \ref{lemma: limsup sum} we get: 
\begin{align*}
\underset{N\to\infty}{\overline{\lim}} \cfrac{1}{N^2}\,\log Z_{\beta}^N &=\underset{N\to\infty}{\overline{\lim}} \cfrac{1}{N^2}\,\log \overline{Q^N_\beta}(\mes)
\\&=\underset{N\to\infty}{\overline{\lim}} \cfrac{1}{N^2}\,\log \left[\overline{Q^N_\beta}(K_L)+\overline{Q^N_\beta}((K_L)^c)\right]\\
&\le \max\left(\underset{N\to\infty}{\overline{\lim}} \cfrac{1}{N^2}\,\log \overline{Q^N_\beta}(K_L), \underset{N\to\infty}{\overline{\lim}} \cfrac{1}{N^2}\,\log \overline{Q^N_\beta}((K_L)^c)\right)\\
&\le \max(-\inf_{K_L}H_\beta,-L)\\
&\le \max(-\inf_{\mes}H_\beta,-L).
\end{align*}
Let $L$ goes to $+\infty$ to get \eqref{exp:tight2}.
\newline
Using this bound with Remark \ref{Remark:largdevconvlow} this gives the large deviations principle of the normalizing factor: $$\lim_{N\to\infty} \cfrac{1}{N^2}\, \log Z_\beta^N=-\inf_{\mes} H_\beta.$$ This gives a more robust approach for the large deviations principle of Section \ref{subsetionlargedevnorm} which was mentioned in Remark \ref{Rem:largedevconstant}.
\end{Remark}
\subsection{Consequences of the large deviations principle}
\label{subsection: cons}
\begin{prop}
\label{prop: value of the minim}
For all $\beta>0$, $$\inf_{\mu\in\mes} H_\beta(\mu)=-F(\beta),\, I_\beta= \widehat {I_\beta}.$$
\end{prop}
\begin{proof}
This is an immediate consequence of Proposition \ref{prop: inffunctionrate}. Indeed, it gives that $\inf_{\mes} \widehat {I_\beta}=0$ which implies the result.
\end{proof}
As a consequence of Section \ref{subsecion: weak large dev} and Section \ref{subsection: expo tightness} we can state the following result.
\begin{theorem}
For all $\beta>0$, $(Q_{\beta}^{N})_{N\in\N}$ satisfies a LDP at speed $N^2$ with good rate function $I_\beta$.
\end{theorem}

\subsection{Computation of the equilibrium measure}
We proved that for all $I_\beta$ reaches its infimum
in a unique $\mu_\beta\in\mes$. The goal of this section is to compute it and accordingly with Wigner theorem, to prove that this measure is the semicircle distribution. 
First let state a lemma to justify that we can reduce to the case $\beta=1$. 

\begin{lemma}
\label{coro opti beta}
$\mu_\beta$ minimizes $I_\beta$ if and only if the measure $\tilde\mu_\beta(A):=\mu_\beta(\sqrt\beta A)$ minimizes $I_1$.
\end{lemma}

\begin{proof}
Thanks to the Proposition \ref{prop:euler lagrange} we have that $\mu_\beta$ minimizes $I_\beta$ if and only if for $\tilde\mu_\beta$ almost all $x$: $$\beta\dis\int_\R\log|\sqrt\beta x-\sqrt\beta y|d\mu_\beta(y)=\cfrac{\beta }{2}\,x^2-2\inf H_\beta+\cfrac{\beta}{2}\dis\int_\R y^2d\mu_\beta(y).$$ 
Since $\inf H_\beta=-F(\beta)$ (by Proposition \ref{prop: value of the minim}), and $-2F(\beta)+\log \sqrt\beta =-2\beta F(1)$, we have $\cfrac{2}{\beta}\inf H_\beta+\log \sqrt\beta =2\inf H_1$.
\newline
Hence $\mu_\beta$ minimizes $I_\beta$ if and only if the measure $\tilde\mu_\beta$ minimizes $I_1$.
\end{proof}

Thanks to Proposition \ref{prop:euler lagrange} it remains to prove that for any $x\in[-\sqrt 2,+\sqrt 2]$ $$\int_\R\log|x-y|d\sigma_1(y)=\cfrac{1}{2}\,x^2-2\inf H_1+\cfrac{1}{2}\dis\int_\R y^2d\sigma_1(y).$$
It is the object of the next proposition.
\begin{prop}
\label{prop:calcul log mu}
For all $x\in[-\sqrt 2,+\sqrt 2]$, $$\dis\int_\R\log|x-y|d\sigma_1(y)=\cfrac{x^2-\log 2-1}{2}.$$
\end{prop}

\begin{proof}
Notice that if we view $x\mapsto\dis\int_\R\log|x-y|d\sigma_1(y)$ as a distribution, we can write $\dis\int_\R\log|x-y|d\sigma_1(y)=\log\ast\sigma_1(x)$ with $\log$ the distribution associated to the function $x\mapsto \log|x|$ and $\sigma_1$ the distribution $f\mapsto\dis\int_\R f(x)d\sigma_1(x)$.
 \newline
It is well known that the distribution $P.V.(1/x)(\phi)=\underset{\varepsilon\to 0}{\lim}\dis\int_{|x|>\varepsilon}\cfrac{\phi(x)}{x}\,dx$ satisfies $P.V.(1/x)=\log'$ in the sense of distributions.
\newline
By the properties of the convolution of distribution we also have that $P.V.(1/x)\ast\sigma_1=(\log\ast\sigma_1)'$.
 \newline
We first compute $P.V.(1/x)\ast\sigma_1(x)=\underset{\varepsilon\to 0}{\lim}\dis\int_{|x-y|>\varepsilon}\cfrac{1}{x-y}\,\sigma_1(dx)$ for $x\in[-\sqrt{2},\sqrt{2}]$. This quantity is called the Hilbert transform of the measure $\sigma_1$ and was introduced in Definition \ref{def:hilbtransfo}. 
\newline
We can compute the Hilbert transform of $\sigma_1$ thanks to the Stieltjes transform of $\sigma_1$ which is well known. We recall that for all $z\in\mathbb H$, we have $$S_{\sigma_1}(z)=\dis\int_\R\cfrac{1}{z-y}\,\sigma_1(dy)=z-\sqrt{z^2-2},$$ with $\sqrt{z'}$ the determination of the square root on $\C-R^+$ such that $\sqrt{-1}=i$ (see Proposition \ref{annexe prop unique stieljes sqrt}).
\newline
Indeed we have:
$$\underset{\varepsilon\to 0}{\lim}\dis\int_{|x-y|>\varepsilon}\cfrac{1}{x-y}\,\sigma_1(dx)=\underset{\varepsilon\to 0}{\lim}\cfrac{S_{\sigma_1}(x+i\varepsilon)+S_{\sigma_1}(x-i\varepsilon)}{2}.$$
Hence, for all $x\in[-\sqrt{2},\sqrt{2}]$, $P.V.(1/x)\ast\sigma_1(x)=x$. 
\newline
\newline
So, we can find a constant $C\in\R$ such that for all $x\in[-\sqrt{2} ,\sqrt{2}]$:
\begin{equation}
\label{eq: log*mu}
\dis\int_\R\log|x-y|d\sigma_1(y)=\cfrac{x^2}{2}+C.
\end{equation}
It remains to compute the constant $C$. 
\newline
We write the equality \eqref{eq: log*mu} at $x=0$, $x=\sqrt{2}$ and $x=-\sqrt{2}$. We have the three following equalities: 

\begin{equation}
\left\{
\begin{split}
 &C=\dis\int_\R\log|y|d\sigma_1(y)\\
&C=-1+\dis\int_\R\log|\sqrt{2}-y|d\sigma_1(y) \\
&C=-1+\dis\int_\R\log|-\sqrt{2}-y|d\sigma_1(y)
\end{split}
\right.
\end{equation}
We make a polar change of variable $y=\sqrt{2}\sin(\theta)$. It gives: 
\begin{equation}
\label{eq:system equation}
\left\{
\begin{split}
&C=\log\sqrt{2}+\cfrac{4}{\pi}\dis\int_{0}^{\pi/2}\log (\sin(\theta))
\cos^2(\theta)d\theta\\
&C=\log\sqrt{2}-1+\cfrac{4}{\pi}\dis\int_{0}^{\pi/2}\log (1-\sin(\theta))\cos^2(\theta)d\theta\\
&C=\log\sqrt{2}-1+\cfrac{4}{\pi}\dis\int_{0}^{\pi/2}\log (1+\sin(\theta))\cos^2(\theta)d\theta 
\end{split}
\right.
\end{equation}
Now if we sum the two last equations of \eqref{eq:system equation}, we get:
 $$C=\log\sqrt{2}-1+\cfrac{2}{\pi}\dis\int_{0}^{\pi/2}\log (1-\sin^2(\theta))
\cos^2(\theta)d\theta=\log\sqrt{2}-1+\cfrac{4}{\pi}\dis\int_{0}^{\pi/2}\log(\cos(\theta))\cos^2(\theta)d\theta.$$
Doing the change of variable $\theta'=\pi/2-\theta$ and summing this equality with the first one of \eqref{eq:system equation} yields: $$C=\cfrac{\log(2)-1}{2}+\cfrac{2}{\pi}\dis\int_{0}^{\pi/2}\log (\sin(\theta))d\theta.$$
It remains to compute $\dis\int_{0}^{\pi/2}\log (\sin(\theta))d\theta$. 
\begin{lemma}
\label{lemma: compute integral}
$$\dis\int_{0}^{\pi/2}\log (\sin(\theta))d\theta=-\pi\cfrac{\log(2)}{2}.$$
\end{lemma}
\begin{proof}
We write $I=\dis\int_{0}^{\pi/2}\log (\sin(\theta))d\theta$. Doing the change of variable $\theta'=\pi/2-\theta$ gives $I=\dis\int_{0}^{\pi/2}\log( \cos(\theta))d\theta$. Summing these two equalities yields $$2I=\dis\int_{0}^{\pi/2}\log (\sin(\theta)\cos(\theta))d\theta.$$
Since $$\dis\int_{0}^{\pi/2}\log (\sin(\theta)\cos(\theta))d\theta=\dis\int_{0}^{\pi/2}\log (\sin(2\theta)/2)d\theta=\dis\int_{0}^{\pi/2}\log (\sin(2\theta))d\theta-\log(2)\cfrac{\pi}{2}$$ and $$\dis\int_{0}^{\pi/2}\log (\sin(2\theta))d\theta=\dis\int_{0}^{\pi}\log (\sin(\theta))\cfrac{d\theta}{2}=\dis\int_{0}^{\pi/2}\log (\sin(\theta))d\theta=I,$$ we have the following relation: $$2I=I-\pi\cfrac{\log(2)}{2}.$$
It yields $I=-\pi\cfrac{\log(2)}{2}$.

\end{proof}
As a consequence of $\eqref{eq: log*mu}$ and Lemma \ref{lemma: compute integral}, for all $x\in[-\sqrt 2,+\sqrt 2]$, $$\dis\int_\R\log|x-y|d\sigma_1(y)=\cfrac{x^2-\log 2-1}{2}.$$
\end{proof}

Thanks to Lemma \ref{coro opti beta} and \ref{prop:calcul log mu} obtain the main theorem of this section. 
\begin{theorem}
The unique minimizer of $I_\beta$ is $\sigma_\beta$.
\end{theorem}
As a consequence of this theorem we get another proof of the Wigner theorem.
\begin{corollary}[Wigner's theorem]
For all $N\ge 1$, let $X_N$ be a $N$ complex (resp. real) Wigner matrix of size $N$. Let $(\lambda_i^N)_{1\le i\le N}$ be the $N$ real eigenvalues of $X_N$ and $\mu_N$ be the spectral measure of $X_N/\sqrt{N}$:$$\mu_N=\cfrac{1}{N}\sum_{i=1}^N\delta_{\frac{\lambda_i^N}{\sqrt{N}}}.$$
Then, almost surely $$\mu_N\overunderset{\mathcal L}{N\to+\infty}{\longrightarrow} \sigma_2\, \text{ (resp. } \sigma_1).$$
\end{corollary}

\begin{proof}
This is an immediate consequence of the LDP of $(Q_\beta^N)_{N\ge 1}$ for $\beta=1$ or $\beta=2$ and  Proposition \ref{prop:almostsureconvlargedev}.
\end{proof}

\section{Applications and further results}

\subsection{Convergence of the smallest and the largest eigenvalue}
\label{subsection : largest eigenalue} 
In this section we focus on the smallest and the largest eigenvalue of real Wigner matrices but we could have stated all the results of this part for complex Wigner matrices. 

\begin{definition}
Let $M\in M_n(\R)$ be a symmetric matrix. We denote $\lambda_1^n(M)\ge...\ge\lambda_n^n(M)$ its eigenvalues. 
\end{definition}

Let us explicit the Kantorovich-Rubinstein distance which metrizes the weak topology of $\mes$. We will denote $d_{KR}$ this distance between two elements of $\mes$. We recall that for all $\mu$, $\nu$ in $\mes$, $$d_{KR}(\mu,\nu)=\sup_{f\in BL(\R);\,||f||_{BL}\le 1}\left\{ \,\left|\dis\int_{\R}f d\mu-\dis\int_{\R}f d\nu\right|\,\right\},$$ with $BL(\R)$ the space of bounded and Lipschitz functions endows with the natural norm: $||f||_{BL}=||f||_{\infty}+||f||_{lip}$.

\begin{prop}
\label{prop: Kanto rub}
Let $\lambda_1^n\ge...\ge\lambda_n^n$ be a family of $n$ real numbers. Assume that there exists $\varepsilon>0$ such that $\sqrt{2}-\varepsilon>\lambda_1^n$. Then we can find $c_\varepsilon>0$ such that: $$d_{KR}(\overline{\mu_n},\sigma_1)\ge c_\varepsilon.$$
\end{prop}

\begin{proof}
Let $f_\varepsilon$ the continuous function which is null on $]-\infty,\sqrt{2}-\varepsilon]$, constant to 1/2 on $[\sqrt{2}-\varepsilon+1,+\infty[$ and linear on $[\sqrt{2}-\varepsilon,\sqrt{2}-\varepsilon+1]$.  We clearly have that $f_\varepsilon$ is in $BL(\R)$ and that $||f||_{BL}\le 1$. 
\newline
Hence, by definition of the Kantorovich-Rubinstein distance, we have: 
$$d_{KR}(\overline{\mu_n},\sigma_1)\ge \left|\dis\int_{\R}f_{\varepsilon} d\overline{\mu_n}-\dis\int_{\R}f_{\varepsilon} d\sigma_1(y)\right|.$$
Since $f_\varepsilon$ is supported on $[\sqrt{2}-\varepsilon,+\infty[$, by hypothesis on the $\lambda_i^n$ we have that $\dis\int_{\R}f_{\varepsilon} d\overline{\mu_n}=0$.
\newline
Since $f_\varepsilon$ is not null on $[-\sqrt{2},+\sqrt{2}]$ we have: $\dis\int_{\R}f_{\varepsilon} d\sigma_1(y)>0$. Let $c_{\varepsilon}:= \dis\int_{\R}f_{\varepsilon} d\sigma_1(y)>0$, we get: $$d_{KR}(\overline{\mu_n},\sigma_1)\ge c_\varepsilon.$$

\end{proof}

\begin{prop}
\label{prop: epslion plus grande va}
Let $\varepsilon>0$ and $X_N$ be a $N$ real Wigner matrix for all $N\ge 1$. Then, we can find $C_\varepsilon>0$ such that $$\overline{\underset{N\to\infty}\lim}\cfrac{1}{N^2}\log\mP\left(\lambda_1^N\left(\cfrac{X_N}{\sqrt{N}}\right)<\sqrt{2}-\varepsilon\right)\le -C_\varepsilon.$$
\end{prop}

\begin{proof}
Using Proposition \ref{prop: Kanto rub} and the fact that the law of the eigenvalues of $X_N/\sqrt{N}$ is $Q_1^N$ we have: 
\begin{equation}
\begin{split}
\mP\left(\lambda_1^N\left(\cfrac{X_N}{\sqrt{N}}\right)<\sqrt{2}-\varepsilon\right)&\le \mP\left(d_{KR}\left(\cfrac{1}{N}\sum_{i=1}^N\delta_{\lambda_i^N(\frac{X_N}{\sqrt{N}})},\sigma_1\right)\ge c_\varepsilon\right)\\
&= Q_1^N(\{\mu\in \mes\,|\,d_{KR}(\mu,\sigma_1)\ge c_\varepsilon\}).
\end{split}
\end{equation}
By the large deviations theorem obtained in Section \ref{section: large dev}, we get: $$\overline{\underset{N\to\infty}\lim}\cfrac{1}{N^2}\log Q_1^N(\{\mu\in \mes\,|\,d_{KR}(\mu,\sigma_1)\ge c_\varepsilon\})\le-\inf_{F_\varepsilon}I_1,$$ with $F_\varepsilon=\{\mu\in\mes ,\, d_{KR}(\mu,\sigma_1)\ge c_\varepsilon)$.
\newline
Since $I_1$ is lsc and $c_\varepsilon>0$, $C_\varepsilon:=\inf_{F_\varepsilon}I_1>0$ by Lemma \ref{coro ferme}. Hence, we have $$\overline{\underset{N\to\infty}\lim}\cfrac{1}{N^2}\log\mP\left(\lambda_1^N\left(\cfrac{X_N}{\sqrt{N}}\right)<\sqrt{2}-\varepsilon\right)\le -C_\varepsilon.$$
\end{proof}

\begin{Remark}
We clearly also have the same result with $\lambda_N^N$ and $-\sqrt{2}$ instead of $\lambda_1^N$ and $\sqrt{2}$.
\end{Remark}
Let us mention two applications. The first one can also be viewed as a consequence of the Wigner theorem. The second one can be viewed as a first step in order to obtain a large deviations principle for the largest eigenvalue of a $N$ real Wigner matrix. 
\begin{corollary}
For all $N\ge 1$ let $X_N$ be a $N$ real Wigner matrix, we have $\mP$ almost surely $$\underset{N\to\infty}{\liminf}\,\lambda_1^N\left(\cfrac{X_N}{\sqrt{N}}\right)\ge \sqrt{2}.$$
\end{corollary}

\begin{proof}
We compute $\mP\left(\underset{N\to\infty}{\liminf}\,\lambda_1^N\left(\cfrac{X_N}{\sqrt{N}}\right)< \sqrt{2}\right)$. By definition of the $\lim\inf$ we have: 
\begin{equation}
\begin{split}
\mP&\left(\underset{N\to\infty}{\liminf}\,\lambda_1^N\left(\cfrac{X_N}{\sqrt{N}}\right)< \sqrt{2}\right)=\\
&\underset{\varepsilon\to 0}{\lim}\,\mP\left(\text{there exists an infinite number of $N$ such that:} \,\,\lambda_1^N\left(\cfrac{X_N}{\sqrt{N}}\right)<\sqrt{2}-\varepsilon\right).
\end{split}
\end{equation}
Let $\varepsilon>0$, by Proposition \ref{prop: epslion plus grande va}, there exists $N(\varepsilon)\in \N$ and $C_\varepsilon>0$ such that $$\forall N\ge N(\varepsilon),\,\mP\left(\lambda_1^N\left(\cfrac{X_N}{\sqrt{N}}\right)<\sqrt{2}-\varepsilon\right)\le\exp\left(-N^2\cfrac{C_\varepsilon}{2}\right),$$ which gives that $\displaystyle\sum_N \mP\left(\lambda_1^N\left(\cfrac{X_N}{\sqrt{N}}\right)<\sqrt{2}-\varepsilon\right)<+\infty$.
\newline
Hence the Borel-Cantelli lemma yields that for all $\varepsilon>0$: $$\mP\left(\text{there exists an infinite number of $N$ such that:} \,\,\lambda_1^N\left(\cfrac{X_N}{\sqrt{N}}\right)<\sqrt{2}-\varepsilon\right)=0.$$
It gives: $$\mP\left(\underset{N\to\infty}{\liminf}\,\lambda_1^N\left(\cfrac{X_N}{\sqrt{N}}\right)< \sqrt{2}\right)=0.$$

\end{proof}

\begin{corollary}
For all $N\ge 1$, let $X_N$ be a $N$ real Wigner matrix, then for all $\varepsilon>0$ we have $$\overline{\underset{N\to\infty}\lim}\cfrac{1}{N}\log\mP\left(\lambda_1^N\left(\cfrac{X_N}{\sqrt{N}}\right)<\sqrt{2}-\varepsilon\right)= -\infty.$$
\end{corollary}

\begin{proof}
By Proposition \ref{prop: epslion plus grande va} we have $C_\varepsilon>0$ such that $$\overline{\underset{N\to\infty}\lim}\cfrac{1}{N^2}\log\mP\left(\lambda_1^N\left(\cfrac{X_N}{\sqrt{N}}\right)<\sqrt{2}-\varepsilon\right)\le -C_\varepsilon.$$
So there exists $N(\varepsilon)$ such that for all $N\ge N(\varepsilon)$ $$\cfrac{1}{N^2}\log\mP\left(\lambda_1^N\left(\cfrac{X_N}{\sqrt{N}}\right)<\sqrt{2}-\varepsilon\right)\le -\cfrac{C_\varepsilon}{2}.$$
This gives that for all $N\ge N(\varepsilon)$ $$\cfrac{1}{N}\log\mP\left(\lambda_1^N\left(\cfrac{X_N}{\sqrt{N}}\right)<\sqrt{2}-\varepsilon\right)\le -\cfrac{C_\varepsilon}{2}\,N.$$
Passing to the limit $N\to+\infty$ gives the result.
\end{proof}
Let us mention the complete result for the large deviations of the largest eigenvalue of a real Wigner matrices. We refer to the section 2.6.2 of \cite{Alicelivre} for a proof of this theorem.
\begin{theorem}
For all $N\ge 1$ let $X_N$ be a $N$ real Wigner matrix. Then $(\lambda_1^N\left(\frac{X_N}{\sqrt{N}}\right))_{N\ge 1}$ satisfies a LDP at speed $N$ with good rate function $$\mathcal I(x)=\left\{
    \begin{array}{ll}
        +\infty & \mbox{if } x<\sqrt{2} \\
        \frac{x^2}{2}-\int_\R\log|x-y|d\sigma_1(dy)-\frac{\log(2)+1}{2} & \mbox{if } x\ge \sqrt{2}.
    \end{array}
\right.$$
\end{theorem}
The function $\mathcal I$ is a good rate function since it is increasing on $[\sqrt{2},+\infty)$, strictly convex on $[\sqrt{2},+\infty)$, continuous on $[\sqrt{2},+\infty)$ (by the dominated convergence theorem on $]\sqrt{2},+\infty)$ and by the monotone convergence for the right continuity in $\sqrt{2}$) and goes to $+\infty$ when $x$ goes to $+\infty$.
\newline
Moreover, $I(x)\ge0$ by the Frostman theorem \ref{thm: frostman} with an equality if and only if $x=\sqrt{2}$. Hence $I$ has a unique minimizer which is $\sqrt{2}$. Using Proposition \ref{prop:almostsureconvlargedev} the previous theorem implies the following result. 
\begin{theorem}
For all $N\ge 1$ let $X_N$ be a $N$ real Wigner matrix. Then $$\lambda_1^N\left(\frac{X_N}{\sqrt{N}}\right)\overunderset{\emph{Almost surely}}{N\to+\infty}{\longrightarrow}\sqrt{2}.$$
\end{theorem}

\subsection{Extensions of the result}
\label{subsection: extenson}
\subsubsection{Coulomb and Riesz gases}
Let $N\ge 1$ and let $X_N$ be the $N$-tuple of points $(x_1,...,x_N)\in (\R^d)^N$, where $d$ is the dimension. We define the following energies $\mathcal H_N\,:\, (\R^d)^N\to (-\infty,+\infty)$ by :
\begin{equation}
\label{def: energy}
\mathcal H_N(X_N):=\cfrac{1}{2}\sum_{1\le i\ne j\le N} g(x_i-x_j)+N\sum_{i=1}^N V(x_i),
\end{equation}
where $g\,:\, \R^d\to(-\infty,+\infty]$ is called the pair interaction potential and $V\,:\,\R^d\to(-\infty,+\infty]$ is called the external field.
\newline
We focus on the cases where $g$ is given by 
\begin{equation}
\label{eq: coulom riesz interaction}
g_{d,s}(x)=\left\{
	\begin{array}{ll}
		\cfrac{1}{s}\,|x|^{-s} & s\ne 0\\
		-\log|x| & s=0,\\
	\end{array}
\right.
\end{equation}
The first case is called the Riesz case and the second one the logarithmic case.
\newline
The case $s=0$ is obtained as the formal limit when $s\to 0$ and so we will name it as a Riesz case also.
\newline
The particular case $s=d-2$ corresponds to the Coulomb case. Indeed, in this case, $g$ is the Coulomb kernel (up to a multiplicative constant); the fundamental solution to The Laplace-Poisson equation
\begin{equation}
\label{eq: laplace poisson eq}
-\Delta g_{d,d-2}=c_d\delta_0,
\end{equation}
where $\delta_0$ is the Dirac mass at the point $0$ and $c_d$ is explicitly given by 
\begin{equation}
\label{eq: constant norm laplace poisson}
c_d=\left\{
	\begin{array}{ll}
		2\pi\ & d=2\\
		 |\mathbb S^{d-1}|& d\ge 3\\
	\end{array}
\right.
\end{equation}
with $\mathbb S^{d-1}$ the unit sphere of $\R^d$.
For instance in dimension $3$, with $s=1$ $g$ corresponds to the gravitational potential. 
\newline
More generally, the cases $d-2\le s<d$ are much more well known. Indeed for $s\ge d$, the potential $g$ is not integrable near $0$ what requires a different study. Moreover, as for the case $s=d-2$, in the cases $d-2\le s<d$, $g_{d,s}$ is the fundamental solution of an operator (up to a multiplicative constant) called the fractional Laplacian: \begin{equation}
\label{eq: laplace eq}
(-\Delta)^{\frac{d-s}{2}} g_{d,s}=c_{d,s}\delta_0,
\end{equation}
where $c_{d,s}$ is explicitly given by 
\begin{equation}
\label{eq: constant norm}
c_{d,s}=\left\{
	\begin{array}{lll}
		2\pi & s=0,\, d=1,\,d=2\\
		 |\mathbb S^{d-1}|& s=d-2>0\\
		 \cfrac{2^{d-s}\pi^{d/2}\Gamma\left(\cfrac{d-s}{2}\right)}{\Gamma\left(\cfrac{s}{2}\right)}& s>\max(0,d-2)
	\end{array}
\right.
\end{equation}
The main difference between the fractional Laplacian and the Laplacian is the fact that the fractional Laplacian is a non local operator. Many equivalent definitions exist but we can for instance define this operator using Fourier representation by $$\mathcal F((-\Delta)^\alpha f)(\xi)=|\xi|^{2\alpha}\mathcal F(f),\, \alpha\in (0,1),$$ or by an integral representation $$(-\Delta)^\alpha f(x)=C_{d,\alpha}\int_{\R^d}\cfrac{f(x)-f(y)}{|x-y|^{d+2\alpha}}dy ,\, \alpha\in (0,1), $$ with $C_{d,\alpha}>0$ an explicit constant. 
From the Fourier definition of the fractional Laplacian, we deduce that the Fourier transform of $g_{d,s}$ is proportional to $\xi\mapsto |\xi|^{s-d}$ making link with Remark \ref{remark : fourier repr}.
\newline
We can associate to the energy $H_N$ the Gibbs measure  which is a probability measure on $(\R^d)^N$ whose density with respect to the Lebesgue measure on $(\R^d)^N$, $dX_N:=dx_1...dx_N$ is 
\begin{equation}
\label{eq: gibbs measue}
d\mP_{N,\beta}(X_N):=\cfrac{1}{Z_{N,\beta}}\exp(-\beta N^{-s/d} H_N(X_N))dX_N,
\end{equation}
where $\beta$ physically corresponds to the inverse of the temperature and $Z_{N,\beta}$ is a normalizing constant. 
\newline
The particular case $s=0$ and $d=1$ corresponds to the case of $\beta$ Hermite ensemble studied in this part. Thanks to a similar approach, one can prove a large deviations principle for the sequence of probability measures $(P_{N,\beta})_{N\in\N}$ extended on $\mes$ as in the Definition \ref{def:extend} at speed $\beta N^{2-\frac{s}{d}}$. See chapter 2 and 3 of \cite{serfaty2024lectures} for the complete and rigorous derivation of this large deviations principle. 
\newline
Let us mention the particular Coulomb case $d=2$, $s=0$ with the potential $V(z)=|z|^2/2$ with $\beta=2$. In this case, 
\begin{align*}
d\mP_{N,2}(z_1,...,z_N)&=\cfrac{1}{Z_{N,2}}\exp\left(2\sum_{i<j}\log|z_i-z_j|-N\sum_{1\le i\le N}|z_i|^2\right)dz_1...dz_N\\
&=\cfrac{1}{Z_{N,2}}\prod_{i<j}|z_i-z_j|^2\,\exp\left(-N\sum_{1\le i\le N}|z_i|^2\right)dz_1...dz_N.
\end{align*}
We recognize the density of the eigenvalues of $M_N/\sqrt{N}$ where $M_N$ is a $N$ complex Ginibre matrix obtained in Theorem \ref{thm : law spectrum ginibre}.
As explained $\mP_{N,2}$ satisfies a large deviations principle at speed $2N^2$. The counterpart of $H_\beta$ that was studied in this section is $$H_\C(\mu):=-\int_{\C^2}\log|z-z'|\mu(dz)\mu(dz')+\cfrac{1}{2}\,\int_\C |z|^2 \mu(dz).$$
By similar arguments as for the $\beta$ ensembles case studied in this section, one can prove that the good rate function for the large deviations of $\mP_{N,2}$ is $I_\C:=H_\C-\inf_{\mu\in\mathcal P(\C)}H_\C(\mu)$.
We can show that the unique minimizer of $I_\C$ is the uniform distribution on the circle. Indeed, the exact same proof as we did for the Frostman theorem \ref{thm: frostman} implies that the unique minimizer of $I_\C$ satisfies the counterpart of \eqref{thm: frost 1} and \eqref{thm: frost 2}. 
Let $\mu_\C$ be the unique minimizer of $I_\C$ and $S_{\mu_\C}$ its support. We obtained that $z\mapsto -\int_\C\log|z-z'|\mu_\C(dz')+\cfrac{|z|^2}{2}$ is constant on $S_{\mu_\C}$. Since $\Delta(\log(.))=2\pi \delta_0(dz)$, computing the Laplacian yields that 
\begin{equation}\mu_\C(dz)=\cfrac{\Delta |z|^2}{4\pi}\,\mathds 1_{S_{\mu_\C}}(z)dz=\cfrac{1}{\pi}\,\mathds 1_{S_{\mu_\C}}(z)dz.
\label{eq:frostcomp}
\end{equation}
By uniqueness of the minimizer, one can show that $\mu_\C$ is invariant by rotation, so its support is a ball of radius $r>0$. Using \eqref{eq:frostcomp}, the unique possibility such that $\mu_\C$ is a probability measure is $r=1$. So $\mu_\C$ is the uniform measure of the disk, giving another proof of the circular law proved in Section \ref{Section:circularlaw}.
\newline
For people interested in this topic we recommend the paper of Chafaï $\&$ al. \cite{chafai2014first} for a very rigorous derivation of a large derivation principles for particles in singular repulsion and a study of the equilibrium measure for radial external potential and the review of Serfaty \cite{serfaty2024lectures} for a more recent overview on this topic. 
\subsubsection{Extension to other random matrices models and large deviations for the largest and smallest eigenvalue}
After the pionner paper of A.Guionnet and G.Ben Arous \cite{Alice1}, the large deviations principle of the empirical mean has been studied for many other classical models of random matrices. For instance a large deviations principle has been obtained for the circular model \cite{arous1998large} and for the unitary one \cite{hiai2000large}.
\newline 
More generally a goal would be to obtain a large deviations principle for random matrices even if we do not know  the density of the distribution of eigenvalues of the model (or even if this density does not exist). An interesting example would be random matrices with independent and identically distributed entries that have Bernoulli or Rademacher distributions. These questions are linked with random graph theory since we can in the Bernoulli case interpret the random matrix as the adjacent matrix of an Erdős–Rényi random graph \cite{augeri2025large}. Recently many works for the Bernoulli and Rademacher cases have been been done to study the large deviations of the largest and smallest eigenvalue of such matrices.
Indeed, as explained in Section \ref{subsection : largest eigenalue}, the behaviour of the largest eigenvalues seems to also obey to a large deviations principle. Indeed for random matrices whose law of eigenvalues are given by $Q^N_\beta$, the largest eigenvalue satisfies a large deviations principle at speed $N$ with an explicit good rate function. See the section 2.6.2 of \cite{Alicelivre} for a complete proof. 
For the study of Rademacher and Bernoulli cases, let us mention the recent works of F.Augeri, A.Guionnet, J.Husson \cite{Alice2,Alice3}.

\newpage
\appendix
\appendixpage
\addappheadtotoc

\section{The Wigner semicircle distribution}
\label{appendixsection}
The goal of this section is to introduce some standard tools that are used in random matrix theory/probability and to present the ideas of how they are used to prove the Wigner theorem. The goal is not to make the proofs of the Wigner theorem but to present these tools for readers that are non familiar with random matrices. For more detailed introductions to random matrices we refer to \cite{Alicelivre,Tao}. 
\subsection{Moments}
\begin{definition}
The Wigner semicircle distribution is the probability distribution on $\mathbb R$ with density function $$\sigma(x):=\mathds{1}_{[-2,2]}(x)\,\cfrac{\sqrt{4-x^2}}{2\pi}.$$
We shall write $\sigma$ for the probability measure and its density. In this section, for all $k\in\mathbb N$ let $m_k:=\dis\int_\R x^kd\sigma(x)$ its moment of order $k$.
\end{definition}

\begin{definition}
We say that a probability measure $\mu$ is characterized by its moments if for all probability measures $\nu$ such that for all $k\in \N$, $\dis \int_\R x^k d\mu(x)=\dis\int_\R x^kd\nu(x)$ we have $\mu=\nu$.
\newline
We say that a real random variable $X$ is characterized by its moments if its law is a probability measure characterized by its moments.
\end{definition}

Let us recall that compact supported measures are determined by their moments.

\begin{prop}
A probability measure $\mu$ on $\mathbb R$ with a compact support is characterized by its moments.
\end{prop}

\begin{proof}
A real probability measure $\mu$ is characterized by its characteristic function $$\phi_\mu(t)=\dis\int_\R\exp(itx)d\mu(x).$$
Fix $R>0$ such that $supp(\mu)\subset[-R,R]$. Then its moments $\mu_k:=\dis\int_{-R}^R x^kd\mu(x)$ satisfy that 
\begin{equation}
\label{annexe bound moment}
|\mu_k|\le R^k.
\end{equation}
\newline
Hence, using the Taylor expansion of the exponential function: $\exp(itx)=\dis\sum_{k=0}^{\infty}\cfrac{(itx)^k}{k!}$ and changing the integral and the sum (which is justified by the bound \eqref{annexe bound moment}), we get that $$\forall t\in\mathbb R,\, \phi_\mu(t)=\displaystyle\sum_{k=0}^{+\infty}\cfrac{\mu_k}{k!}\,i^kt^k.$$
So if a probability measure has the same moments as $\mu$, it necessarily has the same characteristic function and so it is equal to $\mu$.
\end{proof}
We first recall some basic definitions about the convergence of probability measures.
\begin{definition}
Let $\mathbb X$ be a separable Banach space. We endow the set of probability measures on $\mathbb X$; $\mathcal P(\mathbb X)$; with the weak topology which means that a sequence of probability measures $(\mu_n)\in\mathcal P(\mathbb X)^\N$ converges weakly to $\mu\in\mathcal P(\mathbb X)$ if and only if for all $f: \mathbb X\to \R$ continuous and bounded, we have $\int_\mathbb X fd\mu_n\underset{n\to+\infty}{\longrightarrow}\int_\mathbb X fd\mu$.
\newline
When we deal with the law of random variables we shall identify a random variable with its law. For random variables, we said that a sequence of random variables $(X_n)_{n\in\N}$ with values in $\mathbb X$ converges in law towards $X$ if $(\mu_n)_{n\in\N}$ the sequence of law of $(X_n)_{n\in\N}$ converges weakly towards the law of $X$.
\newline
If $(\mu_n)_{n\in\N}$ (resp. $(X_n)_{n\in\N}$) converges weakly (resp. in law) to $\mu$ (resp. $X$), we write $\mu_n\convlaw \mu$ (resp. $X_n \convlaw X$).
\end{definition}

Now, let us recall a result about convergence in distribution of random variables that is called method of moments which is based on Prokhorov's theorem. We first recall the definition of tightness. 

\begin{definition}
Let $\mathbb X$ be a separable Banach space. A sequence of measure $(\mu_n)_{n\in\N}\in\mathcal P(\mathbb X)^\N$ is said to be tight if for all $\varepsilon>0$, there exists a compact $K_\varepsilon\subset \mathbb X$ such that for all $n\in\N$ $\mu_n(K_\varepsilon)>1-\varepsilon$.
\end{definition}

\begin{example}
\label{annexe example 1}
A sequence of real random variables $(X_n)_{n\in\N}$ such that there exists $k\in\N_{\ge 1}$ and $M\ge 0$ such that for all $n\ge 0$, $\E(|X_n|^k)\le M$ is tight by the Markov inequality.
\end{example}
We can now state the main theorem about the convergence in law of random variable (equivalently we could have stated it for the weak convergence of probability measures). We refer to the Chapter 1 of \cite{billingsley2013convergence} for a proof of this standard result.
\begin{theorem}[Prokhorov]
Let $(X_n)_{n\in \N}$ and $X$ be random variables with values in separable Banach metric space.
\newline
Then $(X_n)_{n\in\N}$ is relatively compact for the topology of the convergence in law if and only if $(X_n)_{n\in\N}$ is tight. 
\newline
Furthermore $X_n\convlaw X$ if and only if $(X_n)_{n\in \N}$ is tight and for all subsequences $\phi$ such that $X_{\phi(n)}\convlaw Y$ where $Y$ is a random variable then $X$ and $Y$ have the same law (this second point shall be called uniqueness of the limit).
\end{theorem}
As an application of the Prokhorov theorem, let us prove the so-called method of moments
\begin{prop}[Method of moments]
\label{prop: method of moments}
Let $(X_n)_{n\in \N}$ be a sequence of real random variables and $X$ be a real random variable characterized by its moments. We assume that for all $k\in \N$ $\conv\E(X_n^k)=\E(X^k)$. 
\newline
Then $X_n \convlaw X$.
\end{prop}
First let us prove a lemma that shall give the uniqueness of the limit in order to apply the Prokhorov theorem.
 
\begin{lemma}
\label{lemma: method of moment}
Let $(X_n)_{n\in N}$ and $X$ be real random variables such that $(X_n)_{n\in\N}$ is uniformly integrable and $X_n\convlaw X$, then $\E(X_n)\convf \E(X)$.
\end{lemma}

\begin{proof}
Since $(X_n)_{n\in\N}$ converges to $X$ in law, we know that by the Skorokhod representation theorem (see \cite{billingsley2013convergence}) we can find random variables $(X_n')_{n\in\mathbb N}$ and $X'$ defined on the same probability space, such that for all $n\ge 0$, $X_n'$ and $X_n$ have the same law, $X$ and $X'$ have the same law and $X_n'$ converges almost surely to $X'$. So $(X_n')_{n\in \N}$ is uniformly integrable and converges almost surely and so in probability towards $X'$. Hence $(X_n')_{n\in \N}$ converges in $L^1$ to $X'$. 
\newline
We get that $$\E(X_n)=\E(X_n')\convf \E(X')=\E(X).$$
\end{proof}
We now prove Proposition \ref{prop: method of moments}.
\begin{proof}
We shall use the Prokhorov theorem to show the result. Firstly, since $(\E(X_n^2))_{n\in N}$ converges, this sequence is bounded and so $(X_n)_{n\in N}$ is tight by Example \ref{annexe example 1}. 
\newline
We want to show the uniqueness of the limit. Let $(X_{\phi(n)})_{n\in \N}$ be a subsequence that converges in law to a random variable $Y$.
\newline 
Fix $k\in\N$. $\left(X_{\phi(n)}^k\right)_{n\in \N}$ is uniformly integrable since bounded in $L^2$ by the hypothesis of Proposition \ref{prop: method of moments} and converges in law to $Y^k$. Hence, by Lemma \ref{lemma: method of moment}, we have that for all $k\in\N$, $\E\left(X_{\phi(n)}^k\right)\convf \E(Y^k)$. As for all $k\in \N$ $\conv\E(X_n^k)=\E(X^k)$, we get that for all $k\in \N$, $\E(Y^k)=\E(X^k)$. Since $X$ is characterized by its moments, $X$ is equal to $Y$ in law. This proves the uniqueness of the limit.
 
\end{proof}
In order to apply the method of moments we have to compute the moments of Wigner semicircle distribution. Before let us introduce the Catalan  numbers that are linked with these moments as we shall see. 
\begin{definition}
For all $p\in\N$, we call $C_p$ the $p^{th}$ Catalan number which is the number of plane trees with $p$ edges. 
\end{definition}

\begin{prop}
\label{annexe prop catalan}
We have the following properties for $(C_p)_{p\in\N}$:
\newline
$\bullet$ $(C_p)_{p\in\N}$ is characterized by $C_0=1$ and $C_{p+1}=\dis\sum_{k=0}^pC_kC_{p-k}$. 
\newline
$\bullet$ Let $C(Z)=\displaystyle\sum_{k=0}^{+\infty}C_pZ^p\in\C[[Z]]$ its generative function (where $\C[[Z]]$ is the set of formal series in the variable $Z$). $C(Z)$ satisfies $$ZC(Z^2)-C(Z)+1=0 ,\, \, C(Z)=\cfrac{1-\sqrt{1-4Z}}{2Z}.$$ 
$\bullet$ The explicit values of the Catalan numbers are given by: $$\forall p\in \N, \, C_p=\cfrac{1}{p+1}\binom{2p}{p}.$$
\end{prop}

\begin{proof}
$\bullet$ $C_0$ is clearly equal to 1. For the inductive formula, consider a plane tree with $p+1\ge1$ edges, if we look at the leftmost edge and we isolate it, we have two plane trees with respectively $k\le p$ edges and $p-k$ with $k$ the number of edges on leftmost plane tree. Hence, if we denote $\mathcal T_p$ the set of plane trees with $p$ edges, we have a bijection from $\mathcal T_{p+1}$ to $\dis\bigsqcup_{k=0}^p\mathcal T_k\times\mathcal T_{p-k}$. Taking the cardinality gives the result. 
\newline
$\bullet$ Doing a Cauchy product and using the inductive formula yields $$C(Z)^2=\dis\sum_{k=0}^{+\infty}\dis\sum_{k=0}^pC_kC_{p-k}Z^p=\displaystyle\sum_{k=0}^{+\infty}C_{p+1}Z^p,\, ZC(Z)^2-C(Z)+1=0.$$
Solving this equation in $C(Z)$, we find that $C(Z)=\cfrac{1-\sqrt{1-4Z}}{2Z}$ using the fact $C(0)=1$ to choose the good solution. 
\newline
$\bullet$ Denoting $c(z):= \cfrac{1-\sqrt{1-4z}}{2z}$, we have that $c$ is holomorphic on $D(0,1/4)$, satisfies the equation $zc(z)^2-c(z)+1=0$ for all $z\in D(0,1/4)$. We can do its Taylor expansion at 0 $$\forall z\in D(0,1/4),\, c(z)=\dis\sum_{p=0}^{+\infty}\cfrac{1}{p+1}\binom{2p}{p}z^p.$$ By uniqueness of the coefficients of Taylor expansion we deduce that $$\forall p\in \N, \, C_p=\cfrac{1}{p+1}\binom{2p}{p}.$$
\end{proof}

\begin{prop}
\label{annexe prop moment}The moments of the Wigner semicircle law are given by:
$$\forall p\in \N, \, m_{2p+1}=0,\,  m_{2p}=C_p$$
\end{prop}
\begin{proof}
Let us notice that $\sigma$ is even, so for all $p\in\N$, $m_{2p+1}=0$. 
\newline
Fix $p\in\mathbb N$, we write $$m_{2(p+1)}=\dis\int_{-2}^2x^{2(p+1)}\cfrac{\sqrt{4-x^2}}{2\pi}dx=\dis\int_2^2x^{2p+1}x\cfrac{\sqrt{4-x^2}}{2\pi}dx.$$
Noticing that $g(x)=x\sqrt{4-x^2}$ is the derivative of $h(x)=-\cfrac{1}{3}(4-x^2)^{\frac{3}{2}}$, by integrating by part yields $$m_{2p+2}=\cfrac{2p+1}{3}\dis\int_{-2}^2x^{2(p+1)}(4-x^2)\cfrac{\sqrt{4-x^2}}{2\pi}dx=\cfrac{2p+1}{3}(4m_{2p}-m_{2p+2}).$$
\newline
Hence, we get that $\cfrac{m_{2p+2}}{m_{2p}}=\cfrac{2p+1}{2p+4}=\cfrac{C_{p+1}}{C_p}$. 
\newline
\newline
Since $C_0=m_0=1$, we get that $C_p=m_{2p}$ for all $p\in\mathbb N$.
\end{proof}

As a consequence of the results of this part, we can state the following corollary.

\begin{corollary}
Let $(\mu_n)_{n\in\mathbb N}\in\mes^\N$ be a sequence of probability measures such that for all $k\in \N$, $\left(\dis\int_\R x^kd\mu_n(x)\right)_{n\in \N}$ converges to 0 if $k$ is odd and to $C_{k/2}$ if $k$ is even then $(\mu_n)_{n\in\N}$ converges weakly to $\sigma$.
\end{corollary}
This Corollary is used in random matrix theory to prove the so-called mean Wigner theorem. We defined a $N$ complex Wigner matrix in Definition \ref{def gue,goe}.
For a matrix $M\in \mathcal H_N(\C)$ and $f:\R \to \R$, one can define $f(M)$ by diagonalizing $M=ODO^*$ with $O\in\mathcal U_N(\C)$ and $D$ a real diagonal matrix and defining $f(M)=Of(D)O^{*}$, where $f(D)$ is the diagonal matrix on which we apply $f$ on the diagonal elements of $D$. 
\newline
Let us mention that this definition of $f(M)$ does not actually depend on the choice of basis to diagonalize $M$. Indeed, if we write $M=P_1 D_1 P_1^{-1}=P_2 D_2 P_2^{-1}$ with $D_1$ and $D_2$ diagonal matrix and $(P_1,P_2)\in GL_N(\C)$ and denote $\{\lambda_1(M),...,\lambda_N(M)\}$ the spectrum of $M$. Then we can define by Lagrange interpolation a polynomial $P\in\C[X]$ such that for all $1\le i\le N$, $P(\lambda_i(M))=f(\lambda_i(M))$. We have that $P(M)=P_1 P(D_1) P_1^{-1}=P_2 P(D_2) P_2^{-1}$ (by addition and multiplication of matrices) and so $P(M)=P_1 f(D_1) P_1^{-1}=P_2 f(D_2) P_2^{-1}$. This proves that the definition $f(M)=P_1 f(D_1) P_1^{-1}=P_2 f(D_2) P_2^{-1}$ does not depend on the basis in which we diagonalize $M$.
\begin{definition}[Mean spectral measure]
Let $X_N$ be a N complex Wigner matrix, we define the mean spectral measure $\E(\mu_N)$ of $X_N/\sqrt{N}$ as the unique measure in $\mes$ such that for all $f:\R\to\R$ continuous and bounded: $$\E\left[\cfrac{1}{N}\,\Tr f\left(\cfrac{X_N}{\sqrt{N}}\right)\right]=\int_\R f \,d\E(\mu_N)$$
\end{definition}
Using the method of moments one can prove the Wigner theorem in mean.

\begin{theorem}[Mean Wigner theorem]
\label{annexemeanwigner}
For all $N\in\N$, let $X_N$ be a $N$ complex Wigner matrix and $\E(\mu_N)$ be the mean spectral measure of $X_N/\sqrt{N}$. Then $(\E(\mu_N))_{N\in\N}$ converges weakly to $\sigma$.
\end{theorem}
The starting point is to write $$\int_\R x^k d\E(\mu_N)=\E\left[\cfrac{1}{N}\,\Tr \left(\cfrac{X_N}{\sqrt{N}}\right)^k\right]=\cfrac{1}{N^{1+k/2}}\,\E[\Tr(X_N)^k)].$$
Then, we express $\Tr(X_N^k)$ as $\sum_{i=1}^N\sum_{1\le i_1,...,i_{k-1}\le N}(X_{N})_{i,i_1}...(X_{N})_{i_{k-1},i}$. The next part of the proof consists in doing combinatorics to simplify the expectation of the previous sum. We refer to Theorem 2.1.1 of \cite{Alice1} for a complete proof. Let us mention that this proof does not use that the entries of the Wigner matrix are Gaussian it only uses hypothesis on the moments of the entries of the matrix. 
\newline
Finally, let us also point that in the case of Wigner matrices with Gaussian entries the Wigner theorem can be considered as the first term of a so-called topological expansion. This is obtained thanks to the Dyson-Schwinger equations. This makes a link between the random matrix theory and the enumerative geometry and combinatorics. See for instance \cite{guionnet2019asymptotics,guionnet2005combinatorial}.
\subsection{Stieltjes transformation}
\label{annexesectionstieljes}
\begin{definition}
Let $\mu$ be a real probability measure. We define the Stieltjes transform of $\mu$ by $$\forall z\in \C-\R,\ S_\mu(z):=\dis\int_\R\cfrac{1}{z-t}\,d\mu(t).$$

\end{definition}

\begin{example}
For $\mu=\cfrac{1}{2}\, (\delta_1+\delta_{-1})$, we have $S_\mu(z)=\cfrac{z}{z^2-1}$.
\end{example}

\begin{Remark}
\label{annexe remark stiejes}
For $\mu\in\mes$, for all $z\in\C-\R$, $\overline{S_\mu(z)}=S_\mu(\overline{z})$. Hence we only have to know the Stieljes transform on $\mathbb H=\{z\in\C\,|\,\Im(z)>0\}$ to know it on $\C-\R$.
\end{Remark}
\begin{prop}
\label{annexe prop property stieljes}
Let $\mu$ be a real probability measure, we have the following properties: 
\begin{enumerate}
\item{For all $ z\in \C-\R$, $|S_\mu(z)|\le\cfrac{1}{|\Im(z)|}.$}
\item{If the support of $\mu$ is included in $[-A,A]$, then for all $z$ such that $|z|>A$, we have $$S_\mu(z)=\dis\sum_{k=0}^{+\infty}\dis\int_\R t^kd\mu(t)\cfrac{1}{z^{k+1}}.$$}
\item{ There exists an inversion formula. Let $I=[a,b]$ with $\mu(\partial I)=0$, then 
\begin{equation}
\label{equation inversion stieljes}\mu(I)=\underset{\varepsilon\to0}{\lim}-\cfrac{1}{\pi}\Im\left(\dis\int_IS_\mu(x+i\varepsilon)dx\right).
\end{equation} 
As a consequence, if $\nu\in\mes$ is such that $S_\mu=S_\nu$ on $\C-\R$ then $\nu=\mu$.}
\item{A sequence $(\mu_n)_{n\in \N}$ of real probability measures converges weakly to $\mu\in\mes$ if and only if for all $ z\in \C-\R$, $S_{\mu_n}(z)$ converges to $S_\mu(z)$.}
\item{$S_\mu$ is holomorphic on $\C-supp(\mu)$ where $supp(\mu)$ is the support of the measure $\mu$.}
\end{enumerate}
\end{prop}
\begin{proof}$\,$
\newline
\begin{enumerate}
\item{Notice that if $ z\in \C-\R$ and $t\in\R$, $|\Im(z)|=|\Im(z-t)|\le|z-t|$.}
\item{If $|z|>A$ and $|t|\le A$, then $\cfrac{1}{z-t}=\cfrac{1}{z}\,\cfrac{1}{1-t/z}=\cfrac{1}{z}\displaystyle\sum_{k=0}^{+\infty}t^kz^{-k}$. Exchanging the integral and the sum gives the result.}
\item{Since the atoms of $\mu$ are at most countable the formula \ref{equation inversion stieljes} implies that the measure $\mu$ is characterized by its Stieljes transform. 
\newline
\tab Let $I=[a,b]$ with $\mu(\partial I)=0$. We have that $$\Im\left(\dis\int_IS_\mu(x+i\varepsilon)dx\right)=\Im\left(\dis\int_I\dis\int_\R\cfrac{1}{x+i\varepsilon-t}d\mu(t)dx\right)=-\dis\int_{\R^2}\mathds{1}_{[a,b]}(x) \cfrac{\varepsilon}{(x-t)^2+\varepsilon^2}\, d\mu(t)dx.$$
Changing the variable and using Fubini-Tonelli theorem, yields $$-\cfrac{1}{\pi}\Im\left(\dis\int_IS_\mu(x+i\varepsilon)dx\right)=\dis\int_{\R^2}\mathds{1}_{[a,b]}(\varepsilon x+t) \cfrac{1}{\pi(x^2+1)}d\mu(t)dx.$$  
Let $C$ be a random variable with Cauchy distribution and $X$ be random variable with law $\mu$ and independent from $C$. We can write that 
\begin{equation}
\label{equation cauchy rv}
-\cfrac{1}{\pi}\Im\left(\dis\int_IS_\mu(x+i\varepsilon)dx\right)=\mP(X+\varepsilon C\in[a,b]).
\end{equation} 
Since $\mu(\partial I)=0$ then $\mathds{1}_{X+\varepsilon C\in[a,b]}$ converges almost surely to $\mathds{1}_{X\in[a,b]}$ when $\varepsilon$ converges to 0. By the dominated convergence theorem we get $$-\cfrac{1}{\pi}\Im\left(\dis\int_IS_\mu(x+i\varepsilon)dx\right)=\mP(X+\varepsilon C\in[a,b])\underset{\varepsilon\to0}{\rightarrow}\mP(X\in[a,b])=\mu(I).$$ }
\item{If $(\mu_n)_{n\in \N}$ converges weakly to $\mu$ then, since $t\mapsto\Re(1/(z-t))$ and $t\mapsto \Im(1/(z-t))$ are bounded continuous functions for all $z\in\C-\R$, we have that for all $ z\in \C-\R$, $S_{\mu_n}(z)$ converges to $S_\mu(z)$.
\newline
Conversely, let $(X_n)_{n\in\N}$, $X$, $C$ be independent random variables whose laws are respectively $(\mu_n)_{n\in\N}$, $\mu$ and Cauchy distribution as in the previous point. Let $(a,b)\in\R^2$ such that $\mu(\{a,b\})=0$ then for all $\varepsilon>0$, for all $n\in\N$, $$\mP(X_n+\varepsilon C\in[a,b])=-\cfrac{1}{\pi}\Im\left(\dis\int_IS_{\mu_n}(x+i\varepsilon)dx\right)$$ converges to $$-\cfrac{1}{\pi}\Im\left(\dis\int_IS_\mu(x+i\varepsilon)dx\right)=\mP(X+\varepsilon C\in[a,b]).$$
By the Portemanteau theorem (see \cite{billingsley2013convergence}) we deduce that for all $\varepsilon>0$, $X_n+\varepsilon C\convlaw X+\varepsilon C$. Taking $\varepsilon=1$, the Lévy continuity theorem gives that for all $t\in\R$, $$\E(\exp(it(X_n+C)))\convf \E(\exp(it(X+C))).$$  By independence and since the characteristic function of the Cauchy distribution does not vanish, we get that for all $t\in\R$, $\E(\exp(itX_n))\convf \E(\exp(itX))$. Again by the Lévy continuity theorem $X_n\convlaw X$.}
\item{For $K$ a compact included in $\C-supp(\mu)$ we have that for all $z\in K$ and $t\in supp (\mu)$  $\left|\cfrac{1}{z-t}\right|\le \cfrac{1}{d(K,supp(\mu))}$ with $d(K,supp(\mu))$ the distance between $K$ and $supp(\mu)$. $d(K,supp(\mu))$ is positive since $K$ is compact, $supp(\mu)$ is closed and $K\subset \C-supp(\mu)$. Hence, by the holomorphic parameter integral theorem $S_\mu$ is holomorphic on $\C-supp(\mu)$. }
\end{enumerate}
\end{proof}
We can explicitly compute the Stieljes transform of the semicircle distribution.
\begin{prop}
\label{annexe prop unique stieljes sqrt}
$S_\sigma$ is characterized by the fact that it is the unique holomorphic function on $\C-\R$ such that for all $z\in \C-\R, \,S(z)^2-zS(z)+1=0$ and which converges to 0 when $|\Im(z)|$ converges to $+\infty$. So we get $$\forall z\in \mathbb H,\, S_\sigma(z)=\cfrac{z-\sqrt{z^2-4}}{2}$$ where $\sqrt{z'}$ is the determination of the square root on $\C-\R^+$ such that $\sqrt{-1}=i$.
\end{prop}

\begin{proof}
Since $\sigma$ is supported on the compact $[-2,2]$ and using its moments given in Proposition \ref{annexe prop moment} we have: $$\forall|z|>2,\, S_\sigma(z)=\cfrac{1}{z}\displaystyle\sum_{k=0}^{\infty}\cfrac{m_k}{z^k}=\cfrac{1}{z}\,C\left(\cfrac{1}{z^2}\right)$$ where $C$ is the generative function of the Catalan numbers defined in \ref{annexe prop catalan}.
As explained in Remark \ref{annexe remark stiejes} we only have to prove the result on $\mathbb H$ and not $\C-\R$.
\newline
Since $ZC(Z)^2-C(Z)+1=0$ we have that for all $z\in\mathbb H$ such that $|z|>2$, $S_\sigma^2(z)-zS_\sigma(z)+1=0$. So there is two possibilities: for all $z\in\mathbb H$ with $|z|>2$, $S_\sigma(z)=\cfrac{z-\sqrt{z^2-4}}{2}$ or for all $z\in\mathbb H$, $|z|>2$, $S_\sigma(z)=\cfrac{z+\sqrt{z^2-4}}{2}$ where $\sqrt{z'}$ is the square root on $\C-\R^+$ such that $\sqrt{-1}=i$. 
\newline
By holomorphy and since $\mathbb H$ is connected, we deduce that for all $z\in\mathbb H$, $S_\sigma(z)=\cfrac{z-\sqrt{z^2-4}}{2}$ or for all $z\in\mathbb H$, $S_\sigma(z)=\cfrac{z+\sqrt{z^2-4}}{2}$. 
\newline
However since for all $|z|>2$, $|S_\sigma(z)|\le\cfrac{1}{|z|-2}$, $S_\sigma$ converges to 0 when $|z|$ converges to $+\infty$. So the unique possible solution is given by $S_\sigma(z)=\cfrac{z-\sqrt{z^2-4}}{2}$ for all $z\in\mathbb H$.
\end{proof}
In the end of this section, we briefly explain how we usually use the Stieljes transform to prove the mean Wigner theorem as we did in the last section with the method of moments.
\newline
\tab Let us recall the notion of convergence for holomorphic functions and the Montel theorem which can be viewed as an Ascoli's theorem for holomorphic function.

\begin{definition}
Let $U$ be an open set of $\C$, we write $\mathcal H(U)$ the set of holomorphic functions on $U$ and denote $(K_p)_{p\in\N}$ a sequence of compacts subets of $U$ such that for all $p$, $K_p\subset  \mathring{K_{p+1}}$ and $\dis\bigcup_{p\in\N}K_p=U$ (we shall call a such sequence of compact sets a covering of $U$). 
\newline
\newline
Let $f,g\in \mathcal H(U)$, we introduce the distance $$d_{\mathcal H(U)}(f,g):=\dis\sum_{p=0}^{+\infty}\cfrac{\min(1,||f-g||_{\infty,K_p})}{2^{p+1}},$$ where $||f||_{\infty,K_p}$ is the infinite norm of $f$ viewed as function defined on the compact set $K_p$.
\newline
The topology defined by this distance $d_{\mathcal H(U)}$ does not depend on the covering $(K_p)_{p\in\N}$. A sequence $(f_n)_{n\in\N}$ in $\mathcal H(U)$ converges to $f\in \mathcal H(U)$ for this distance if and only if for all compact sets $K$ included in $U$, $||f_n-f||_{\infty,K}\convf 0$.
\newline
Moreover $(\mathcal H(U),d_{\mathcal H(U)})$ is a complete metric space. 
\end{definition}
\begin{Remark}
Given an open set $U\subset \C$ a covering of $U$ always exists. Indeed we can consider the sequence of compacts $K_p:=\{z\in U\,|\, |z|\le p\}\cap\{z\in U\,|\,d(z,\partial U)>1/p\}$.
\end{Remark}
We now state the main result about relative compact subsets of $\mathcal H(U)$ for this topology.
\begin{theorem}[Montel]Let $U$ be an open subset of $\C$, and $A\subset \mathcal H(U)$. $A$ is relatively compact in $(\mathcal H(U),d_{\mathcal H(U)})$if and only if for all $K$ compact included in $U$, $\{f_{|K},\, f\in A\}$ is bounded for $||.||_{\infty,K}$.
\end{theorem}
This result can be surprising at first sight when we compare it to the Ascoli theorem because there is no counterpart to the hypothesis of uniform uniform continuity for elements of $A\subset \mathcal H(U)$. Indeed, thanks to the Cauchy formula the uniform boundedness assumption of the elements of $A\subset \mathcal H(U)$ in the Montel theorem implies a uniform uniform continuity for elements of $A$.
\newline
\tab As an application of the Montel theorem we have the following result that can be used to prove the mean Wigner theorem.
\begin{corollary}
\label{annexe:coro}
Let $(\mu_n)_{n\in\N}$ be a sequence of real probability measures. If for all $z\in\C-\R$, $S_{\mu_n}(z)^2-zS_{\mu_n}(z)+1\convf 0$ then $(\mu_n)_{n\in\N}$ converges weakly to $\sigma$.
\end{corollary}

\begin{proof}
Let $K$ be a compact included in $\C-\R$, since $z\mapsto \Im(z)$ is continuous on $K$ it has a minimum on $K$, which will be denoted $M_K$. For all compact $K$ included in $\C-\R$, we have that  for all $n$, $||S_{\mu_n}||_{\infty,K}\le 1/M_K$ by Proposition \ref{annexe prop property stieljes}. By the Montel theorem, $(S_{\mu_n})_{n\in\N}$ has an accumulation point for the distance $d_{\mathcal H(\C-\R)}$. 
\newline
Let $S$ be a accumulation point of $(S_{\mu_n})_{n\in\N}$. Taking the limit, we get that for all $z\in\C-\R$, $S(z)^2-zS(z)+1=0$. We first notice that $\overline{S(z)}=S(\overline z)$ by Remark \ref{annexe remark stiejes} and passing to the limit. So we only have to identify $S$ on $\mathbb H$. As in Proposition \ref{annexe prop unique stieljes sqrt} we have two possibilities: either for all $z\in\mathbb H$, $S(z)=\cfrac{z-\sqrt{z^2-4}}{2}$ or for all $z\in\mathbb H$, $S(z)=\cfrac{z+\sqrt{z^2-4}}{2}$. However, since for all $n\in\N$, for all $t\in\R^+$, $|S_{\mu_n}(it)|\le \cfrac{1}{t}$ by Proposition \ref{annexe prop moment}, we get that $|S(it)|\le \cfrac{1}{t}$. So we necessarily have that $z\in\mathbb H$, $S(z)=\cfrac{z-\sqrt{z^2-4}}{2}=S_\sigma(z)$.
\newline
So $(S_{\mu_n})_{n\in\N}$ is a relatively compact sequence of $(\C-\R,d_{\mathcal H(\C-\R)})$ with a unique accumulation point $S_\sigma$. So $(S_{\mu_n})_{n\in\N}$ converges to $S_\sigma$.
\end{proof}
In particular to prove the Wigner theorem in mean, with the notations of Theorem \ref{annexemeanwigner}, we classically first prove that for all $z\in\C-\R$, $\E[S_{\mu_N}^2(z)]-z\E[S_{\mu_N}(z)]+1=0$. Then using some concentration inequalities as those obtained in Section \ref{section:concentrationineq}, we prove that for all $z\in\C-\R$, $\E[S_{\mu_N}^2(z)]-\E[S_{\mu_N}(z)]^2\underset{N\to+\infty}{\longrightarrow}0$.
So we deduce that for all $z\in\C-\R$, $\E[S_{\mu_N}(z)]^2-z\E[S_{\mu_N}(z)]+1\underset{N\to+\infty}{\longrightarrow}0$. Thanks to Corollary $\ref{annexe:coro}$ we obtain that for all $z\in\C-\R$, $\E[S_{\mu_N}(z)]=S_{\E[\mu_N]}(z)\underset{N\to+\infty}{\longrightarrow}S_\sigma(z)$. By Proposition \ref{annexe prop property stieljes}, $(\E[\mu_N])_{N\in\N}$ converges weakly towards $\sigma$. See Section 2.4 of \cite{Alicelivre} for a complete proof.

\subsection*{Acknowledgments} I acknowledge Charles Bertucci for its encouragement for the writing of these notes and Edouard Maurel-Segala for its course on random matrices that was a motivation for working on this topic. Finally, I also acknowledge the seminar MEGA for the hospitality and for the diversity of talks that made me discover the richness of this topic.

\newpage
\bibliographystyle{plain}
\bibliography{referencematrice}

\end{document}